\documentclass[12pt]{amsart}
\usepackage{amssymb,epsfig,amsmath,latexsym,amsthm}
\usepackage{Bib-formating}
\usepackage{graphicx}
\usepackage{subcaption}
\usepackage{tikz-cd}
\usepackage{pinlabel}
\usepackage{mathdots}
\theoremstyle{plain}
\newtheorem{theorem}{Theorem}[section]
\newtheorem{lemma}[theorem]{Lemma}

\newtheorem{proposition}[theorem]{Proposition}
\newtheorem{corollary}[theorem]{Corollary}

\theoremstyle{definition}
\newtheorem{definition}[theorem]{Definition}
\newtheorem{definition/construction}[theorem]{Definition/Construction}
\newtheorem{construction}[theorem]{Construction}
\newtheorem{remark}[theorem]{Remark}

\newtheorem{acknowledgements}[theorem]{Acknowledgements}

\newtheorem{remarks}[theorem]{Remarks}

\newtheorem{notation}[theorem]{Notation}
\newtheorem{convention}[theorem]{Convention}

\newtheorem{examples}[theorem]{Examples}

\DeclareMathOperator{\Diff}{Diff}

\DeclareMathOperator{\Emb}{Emb}

\DeclareMathOperator{\id}{id}

\DeclareMathOperator{\tr}{tr}

\DeclareMathOperator{\inte}{int}

\DeclareMathOperator{\std}{std}

\DeclareMathOperator{\IA}{IA}
\DeclareMathOperator{\EA}{EA}

\DeclareMathOperator{\REA}{R-EA}
\DeclareMathOperator{\GEA}{G-EA}
\DeclareMathOperator{\bfI}{\textbf{I}}
\DeclareMathOperator{\I}{I}
\DeclareMathOperator{\LB}{LB}

\newcommand{\BN}{\mathbb N}

\newcommand{\BR}{\mathbb R}

\newcommand{\BZ}{\mathbb Z}

\newcommand{\mRs}{\mR^{\std}}
\newcommand{\mGs}{\mG^{\std}}
\newcommand{\Rs}{R^{\std}}
\newcommand{\Gs}{G^{\std}}

\newcommand{\mC}{\mathcal{C}}
\newcommand{\mD}{\mathcal{D}}
\newcommand{\mE}{\mathcal{E}}
\newcommand{\mF}{\mathcal{F}}
\newcommand{\mG}{\mathcal{G}}

\newcommand{\mK}{\mathcal{K}}
\newcommand{\mL}{\mathcal{L}}

\newcommand{\mR}{\mathcal{R}}

\newcommand{\mU}{\mathcal{U}}

\newcommand{\mW}{\mathcal{W}}

\newcommand{\stwostwo}{S^2\times S^2}

\newcommand{\mfw}{(\mF,\mW)}

\newcommand{\mfii}{\mF_{ii}}
\newcommand{\mwii}{\mW_{ii}}

\newcommand{\bC}{\textbf{C}}
\newcommand{\bI}{\textbf{I}}

\begin{document}

\title{Pseudo-Isotopy and Diffeomorphisms of the 4-Sphere I:  Loops of Spheres}

\author{David Gabai}
\address{Department of Mathematics\\Princeton
University\\Princeton, NJ 08544}
\email{gabai@princeton.edu}

\author{David Gay}
\address{Department of Mathematics\\University of Georgia\\Athens, GA 30602}
\email{dgay@uga.edu}

\author{Daniel Hartman}
\address{Department of Mathematics\\Duke University\\Durham, NC 27708}
\email{daniel.hhartman@gmail.com}

\thanks{
\newline\noindent\emph{Primary class:} 57R35
\newline\noindent\emph{secondary class:} 57R52, 57R50, 57N37
\newline\noindent\emph{keywords:} Smale, 4-sphere, diffeomorphism, pseudo-isotopy}

\begin{abstract}  We introduce new methods in pseudo-isotopy  and embedding space theory.  As an application we introduce an invariant $\mD$ that detects a non trivial loop of embedded 2-spheres in $\stwostwo$. 
In the sequel \cite{GGH} we will use these techniques to expand upon the applicability of the invariant and prove that $\Diff_+(S^4)$ has an exotic element.  \end{abstract}

\maketitle

\setcounter{section}{-1}

\section{Introduction}

August 2025 marks the 100'th anniversary of Emil Artin's introduction of spun knots and the initiation of 2-knot theory \cite{Ar}.  Since then, the theory of 2-knots in $S^4$ and other 4-manifolds has experienced explosive growth. Using the notation $\Emb(A,B)$ for the space of embeddings of $A$ into $B$, this $2$-knot theory is about $\pi_0(\Emb(S^2,X^4))$.  On the other hand, only recently has the first nontrivial element of $\pi_1(\Emb(S^2,S^4))$ been discovered \cite[Theorem 10.1 (4)]{BG1}. More recently \cite{FGHK} exhibited nontrivial loops of spheres in $S^2\times B^2$.   All of these examples rely on the study of 1-manifolds in 4-manifolds, to which embedding calculus is amenable \cite{Weiss,GW}.  In particular, crucial  to \cite{BG1} is the study of $\pi_2(\Emb(I, S^1\times B^3))$ and \cite{FGHK} relies on the analysis of $\pi_1(\Emb(S^1, S^1\times S^3),S^1 \times x_0)$ given in \cite{BG1}.  Theorem 10.1 (1) of \cite{BG1} computes $\pi_1(\Emb(B^2, S^2\times B^2), x_0 \times B^2)$,  giving loops of $S^2$'s in $S^2\times S^2$ via extension.  That proof derives from a Cerf fibration sequence one term of which is $\pi_2(\Emb(I, B^4))$, again a reduction to arcs in 4-manifolds.  

Note that embedding space theory has also undergone intense development in the last 30 years, but that theory is based on codimension $\ge 3$. Arguably, the second Hatcher-Wagoner obstruction  \cite{HW} from the early 1970's contributes to the codimension-2 theory for dimension $\ge 6$ for non-simply connected manifolds, and these techniques were partially extended to dimension four by Singh and Igusa \cite{Ig3, Singh}. The results in dimension four rely on the manifolds having nontrivial $\pi_1$ and $\pi_2$. Our main contribution is a first step in codimension-2 for simply connected 4-manifolds.

Our main result is the introduction of an invariant $\mD$ which detects loops of spheres in $S^2\times S^2$ that cannot be homotoped into loops arising from $\pi_1(\Emb(B^2, S^2\times B^2), x_0\times B^2)$, and more generally a similar statement for $k\ge 1$ spheres embedded in a connected sum of $k$ copies of $S^2 \times S^2$.  Here is the precise statement.

\begin{definition} Let $(g_0, r_0)\in \stwostwo$, let $\Gs:=S^2\times r_0$ and let $\Rs:=g_0\times S^2$. In $\#^k S^2 \times S^2$, let  $R_i^{\textrm{std}}$ (resp. $G_j^{\textrm{std}}$) denote the i'th $g_0\times S^2$ (resp. j'th $S^2\times r_0$) and $\mRs$ (resp. $\mGs)$ denote $\cup R_i^{\textrm{std}}$  (resp. $\cup G_j^{\textrm{std}})$  These are the \emph{standard red and green spheres}. 
\end{definition}

We will denote arbitrary elements of $\Emb(\sqcup^k S^2, \#^k S^2\times S^2)$ by $\mR$ and denote the $i$'th component of $\mR$ by $R_i$, thinking of these as {\em red spheres}. 

\begin{definition}
The space of {\em light-bulb embeddings}, denoted $\LB$, is the subset of $\Emb(\sqcup^k S^2, \#^k S^2\times S^2)$ consisting of embeddings $\mR$ which intersect $\mGs$ transversely, with the property that $|R_i \cap G^{std}_j| = \delta_{ij}$ for all $i,j$.   
\end{definition}

The point of this definition is that, when $k=1$, $\pi_1(\LB,\mRs)$ is precisely the space of loops arising from $\pi_1(\Emb(B^2, S^2\times B^2), x_0\times B^2)$ as in \cite{BG1}, and is so named because these are the embeddings to which the light-bulb theorem applies \cite{Ga1}.

\begin{theorem} \label{T:mainthm}
    For each $k$ there exists a surjective homomorphism 
    \[ \mD : \pi_1(\Emb(\sqcup^k S^2, \#^k (S^2 \times S^2)), \mRs) \to \BZ_2 \]
    with $\pi_1(\LB,\mRs)$ contained in the kernel of $\mD$.     Furthermore, there exists $\beta\in \pi_1(\Emb(S^2, S^2 \times S^2),\mRs)$ such that for all $k$, after the natural inclusion into $\pi_1(\Emb(\sqcup^k S^2, \#^k (S^2 \times S^2)), \mRs)$,   $\mD(\beta)=1$.    
  \end{theorem}
  
  Consider the case $k=1$, in which case we use $I$ instead of $\mD$.  Represent $[\alpha]$ by a generic isotopy $\alpha: S^2 \times [0,1]\to \stwostwo$ based at $\Rs$.  Here we denote $\alpha(S^2 \times t)$ by $R_t$.  As $t$ increases from 0 to 1, $R_t$ undergoes finger moves and Whitney moves with respect to $\Gs$, e.g see \cite[\S 1]{FQ}.  As in \cite[\S 4]{Qu} we can arrange that the finger moves occur before $t=1/2$ and the Whitney moves after $t=1/2$ and at $t=1/2$ the union of the finger and Whitney discs intersect both $\Gs$ and $R_{1/2} $ in embedded arcs.  By suitably adding up the intersections between the interiors of the Whitney and finger discs and reducing mod-2 we obtain our invariant, defined precisely in Definition \ref{invariant}.   Here is our fundamental example $\beta$.  
  
  \begin{examples}  \label{example} i) Figure 1, on the left, shows $R_{1/2} \cup \Gs$ together with the \emph{standard} finger and Whitney discs $f $ and $w'$.  Seen is the projection of $R_{1/2}$ to the present i.e. $S^2\times S^1\times 0$, with the dark lines indicating the intersection with the present. The middle of Figure 1 shows a 2-sphere $J$ linking the finger which intersects $f$ once.  Figure 1, on the right, shows a new Whitney disc $w$ obtained by tubing $w'$ to $J$.  I.e. we remove an open disc from each and attach a annulus. The annulus is pushed slightly in the future, from where it starts in $w'$ until after it crosses $G$. This is the example from \S 7.2 of  \cite{GGHKP}.  The finger and Whitney discs $f$ and $w$ give rise to a loop $\beta \in \Emb(S^2,\stwostwo)$ based at $\Rs$ as follows.   Starting at $R_{1/2}$, going backwards in $t$, construct $R_t, t\in [1/4,1/2]$ by first doing a Whitney move using $f$ to obtain $R_{1/4}$, geometrically dual to $\Gs$.   Staying geometrically dual, extend to $R_t,  t\in[0,1/4]$ so that $R_0=\Rs$.  In the other direction, use the Whitney disc $w$ to obtain the geometrically dual $R_{3/4}$,  then isotope staying geometrically dual, to obtain $R_1=\Rs$.  This can either be done explicitly or by appealing to  \cite{Ga1}.    Note that $|\inte(f)\cap\inte(w)|=1$.  Here $I(\beta)=1$.  
\vskip 6pt
ii) The loop $\alpha'$ corresponding to $f$ and $w'$   is homotopic to the trivial one and has $I(\alpha')=0$. \end{examples} 

\begin{figure}
    \centering
    \includegraphics[width=1.0\linewidth]{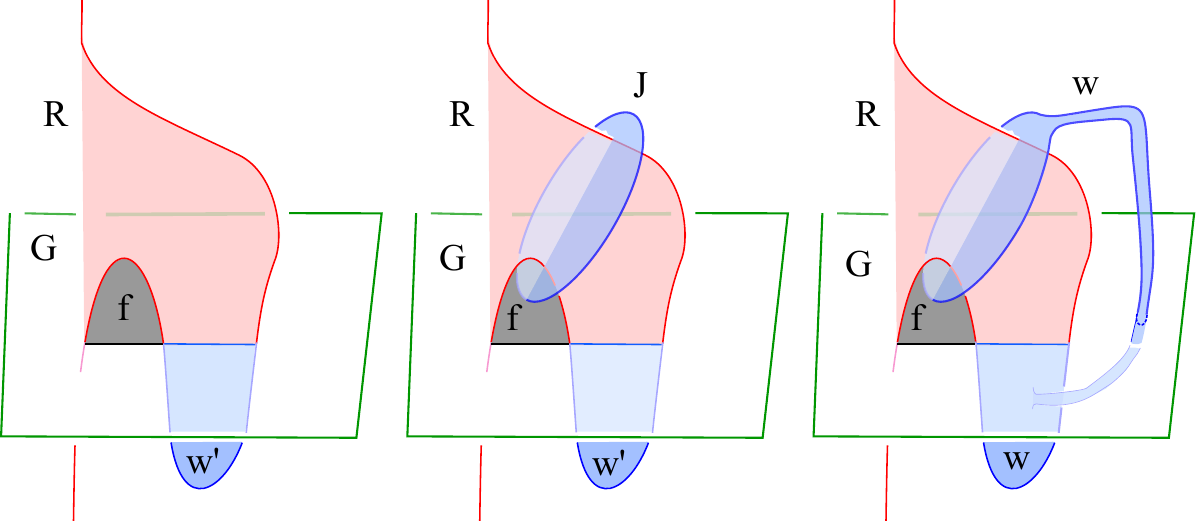}
    \caption{The Key Example}
    \label{fig:example}
\end{figure}

\begin{remark}\label{drc} The discs $w$ and $w'$ satisfy Quinn's Disc Replacement Criterion \cite[Lemma 4.5]{Qu}, discussed as problematic in \cite{GGHKP}.  Paper \cite{GGH} will show that these examples  give an explicit counterexample to \cite[Lemma 4.5]{Qu}.\end{remark}

 We now  precisely define the invariant, give an outline of the proof which uses purely codimension-2 methods, and then discuss the relation with pseudo-isotopy theory.   To start with as Quinn

\begin{definition}\label{D: fw system} Given $[\alpha] \in \pi_1(\Emb(\sqcup_{i=1}^kS^2, \#^k\stwostwo), \mRs)$, we say that the representative $\alpha$  is a \emph{finger first loop} if it intersects $\mGs$ first in finger moves say at $t=1/4$ and then in Whitney moves say when $t=3/4$.  We call such a loop a \emph{finger-first representative}.\end{definition}

\begin{remark} Any $[\alpha]$ has a finger first representative, by an argument due to Quinn \cite{Qu} that is in the same spirit that Morse functions can be modified so that critical points appear in monotonically increasing index and further that all indices of the same index can appear at once.  \end{remark}

 \begin{definition} For a finger first representative let $\mR_t = \alpha_t(\mR)$,   $R_1, \cdots, R_k$ be the components of $\mR_{1/2}$ and $G_1,\cdots, G_k$ be the components of $\mGs$. Let $\mF$ and $\mW$ denote the finger and Whitney discs corresponding to $\mR_{1/2}$ that cancel the excess points of $\mR_{1/2}\cap \mGs$.  For $1\le i,j\le k$ let $\mF_{ij}=(f^{ij}_1, \cdots, f^{ij}_{n_{ij}})$ and $\mW_{ij}=(w^{ij}_1, \cdots, w^{ij}_{n_{ij}})$ respectively denote the elements of $\mF$ and $\mG$ which cancel points of $R_i\cap G_j$. The tuple $(\mR_{1/2}, \mG^{\textrm{std}},\mF, \mW)$ is called a \emph{finger/Whitney system}. We will sometimes denote a finger/Whitney system by $(\mF, \mW)$, when $\mR_{1/2}$ and $\mG$ are understood.
\end{definition}

\begin{definition}
We say that a finger/Whitney system for a finger first loop $\alpha$  satisfies the  \emph{immersed arc condition} if each of $(\mF_{ii}\cup\mW_{ii})\cap R_i$ and $(\mF_{ii}\cup \mW_{ii}) \cap G_i$ is an immersed arc, and we write $(\mF,\mW) \in \IA$.  If these arcs are embedded we say that the finger/Whitney system satisfies the \emph{embedded arc condition}, $\alpha\in\EA$.   Orient the arcs from the end which contains an element of $\mF$ and rename the elements of $\mF_{ii}$ and $\mW_{ii}$ so that $f^{ii}_1, w^{ii}_1, f^{ii}_2, \cdots, w^{ii}_{n_{ii}}$ appear in succession. The induced orderings on $\mF_{ii}$ and $\mW_{ii}$ are called \emph{IA-orderings}. \end{definition}

In 1986 Frank Quinn \cite[\S4]{Qu} introduced the embedded arc condition and proved that any $[\alpha]$ has a finger-first loop representative $\alpha$ with a finger/Whitney system $(\mF, \mW)\in EA$.  He gave the proof for $k=1$, though the proof works in general.

\begin{definition} \label{invariant} We say that $\mfw$ is in \emph{full embedded arc position} or $\mfw\in$ FEA if $\mfw\in$ EA and for all $i\neq j$,  $(\mF_{ij}\cup\mW_{ij})\cap R_i$ and $(\mF_{ij}\cup\mW_{ij})\cap G_j$ is a union of $n_{ij}$ pairwise disjoint embedded circles bounding pairwise disjoint  discs that are disjoint from $(\mF\setminus \mF_{ij})\cup(\mW\setminus \mW_{ij})$. In the following definition  intersections among finger and Whitney discs refer to intersections of their interiors. Also for  $1\le i\le k$ each  $\mF_{ii}, \mW_{ii}$  has the IA ordering. If $\mfw\in FEA$, then define
 
 \begin{align} 
{\I_j}([\alpha])\in \BZ_2:&=\sum_{p\le q} | f^{jj}_p\cap w^{jj}_q |\  \textrm{mod}\ 2,\nonumber
\\
\I([\alpha])\in \BZ_2 :&=\sum_{j=1}^k \I_j([\alpha]) \  \textrm{mod}\ 2,\quad \nonumber  
\\
{C_{ij}}([\alpha])\in \BZ_2:&=| \mF_{ij}\cap \mW_{ji}|\  \textrm{mod}\ 2,\quad \nonumber 
\\
{C}([\alpha])\in \BZ_2:&=\sum_{i< j} C_{ij}\  \textrm{mod}\ 2\ \textrm{and} \quad \nonumber
\\
\mD([\alpha])\in \BZ_2:&=\I([\alpha])+C([\alpha])\ \textrm{mod}\ 2.\quad \nonumber
 \end{align}
 \end{definition}  

Much of this paper is about showing $\mD([\alpha])$ is well defined.  There are manifold ways to bring an $\alpha$ to full embedded arc position and we need to show $\mD$ is independent of how this is done.   We do this by slowly building up the set of representatives $\alpha$ for which we can define $\mD(\alpha)$ independently of previous choices.    In \S2 and \S3 we start with $k=1$.  A technical point explained in \S1 is that our finger and Whitney discs need to be \emph{untwisted}, which can be readily done, to avoid unwanted intersections between the finger and Whitney discs. We show that if $\mfw\in$ IA, then we can transform it using \emph{disc slides} to one $\in$ EA.  We show that the resulting value of $\I(\alpha)$ is independent of the choice of disc slides.  In general the union of finger and Whitney discs intersect $\Gs$ in a union of immersed circles plus possibly an arc.  We show in \S 3 how to transform such a loop to one satisfying the immersed arc condition by a technique called \emph{switching}.  This requires an a priori  ordering of the Whitney discs.  Again we show that $\I(\alpha)$ does not depend on the choice of switching.  In \S4 we extend these results to $k\ge 1$  to produce a well defined $\I(\alpha)$, provided that each $\mW_{ii}$ is ordered.  In \S5 we introduce the \emph{cross term invariant} $C(\alpha)=c(\alpha)+CU(\alpha)$, where $c(\alpha)$ and $CU(\alpha)$ are defined for $\alpha$ not necessarily in FEA.  We also introduce the \emph{special sum square move} which is a special case of Quinn's sum square move \cite{Qu}.  In \S7 we introduce the \emph{clasping} move.  In \S 8 we use special cases of clasping, disc sliding, switching and the special sum square move to show $\mD =\I+C$ is independent of the ordering of $\mW$.  This completes the proof that $\mD(F,W)$ is well defined and uses technical lemmas proved in \S6.  In \S9 we show $\mD(\mF,\mW)=\mD(\mW,\mF)$.  

It remains to show that $\mD(\alpha)$ is well defined on the homotopy class $[\alpha]$.  
Implicit until now is the assumption that the finger moves occur before the Whitney moves. In \S 11 we study how the data changes under a generic homotopy of $\alpha$.     This section also identifies the  indeterminacy in transforming a non-finger-first $\alpha$ to a finger-first one.  In \S12 we show, using \S2 - \S10, that all such operations give a well defined invariant, which is in fact a group homomorphism, thereby completing the proof of Theorem~\ref{T:mainthm}.   

We view the contents of \S11 to be of independent interest. Here we develop the tools for studying generic loops of embeddings of surfaces and based homotopies between them \emph{relative to a fixed embedded surface in a general 4-manifold}. In particular, we give a complete account of how the finger/Whitney moves between a family of embeddings $\mR_{s,t}$ and $\mG$ propagate in general 2-parameter families of embeddings, and use this to prove an ``ordering'' theorem; that is to say two paths that have all the finger moves happening before the Whitney moves and are homotopic are homotopic by a homotopy where each path has all the finger moves happening before all the Whitney moves \emph{except at controlled instances} (see Theorem \ref{T: 2-par ordering} for details). From Theorem \ref{T: 2-par ordering}, we are able study how the finger/Whitney systems (Definition \ref{D: fw system}) for two relatively homotopic paths of embeddings are related. We show in Theorem \ref{T: fw system moves} that any two finger/Whitney systems for homotopic paths of embeddings differ by the following five transformations, up to isotopy:
\begin{enumerate}
    \item disc slides (Definition \ref{hsf})
    \item sphere slides (Definition \ref{D: sphereslide}),
    \item birth/death moves (Definition \ref{D: b/d move}),
    \item $x^3$-moves (Definition \ref{D: x^3}) and,
    \item saddle moves (Definition \ref{D: saddle}).
\end{enumerate}

Theorem \ref{T: fw system moves} is a foundational result for identifying potential smooth 4-dimensional pseudo-isotopy invariants, as we now explain. Let $Y$ be any closed smooth 4-manifold and $X = Y\#^k S^2\times S^2$, and consider the set $\pi_1(\Emb(\sqcup^k S^2, X),\LB))$. Below we give more details on the following fact in the case where $Y = S^4$, but in short, any representative $\alpha$ of a class $[\alpha] \in \pi_1(\Emb(\sqcup^k S^2, X),\LB))$ can be used to construct a pseudo-isotopy of $Y$ from the identity; that is, a diffeomorphism $g:Y\times I\rightarrow Y\times I$ such that $g|_{Y\times 0  }=\id_Y$ and $g|_{Y\times 1}$ is some element in $\Diff^+(Y)$. If $[\alpha]$ contains a representative $\alpha$ such that $\mR_t\in \LB$ for all time, then the pseudo-isotopy determined by \emph{any} representative of $[\alpha]$ is always isotopic to an isotopy. Since Quinn's result about the ordering of finger and Whitney moves holds, one can study these relative classes by studying their associated finger/Whitney systems. A natural corollary of Theorem \ref{T: fw system moves} is that if a relative class $[\alpha]$ is trivial in $\pi_1(\Emb(\sqcup^k S^2, X),\LB))$, then that system can be modified by our five moves to the ``trivial'' finger/Whitney system. Thus, any potential invariant of a finger/Whitney system that is unchanged by the five listed moves is an invariant of the relative homotopy class of embedding that the system determines and is a good candidate for a smooth 4-dimensional pseudo-isotopy invariant.

We now discuss the connection between this paper and \cite{GGH}, the main result of which is:

\begin{theorem}  \label{main} There exists an exotic element of $\Diff_+(S^4)$, i.e. a diffeomorphism of $S^4$ topologically isotopic but not smoothly isotopic to $\id$.  \end{theorem}

Our approach to Theorem \ref{main} is via pseudo-isotopy theory.  To start with since $S^5$ has a unique smooth structure \cite{Sm4}, \cite{KM},  any orientation preserving diffeomorphism $\varphi:S^4\to S^4$ is pseudo-isotopic to $\id$, see \cite[P. 157]{Br}.  This means that there is a diffeomorphism $g:S^4\times I \to S^4\times I$ such that $g|S^4\times 0=\id$ and $g|S^4\times 1=\varphi$. We denote the space of such diffeomorphisms by $\Diff_{\partial_0}(S^4\times I)$.

Every pseudo-isotopy induces a nonsingular map $q_1:S^4\times I \to I$ by post-composing with the standard projection map $q_0$ from $S^4\times [0,1]$ to $[0,1]$.  Since the space of smooth maps to $[0,1]$ is contractible, there is a 1-parameter family of smooth maps $q_t$ from the trivial projection $q_0$ to $q_1$. Standard singularity theory implies that a generic path $q_t$ is a path of generalized Morse functions, for which one can take a 1-parameter family of gradient like vector fields (g.l.v.f.) $V_t$ for the functions $q_t$ as well. Since the 4-sphere is a simply connected closed 4-manifold, the first pseudo-isotopy Hatcher-Wagoner obstruction vanishes and hence by  \cite[Proposition 3 p. 214]{HW}, the pair $(q_t,V_t)$ can be chosen to be a ``nested eye'' family. By this we mean that the family $q_t$ starts out with no critical points, has $k$ standard births of critical point pairs of index 2 and 3. Later, those same pairs of critical points are canceled, returning $q_t$ back to a function with no critical points. The g.l.v.f. for the births and deaths are independent (see \cite[Ch 1 \S6]{HW}). Furthermore, there are no flow lines between critical points of the same index, i.e. there are no handle slides. For more details see \cite{HW}, \cite[\S4]{Qu} or \cite[\S2]{GGHKP}.  

We can further assume that all births (resp. deaths) occur in $t\in (1/8, 1/4)$ (resp. $t\in (3/4,7/8))$ and for $t\in [1/4, 3/4]$ all the index-2 (resp. index-3) critical points occur in $y\in (1/8, 1/4)$ (resp. $y\in (3/4,7/8))$ where $y$ denotes the I coordinate of $S^4\times I$.  For $t\in [1/4,3/4]$ one sees the ascending spheres of the index-2 critical points and the descending spheres of the index-3 critical points in $q_t^{-1}(1/2)=S^4\#^k S^2\times S^2$, where $k$ is the number of eyes.  Except for finitely many t's corresponding to finger and Whitney moves these spheres are transverse.  As in \S 4 \cite{Qu} we can assume all the finger moves occur before $t=1/2$ and all the Whitney moves after $t=1/2$.      Thus, essentially all the data for the 1-parameter family is contained in the \emph{middle middle} level i.e. when $y=t=1/2$. As shown in Example \ref{example} from the  finger and Whitney discs  we obtain a loop in $\Emb(\sqcup^k S^2, \#^k S^2\times S^2; \Rs)$ where $\Gs$ is the union of ascending spheres.  For more details see \S 2  \cite{Ga3} or \cite{Gay}.   

The upshot is that a pseudo-isotopy induces a loop in $ \Emb(\sqcup^k S^2, \#^k S^2\times S^2;\Rs)$ and an element of $\pi_1( \Emb(\sqcup^k S^2, \#^k S^2\times S^2),\Rs)$ induces a pseudo-isotopy up to isotopy. First, if the two functions $q_0$ and $q_1$ can be connected by a path \emph{of nonsingular functions}, then the pseudo-isotopy can be isotoped to be an isotopy. Second, the criteria for performing a standard cancellation of all the critical points for a nested eye family is precisely if the associated loop of spheres is homotopic to a loop in $\LB$. Therefore knowing a loop cannot be deformed into $\LB$ is necessary for both the pseudo-isotopy and the diffeomorphism to be nontrivial. However, two Hatcher-Wagoner 1-parameter families associated to a pseudo-isotopy potentially need not correspond to homotopic loops in $ \Emb(\sqcup^k S^2, \#^k S^2\times S^2;\Rs)$, even allowing for a stabilization operation on $\Emb(\sqcup^k S^2, \#^k S^2\times S^2;\Rs)$.  On the other hand, using the contractibility of smooth maps $S^4\times I \to I$, they are connected by 2-parameter families of generalized Morse functions, possibly with critical points of all indices, swallowtail singularities, flow lines between critical points of the same index and flow lines from critical points of index $i$ down to critical points of index $i+1$   The technical content of \cite{GGH} is to extend the definition of $\mD([\alpha])$ and show invariance under such 2-parameter deformations by reducing to the invariance theorems proven in this paper. This gives that $\mD$ is a well defined invariant of $\pi_0 \Diff^+_{\partial_0}(S^4\times I)$. Consider the fibration sequence,
    \begin{equation*}
        \begin{tikzcd}
            \Diff^+_\partial (S^4\times I) \arrow[r, "i"]& \Diff^+_{\partial _0}(S^4\times I) \arrow[r,"\rho_1"]& \Diff^+(S^4)
        \end{tikzcd}
    \end{equation*}
where $\rho_1$ is the fibration given by restricting to $S^4\times \{1\}$ and $i$ is the natural inclusion. Now for $S^4$, $\pi_0\Diff^+_{\partial _0}(S^4\times I)/\pi_0\Diff_{\partial}(S^4\times I) \simeq \pi_0 \Diff^+(S^4)$ since every diffeomorphism is pseudo-isotopic to the identity. Furthermore it follows from Cerf's theorem that $\pi_0 \Diff^+_{\partial}(S^4\times I)\simeq \BZ_2$ \cite{Ce1}. The $\BZ_2$ is generated by the 5-dimensional ``Dehn twist'', which is realized by the trace of an essential loop in $\Diff(S^4)$, a.k.a. an isotopy. As such, if the pseudo-isotopy generated by the family of spheres is not isotopic to an isotopy, it is not in the image of $\Diff_\partial(S^4\times I)$, and therefore the diffeomorphism of $S^4$ given by restricting is not isotopic to the identity.

\begin{acknowledgements}  Our collaboration started in November 2022 during the Banff conference ``Topology in Dimension 4.5" where we ``\emph{proved}" that the pseudo-isotopy induced from the loop of 2-spheres in Example \ref{example} i) is trivial.  Subsequently, Slava Krushkal and Mark Powell informed us that our proof was probably wrong since it depended on Quinn's Disc Replacement Criterion, whose proof they found wanting.  We are also grateful to Krushkal and Powell for many helpful discussions on pseudo-isotopy over the years.  David Gabai thanks Ryan Budney for many conversations. David Gabai was partially supported by grants DMS-2003892 and DMS-2304841 of the National Science Foundation. Part of this research was conducted while he was a Member of the Institute for Advanced Study in fall 2025 and as a Clay Senior Scholar for the program \emph{Topological and Geometric Structures in Low Dimensions} in spring 2026. David Gay was supported by NSF Grant DMS-2005554, NSF Grant DMS-2342252 and Simons Foundation Travel Support for Mathematicians Grant MP-TSM-00002714. This work was also supported by the Centre International de Recontres Mathématiques (CIRM) in Marseille while David Gay was in residence from July to December 2025 through the CIRM Chaire Jean Morlet program. This program also supported two visits by the other co-authors to CIRM during that time.  Daniel Hartman thanks Peter Teich\-ner, Danica Kosanovi\'c, and Daniel Galvin for many conversations, and the Max Planck Institute for Mathematics for its hospitality and support. Daniel Hartman was partially supported by the Visiting Student Research Collaborator program at Princeton University and Max Planck Institute of Mathematics.     \end{acknowledgements}

\section{preliminaries}

In what follows we frequently picture  $S^2\times S^2$ as $S^2\times (S^1\times[-\infty,\infty]/\sim)$, where each of  $S^1\times \pm\infty$ are identified with points.  The first $S^2$ is often viewed as a compactified  $\BR^2$ and the $S^1$ as a compactified $\BR$.  So fixing $(g_0, r_0)\in S^2\times S^2$, a 3D slice often seen in this paper is $\BR^2\times (\BR\times 0)$ with $(g_0,r_0)$ corresponding to $(0,0)\times (0,0)$ and hence $\Gs$ corresponds to $\BR^2\times (0,0)$ with $\Rs$ intersecting this slice in the $z$-axis.

We give $\Rs$ and $\Gs$ the standard orientations.  I.e. that of $\Gs$ is induced from $\BR^2$ and $\Rs$ from $\BR\times \BR$.   We give $S^2\times S^2$ the product orientation.  Here $R$ will denote a surface isotopic to $\Rs$ and unless said otherwise $G$ will denote $\Gs$.  A set of Whitney discs for $R\cup G$ is \emph{complete} if they are pairwise disjoint and cancel all but one point of $R\cap G$.  

If $X\subset Y$, then $N(X)$ denotes a closed regular neighborhood.  It may have corners, e.g. $N(R\cup G)\subset \stwostwo$.  We use $|X|$ to denote the number of components of $X$.  

As explained in the introduction a representative $\alpha$ gives rise to a finger/Whitney system $\mfw$.  The goal of \S 1-10 is to transform $\mfw$ into a finger/Whitney system for which we can calculate the invariant independent of all choices.     In \S11-12 it will be shown that all finger/Whitney systems arising from $[\alpha]$ produce the same value.  Henceforth, functions such as $\I([\alpha])$ or $\mD([\alpha])$ will be denoted by $\I\mfw$ or $\mD\mfw$.

\begin{remark} Algebraic intersection number of properly mapped oriented surfaces with fixed disjoint boundaries into an oriented 4-manifold depends only on their relative homology class.  In our setting  the complete sets of Whitney discs  $\mF$ and $\mW$ have boundary on $G\cup R$, so  to take intersection number we need to view  them in  $E:=\stwostwo \setminus \inte(N(R\cup G))$.  How to do this near $\partial \mF\cup \partial \mW$ is a function of $N(\partial \mF)\cup N(\partial \mW)$, i.e. the \emph{germs} of the boundaries of  $\mF\cup \mW$.  It will be crucial to keep track of how the germ of a disc $D\subset \mF\cup \mW$ twists relative to other such discs.   Twisting of Whitney discs relative to other Whitney discs was studied in \cite{Ga3} while here we need to also keep track of twisting with respect to $G$ and $R$.  \end{remark}

\begin{definition}  Since $R$ and $G$ have trivial normal bundles we  identify $N(R):=R\times D^2$ and $N(G)=G\times D^2$.  Fix $[-1,1]$ subbundles of $N(R)$ and $N(G)$ where $[-1,1]$ corresponds to a diameter of $D^2$.  Call the $[0,1]$ (resp. $ [-1,0]$) subsubbundle the \emph{positive} (resp. \emph{negative}) $I$-bundles.  We can assume that $R$ (resp. $G$) intersects $N(G)$ (resp. $N(R)$) in $D^2$ fibers and these bundles align as in Figure \ref{fig:example}.  We say that $f$ (resp. $w$) is \emph{untwisted} if $f\cap N(R)$ (resp. $w\cap N(R))$ and $f\cap N(G)$ (resp. $w\cap N(G)$) are contained in the positive (resp. negative) $I$-bundles. Such discs are also respectively called \emph{finger germed} or \emph{Whitney germed}.  It will frequently be useful to isotope a finger germed $f\in \mF$ to be Whitney germed or a Whitney germed $w\in \mW$ to be finger germed.  Note that untwisted discs are determined near $R\cap G$. An  \emph{untwisted isotopy} is an isotopy of untwisted discs through untwisted discs.  Two untwisted finger/Whitney systems are \emph{equivalent} if the finger (resp. Whitney) discs  differ by disjointly supported untwisted isotopies.   By \emph{rotating} an untwisted disc near either $R$ or $G$ we obtain a new untwisted disc.   Such an isotopy is called a \emph{boundary twisting} or an \emph{$R$-rotation} or \emph{$R$-twist} or \emph{$G$-rotation} or \emph{$G$-twist} and an isotopy of an untwisted disc which is a concatenation of boundary twisting is called a \emph{reuntwisiting}. 

A  \emph{boundary bigon} is an embedded disc $D\subset R$ or $G$ with $\inte (D)\cap (\mF\cup \mW)=\emptyset$, $\partial D \subset \mF\cup \mW$, $|D\cap\mF|=D\cap \mW|=1$ and $|D\cap R\cap G|=1$.  See Figure 2 of Quinn \cite{Qu}. 
Unless said otherwise $G$ denotes $\Gs$ and $R$ denotes a surface isotopic to $\Rs$ transverse to $G$.  If $\mF, \mW$ are each complete sets of pairwise disjoint Whitney discs for $R$ and $G$, then as in Example \ref{example} they determine a loop $\alpha$ in $\Emb(S^2, S^2\times S^2; \Rs)$ well defined up to paths geometrically dual to $G$.  By this we mean starting at $R_{1/2}:=R$ we naturally obtain $R_{1/4}$ and $R_{3/4}$ respectively by doing  Whitney moves corresponding to $\mF$ and $\mW$.  Since $R_{1/4}$ and $R_{3/4}$ are geometrical dual to $G$ and homotopic to $\Rs$, they can be isotoped to $\Rs$ staying geometrically dual to $G$ \cite{Ga1}.  However, two such isotopies can differ by a loop which stays geometrically dual to $G$, i.e. a loop in $\LB$.  \end{definition}

\begin{lemma} \label{normal untwisting} If $\mfw$ is in immersed arc position, then we can isotope each $\mF_i$, $\mW_i$ and $R_i$ such that each disc in $\mF_i\cup \mW_i$ is untwisted and  there are no boundary bigons.  The $R_i$ are fixed setwise throughout the isotopy. \end{lemma}

\begin{proof}  For each $i$, first isotope  $w^i_{n_i}$ near $R_i$ to be untwisted along $R_i$, then isotope $w^i_{n_i}$ and $R_i$ near $G_i$ to also be untwisted along $G_i$.  Next isotope $f^i_{n_i}$ to be untwisted near $R_i$ and to have no boundary bigons with $w^i_{n_i}$ and then isotope $f^i_{n_i}$ and $R_i$ near $G_i$, but away from $w^i_{n_i}$ to be also untwisted along $G_i$. Inductively isotope $R_i$ and $w^i_{n_i-1}, f^i_{n_i-1}, w^i_{n_i-2}, \cdots, f^i_1$ so that each disc is untwisted and has no boundary bigons with the preceding disc.  Being local, the isotopies for the various $i$'s can be done simultaneously.\end{proof}

\begin{definition}  Call  $(\mF,\mW)$ \emph{normalized} if each finger disk and each Whitney disk is untwisted and there are no boundary bigons.\end{definition}

\begin{remarks}\label{normalized choice} i) There is choice in how to isotope a Whitney or finger disc to make it untwisted,  by $2\pi$ rotations near either $G$ or $R$. However for an $\alpha\in \IA$, once each $w^i_{n_i}$ is chosen, the no boundary bigon condition  imposes the choices for the remaining discs.  

ii) A normalized $(\mF,\mW)$ gives rise to an intersection number that is well defined as a function of $\mF$ and $\mW$ and independent of the choice of normalization, however Remark \ref{new ints} will explain that such an intersection number is unsuitable for computing $\I\mfw$.     

iii) If $(\mF, \mW)\in EA$, then it can be normalized staying within EA and without changing the various $|\inte(w_i)\cap \inte(f_j)|$.  In this case $\I(\mF, \mW)$ does not require $\mF$ and $\mW$ be untwisted.

iv) In this paper we will define operations and functions on untwisted finger/Whitney systems.  We leave it to the reader to check that these operations preserve untwisted isotopy classes of finger/Whitney systems and these functions are well defined on such classes.\end{remarks}

\section{Immersed arc position} \label{immersed arc}

In this section  we introduce a method of transforming an $(\mF,\mW)\in \IA$ to  $(\mF',\mW')\in \EA$ and then define $\I(\mF,\mW)=\I(\mF', \mW')$.  We do this by a sequence of \emph{finger disc slides} and \emph{Whitney disc slides} starting from a normalized $(\mF,\mW)$.   We then show that $\I(\mF,\mW)$ is independent of both the slide sequence and normalization.

\begin{definition}  \label{hsf}  Let $\mW=\{w_1, \cdots, w_n\}$ be a union of untwisted Whitney discs.  We say that $\mW'=\{w_1', \cdots, w_n'\}$ is obtained by a \emph{G-Whitney disc slide} if $w'_i=w_i$ for all but one $i$, and for that $i,  w_i'$ is obtained by sliding $ w_i$ over a $w_j$  as in Figure \ref{fig:disc slide}.   Note that Figure \ref{fig:disc slide} b) shows finger discs intersecting $w_i'$.   In more detail construct the \emph{G-Whitney cap} $K$  about $w_j$ by taking two slightly enlarged parallel copies of $w_j$ and glue them together along a rectangle with two edges on $G$ and two on the parallel copies.  See Figure \ref{fig:disc slide} a).   Obtain $w_i'$ by banding  $w_i$ to $K$ along an embedded path $\omega\subset G$.  The band should lie in the negative I-bundle of $G$.  Similarly define an \emph{R-Whitney slide} of $\mW$, where the roles of $G$ and $R$ are reversed.  We analogously define $G$-\emph{finger} and \emph{R-finger disc} slides.  Here the band is required to lie in the positive I-bundle to $G$ or $R$.  It may be called a \emph{disc slide} when the type is clear from context or not specified.  

   \begin{figure}[!htbp]
    \centering
    \includegraphics[width=1.0\linewidth]{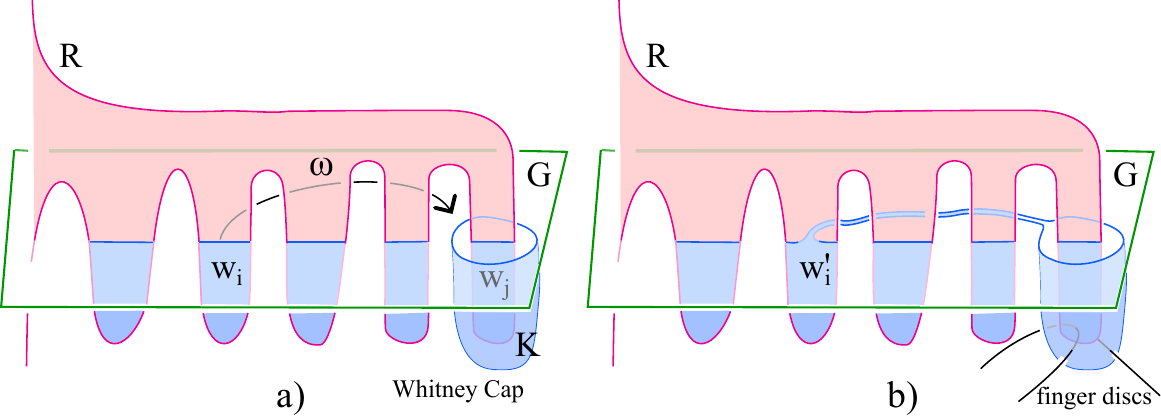}
    \caption{$G$-Disc Sliding $w_i$ over $w_j$ along $\omega$}
    \label{fig:disc slide}
\end{figure}

If  $\mW$ and $\mW'$ (resp. $\mF$ and $\mF'$) have the same boundary germs, then a disc slide on $\mW$ (resp. $\mF$) naturally \emph{corresponds} to one on $\mW'$ (resp. $\mF')$. We say that $(\mF_2, \mW_2)$ is obtained from $(\mF_1, \mW_1)$ by a \emph{disc slide} if exactly one of $\mF_2$ or $\mW_2$ is obtained that way. 

Define $T^G_{w_j}$, the $G$-\emph{normal torus} to $\partial K$, to be the tube about $\partial K$ disjoint from $G$.  In an analogous manner define $T^R_{w_j}$.  \end{definition}

\begin{lemma}\label{handle slide is homotopic}  If $(\mF', \mW')$ is obtained from $(\mF, \mW)$ by a disc slide, then the path $\alpha(\mF,\mW)$ is relatively homotopic to $\alpha(\mF', \mW')$.  \qed \end{lemma}

\begin{remarks}  i) The proof can be found in \S 11.  However, for the purposes of this section, this lemma is not needed. In fact, in this paper we will perform operations on finger/Whitney systems that may move it out of its relative homotopy class.  In analogy, we might prove that a certain function on loops in a space is in fact a function on the homotopy class by showing that two homotopic loops are homotopic in a larger space through loops for which we know the function is invariant.

ii) A disc slide can be viewed as resulting from a special type of factorization again showing that the induced 1-parameter family does not change the associated pseudo-isotopy class. See  \cite[Lemma 3.15]{Ga3}. \end{remarks}

\begin{definition}  We say that $\mW'$ is obtained from $\mW$ by a $p$-twisted G-disc slide if it arises from a disc slide except that the band twists $p$-times around $G$ as it traverses $\omega$.  A 3D local slice of the band appears as in Figure \ref{fig:pspinning}  and from that one sees how to rotate in the $t,z$ plan as the band travels along $\omega$.  Here $G\subset x,y$ plane.  The reader can use the orientation of $G$ and $S^2\times S^2$ to define an orientation convention for this twisting.  \end{definition}

   \begin{figure}[!htbp]
    \centering
    \includegraphics[width=.2\linewidth]{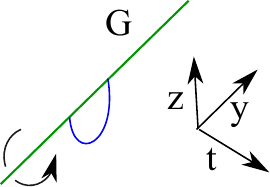}
    \caption{Constructing a Twisted $G$-Whitney Disc Slide}
    \label{fig:pspinning}
\end{figure}

\begin{remarks}  \label{renormalizing} i) Twisted disc slides naturally occur by choosing different untwisted representatives $\hat w_i$ (resp. $\hat w_j$)  for $w_i$ (resp. $w_j$).  For example, if $\hat w_i$ (resp. $\hat w_j$) is obtained from $w_i$ (resp. $w_j$) by locally rotating $p$ times about $G$ near $w_i\cap G$ (resp. $w_j\cap G$), the disc $\hat w_i'$ obtained by doing an untwisted G-disc slide to $\hat w_i$ is isotopic to a $-n$ twisted G-disc slide to $w_i$.   Further if $w_i'$ denotes the untwisted disc slide of $w_i$ over $w_j$ and $\hat w_i'$ is isotoped to have the same $\partial$ germ as $w_i'$, then $\hat w_i'$ is obtained from $w_i'$ by doing cut and paste with n copies of $T_{w_j}^G$.  Since we are working with $\BZ_2$ intersection numbers what matters is the parity of p.

ii) On the other hand if  $\hat w_i$ or $\hat w_j$ are obtained by twisting near $R$, then the corresponding G-disc slide is unchanged. \end{remarks}

\begin{definition}  \label{clifford}  Recall that if  $c$ is a point of transverse intersection of $R$ and $G$, then  $N(c)$ is a 4-ball with $\partial N(c)\cap (R\cup G)$ the Hopf link, so define $C_c$, a \emph{Clifford torus} to be an unknotted torus in $\partial N(c)$ separating the components.  If $w$ is a Whitney disc, then define $C_w$ to be a pair of Clifford tori for $w\cap R\cap G$.  Let $E:=\stwostwo\setminus \inte(N(G)\cup N(R))$.  We say that the complete sets of Whitney discs $\mW, \mW'$ \emph{differ by Clifford tori} or are \emph{Clifford equivalent} if the boundary germs of $\mW$ are isotopic to $\mW'$; for each $i, [w_i']=[w_i]+\sum_{j=1}^n n_{i,j}[C_{w_j}]\in  H_2(E, \partial w_i; \BZ_2)$ and $\sum_{i=1}^n n_{i,i}=0 \in \BZ_2$.  In a similar manner we define Clifford equivalence for $\mF, \mF'$ and use the notation $[f_i']=[f_i]+\sum_{j=1}^n m_{i,j}[C_{f_j}]\in  H_2(E, \partial f_i; \BZ_2)$.  If each of $\mF, \mF'$ and $\mW, \mW'$ are Clifford equivalent we say that $(\mF,\mW)$ and $(\mF',\mW')$ are \emph{Clifford equivalent}.   

We say that $\mW'$ is \emph{$H_2$-equivalent} to $\mW$ if after isotopy of $\mW'$ their boundary germs coincide and  for each $i$, $\partial w_i'=\partial w_i$ and $[w_i']=[w_i]\in H_2(E, \partial w_i; \BZ_2)$.    \end{definition}

\begin{lemma} \label{clifford homology} If $(\mF,\mW)$ is in embedded arc position and differs from $(\mF',\mW')$ by Clifford tori, then $\I(\mF,\mW)=\I(\mF',\mW')$.  \end{lemma}

\begin{proof} This argument will repeatedly use the fact that only terms of the form $\langle f_i', w_j'\rangle, i\le j$ contribute to the invariant.  Assume that the various finger and Whitney discs are ordered as in Definition \ref{invariant}.  The result follows if each $m_{i,j}=n_{i,j}=0$, so it suffices to consider the $\BZ_2$-homological effect of raising one  value by 1 if $i\neq j$ or two values by 1 possibly the same one, if $i=j$.  The argument is symmetric in finger and Whitney discs so it suffices to consider where only values of the form $n_{i,j}$ are changed.  There are basically two cases, when $j\neq i$ and when $j=i$.    If $i<j$, then the added Clifford tori are disjoint from $f_k'$ for $k\le i$ and hence contribute 0 to the invariant.    If $j<i$, then the new Clifford tori contribute $1$ to $\langle f_j', w_i'\rangle$ and $1$ to $\langle f_{j+1}', w_i'\rangle$ and again the invariant is unchanged.  Now suppose $n_{i, i}$ is changed.  It follows that either $n_{i,i}$ is changed by two or some other $n_{j,j}$ is changed.  In the former case the invariant is unchanged and in the latter we have $1=\langle f_i, w_i'\rangle -\langle f_i, w_i\rangle= \langle f_j, w_j'\rangle-\langle f_j, w_j\rangle$ mod 2 and $\langle f_p,w_q'\rangle=\langle f_p,w_q\rangle$ for all $p< q$ and $(p,q)\neq (i,i)$ or $(j,j)$.  Again, the invariant is unchanged.  \end{proof}

\begin{lemma} \label{clifford persists}   If $(\mF_1, \mW_1)$ is Clifford equivalent to $(\mF_1', \mW_1')$, $(\mF_2,\mW_2)$ is obtained from $(\mF_1,\mW_1)$ by disc slides, $(\mF_2',\mW_2')$ is obtained from $(\mF_1', \mW_1')$ by the corresponding disc slides, then $(\mF_2, \mW_2)$ is Clifford equivalent to $(\mF_2', \mW_2')$.\qed\end{lemma}

\begin{lemma} \label{twist hsf} If $w_{i,j}^0$ and $w_{i,j}^n$ are respectively obtained by G-disc slides with 0 and n twisting  along $\omega$, a path from $w_i$ to $w_j$, then up to isotopy, the boundary germ of $w_{i,j}^0$ is equal to that of $w_{i,j}^n$.  Furthermore, $[w_{i,j}^n]=[w_{i,j}^0] + n [C_{w_j}]\in H_2(E, \partial w_{i,j}^0, \BZ_2)$.
 \end{lemma} 
 
 \begin{proof} It follows by Remarks \ref{renormalizing} i) that the boundary germ of $w_{i,j}^n$ can be isotoped to be equal to that of $w_{i,j}^0$ and that  $[w_{i,j}^n]-[w^0_{i,j}]=n[T^G_{w_j}]=n[C_{w_j}]\in H_2(E, \partial w_i^0, \BZ_2).$  \end{proof}

\begin{corollary} \label{normalizing equivalent}  Let $(\mF,\mW)$ be in immersed arc position and $(\hat\mF,\hat\mW)$ be in embedded arc position and obtained from $(\mF,\mW)$ by a finite sequence of disc slides.  

i) Suppose up to twisting along bands, $( \hat\mF', \hat\mW')$ is obtained from $(\mF,\mW)$  by the corresponding sequence of disc slides, then $I(\hat\mF,\hat\mW)=I(\hat\mF', \hat\mW')$. 

ii) $I(\hat\mF,\hat\mW)$ is independent of the initial normalization, i.e. if we start with a different normalization but do the corresponding sequence of disc slides to obtain $(\hat\mF',\hat\mW')$, then $I(\hat\mF',\hat\mW')=I(\hat\mF,\hat\mW)$.\end{corollary}

\begin{proof} i) By Lemmas \ref{twist hsf} and \ref{clifford persists}  $(\hat \mF,\hat\mW)$ differs from $(\hat\mF',\hat \mW')$ by Clifford tori.  The result now follows from  Lemma \ref{clifford homology}.

ii)  It follows from Remark \ref{renormalizing} that  a different normalization of $(\mF,\mW)$ modifies the sequence of disc slides's by twisting along bands and hence the resulting $(\hat\mF',\hat \mW')$ differs from $(\hat \mF,\hat\mW)$ by Clifford tori.  The result now follows from i).\end{proof}

\begin{definition}  We say that  $(\mF_1,\mW_1)=(\mF,\mW), \cdots, (\mF_n,\mW_n)$ is a \emph{disc slide sequence} if for each $i>1$ either 

i) $(\mF_i,\mW_i)$ is obtained from $(\mF_{i-1},\mW_{i-1})$ by a single disc slide plus possibly untwisted isotopy of the resulting discs or 

ii) an isotopy where a single $D\in \mW_{i-1}\cup \mF_{i-1}$ is replaced by a new untwisted representative, i.e. $D$ undergoes a G-twist or an R-twist.
\vskip 6pt
If the sequence involves only R-disc slides and G-twistings (resp. G-disc slides and R-twistings), then the sequence is an \emph{R-disc (resp. G-disc) slide sequence}.
\end{definition}

\begin{definition}  We say that the untwisted $(\mF,\mW)\in \REA$ (resp. $\in \GEA$) if $(\mF\cup\mW)\cap R$ (resp.  $(\mF\cup\mW)\cap G)$ is an embedded arc.  If either hold, then we say that $(\mF,\mW)\in $ \emph{Half-EA}.  \end{definition}

\begin{lemma} \label{ia to ea} Let $(\mF,\mW)$ be normalized in immersed arc position.  Then there is a disc slide  sequence $(\mF_1,\mW_1)=(\mF,\mW), \cdots, (\mF_n,\mW_n)$ such that $(\mF_n,\mW_n)\in \EA$.  If  $(\mF,\mW)\in \REA$ (resp. $(\mF,\mW)\in \GEA$) , then the sequence can be chosen to  not involve $R$-disc (resp. $G$-disc) slides. \end{lemma}  

\begin{proof}  We will first do a sequence of G-disc slides to obtain $(\mF',\mW')\in \GEA$, then do a sequence of R-disc slides to obtain the desired $(\mF_n,\mW_n)$.  We begin with a sequence of G-Whitney disc slides to eliminate excess intersections of $f_1\cap G$ with $\mW\cap G$.  See Figure \ref{fig:iaea} a) which shows a schematic of $(\mF\cup\mW)\cap G$ where $\mW$ looks standard and intersects various components of $\mF$.  Also shown is the first intersection $x\in w_i$ of $f_1\cap \mW\cap G$ where $f_1\cap G$ starts at $a_0$.  By swinging $w_i\cap G$ about the 2-sphere G we go from Figure \ref{fig:iaea} a) to \ref{fig:iaea} c) by an isotopy which passes through each $w_j, j\neq i$, once.  Thus by doing a suitable sequence of G-disc slides of $w_i$, one for  each $w_j, j\neq i$ we can transform $w_i$ to appear as in Figure \ref{fig:iaea} c).  An example is shown in Figure \ref{fig:iaea} b).   Note that $x$ has been eliminated from $w_i$ at the cost of creating  boundary bigons  with $f_i$ and $f_{i+1}$, or just $f_i$ if $i=n$ which are then eliminated by replacing $w_i$ by a new untwisted representative.  The net result is the elimination of this single point of intersection.    A sequence of such operations eliminates all the  excess $f_1\cap\mW\cap G$ intersections.  

Next, by repeated G-finger disc slides of $\mF$ over $f_1$, one for each point of $\mF\cap\inte(w_1\cap G)$ we eliminate excess intersections of $w_1\cap G$ with $\mF$.   A  sequence of G-Whitney disc slides over $w_1$ then transfers all the excess $f_2\cap\mW\cap G$ intersections to $f_1$ and these intersections are then eliminated as in the previous paragraph.  By  induction we eliminate all excess $\mF, \mW$ intersections in $G$ to obtain $(\mF', \mW')$.   

Note that in the process we did not create any new $\mF\cap \mW\cap R$ intersections.  By a sequence of R-disc slides starting from $(\mF', \mW')$ we obtain an $(\mF_n,\mW_n)\in \EA$.  \end{proof} 

  \begin{figure}[!htbp]
    \centering
    \includegraphics[width=.8\linewidth]{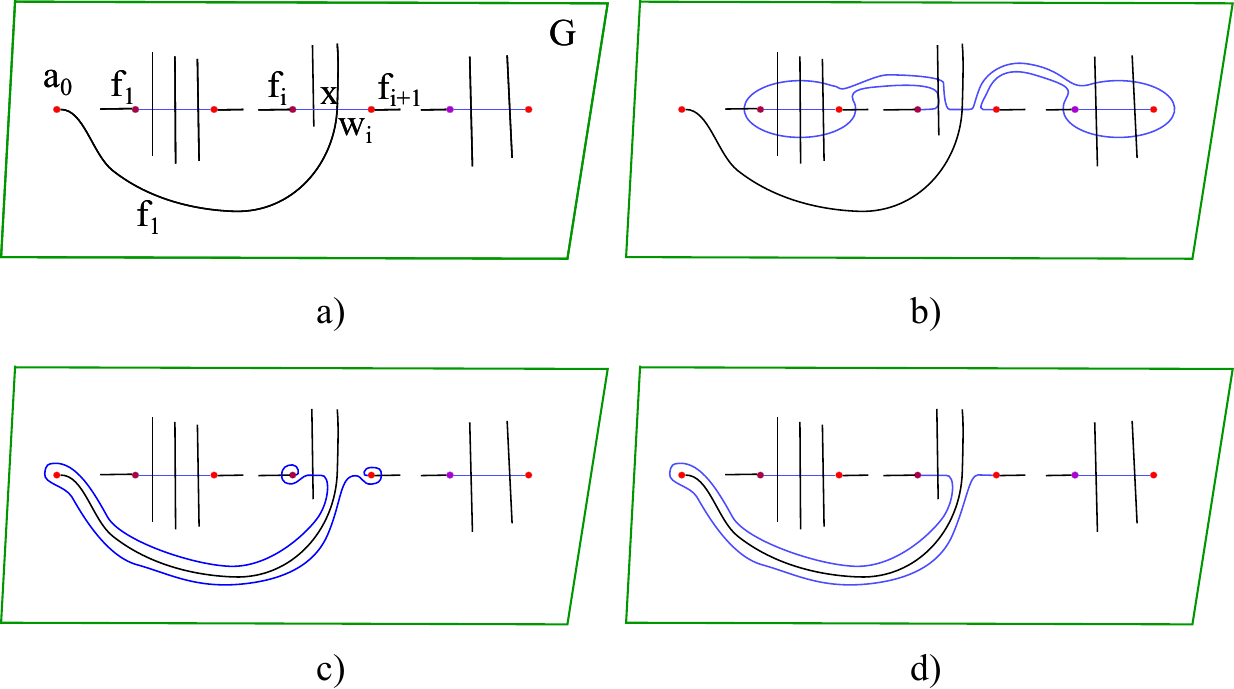}
    \caption{Going from Immersed Arc to Embedded Arc Position}
    \label{fig:iaea}
\end{figure}

\begin{lemma} \label{no new ints}  Let $\mF,\mW$ be untwisted sets  of discs  such that $(\mF\cup\mW)\cap R$ is the embedded union of circles and possibly an arc.  Let $(\mF',\mW')$ be obtained from $(\mF,\mW)$ by either an untwisted  G-disc slide or by an $R$-twist  of a single disc supported near $R$, then for all $i,j$,\ $|\inte(f_i)\cap\inte(w_j)|=|\inte(f_i')\cap\inte(w_j')|$ mod-2 where $\mF=\{f_1,\cdots, f_n\}$ and $\mW=\{w_1, \cdots, w_n\}$.  A similar statement holds with $R$ and $G$ reversed.  \end{lemma}

\begin{proof}  Suppose the G-disc slide arises from  sliding $w_i$ over $w_j$ along $\omega$ with $w_i'$ the resulting disc.  Since $\mF, \mW$ are untwisted the interior of the  band is disjoint from $\mF$.  Each intersection of $\inte(w_j)$ with $\mF$ gives rise to two new intersections of the Whitney cap with $\mF$.  Since $(\mF\cup\mW)\cap R$ is embedded  there are no other new intersections.  Similarly an $R$-twist corresponds to rotating a disc near its intersection with $R$. Since $\mF\cup \mW$ is embedded near $R$, there are no new intersections.  \end{proof}

\begin{remarks}  \label{new ints} i) If $|w_j\cap\inte(f_i\cap R)|=1$, then a G-disc slide over $w_j$  will create one new point of intersection mod 2 with $f_i$.  It is for this reason that an analogous invariant of $(\mF, \mW)$ in immersed arc position that just uses normalized discs is unsuitable.

ii) The next lemmas show that at the cost of possibly adding Clifford tori, G-disc sliding commutes with R-disc sliding, disc sliding commutes with boundary twisting, and G-twisting commutes with R-twisting. \end{remarks}

\begin{lemma} (Clifford Commutation Lemma 1) \label{clifford commutes}Let $\mW=\{w_1, \cdots, w_n\}$ be a complete system  of untwisted Whitney discs.  Let $\omega_g$ (resp $\omega_r$) be an embedded path in $G$ (resp. $R$) from $w_i$ to $w_j$ (resp. $w_p$ to $w_q$).  Let $\mW_1$ be obtained by doing a G-disc slide  to $\mW$ along $\omega_g$ and $\mW_2$ be obtained from $\mW_1$ by doing a R-disc slide along $\omega_r$.  Let $\mW_3$ be obtained by doing an R-disc slide to $\mW$ along $\omega_r$ and $\mW_4$ be obtained by doing a G-disc slide  to $\mW_3$ along $\omega_g$.  Then $\mW_2$ is Clifford equivalent to $\mW_4$.\end{lemma}

\begin{proof}  Since the boundary germ of a disc near $G$ (resp. $R$) resulting from a $G$ (resp. $R$) disc slide is determined by $\omega_g, w_i\cap G$ and $w_j\cap G$ (resp. $\omega_r, w_p\cap R$ and $w_q\cap R$) it follows that the germs of $ \mW_2$ and $\mW_4$ are isotopic.  It remains to show that their homology classes are Clifford equivalent.
\vskip 8pt
\noindent\emph{Case 1:} $ i,j,p,q$ are all distinct.
\vskip 8pt
\noindent\emph{Proof of Case 1:} Being distinct the disc slides can be done independently.\qed
\vskip 8pt
\noindent\emph{Case 2:} $i=p$ or $j=q$.
\vskip 8pt
\noindent\emph{Proof of Case 2:}  In all cases $\mW_2$ is isotopic to $\mW_4$.\qed
\vskip 8pt
\noindent\emph{Case 3:} $p=j$ and $q\neq i$.
\vskip 8pt
\noindent\emph{Proof of Case 3:}  Let $\mW^k:=\{w^k_1,\cdots, w^k_n\}$ for $1\le k\le 4$.  Here $w^2_r$ are $w^4_r$ are isotopic for $r\neq i$,  $w^2_i$ is the result of G-disc sliding $w_i$ over $w_j$ and $w^4_i$ is the result of G-sliding $w_i$ over $w^2_j$.  In this case $[w^4_i]=[w^2_i] + [T^R_{w_q}]=[w^2_i]+[C_{w_q}]$.\qed
\vskip 8pt
\noindent\emph{Case 4:} $p=j$ and $q=i$.
\vskip 8pt
\noindent\emph{Proof of Case 4:}  Let $w_i'$ denote the result of sliding $w_i$ over $w_j$ using $\omega_g$ and $w_j'$ the result of sliding $w_j$ over $w_i$ using $\omega_r$.  Let $w_i''$ the result of sliding $w_i$ over $w_j'$ using $\omega_g$ and $w_j''$ the result of sliding $w_j$ over $w_i'$ using $\omega_r$.  Then with notation as in the proof of Case 3, $w^2_i=w_i', w^2_j=w_j'', w^4_i=w_i'', w^4_j=w_j' $ and $[w_i'']=[w_i']+[T^R_{w_i}]=[w_i']+[C_{w_i}]$ while $[w_j'']=[w_j']+[T^G_{w_j}]=[w_j']+[C_{w_j}]$.\end{proof}  

\begin{lemma} (Clifford Commutation Lemma 2) \label{clifford commutes two} i)  Up to Clifford equivalence G-disc sliding commutes with G-twisting and R-disc sliding commutes with R-twisting,  

ii) G-disc sliding commutes with R-twisting and R-disc sliding commutes with G-twisting and

iii) G-twisting commutes with R-twisting.
\end{lemma}

\begin{proof} i) If the disc sliding involves finger discs and the twisting involves a Whitney disc or vice versa, then these operations can be done independently.  When they both involve say Whitney discs and G-twisting and disc sliding, then it is immediate that they have the same boundary germ.  The only nontrivial case is when the twisting involves one of $w_i$ or $w_j$, where $w_i$ G-slides over $w_j$. In both cases they differ by $C_{w_j}$. 

ii) This is immediate. 

iii) The nontrivial case is when both twistings involve the same disc, say $w_i$ and again it is immediate that they have the same boundary germ and are homologous rel $\partial$.  \end{proof}

\begin{definition}  If $(\mF,\mW)\in IA$, then define $\I(\mF, \mW)=\I(\mF', \mW')$ where $(\mF', \mW')\in \EA$ is obtained from $(\mF, \mW)$ by a disc slide sequence.\end{definition}

\begin{proposition}  \label{hsf independence} Let $(\mF,\mW)$ be normalized and in immersed arc position, then $\I(\mF,\mW)$ is independent of the normalization and of the disc slide sequence.\end{proposition}

\begin{proof} We show that if $(\mF', \mW'), (\mF'', \mW'')\in EA$ are each obtained from a normalized $(\mF,\mW)$ by a disc slide sequence, then $\I(\mF', \mW') = \I(\mF'', \mW'')$.   Independence of normalization was proven in Corollary \ref{normalizing equivalent}.

By repeated applications of Lemmas \ref{clifford commutes}, \ref{clifford commutes two}, \ref{clifford homology}, \ref{clifford persists},  and  \ref{twist hsf} we can modify the sequence without altering the values $\I(\mF', \mW')$ and $ \I(\mF'', \mW'')$ so that all G-disc slides and R-twistings are done before all the R-disc slides and G-twistings and further all the disc slides are untwisted.  

Let $(\hat \mF', \hat\mW')$ (resp. $(\hat \mF'', \hat\mW'')$) denote the  pair  which occurs after the last G-disc slide or R-twisting in the disc slide sequence to $(\mF', \mW')$ (resp. $(\mF'', \mW'')$).   We obtained $(\mF', \mW')$ from $(\hat \mF', \hat\mW')$ by a sequence of R-disc slides and G-twistings. On the other hand we can obtain $(\mF''', \mW''')\in \EA$ from $(\hat \mF', \hat\mW')$ by using the R-disc slides and G-twistings corresponding to the $(\hat\mF'', \hat\mW'')$ to $(\mF'',\mW'')$ sequence.  By Lemma \ref{no new ints} we have $\I(\mF', \mW')=\I(\mF''', \mW''')$.  The sequences from $(\mF, \mW)$ to $(\mF'', \mW'')$ and $(\mF,\mW)$ to $(\mF''', \mW''')$ have the same subsequence of R-disc slides and G-twistings.  Using Lemmas \ref{clifford commutes} and \ref{clifford commutes two} we can commute all the R-disc slides and G-twistings with all the G-disc slides and R-twistings and then apply Lemma \ref{no new ints} to conclude $\I(\mF'', \mW'')=\I(\mF''', \mW''')$.\end{proof}

\begin{definition}\label{hat i}  Let $(\mF,\mW)\in IA$ be untwisted but not necessarily normalized such that $\mF$ is finger germed  and $\mW$ is Whitney germed or vice versa, then with the IA ordering define 
$$\hat \I(\mF,\mW)=\sum_{i\le j} |\inte(f_i)\cap \inte(w_j)|\ \textrm{mod}\ 2 \in \BZ_2.$$\end{definition}

\begin{lemma} \label{reverse} ($\hat \I$ and $\I$-symmetry for IA) If $\mfw\in$ IA, then $\hat \I(\mF,\mW)=\hat \I(\mW,\mF)$ and $\I(\mF,\mW)=\I(\mW,\mF)$.\end{lemma}

\begin{proof}  The first conclusion follows from Definition \ref{hat i}.  Since the second conclusion holds when $\mfw$ is in embedded arc position, the result follows from Lemma \ref{ia to ea} and Proposition \ref{hsf independence}.\end{proof}

The next result follows from Lemma \ref{no new ints} and Proposition \ref{hsf independence}.  

\begin{lemma}  \label{hati half ea}  If $(\mF, \mW)$ is untwisted and half-EA, then $\hat \I(\mF,\mW)=\I(\mF,\mW)=I(\mW, \mF)$.  \end{lemma}

\begin{proof}  If say $(\mF, \mW)\in \REA$, then by Lemma \ref{ia to ea} there exists $(\mF', \mW')\in EA$ obtained from $(\mF, \mW)$ by a G-disc slide sequence and by Lemma \ref{no new ints}, $\hat \I(\mF', \mW')=\I(\mF, \mW)$.   Lemma \ref{no new ints} also implies that when $\mfw\in \REA, \hat \I\mfw$ is unchanged under a $G$-disc slide sequence.   The last equality follows from Lemma \ref{reverse}.   \end{proof}

\section{Getting to immersed arc position} \label{finger first}

In this section   we give a process for transforming the finger/Whitney system $(\mF,\mW)\notin$ IA to one $(\mF, \mW_1)\in$ IA.  Here we require that the elements of $\mW$ are ordered.   While there are many ways to carry this process out, Proposition \ref{switch invariance}, the main result of this section, asserts that each gives the same $\I(\mF, \mW_1)$ value denoted $\I(\mF,\mW)$.   We start by giving a criterion for recognizing when $(\mF_i,\mW_i)\in$ IA, $i=1,2$ produce the same $\I$-value when transformed to an $EA$ pair by disc slide sequences.

\begin{definition}  We say that the  sets $\mW, \mW'$ of Whitney discs for $R$ and $ G$ \emph{similarly match} if there is a 1-1 correspondence such that if $w\in \mW$, then  the corresponding $w'\in \mW'$  pairs the same points of $R\cap G$.   Here both $\mW,\mW'$ can be  sets of Whitney discs or finger discs or one set of each.  More generally two Whitney discs are \emph{matched} if they cancel the same points of $R\cap G$.  \end{definition}

\begin{lemma}\emph{\textbf{(Parity Lemma)}}\label{parity}  For $i=1,2$ let $\mF_i=\{f_{i,1}, \cdots, f_{i,n}\}, \mW_i=\{w_{i,1}, \cdots, w_{i,n}\}$.  Assume that 
\vskip 8pt
i) $\mF_1, \mF_2$ are similarly matched as are $\mW_1, \mW_2$ with $f_{1,j}$ (resp. $w_{1,j}$) corresponding to $f_{2,j}$ (resp. $w_{2,j}$).  Furthermore, $\mF_1, \mF_2, \mW_1, \mW_2$ are untwisted, though not necessarily normalized, such that both $\mF_1, \mF_2$ are finger germed and $\mW_1, \mW_2$ are Whitney germed or vice versa.  
\vskip 8pt
ii) $(\mF_1, \mW_1), (\mF_2, \mW_2)\in IA$,
\vskip 8pt
iii) For each $i, j,\  |f_{1,i}\cap w_{1,j}\cap R|-|f_{2,i}\cap w_{2,j}\cap R|=|f_{1,i}\cap w_{1,j}\cap G|-|f_{2,i}\cap w_{2,j}\cap G|=0 \ \textrm{mod}\ 2$, 
\vskip 8pt
then $\I(\mF_1, \mW_1)=\I(\mF_2, \mW_2)$ if and only if $\hat \I(\mF_1, \mW_1)=\hat \I(\mF_2, \mW_2)$.  \end{lemma}  

\begin{proof}  First assume that $\hat \I(\mF_1, \mW_1)=\hat \I(\mF_2, \mW_2)$.  We will replace each $(\mF_i, \mW_i)$ with an R-EA $(\mF_i', \mW_i')$ by an R-disc slide sequence such that $\hat \I(\mF_1', \mW_1') =\hat \I(\mF_2', \mW_2')$.  It will then follow that $\I(\mF_1, \mW_1)=\I(\mF_1', \mW_1')=\hat \I(\mF_1', \mW_1')=\hat I(\mF_2', \mW_2')=\I(\mF_2', \mW_2')=\I(\mF_2, \mW_2)$.  The second and fourth equalities follow from Lemma \ref{hati half ea} and the first and fifth equalities follow from Proposition \ref{hsf independence}. 

 Let $\rho_i$ denote the immersed arc $(\mF_i\cup\mW_i)\cap R$ oriented from $f_{i,1}$ to $w_{i,n}$.  To minimize notation we will usually suppress the subscript i.  Let $c_1, \cdots, c_p$ denote the points of $\inte(f_1\cap R)\cap \mW$ as they appear along $ f_1$.  We first eliminate $c_1$ as in the proof of Lemma \ref{ia to ea} except that we do not eliminate the resulting boundary bigons and so this operation is supported away from $G$.  The newly created points of $f_1\cap w_1\cap R$ will be eliminated later.  If $c_1\in w_j$, then let $w_j'$ be the result of eliminating $c_1$.   Compare Figures \ref{fig:iaea} a) and \ref{fig:iaea} c).  Here $w_j'$ is obtained by doing an R-Whitney disc slide over each $w_k, k\neq j$ followed by isotopy.  
We see that $\partial w_j'$ has two new intersections with $\mF$ if $j\neq n$ and one if $j=n$ which consists of $\partial$ bigons with $f_j$ and $f_{j+1}$, the latter not existing when $j=n$.  Also mod 2, $\inte(w_j')$ picks up new intersections with $\mF$, one for each point of $\inte(w_k\cap G)\cap \mF$, where $k\neq j$.  See Figure \ref{fig:disc slide} b).  Note that if $f_{1,1}\cap\mW_1\cap R$ corresponds to a single boundary bigon with $w_{1,1}$, then $f_{1,1}\cap\mW_1'\cap R$ will have the same feature, though $w_{1,1}'$ may have many new intersections with $\mF_1$.  These are needed to have the same parity with the resulting $w_{2,1}'$.

In a similar manner we sequentially eliminate $c_2, \cdots, c_p$ and continue to call $w_j'$ the result of these operations to $w_j$ as well as future ones and define $\mW':=\{w_1', \cdots, w_n'\}$ the result of these operations.  Note that $\mW'$ will change as the proof proceeds.   In what follows $\bar m$ will denote $ m\ \textrm{mod} \ 2$.  Before cataloging the new intersections we note that except for the $c_i$'s no old ones are lost.  

If $b_j=|\inte(f_1\cap R)\cap w_j|$, then $b_j$ new intersections will occur on $w_j'\cap f_j\cap R$ and also on $w_j'\cap f_{j+1}\cap R$ if $ j\neq n$.  Restoring the subscript $i\in \{1,2\}$ we have $\bar b_{1,j}=\bar b_{2,j}$ and for all $i,j, \ |f_{1,i}\cap \inte(w_{1,j}'\cap R))|-|f_{2,i}\cap\inte(w_{2,j}'\cap R)|=0\ \textrm{mod}\ 2$.  

We now consider new $\mF\cap \inte(\mW')$ intersections.  Since each operation is supported away from $G, w_j'\cap G=w_j\cap G$ and mod 2, $\inte (w_j)$ picks up $d_j=b_j(\sum_{k\neq j} \mF\cap \inte(w_k\cap G))$ intersections.  It follows that for all $j, \bar d_{1,j}=\bar d_{2,j}$ and hence $\hat \I(\mF_1, \mW_1')=\hat \I(\mF_2, \mW_2')$.  

By construction $f_1\cap R$ is disjoint from $w_k'\cap R$, all $k\neq 1$ and $|\inte (w_1'\cap R)\cap f_1|=b_1$.  After an isotopy supported near R and away from G we can assume that $w_1'$ spirals around 
$w_1'\cap f_1\cap R$, both geometrically and algebraically some $\hat b_1$ times.  This isotopy reduces $|f_1\cap w_1'|$ an even number of times  and leaves $\hat \I(\mF, \mW')$ 
unchanged.  Finally do $\pm \hat b_1$ G-rotations of $w_1'$ to make $\inte(w_1'\cap R)\cap f_1=\emptyset$.  Mod 2, these 
rotations create $|\hat b_1|\sum_{k=1}^n|\inte(w_1\cap G)\cap f_k|$ new $ \inte (w_1')\cap \mF$ intersections and $|\hat 
b_1|$ new $w_1'\cap R\cap f_2$ points.  In summary, we have that $(\mF_1, 
\mW_1'), (\mF_2, \mW_2')$ satisfy the hypotheses of the lemma and both $f_{1,1}\cap R$ and $f_{2,1}\cap R$ have no excess intersections.

An induction completes the proof following the organization of the proof of Lemma \ref{ia to ea}.   For example, using a similar  argument as in that proof we  reduce to the case that for $i=1,2$ each of $f_{i,1}\cap R$ and $w_{i,1}\cap R$ has no excess intersections.  Then  we clear each $f_{i,2}\cap R$ of excess intersections with $\mW_i$ at the cost of moving them to $f_{i,1}\cap R$ and then eliminate them as above, except that we eliminate resulting excess $f_{i,2}\cap w_{i,2}\cap R$ intersections by isotopy and G-twisting $w_{i,2}$.

If $\hat \I(\mF_1, \mW_1)\neq\hat \I(\mF_2, \mW_2)$, then the pairs of $\hat \I$ values remain unequal throughout.\end{proof}

\begin{definition}  Recall that two complete sets of Whitney discs $\mW_1$, $\mW_2$ are \emph{$H_2$-equivalent} if after an isotopy, for all $i,\ w^1_i, w^2_i$ have the same boundary germ and represent the same class in $H_2(E, \partial E; \partial w^1_i)$.  They are \emph{strongly $H_2$-equivalent} if they have the same boundary germs to start with. \end{definition}

\begin{corollary} \label{homological independence} If $(\mF, \mW)$ satisfies the immersed arc condition, both $\mF, \mW$ are untwisted and $\mF', \mW'$ are untwisted and respectively $H_2$-equivalent to $\mF$ and $\mW$, then $\I(\mF,\mW)=\I(\mF', \mW')$.\end{corollary}

\begin{proof}  If they are strongly $H_2$ equivalent, then this is an immediate consequence of Lemma \ref{parity}. Now suppose that $\mW'$ is $H_2$-equivalent to $\mW$.   Then $\mW$ is properly isotopic to $\mW''$ which is strongly $H_2$-equivalent to $\mW'$.   We conclude that $\I(\mF,\mW)=\I(\mF, \mW'')=\I(\mF,\mW')$, where the second equality follows from Lemma \ref{parity} and the first follows from Proposition \ref{hsf independence}.  In a similar manner we conclude $\I(\mF,\mW')=\I(\mF', \mW')$. \end{proof}

\begin{definition}  Recall from 3-manifold theory  that a \emph{boundary compression} is the operation of compressing a properly embedded surface $S$ in the manifold with boundary $M$ along a framed half disc $H$ which has  one boundary component on $S$ and the other on $M$.  For example, see Figure \ref{fig:braiding3}.  Let $w$ be an untwisted Whitney disc whose boundary germ lies to the $+$-side of $G\cup R$.  A surface $w'$ obtained by $\partial$-compressing $w$ into the $+$ side of $G\cup R\setminus N(G\cap R)$ is called a \emph{$\partial_+$-compression}.  In a similar manner we define \emph{$\partial_-$-compression}.\end{definition}

\begin{remark}  Note that if $\mW'=\{w_1', \cdots, w_n'\}$ is obtained from the untwisted $\mW=\{w_1,\cdots, w_n\}$  by a sequence of $\partial_+$-compressions, then the $w_i'$'s need not be connected, embedded or pairwise disjoint.  Nevertheless, we can extend the definition of $\hat I$ to pairs of the form $(\mF, \mW')$.\end{remark}

\begin{lemma}  \label{boundary compression} Given $(\mF, \mW)$ and $(\mF, \mW')$ where $\mW$ is finger germed and $\mW'$ is obtained from $\mW$ by a sequence of compressions and  $\partial_+$-compressions, then $(\mF,\mW)$, $(\mF, \mW')$ satisfy hypotheses ii), iii) of Lemma \ref{parity}.\qed\end{lemma}

\begin{remark} In what follows we will need to show that certain pairs $(\mF, \mW_1), (\mF, \mW_2)$ satisfy the hypotheses of Lemma \ref{parity}.  We will sometimes do this by showing $(\mF, \mW_1), (\mF, \mW_2')$ satisfy these hypotheses where $\mW_2'$  is obtained from $\mW_2$ by a sequence of $\partial_+$- or $\partial_-$-compressions, the type depending on context.\end{remark}

Given $(\mF, \mW)$ with an ordering on $\mW$ we describe a way of modifying $\mW$ to $\mW^*$ so that $(\mF, \mW^*)$ satisfies the immersed arc condition.

\begin{definition} \label{switching}  Let $\mW=\{w_1, \cdots, w_n\}$ be a complete untwisted system of Whitney discs.  We say that the complete system $\mW^*=\{w_1^*, \cdots, w_n^*\}$ is a \emph{k-switching} if
\vskip 8pt

a) $a_{2i-1}\in w_i\cap w_i^*$, where $R\cap G=\{a_0, a_1, \cdots, a_{2n}\}$ such that $\langle R,G\rangle=+1$ at $a_0, a_2, \cdots, a_{2n}$ and $w_i\cap R\cap G=a_{2i-1}\cup a_{2i}$.
\vskip 8pt
b) Up to isotopy $w_i^*$=$w_i$ except for exactly $k\ w_i$'s.  The $w_j^*$'s with $w_j\neq w_j^*$ are called the \emph{switch discs} and their union is denoted $\mW^*_S$.  The $w_j$'s that get switched are called the \emph{switch out discs} the others, $\mW^*\setminus \mW^*_S$ are called the \emph{non switch discs}. 
\vskip 8pt
c) $\mW^*_S$, $\mW$ can be respectively isotoped to be finger germed and Whitney germed such that $\mW^*_S\cap \mW\subset R\cap G$ and $(\mW^*_S\cup\mW)\cap G$ contains no cycles.
\vskip 8pt

Call $k$ the \emph{switch index}. We say that  $\mW^*$ is \emph{$\mW$-framed} if $\mW$ is Whitney germed and $\mW^*_S$ satisfies c) without the need to isotope $\mW$.   A \emph{full switching} is one with switch index $= |\mF|=|\mW|$.  

Given $(\mF,\mW)$, then $\mW^*$ is a \emph{minimal switching} if $(\mF,\mW)\in$ IA and the switch index equals the number of immersed cycles in $(\mF\cup\mW)\cap G$.
\end{definition}

\begin{remark} Recall from \S0 that an $(\mF, \mW)$ gives rise to $[\alpha]\in\pi_1(\Emb(\sqcup_{i=1}^kS^2, \#^k\stwostwo), \mRs)$, an isotopy class of pseudo-isotopies and its associated diffeomorphism of $S^4$.  A priori an $(\mF, \mW^*)$ may give rise to different classes than $(\mF, \mW)$, however by Lemma 0.55, proven in \S11 we will show that they induce the same $[\alpha]$.  This is true both for order induced switching introduced in this section and the general form given in Proposition \ref{sum invariance}.  \end{remark}

 We now describe ways of viewing $R\cup G$ that will facilitate both visualizing and working with finger and Whitney discs.  

\begin{definition} \label{framed finger} The pair $R$ and $G$ are in \emph{framed finger form} if $G=\Gs$ and $R$ appears as in Figure \ref{fig:framed finger} a) or b). The choice will be clear from context.  Further the positive bundles of $G$ and $R$ are as indicated.  This is the \emph{standard framing} on $N(R\cup G)$.   In both cases the local 3D projection of $R$ to $G$ lies in the $y$-axis with $x\ge 0.$   We denote the points of $R\cap G$ by $a_0, a_1, \cdots, a_{2n}$ where these are points along the $x$-axis.   In the figures, the  $w_1, \cdots, w_n$ are the \emph{standard ordered Whitney germed discs} and the discs $f_1, \cdots, f_n$ are the \emph{standard ordered finger germed discs}. \end{definition}

\begin{remark} \label{fff differences} Here are two features of the version of Figure \ref{fig:framed finger} a).  First the various fingers can be locally rotated about the $z$-axis.  This will enable correcting a misframed R as in the proof of Lemma \ref{trivial w}.  Second, we readily see how to isotope $R\cup G$ so that $R$ becomes $\Rs$ and $G$ has fingers poking into $\Rs$.  The version of Figure \ref{fig:framed finger} b) will be the one most commonly used, e.g. in Figures \ref{fig:switch} and Definition \ref{1d}. \end{remark}

   \begin{figure}[!htbp]
    \centering
    \includegraphics[width=1.0\linewidth]{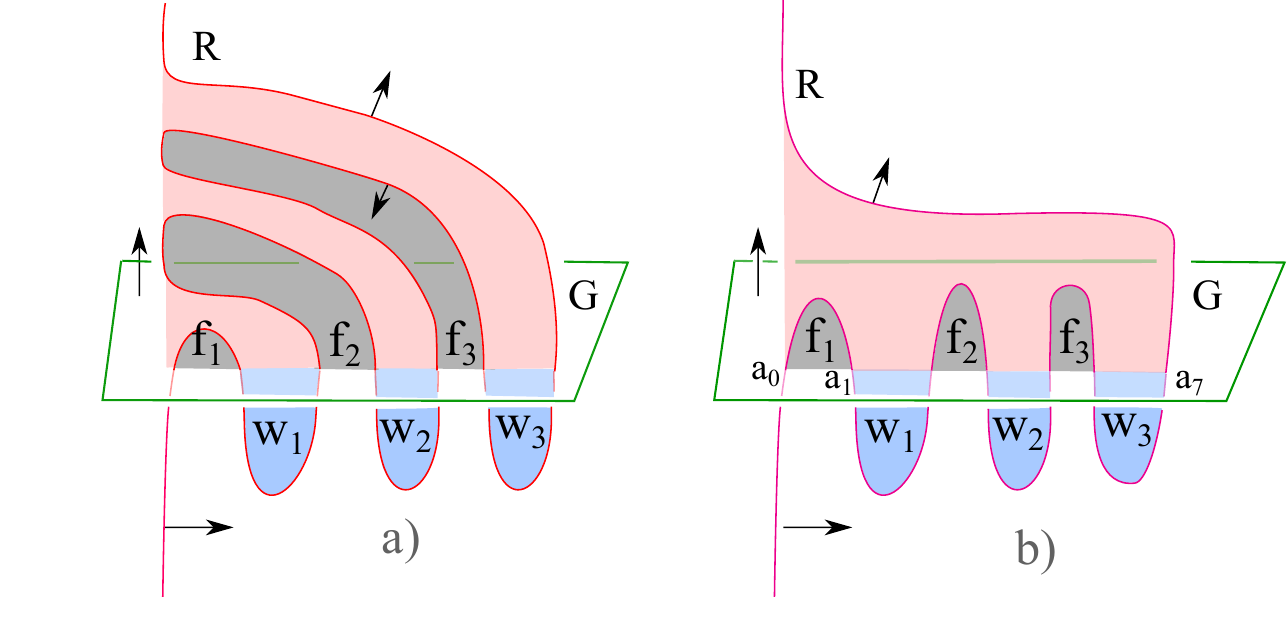}
    \caption{Framed Finger Form}

    \label{fig:framed finger}
\end{figure}

\begin{lemma} \label{trivial w} Let $\mW$ be a complete ordered system of Whitney germed Whitney discs, where $G=\Gs$.  Then there exists an ambient isotopy fixing $G$ setwise whose time 1 map $\kappa$ has the property that $G$ and $\kappa(R)$  are in framed finger form.  Furthermore, $\kappa(\mW)$ are the standard Whitney discs such that the ordering on $\mW$ equals the induced ordering on  $\kappa(\mW)$.  Similarly, if $\mW$ is finger germed, then there exists a $\kappa'$ such that $\kappa'(\mW)$ is the standard set of ordered finger discs.  \end{lemma}

\begin{proof} Without the condition on the normal framing, this is essentially \cite{Ga3} Lemma 3.11.  View $R\cup G$ as in Figure \ref{fig:framed finger} a).  After an initial isotopy of $N(R\cup G)$ we can assume that the framing is standard near $a_0$ and since $\mW$ is Whitney germed, we can assume that the normal framing of $R\cup G$ is standard near $\mW$.  After an isotopy of $N(R)$ fixing $R$ pointwise and supported away from $G$ we can make it standard near all the fingers of $R$, however, the positive bundle can rotate by an integral multiple of $2\pi$ as one goes from one finger to the next or to $a_0$.   After rotating the various fingers by multiples of $2\pi$ as in Figure \ref{fig:framed hinge} we can assume that the positive bundle of $R$ is standard in a 2-disc containing the fingers and $a_0$.  Since $R$ has trivial normal bundle we can isotope $N(R)$ to have the  standard framing on all of $N(R)$.  Note that if the positive bundle of $G$ was already standard, then it would stay standard throughout this process.  Since framed finger form is basically symmetric in $R$ and $G$ we can now correct the framing of $N(G)$ by ambient isotopy, preserving that of $N(R)$.\end{proof}

   \begin{figure}[!htbp]
    \centering
    \includegraphics[width=0.5\linewidth]{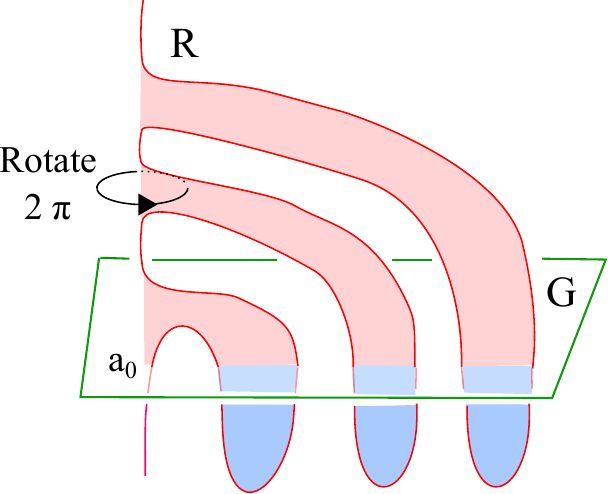}
    \caption{Rotating a Finger to Correct the Positive Framing}
    \label{fig:framed hinge}
\end{figure}

\begin{remark} In this paper given a complete set of   Whitney discs $\mW$ for $R\cup G$ we will frequently invoke Lemma \ref{trivial w} to put $R\cup G$ into framed finger form so that  $\mW$ becomes the standard set of finger or Whitney discs.  It will go without saying that $\mW$ first undergoes a preliminary isotopy to become finger germed or Whitney germed as needed. \end{remark}

\begin{definition} The $\kappa$ of Lemma \ref{trivial w} is called a \emph{standardizing} map.\end{definition}

\begin{corollary} \label{switch exists} Given a complete ordered set $\mW$ of $n$ untwisted Whitney discs, then for each $0\le k\le n, \mW$ has a $\mW$-framed switching of index $k$.  See Figure \ref{fig:switch}.  \qed\end{corollary}

   \begin{figure}[!htbp]
    \centering
    \includegraphics[width=1.0\linewidth]{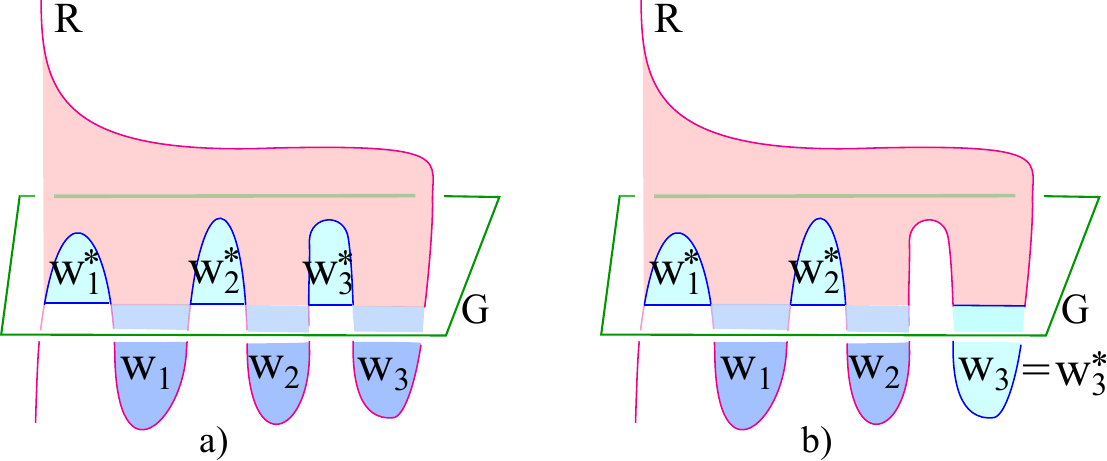}
    \caption{$\qquad\quad$  a)  3-Switch  $\qquad\quad$                b) 2-Switch}
    \label{fig:switch}
\end{figure}

\begin{definition}  \label{switch definition} If $R$ and $G$ are in framed finger form with $\mW$ the standard ordered Whitney discs, we call the naturally defined $\mW^*$ the \emph{standard k-switch}.  By \emph{natural} we mean the switch discs of a $k$-switch are the standard finger discs $f_1, \cdots, f_k$.  
  The cases of n=3 and k=2,3 are shown in Figures \ref{fig:switch} a) and b).\end{definition}

 \begin{remarks}  i) As we shall see, there is a multitude of $k$-switchings for an ordered $\mW$, however Figure \ref{fig:switch} indicates that if $R\cup G$ is in framed finger form with the Whitney discs ordered and standard, then there is a canonical one.
 
 ii)  Given $(\mF, \mW)$ the proof of the next lemma addresses how to choose the switch out discs of $\mW$.   \end{remarks}

\begin{lemma} \label{ff to ia}  Given $(\mF, \mW)$ with $\mW$ ordered and untwisted, then the switch out discs are canonically chosen  and there is a corresponding reordering $w_{i_1}<w_{i_2}< \cdots <w_{i_n}$ of the elements of $\mW$, called the \emph{$\mF$-reordering} with the property that all the switch out discs appear first followed by the non switch discs and the ordering among the switch out discs and the ordering among the non switch discs is induced from that of $\mW$.  \end{lemma}

\begin{definition} \label{order induced} A minimal switching arising from an ordering as above is called an \emph{order induced switching}.\end{definition}

\begin{remark} The reader should glean from the proof below that given $\mfw$, a minimal switching $\mW^*$ is obtained by choosing as switch out discs exactly one $w\in \mW$ in each immersed cycle.  On the other hand this is an order induced switching if $\mW$ is ordered, each such $w$ is minimal among $w'\in \mW$ in its immersed cycle and these switch out discs get switched out according to their order.\end{remark}

\begin{proof}  Once the switch out discs are chosen  the ordering on $\mW$ induces the new ordering.  We now show how to choose the switch out discs.  If $(\mF, \mW)\in EA$, then there is nothing to do.  Otherwise our $(\mF\cup \mW)\cap G$ forms a union of $k\ge 1$ immersed cycles and an immersed arc which is possibly a point.  Let $w_{i_1}$ be the minimal element of $\mW$ lying in some cycle.  Having inductively constructed $w_{i_1}, \cdots, w_{i_p}, p<k$, then define $w_{i_{p+1}}$ to be the least element lying in a cycle that does not contain any of $w_{i_1}, \cdots, w_{i_p}$.  These $w_{i_1}, \cdots, w_{i_k}$ are our proposed  switch out discs and with respect to the new ordering on $\mW$, Corollary \ref{switch exists} shows that there is a $k$-switching with these as the switch out discs.\end{proof}   

\begin{convention}  \label{minimal switch}  Given $\mfw$ with $\mW$ ordered, unless said otherwise, all switchings of  $\mW$ will be order induced. In particular they will be minimal.  In \S10 we will consider switchings that may not be minimal.  Going forward we denote the elements of $\mW$ by $\{w_1, \cdots, w_n\}$ whose indicated order coincides with its  canonical reordering.  \end{convention}

 Certain symmetric finger and Whitney germed discs can be represented by arcs in $G$ with endpoints in $R\cap G$.  
 
 \begin{definition}  \label{1d} (1D representations of Whitney discs)  Let $R\cup G$ be in framed finger form, $\alpha\subset G$  an embedded path with $\alpha\cap R=\partial \alpha$, one point being positive and one negative, both ends of which approache $R$ from either the positive or negative side.   Then as in Construction \ref{1d construction}, $\alpha$ naturally extends to a framed Whitney disc $D_\alpha$ called the \emph{framed extension} well defined up to isotopy. See Remark \ref{1d remarks} i).     Figure \ref{fig:1d rep} a) shows  two examples of arcs and Figure \ref{fig:1d rep} b) shows their finger and Whitney germed framed extensions. \end{definition}

  \begin{figure}[!htbp]
    \centering
    \includegraphics[width=1.0\linewidth]{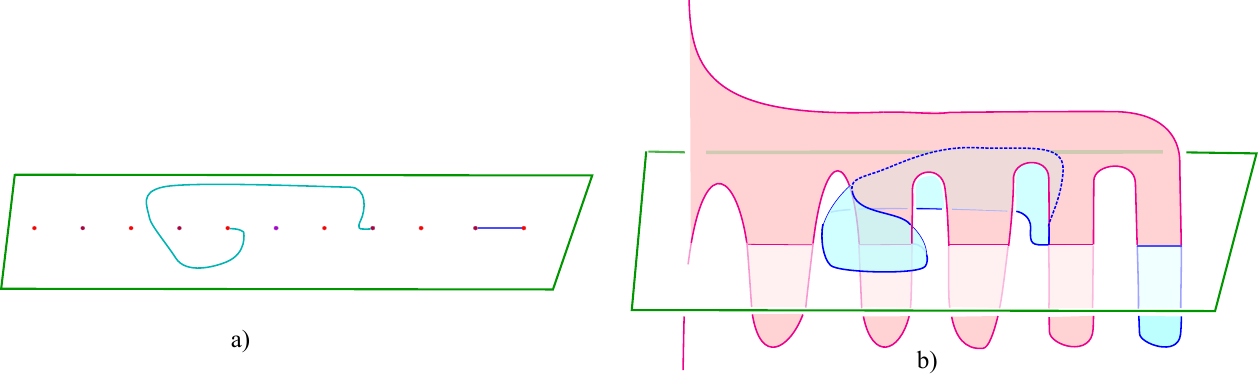}
    \caption{1D representations of Whitney discs}
    \label{fig:1d rep}
    \end{figure}

 \begin{construction} \label{1d construction} After a small isotopy of $R$, there exists a subdisc $F_R\subset R$ which 1-1 vertically projects to an embedded disc $F_G\subset G$ containing $(\mW\cup\mW^*)\cap G$, where $\mW^*$ is a full switch of $\mW$.  Therefore, any arc $\alpha$ as in Definition \ref{1d} can be isotoped to lie in $F_G$.  A framed extension $D_\alpha$ can be constructed so that the vertical projection to $G$ equals $\alpha$.  In general while Whitney framed this $D_\alpha$ is usually not finger or Whitney germed.  \end{construction}
 
\begin{remarks} \label{1d remarks}  i)  There are two ways of isotoping $R$ to have the desired projection property.  In the local model $R$ lies in the $y=0$ plane.  It can be isotoped so that the points with $t\ge 0$ are pushed to $y\ge 0$ and those with  $t\le 0$ pushed to $y\le 0$ or vice versa.  Once the convention is fixed our 1d representation is unique up to isotopy.  Switching the convention changes the representation by reflection in the $x,y,z$ plane.  

ii) The $\alpha$'s used in this paper are all sufficiently simple so that one can readily isotope $D_\alpha$ to either be Whitney germed or finger germed as needed.   \end{remarks}

The following result is needed for \S 8, where we prove that $\mD(\mF, \mW)$ is independent of the ordering on  $\mW$.

\begin{lemma}  \label{unknotted finger} Given $(\mF, \mW)$ there exists a complete set $\hat\mW$ of Whitney discs that similarly matches $\mF$ with the following property.  Given any ordering on $\mW$ there exists an order induced switching $\mW_1$ that is also an order induced full-switching of $\hat\mW$ for some ordering on $\hat \mW$.\end{lemma}

\begin{proof} We will find a $\hat\mW$ which similarly matches $\mF$ with the property that given any ordering on $\mW$ there exists an order induced switching $\mW^*$ such that  $\hat\mW\cap\mW^*\subset R\cap G$.  Since $(\hat\mW,\mW^*)\in EA$ it follows that $\mW^*$ is an $n$-switch of $\hat\mW$.  

View $\mW$ as the standard Whitney discs.  After an initial isotopy of $R\cup G$ which fixes $\mW$ setwise, takes framed finger form to framed finger form and giving the discs of $\mW$ the induced renaming and also renaming $\mF$ we can assume that $w_1, f_1, w_2, f_2, \cdots, w_{i_1}, f_{i_0}$ forms an immersed arc, which may be a point in which case $i_0=0$.  Also, $w_{i_0+1}, f_{i_0+1}, \cdots, w_{i_1}, f_{i_1};\\ \cdots; w_{i_k+1}, f_{i_k+1}, \cdots, w_{n}, f_{n}$ form $k$-immersed cycles.  Now choose $\hat\mW$ as  Whitney  discs which similarly match $\mF$ and such that $(\hat\mW\cap\mW)\subset G\cap R$ where a typical $(\hat\mW\cap\mW)\cap G$ appears as in Figure \ref{fig:nswitch1}  with the same pattern appearing in $(\hat\mW\cap\mW)\cap R$.  Here $\hat\mW$ is shown in black and Whitney germed and $\mW$ is blue and finger germed. 

   \begin{figure}[!htbp]
    \centering
    \includegraphics[width=.6\linewidth]{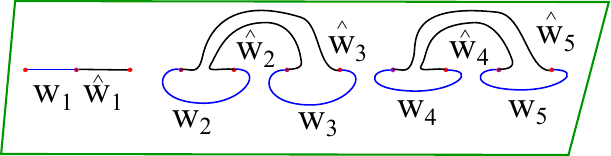}
    \caption{Constructing $\hat W$}
    \label{fig:nswitch1}
\end{figure}

Given $\hat\mW$ and an  ordering on $\mW$, the relevant data for determining an order induced $k$-switch is first, for each $i$ picking out a single $w_j$ in the i'th cycle and second, ordering these $w_j$'s.   For a given order on $\mW$ one readily constructs  the desired order induced $k$-switch $\mW^*$ as in Figures \ref{fig:nswitch2} a), b).  Those figures show the construction, restricted to $G$ with the same pattern appearing in $R$, of the 2-switches $\mW_1$ and $\mW_2$ of $\mW$, where $w_2, w_5$ are the picked out discs and $\mW_1$ arises from $w_2<w_5$ and $\mW_2$ arises from $w_5<w_2$.   \end{proof}

   \begin{figure}[!htbp]
    \centering
    \includegraphics[width=1.0\linewidth]{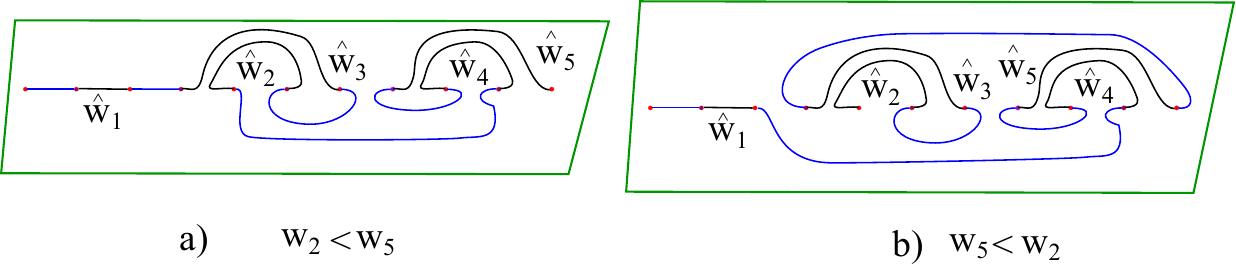}
    \caption{a) Constructing $\mW_1\qquad\qquad$ b) Constructing $\mW_2$}
    \label{fig:nswitch2}
\end{figure}

Given $(\mF,\mW_1)\in$ IA  where $\mW_1$ is an order induced switching of the ordered  $\mW$, then the ordering on $\mW_1$ induced from $\mW$ will usually be different from the induced IA-ordering on $\mW_1$ from Definition \ref{invariant}.  The next lemma explains how to compute the latter.

\begin{lemma} \label{ia order} Given $(\mF,\mW) $ with $\mW$ ordered  such  that $\mW_1$ is an order induced switching of $\mW$, then the IA-ordering on $\mF\cup\mW_1$ restricted to $\mW_1$ is as follows.  Suppose that $(\mF\cup \mW)\cap G$ is a union of an arc $\sigma_0$, which is possibly a point, and cycles $\sigma_1, \cdots, \sigma_k$.  Then the IA ordering on $\mF\cup\mW_1$ is determined by the following two rules.  

i) The IA ordering on $((\mF\cup\mW^1)\cap (\sigma_i\setminus w_i))\cup w^1_i$ is induced from the linear order on that immersed segment having $w^1_i$ as the maximal element. 

ii) If $r<s$, then all elements in $\sigma_s\cup w^1_s$ appear before elements in $\sigma_r\cup w^1_r$.  See Figure \ref{fig:new order}.\qed\end{lemma}

   \begin{figure}[!htbp]
    \centering
    \includegraphics[width=1.0\linewidth]{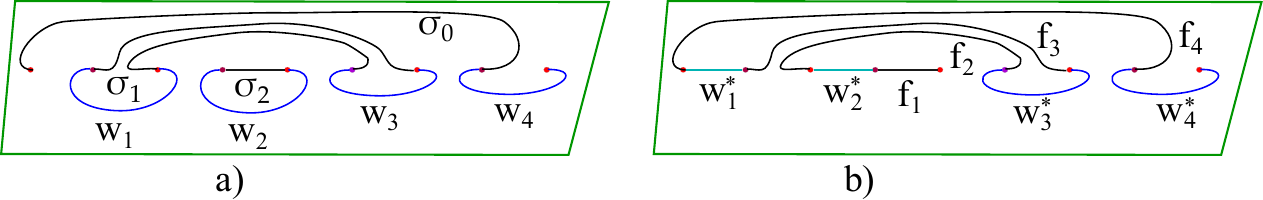}
    \caption{a)  $\mW$ has the standard ordering\\            b) The IA Ordering: $f_1<w_2^*<f_2<w_3^*<f_3<w_1^*<f_4<w_4^*$}
    \label{fig:new order}
\end{figure}

\begin{definition} \label{restandardization}  Given $R$ and $G$ in framed finger form, then a \emph{restandardization map} is the time $1$ map $\psi$ of an ambient isotopy that fixes $G\cup R$ setwise, fixes the standard Whitney discs $\mW$ pointwise and such that $\psi(R\cup G)$ is in framed finger form.

If now $\mW$ is any complete system of untwisted Whitney discs for a general $R$ and $\Gs$ with standardizing map $\kappa$ and   $\mW_1$ is a $\mW$-framed $k$-switch of $\mW$, then $\mW_2=\kappa^{-1}\circ\psi\circ\kappa (\mW_1)$ is also a $\mW$-framed $k$-switch of $\mW$ called a \emph{restandardization} of $\mW_1$.

If $\mW$ is ordered, then define $\I(\mF, \mW)=\I(\mF, \mW^*)$ where $\mW^*$ is an order induced switching, not necessarily $\mW$-framed.
\end{definition}

The following is the main result of this section.  It will follow from the next three lemmas and Corollary \ref{homological independence}.

\begin{proposition} \label{switch invariance}  If the finger/Whitney system $(\mF, \mW)$ has $\mW$ is ordered, then $\I(\mF,\mW)$ is well defined.\end{proposition}

\begin{lemma}\label{standardizing} If $\mW$ is a  complete set of ordered untwisted Whitney discs with standardization map $\kappa$, then any $\mW$-framed order induced switching  $\mW_1$ of $\mW$ is a restandardization of $\kappa^{-1}(\mW_0)$, where $\mW_0$ is the standard minimal switching of $\kappa(\mW)$.\end{lemma}

\begin{lemma}  \label{ff independence}  Given $(\mF, \mW)$ with $\mW$ ordered and $\mW_1, \mW_2$  order induced switchings such that $\mW_2$ is a restandardization of $\mW_1$, then $\I(\mF, \mW_1)=\I(\mF,\mW_2)$.\end{lemma}

\begin{lemma} \label{wframe independence}  Given $\mF$ and  ordered isotopic sets of untwisted Whitney discs $\mW_i, i=1,2$ with $\mW_i$-framed order induced switchings  $\mW_i^*$, then $\I(\mF, \mW_1^*)=\I(\mF,\mW_2^*)$.\end{lemma}

\noindent\emph{Proof of Proposition \ref{switch invariance} assuming Lemmas \ref{standardizing} - \ref{wframe independence}}:  After  initial individual isotopies of $\mF$ and $\mW$ supported near $\partial \mF$ and $\partial \mW$ we can assume that $\mF$ and $\mW$ are untwisted.  Up to ambient isotopy the ambiguity of this process is changing $\mW$ by R and G-twisting.

	We define $\I(\mF, \mW)=\I(\mF, \mW^*)$ where $\mW^*$ is an order induced $\mW$ framed switching.  Lemma  \ref{standardizing} implies that any two choices are equal up to strong $H_2$-equivalence and restandardization.  Lemma \ref{homological independence} implies that $\I(\mF, \mW)$ is well defined up to $H_2$-equivalence and Lemma \ref{ff independence} implies that $\I(\mF, \mW)$ is well defined up to restandardization.  Finally Lemma \ref{wframe independence} implies that $\I(\mF, \mW)$ is well defined up to R and G-twisting of $\mW$.\qed
	
\begin{corollary} \label{i-invariance boundary twisting} If $\mW$ is ordered and $(\mF', \mW')$ is obtained from $\mfw$ by boundary twisting then $\I(\mF', \mW')=\I\mfw$.\end{corollary}

\begin{proof} $\I$-invariance under boundary twisting $\mW$ follows from Lemma \ref{wframe independence}.  If $\mfw\in$ IA, then the result follows from Proposition \ref{hsf independence}.  Otherwise, to compute $\I\mfw$ we first apply order induced switching to $\mW$ to obtain $(\mF, \mW^*)\in$ IA.  Thus, if $\mF'$ is obtained by boundary twisting $\mF$ we conclude $\I(\mF', \mW)=\I(\mF', \mW^*)=\I(\mF, \mW^*)=\I\mfw$.\end{proof}

\vskip 8pt
\noindent\emph{Idea of the Proof of Lemmas \ref{standardizing} and \ref {ff independence}}:  To prove Lemma \ref{standardizing} we will describe six \emph{basic} restandardization maps and show that if $\mW$ is the standard ordered Whitney system for $R$ and $G$ in  framed finger form with $\mW_0$  the standard k-switch, then up to $H_2$-equivalence, any $\mW$-framed $k$-switch $\mW_1$ is a restandardization of $\mW_0$ via a composition of these basic maps. 
\vskip 8pt
  Given $R$ and $G$ in framed finger form we first describe how, after a suitable isotopy, an untwisted Whitney germed $\mF$ appears near the \emph{solid fingers}.  One by one we then describe the six basic restandardization maps and for each such map $\psi$, show that if $R$ and $G$ are in framed finger form with $(\mF,\mW_1)\in$ IA where $\mW_1$  is a  $\mW$-framed k-switch of the standard ordered $\mW$, then  $\I(\mF, \psi(\mW_1))=\I(\mF, \mW_1)$.  The main observation being that except for tubes disjoint from $\mF$, $\psi$ is supported near $R\cup G \cup \textrm{fingers}$, so with this explicit understanding we can compare the invariant using Lemma \ref{parity}.  

\vskip 10pt
\noindent\emph{Proof of Lemma \ref{wframe independence}:} Since $\mW_1, \mW_2$ differ by G and R-twisting it follows  by Lemma \ref{standardizing} that we can assume that $\mW_1^*$ and $\mW_2^*$ differ up to restandardization and $H_2$-equivalence by G and R-twisting.  Since by Lemma \ref{ff independence} restandardization does not change the invariant and sequences of twistings are special cases of disc sliding sequences the result follows by Lemma \ref{hsf independence}.\qed

\vskip 10pt

We now prove the first two lemmas simultaneously.  For the remainder of the section any switch of the standard ordered $\mW$ will be $\mW$-framed.

\begin{definition} \label{solid finger} Viewing the framed finger form as in Figure \ref{fig:framed finger} b) we define the \emph{solid finger} $\sigma_F$ as the convex hull of the finger $F$ with $w$ denoting its standard Whitney disc.  Two local views of a solid finger are shown in Figure \ref{fig:solidfinger1}.  Figure \ref{fig:solidfinger1} a) shows  the $z=0$ 3D slice and Figure \ref{fig:solidfinger1} b) shows the 3D slice $y=0$.  Note that $F$ is contained in the $y=0$ plane, $\sigma_F$ is a 3-ball with corners where $\partial \sigma_F=F\cup \partial_e\sigma_F$, where $\partial_e$ is called the \emph{exterior boundary}.  Let $g_{\sigma_F}$ denote the arc $\sigma_F\cap G$ and $r_{\sigma_F}:=\sigma_F\cap R$.  We will drop the subscript $F$ when clear from context.\end{definition}

   \begin{figure}[!htbp]
    \centering
    \includegraphics[width=.6\linewidth]{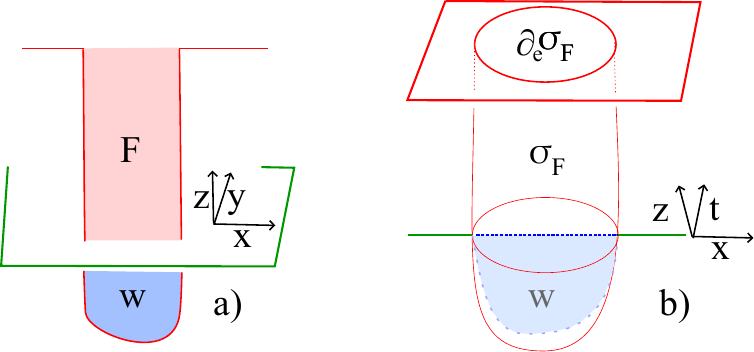}
    \caption{Two Views of the Solid Finger $\sigma_F$}
    \label{fig:solidfinger1}
    \end{figure}

\begin{lemma} \label{f in solid}  Let $R$ and $G$ be in framed finger form, $\sigma:=\sigma_F$  a solid finger with associated standard Whitney disc $w$ and $f$  another Whitney germed Whitney disc with possibly $f\cap w\neq \emptyset$.  Then $f $ can be isotoped, as a Whitney germed Whitney disc, so that each component $C$ of  $f\cap\sigma$ appears as in Figure \ref{fig:solidfinger2} with a neighborhood as described below.  There are four cases.
\vskip 8pt
a) $C\cap F=\emptyset$:  Here $|C\cap w|=1$ and $C\cap \sigma$.  The $t=0, y=0, x=0$ 3D slices of $N(C)$ are shown in \ref{fig:solidfinger2} a). 
\vskip 8pt
b) $C\cap G\neq\emptyset$ and $C\cap R=\emptyset$:  Here $C\cap \sigma$ is an arc with one endpoint on $G$ and $N(C)\cap G$ is an arc with one endpoint on $\partial_e\sigma$.  See Figure \ref{fig:solidfinger2} b) which shows $x=0$ and $y=0$ 3D slices.  
\vskip 8pt
c)  $C\cap R\neq\emptyset$ and $C\cap G=\emptyset$:  Here $C\cap \sigma$ is a disc  where $\partial C$ is the union one arc on  $C\cap R$ one  arc $\delta C$ and two arcs on $\partial_e\sigma$.  Further $\delta C$ extends in the $y\ge 0$ direction.  See Figure \ref{fig:solidfinger2} c) which shows $x=0$ and $y=0$ 3D slices.  While $C$ is smooth it is drawn with a corner.
\vskip 8pt
d) $C\cap R\cap G\neq\emptyset$:  First $|C\cap R\cap G|=1$ and there are two possibilities for that.  Here $C\cap \sigma$ is a disc with $\partial C$ consisting of an arc $\subset R$ and an arc $\subset G$ and an arc $\subset\partial_e\sigma$ and the arc $\delta C$ which extends in the $t\ge 0$ direction.   See Figure \ref{fig:solidfinger2} d) which shows $y=0$ and $x=0$ 3D slices.  While $C$ is smooth it is drawn with a corner. \end{lemma}

\begin{proof} This follows from the following two observations.  First, $\sigma_F$ deformation retracts to $F\cup w$.  Second, since $f$ is Whitney germed it's neighborhood near $f\cap (G\cup R)$ is determined.  In Case b) we can insist that $N(C)$ has $t\ge 0$ at the cost of creating a component of type a).  Similarly in Case c) we can insist that $\delta C$ continues to $y\ge 0$ at the cost of creating an extra type a) component.  Similarly in Case d) we can insist $C\cap R$ has $y\ge 0$ and $N(C)$ has $t\ge 0$ at the cost of extra components of the other types.   \end{proof} 

  \begin{figure}[!htbp]
    \centering
    \includegraphics[width=.9\linewidth]{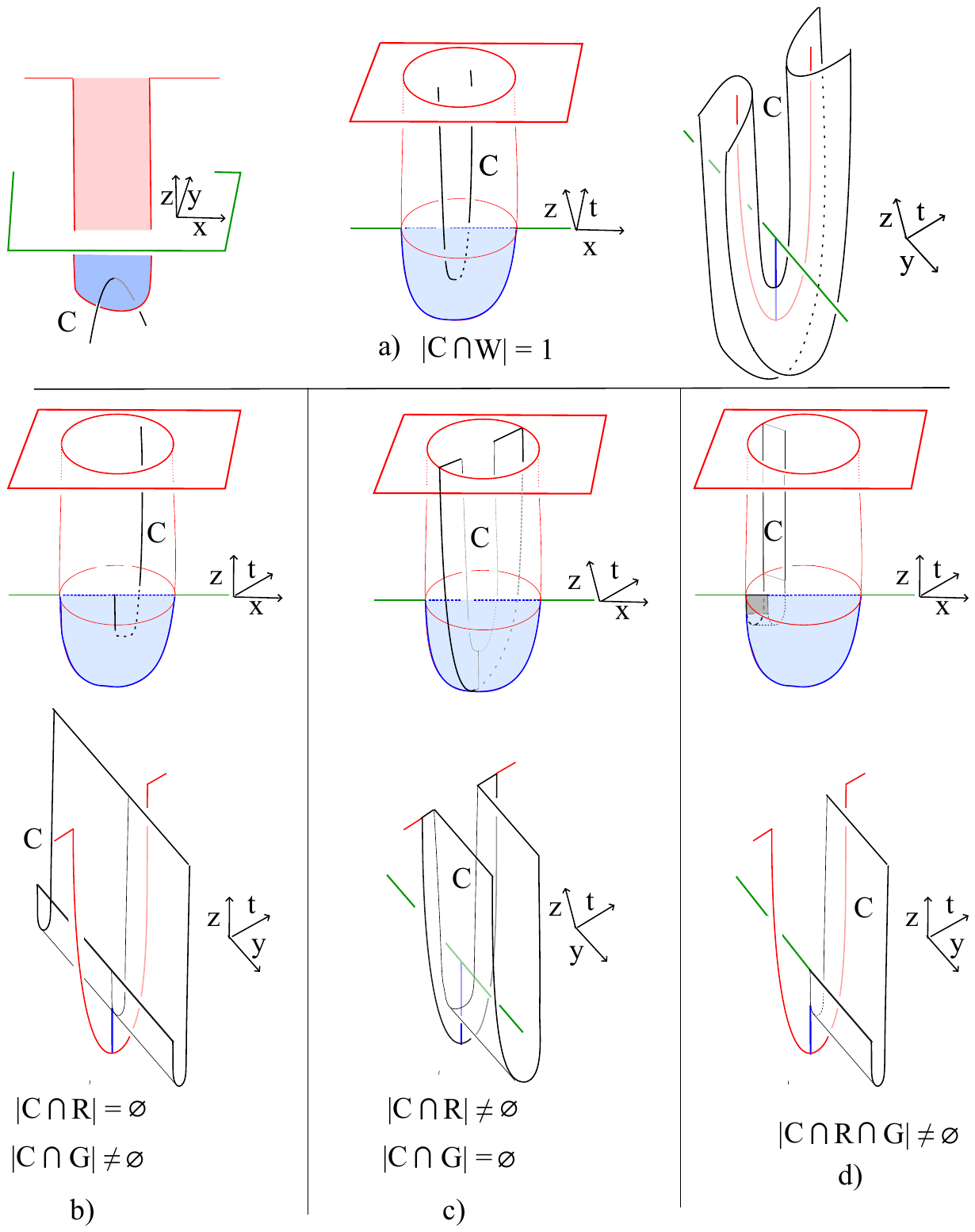}
    \caption{Types of Components of a Whitney germed $f$ in a Solid Finger}
    \label{fig:solidfinger2}
    \end{figure}

We now define the six basic restandardization maps and for each one show that they do not change our proposed invariant.  We will then show that these six maps suffice.  We will assume that $R$ and $G$ are in framed finger form.  We need a preliminary lemma.

\begin{definition} \label{winding} Let  $w$ an untwisted Whitney disc either Whitney or finger germed.  Let $w'$ another untwisted Whitney disc, either Whitney or finger germed, with $w'\cap R\cap G=w\cap R\cap G$.  Here $G\setminus \inte N(w\cap R\cap G)$ is an annulus and $w\cap G$ is a fiber of an $I\times S^1$ product structure.  Define the \emph{$G$-winding} of $w'$ \emph{with respect to $w$} to be the mod-2 number of times $w'\cap G$ (resp. $w'\cap R$) winds about this $S^1$ factor with $R$-winding defined similarly. Define the \emph{winding} of $w'$ with respect to $w$ to be the sum of the $G$ and $R$-windings. \end{definition}

\begin{lemma}\label{even winding} The winding  of $w'$ with respect to $w$ is 0 mod 2.\end{lemma}

\begin{proof} Using $w_2(S^2\times S^2)=0$, the proof follows essentially that of Lemma 5.3 \cite{Ga3}.\end{proof}

\begin{definition} \label{twist} We define the \emph{G-finger twisting} restandardization map $\psi$.  Given $i$, this map is obtained by first twisting the finger near $w_i$ by $2\pi$ preserving the framed $G$ setwise and extending to the rest of the solid finger.  Figure \ref{fig:solidfinger3} a) shows the switch disc $w_i^*$ and Figure \ref{fig:solidfinger3} b) shows $\psi(w_i^*)$.  
The effect of  G-finger twisting on $G$  is a Dehn twist on $\partial N(w_i)\cap G$.   We can arrange that the effect on $R$ is a Dehn twist on $\sigma_RF$ is also a Dehn twist which in particular fixes the normal framing of $R$ setwise.   Note that $\psi$ is isotopic to $\id$ via an isotopy supported near $\sigma_F$ and fixing the framed $G$ setwise.\end{definition}

 \begin{figure}[!htbp]
    \centering
    \includegraphics[width=.7\linewidth]{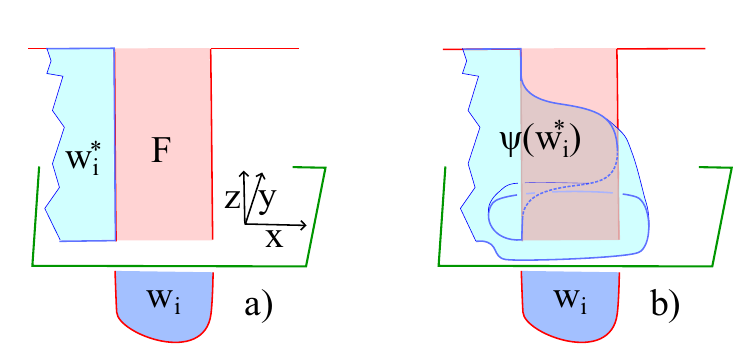}
    \caption{$G$-Finger Twisting a Switch Disc}
    \label{fig:solidfinger3}
    \end{figure}

\begin{lemma} \label{twist invariance} Given $\mF$, let $\mW_1=\{w_1^1, \cdots, w_n^1\}$ be a minimal switching of $\mW$.  Let $\psi$ be a G-finger twist map with $\mW_2=\psi(\mW_1)$.  Then $\I(\mF, \mW_1)=\I(\mF, \mW_2)$.\end{lemma}

\begin{proof} Here $G\cup R$ are in framed finger form, $\mW$ are the standard Whitney discs with $\mW_1$  a  $k$-switch.  However, for $i=1,2$ we view $\mW_i$ as finger germed and $\mF$ as Whitney germed.
It suffices to consider the case of a $j$-twist where $j\le k$, since when $j>k$, $G$-finger twisting acts trivially on nonswitch discs.  Also when $ j\le k$, the $j$-twist only modifies $w^1_j$ and $w^1_{j+1}$, unless $j=k$, in which case only $w^1_k$ is modified.     
\vskip 8pt
\noindent\emph{Case 1}: There exists $f\in \mF$ that matches $w_j$.
\vskip 8pt
\noindent\emph{Proof of Case 1}:   Let $f_i\in \mF$.  To start with for $r=j$ or $ j+1$ if $j<k$, $\psi(w_r^1) \cap G$ is obtained by doing a Dehn twist along a  circle $c$ about $w_j \cap G$ and for all $i, c\cap f_i $ is even.  Thus $|w_r^1\cap f_i\cap G|=|\psi(w_r^1)\cap f_i\cap G|$ mod 2.  A similar argument shows that $|w_r^1\cap f_i\cap R|=|\psi(w_r^1)\cap f_i\cap R|$  mod 2.  By Lemma \ref{parity} it suffices to show that $\hat \I(\mF,\mW_1)=\hat \I(\mF, \mW_2)$.

Let $\sigma$ be the solid finger that contains $\partial w_j$.  If $f_i\neq f$, then $ f_i$ intersects the solid finger $\sigma$ only in components of type a), b) and c).  For $r=j,j+1$, each component of type a) contributes 2 to $|\inte(w^1_r)\cap f_i|$ and each component of type b) or c) contributes 1.  If in the IA ordering $w^1_{j+1}<w^1_{j}<f_i$, then none of the $\psi(w_j), \psi(w_{j+1})$ intersections contribute to $\hat \I(\mF, \mW_2)$.  If $f_i<w^1_{j+1}<w^1_{j}$, then each new $\psi(w_j)$ intersection is paired with a new $\psi( w_{j+1})$ intersection so again $\hat \I(\mF, \mW_2)$ is unchanged.  In the special case $j=k$ we never have $f_i<w_j^1$.  Indeed, in that case since $f$ matches $w_j$ we have $k=n$.  This implies that $w^1_j$ is the minimal element in $\mW^1$ in the IA ordering and hence $f_i<w^1_j$ implies $f_i=f$.  

We now consider when $w^1_{j+1}<f_i<w^1_{j}$ which is equivalent to $f_i=f$, unless $k=j$ in which case we  just have $f<w^1_j$.  Here only the new intersections of $\inte(\psi(w^1_j))\cap f$  count towards $\hat \I(\mF, \mW_2)-\hat \I(\mF, \mW_1)$.  Again, components of $f\cap \sigma$ of type a) contribute 0 mod 2 and 
components of type b) and c) each contribute 1.  The two type d) components each contribute 1. Recall that we use the special form of type d) components.  The G-winding of $f$ with respect to $w_j$  is the number of type $b)$ components $:= b$ and the R winding
  is the number of c) components $:= c$.    We conclude that the total winding is $b+c$.   Since by Lemma \ref{even winding} the mod-2 winding equals 0, the proof of Case 1 follows. \qed

\vskip 8pt
\noindent\emph{Case 2}: There exists no $f\in \mF$ that matches $w_j$.
\vskip 8pt

\noindent\emph{Proof of Case 2}:  Being in Case 2, we have $k<n$ and case $w_j\cap G$ is part of an immersed cycle $\omega$ of $2r$ arcs, $r>1$,  of the form $w_{p_1}\cap G, f_ {q_1}\cap G, \cdots, w_{p_m}\cap G, f_{q_m}\cap G$, where  $w_j=w_{p_1}$.   See Figure \ref{fig:gtwist1} a) which shows an example with $j=1$. As noted above we can assume that $j\le k$.   Unlike in Case 1, $\mW_1$ and  $\mW_2:=\psi(\mW_1)$ will  not satisfy hypothesis iii) of Lemma \ref{parity}. Indeed it will fail along the discs of $\mF$ with endpoints on $w_j\cap R\cap G$.  Compare Figure \ref{fig:gtwist1} b) with Figure \ref{fig:gtwist1} c).  However, we can modify $\mW_2$ to $\mW_3$ by a sequence of disc slides  so that $\mW_1$ and $\mW_3$ satisfy the conclusions of Lemma \ref{parity}.  See Figure \ref{fig:gtwist1} d).  For each element $\hat w\neq w_j$ of $\mW$ in the cycle containing $w_j$, do a G-disc slide to $w_j^2$ and if necessary $ w_{j+1}^2$ (i.e if  $j<k$) over $\hat w$ and similarly do an R-disc slide  of $w_j^2$ and if necessary $w_{j+1}^2$ over $\hat w$.  The resulting $\mW_1$ and $\mW_3$ now satisfy condition iii) of Lemma \ref{parity}, independent of the choice of disc slides.  If we can show $\hat \I(\mF, \mW_1)=\hat \I(\mF, \mW_3)$ it then follows by Lemma \ref{parity} that $\I(\mF, \mW_1)=\I(\mF, \mW_3)$ and hence $\I(\mF, \mW_2)=\I(\mF,\mW_3)=\I(\mF, \mW_1)$, since by Proposition \ref{hsf independence} the invariant is independent of choice of disc slide sequence.

 \begin{figure}[!htbp]
    \centering
    \includegraphics[width=1.0\linewidth]{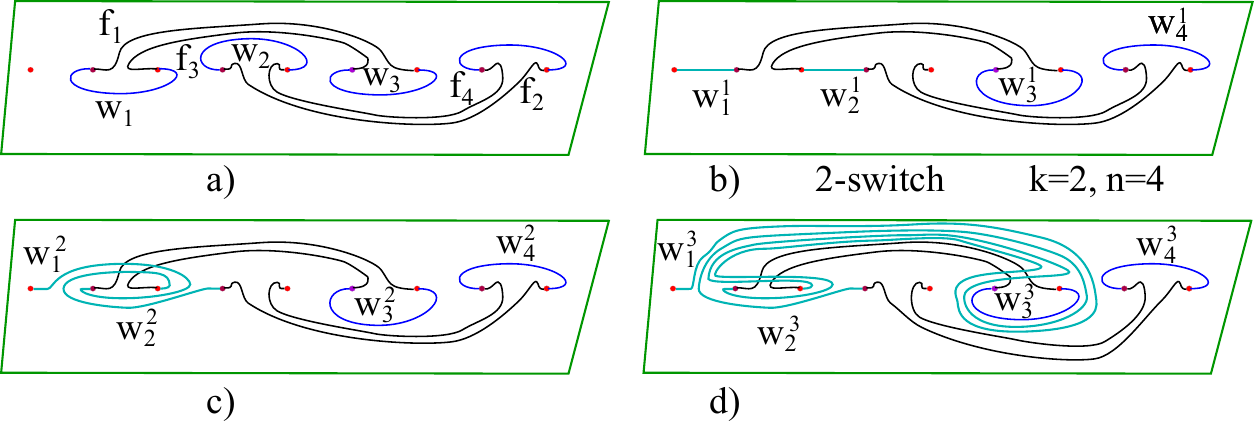}
    \caption{Invariance under $G$-Finger Twisting}
    \label{fig:gtwist1}
    \end{figure}

We prove $\hat \I(\mF, \mW_1)=\hat \I(\mF, \mW_3)$.  First consider the case $j<k$.  By Lemma \ref{ia order} the IA ordering on $(\mF,\mW_3)$ contains the following linear sequence
$w_{j+1}^3<\phi_1<\tau_1<\phi_2<\tau_2<\cdots<\phi_m<w_j^3$ where $\{\phi_1, \cdots, \phi_m\} = \mF\cap \omega$ and $\{\tau_1, \cdots\tau_{m-1}\}=(\mW\setminus w_j)\cap \omega$.  
We now catalogue the difference between $\mF\cap \inte(\mW_1)$ and $\mF\cap\inte(\mW_3)$.  First, the only changes involve $w_{j+1}^3$ and $w_j^3$.   Since we obtain $w_j^2$ from $w_j^1$ by modifying it explicitly near the finger $\sigma$, the former keeps the $w_j^1$ intersections and possibly picks up new intersections near $\sigma$ as in Case 1). Since $w_j^3$ is obtained from $w_j^2$ by a G-disc slide and an R-disc slide it keeps the $w_j^2$ intersections.  If $w_j^2$ has a G-disc slide over say $\hat w$ it also has a R-disc slide over $\hat w$.  From these disc slides our $w_j^3$ picks up four intersections for each point of $\mF\cap \hat w$, two from the G-disc slide and two from the R-disc slide.  It also picks up one intersection for each point of $\inte(\hat w\cap R)\cap \mF$ from the G-disc slide and one for each point of $\inte(\hat w\cap G)\cap \mF$ from the R-disc slide.  Figure \ref{fig:disc slide} b) shows the new intersections coming from the G-disc slide.  Similar new intersections happen when comparing $w_{j+1}^1$ with $w_{j+1}^3$.

If in the IA-ordering $f<w_{j+1}^3<w_j^3$, each of $\inte(w_{j+1}^3), \inte(w_j^3)$ pick cancelling new intersections with $f$.    If $w_{j+1}^3<w_j^3<f$, then none of these intersections count towards $\hat \I(\mF, \mW_3)$ or $\hat \I(\mF, \mW_1)$.  It remains to consider those $w_{j+1}^3<f<w_j^3$ which are exactly those $f\in \omega\cap \mF$.  Therefore it suffices to show that $$\sum_{i=1}^m |\inte(\phi_i\cap R)\cap((w_j\cup \tau_1\cup\cdots\cup\tau_{m-1})\cap R)|+\sum_{i=1}^m |\inte(\phi_i\cap G)\cap((w_j\cup \tau_1\cup\cdots\cup\tau_{m-1})\cap G)|=0\ \textrm{mod}\ 2.$$ 

Note that the terms in this sum of the form $|\inte(\phi_i)\cap G\cap w_j|$ and  $|\inte(\phi_i)\cap R)\cap w_j|$ respectively correspond to the new intersections with $\inte(w_j^2)$ coming from type b) and type c) components  of $\phi_i\cap\sigma$.  Also there are two type d) components, one each from $\phi_1$ and $\phi_n$, but these contribute one each when comparing $\hat \I(\mF, \mW_1)$ and $\hat \I(\mF, \mW_3)$ and hence cancel.

We give the proof of the above equality for $m=3$, the general case being similar.   Let  $u_1, u_2, u_3$ be untwisted finger germed discs which similarly match $\phi_1, \phi_2, \phi_3$ such that $(u_1\cup \tau_1\cup u_2\cup \tau_2\cup u_3\cup w_j)\cap G$ and  $(u_1\cup \tau_1\cup u_2\cup \tau_2\cup u_3\cup w_j)\cap R$ form  embedded cycles $\alpha_G, \alpha_R$ respectively in $G$ and $R$, where we ignore the intervals where they coincide. 
We show that equation is true because the $\phi_i$'s are pairwise disjoint and have 0-winding with respect to the $u_i$'s.  First, the ends of each $\phi_i\cap G$ and $\phi_i \cap R$ occur on the same sides of $\alpha_G$ and $\alpha_R$ so for all $i,\  |\phi_i\cap \alpha_G|=|\phi_i\cap \alpha_R|=0$ mod 2.  Second, 0-winding implies that for every $i,\  |\phi_i\cap u_i\cap G|+|\phi_i\cap u_i\cap R|=0$ mod 2.  Third, each $\phi_p\cup u_p\cap G$ is an immersed cycle $\beta_p^G$ and if $p\neq q$, then $u_p\cap u_q=\emptyset, \phi_p\cap \phi_q=\emptyset$ and  $|\beta_p^G\cap \beta_q^G|=0$ mod 2  which imply $|\phi_p\cap u_q\cap G|+|\phi_q\cap u_p\cap G|=0$ mod 2 with similar results holding for $R$.  Therefore in total the $\phi_p$'s intersect the $u_q$'s an even number of times and hence they intersect the $\tau_i$'s $\cup w_j$ an even number of times.\end{proof}

\begin{definition} \label{g braiding}  We define \emph{G-braiding}.  Informally it's the time one map of an ambient point pushing map on $G$ which sends $w_i$ about $w_j$ and fixes $N(R)$ pointwise.  Formally, let $\mW$ be the standard Whitney discs with $R\cup G$ in framed finger form.  For $p=i,j$, let $\alpha_p\subset G$ be a circle about $w_p\cap G$.  If $i<j$, then let $\beta$ be an embedded arc from $\alpha_i$ to $\alpha_j$ disjoint from $\mW$ and let  $\alpha_{ij}$ be the component of $\partial N(\alpha\cup \alpha_i\cup \alpha_j)$ not isotopic to $\alpha_i$ or $\alpha_j$.  Define $\eta_{i,j}:G\to G$ to be the composition of  L Dehn twists about $\alpha_i$ and $\alpha_j $ with a R Dehn twist about $\alpha_{ij}$.  See Figure \ref{fig:gbraiding1} a).   Since $G$ has a positive side, right and left make sense.  Note that $\eta_{i,j}$ is isotopic to  the time one map of an ambient isotopy $\zeta^t_{i,j} $ of $G$ which  sends $w_i$  to $w_j$ along $\beta$, then swings it counterclockwise about $w_j$ then returns it along $\beta$ to its original position.   See Figure \ref{fig:gbraiding1} b).  Now extend $\zeta^t_{i,j}$ to an ambient isotopy of $\stwostwo$ that fixes each $F_r, r\neq i$   pointwise such that $F_i$ goes into the future as its 3D projection crosses $F_j$.     The time one map $\psi$ is required to fix $N(R)$ and the solid fingers pointwise.   The map $\psi$ is  a \emph{$G$-braiding} of $F_i$ about $F_j$.\end{definition}

 \begin{figure}[!htbp]
    \centering
    \includegraphics[width=.7\linewidth]{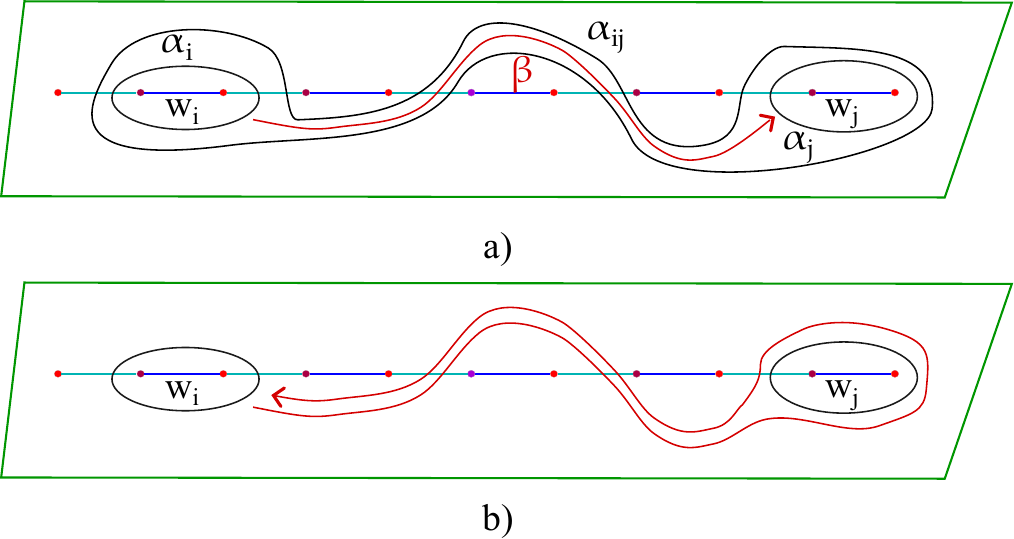}
    \caption{$G$-Braiding}
    \label{fig:gbraiding1}
    \end{figure}

\begin{lemma}\label{gbraiding invariance} Given $\mF$ and let $R$ and $G$ be in framed finger form with $\mW$ the standard Whitney discs.  Suppose $\mW_1$ is a minimal switching of $\mW$.  If $\mW_2:=\psi(\mW_1)$ where $\psi$ is a $G$ braiding of $F_i$ about $F_j$, then $\I(\mF, \mW_1)=\I(\mF, \mW_2)$.\end{lemma}

\begin{proof}  We will prove this with $\mW_1, \mW_2$ finger germed and $\mF$ Whitney germed.  Note that $\alpha_{ij}$ is obtained by banding $ \alpha_i$ to $ \alpha_j$ along the thin band $b_{ij}\subset G$ with core  $\beta$.    We  assume that $b_{ij}$ is transverse to $\mW_1\cap G$ and call a component of $b_{ij}\cap \mW_1$ a \emph{middle arc}.

Up to $\partial_+$-compressing, if $i\le k$, $w_i^2=w_i^{1}$ union a future cap about $ w_j$ and if $i\le k-1$, $w_{i+1}^2=w_{i+1}^{1}$  union a future cap about $ w_j$.   See Figure \ref{fig:braiding3} a).  By \emph{future cap} we mean a 2-disc whose 3D projection is as in Figure \ref{fig:braiding3} b)  and which avoids the finger by going into the future.  A \emph{past cap} is one with the same 3D projection but goes into the past.  Also if $j\le k$, $w_j^2=w_j^{1}$ union a past cap about $ w_i$ and if $j\le k-1$, $w_{j+1}^2=w_{j+1}^{1}$  union a past cap about $ w_i$.  If $j=i+1\le k$, then up to boundary compressing $w^2_j$ has both a past cap about $w_i$ and a future cap about $w_j$.    Up to $\partial_+$-compressing, each middle arc in a $w_p^1$ adds two past caps about $w_i$.  To prove $(\mF, \mW_1), (\mF,\mW_2)$ satisfy the hypotheses of Lemma \ref{parity} we can ignore middle arcs since the resulting caps come in parallel pairs.  There are three cases.
  
   \begin{figure}[!htbp]
    \centering
    \includegraphics[width=1.0\linewidth]{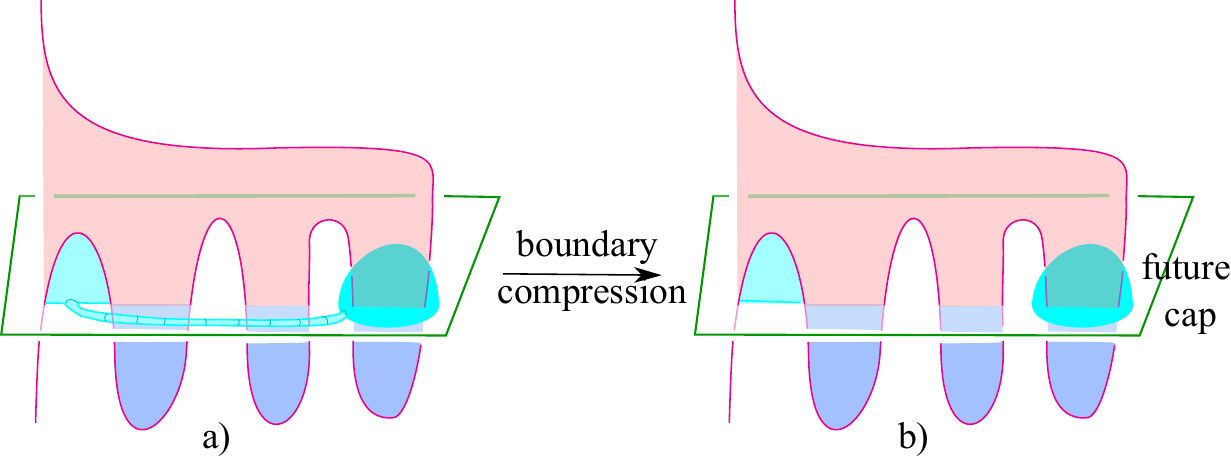}
    \caption{$\qquad\qquad$  a)  $\psi(w_1^1)$       $\qquad\qquad$        b) After the boundary compression}
    \label{fig:braiding3}
\end{figure}

\vskip 8pt 
\noindent\emph{Case 1}: There are $f_p, f_q\in \mF$ that respectively match $w_i,w_j\in \mW$ and $j\le k$.   \vskip 8pt

\noindent\emph{Proof of Case 1}:  Since the boundary of each cap around $w_i$ and $w_j$ intersects each component of $\mF$ an even number of times hypothesis iii) is satisfied for $G$.  Since $\mW_1\cap R=\mW_2\cap R$ it is also satisfied for $R$.    It remains to show $\hat \I(\mF, \mW_1)=\hat \I(\mF, \mW_2)$.  We now consider how a cap over a finger intersects  neighborhoods of  components of an $f_i$ with its solid finger.  Each component of type a) gives rise to two intersections with the cap.  Components of type b) gives rise to either 0 or two intersections.   Components type c) with gives rise to one intersection.   A component of type d) gives rise to one intersection.  In all cases caps can be future or past.  By Lemma \ref{ia order} 
 we have in the IA order $w^2_{j+1}<f_q<w^2_j\le w^2_{i+1}<f_p<w^2_i$ and the first three and last three are consecutive in the linear order, though if $j=k=n$ then we don't have the $w^2_{j+1}$ term.   Using Lemma \ref{boundary compression}, $\hat \I(\mF, \mW_2)$ is modified by the number of type c) $f_p$ components in $\sigma_j$  plus the number of c) $f_q$ components in $\sigma_i$, where $\sigma_r$ is the solid finger  corresponding to $w_r$.  But these are equal mod 2 since $(f_p\cap R)\cup (w_i\cap R)$ is an immersed cycle in R as is $(f_q\cap R)\cup (w_j\cap R)$ and these cycles intersect each other an even number of times and  $w_i\cap w_j=\emptyset$, $f_p\cap f_q=\emptyset$.  \qed

\vskip 8pt
\noindent\emph{Case 2}: Case 1 does not hold and $j\le k$.
\vskip 8pt
\noindent\emph{Proof of Case 2}:  The proof is similar to that of Case 2  of  Lemma \ref{twist invariance}.  Since at least $w^1_i, w^1_{i+1}, w^1_j$ are switch discs as is $w^1_{j+1}$ if $j<k$, $w_j\cap G$ is part of an immersed cycle of $2r$ arcs, $r\ge 1$,  of the form $w_{p_1}\cap G, f_ {p_1}\cap G, \cdots, w_{p_r}\cap G, f_{p_r}\cap G$, where  $w_{p_1}:=w_j$ and  $w_i\cap G$ is part of an immersed cycle of $2s$ arcs, $s\ge 1$,  of the form $w_{q_1}\cap G, f_ {q_1}\cap G, \cdots, w_{q_s}\cap G, f_{q_s}\cap G$, where  $w_{q_1}:=w_i$ and $r s\ge 2$.   We  modify $\mW_2$ to $\mW_3$ by a sequence of G-disc slides.   From each of $w^2_i$ and $w^2_{i+1}$ we do a G-disc slide to each $\hat w\in w_{p_2}, \cdots w_{p_r}$.  From each of $w^2_j$ and $w^2_{j+1}$, if it exists, we do a G-disc slide to each $\hat w\in w_{q_2}, \cdots w_{q_s}$.   The resulting $(\mF, \mW_1), (\mF,\mW_3)$ satisfies hypothesis iii). To complete the proof it suffices to show that $\hat \I(\mF, \mW_1)=\hat \I(\mF, \mW_3)$.  After the boundary compressions there is a natural inclusion from $\inte(\mW_1)\cap\mF$ to $\inte(\mW_3)\cap\mF$.  The new points of $\inte(\mW_3)\cap\mF$ come from intersections of caps with $\mF$ near solid fingers  plus those discs arising from the disc slides.   As before if there is a disc slide  of some $w^2_t$ over a $\hat w$, then the resulting $w^3_t$ sees pairs of intersections coming from $\inte(\hat w) \cap \mF$ plus single intersections coming from $\inte(\partial\hat w\cap R)\cap\mF$.  Most of these new  $\inte(\mW_3)\cap\mF$ intersections come in cancelling pairs.  What follows organizes those that do not. 

\vskip 8pt
i) Intersections between $w_j\cap R$ with a $f_{q_t}\cap R$.  Note that $w^3_{i+1}<f_{q_t}<w^3_i$.  These  intersections involve the $w^3_i$ cap over $w_j$ with a neighborhood of a type c) $f_{q_t}$ component in $\sigma_j$.

i') Intersections that arise from a disc slide of $w^2_i$ over the finger germed $w_{p_u}$.  This corresponds to an intersection between a $f_{q_t}\cap R$ with $w_{p_u}\cap R$ and $u\ge 2$.  Again note that $w^3_{i+1}<f_{q_t}<w^3_i$. 

\vskip 8pt
ii) Intersections between $w_i\cap R$ with a $f_{p_t}\cap R$.

ii')  Intersections between a $f_{p_t}\cap R$ with $w_{q_u}\cap R$ and $u\ge 2$.
\vskip 8pt

Now  the immersed cycle $w_{p_1}\cap R, f_ {p_1}\cap R, \cdots, w_{p_r}\cap R, f_{p_r}\cap R$ intersects the immersed cycle $w_{q_1}\cap R, f_ {q_1}\cap R, \cdots, w_{q_s}\cap R, f_{q_s}\cap R$ in an even number of points.  Since the $w_u$'s are pairwise disjoint as are the $f_v$'s there is a bijection between the intersections of the cycles and  the points corresponding to  i), i'),  ii) or ii'), proving Case 2.
\vskip 8pt
\noindent\emph{Case 3}: $j>k$.
\vskip 8pt
\noindent\emph{Proof of Case 3}:  After boundary compressing and ignoring the pairs of parallel discs arising from the middle arcs, the result of the G-braiding is that $w_i$ (resp. $w_{i+1}$) has a positive cap about the $w_j$ solid finger if $i\le k$ (resp. $i\le k-1)$.  To the extent they exist, G-disc slide each of these caps over $w_j$.  If $\mW_3$ is the result of these disc slides it follows that $(\mF,\mW_1), (\mF, \mW_3)$ satisfy conditions i), iii) of Lemma \ref{parity}, thus it remains to show that $\hat \I(\mF, \mW_1)=\hat \I(\mF,\mW_3)$.   Since $w^r_{i+1}<f_i<w^r_i$ for $r=1,2,3$ all $\hat \I(\mF,\mW_3)-\hat \I(\mF, \mW_1)$ intersections cancel except possibly of  those of $\inte(w^3_i)$ with $f_i$.  The non cancelling intersections with the cap are from the type c) components and they are in 1-1 correspondence with $f_i\cap \inte(w_j)\cap R$.  That set is also in 1-1 correspondence with the non cancelling intersections with the slid disc.
\end{proof}

\begin{definition}\label{rbraiding}  We say that $G\cup R$ is in  \emph{R-framed finger form} if  $R$ is isotoped to be $\Rs$ and $G$ has fingers poking into $R$ analogous to regular framed finger form.  Similarly we can define $F^G_1, \cdots, F^G_n$ the solid $G$ fingers.  There is a natural isotopy taking framed finger form to R-framed finger form.  Let $\rho$ denote the time one map of this isotopy.

We now define \emph{pre-$R$-braiding} viewed from R-framed finger form.  Let $\mW$ be the standard Whitney discs.  For $p=i,j$, let $\alpha_p\subset R$ be a circle about $w_p\cap R$.  If $i<j$, then let $\beta$ be an embedded arc from $\alpha_i$ to $\alpha_j$ disjoint from $\mW$ and let  $\alpha_{ij}$ be the component of $\partial N(\alpha\cup \alpha_i\cup \alpha_j)\cap R$ not isotopic to $\alpha_i$ or $\alpha_j$.  Define $\eta_{i,j}:R\to R$ to be the composition of  L Dehn twists about $\alpha_i$ and $\alpha_j $ with a R Dehn twist about $\alpha_{ij}$.  Now  $\eta_{i,j}$ is isotopic to  the time one map of an ambient isotopy $\zeta^t_{i,j} $ of $R$ which  sends $w_i$  to $w_j$ along $\beta$, then swings it  about $w_j$ then returns it along $\beta$ to its original position.   Now extend $\zeta^t_{i,j}$ to an ambient isotopy of $\stwostwo$ that fixes each $F^G_r, r\neq i$   pointwise such that $F^G_i$ goes into the future as its 3D projection crosses $F^G_j$.  The time one map $\phi$ is required to fix $N(G)$ and the solid fingers pointwise.   The map $\phi$ is  a \emph{pre-$R$-braiding} of $F^G_i$ about $F^G_j$.  Define \emph{$R$-braiding} by $\psi=\rho^{-1}\phi\rho$.\end{definition}

\begin{lemma}\label{rbraiding invariance} Given $\mF$, let $R$ and $G$ be in framed finger form with $\mW$ the standard Whitney discs.  Suppose $\mW_1$ is a minimal switching of $\mW$.  If $\mW_2:=\psi(\mW_1)$ where $\psi$ is the $i,j$ $R$ braiding, then $\I(\mF, \mW_1)=\I(\mF, \mW_2)$.\qed\end{lemma}

Our next restandardization map is based on the operation of \emph{spinning} one arc about another in a four manifold, an operation sometimes called \emph{double curve resolution}.  See \S 3 \cite{Ga2} for details on spinning.    First, we have the elementary:

\begin{lemma} \label{e' homology} Let $E'=S^2\times S^2\setminus \inte(N(G\cup R\cup\mW))$, where $\mW$ is the standard set of Whitney discs.  Let $S_i$ denote a 2-sphere which links the $i$'th solid finger.  Then $H_2(E')$ is freely generated by the classes $[S_i]$ and with suitable orientation on the $S_i$'s, $\langle S_i, f_j \rangle=\delta_{ij}$, where $f_j$ is the $j$'th standard finger.\qed \end{lemma}

\begin{definition}\label{ij spinning}  Let $R$ and $G$ be in framed finger form, then the \emph{$i,j$-spinning} is the time one map $\psi$ corresponding to spinning the $i$'th solid finger about the $j$'th solid finger where $i<j$.  Our $\psi$ is  supported in a small neighborhood of 2-disc sections of the solid fingers and an arc $\alpha$ from one disc to the other whose interior is disjoint from $R\cup G\cup \mW_1\cup$ solid fingers, where $\mW_1$ is a k-switch of $\mW$.  Strictly speaking we need to establish orientation conventions to differentiate positive and negative spinning, however we will be vague on this point since we are working  $\BZ_2$-intersection numbers.\end{definition}

\begin{lemma} \label{ij spinning homology} Let  $\mW_1$ be a $k$-switch of $\mW$ and $\psi$ an $i,j$ spinning along the arc $\alpha$.  Let $\mW_2=\psi(\mW_1)$ and $z_q\in H_2(E)$ denote the class of a 2-sphere $S_q$ that links the $q$'th solid finger.  For all $p, w^2_p, w^1_p$ have the same boundary germ.  If $k\ge i$ (resp. $k\ge j$), then with suitable orientations on $S_q$'s, $[w^2_i]= [w^1_i]+z_j\in H_2(E,\partial w^2_i$), (resp. $[w^2_j]= [w^1_j]-z_i\in H_2(E,\partial w^2_j)$), while if  $ k>i$ (resp. $k>j$) then $[w^2_{i+1}]= [w^1_{i+1}]-z_j$, $[w^2_{j+1}]= [w^1_{j+1}]+z_i$, otherwise $[w^2_r]=[w^1_r]$.  \qed\end{lemma}

\begin{remark}   Represent $z_p, p=i,j$ by a 2-sphere $S_p$ linking $\sigma_p$.   Here $w^2_i$ (resp. $w^2_j$) when $i\le k$ (resp. $j\le k$) is obtained by tubing $w^1_i$ (resp. $w^1_j$)  to $S_j$ (resp. $S_i$) along a tube that parallels $\alpha$.  Similar statements hold for $w^2_{i+1}$ and $w^2_{j+1}$.  The tubes from $w^1_j, w^1_{j+1}$ to $S_i$ are linked with those from $w^1_i, w^1_{i+1}$ to $S_j$.  All but the $S_p$'s can be ignored, since by Lemma \ref{boundary compression} for the purposes of calculating $\hat \I(\mF, \mW_2)$ it suffices to compress these tubes.      \end{remark}

\begin{lemma}  \label{spin invariance}  Given $(\mF, \mW)$ with $\mW_1$  a minimal switching of $\mW$ and $\psi$ an $i,j$-spinning with $\mW_2:=\psi(\mW_1)$, then $\I(\mF, \mW_1)=\I(\mF, \mW_2)$.\end{lemma}

\begin{proof} Since $\mW_1$ has the same boundary germ as $\mW_2$ by Lemma \ref{parity} it remains to show that $\hat \I(\mF, \mW_1)=\hat \I(\mF, \mW_2)$.  Since $\hat \I(\mF, \mW_2)$ is a $\BZ_2$ intersection invariant it suffices to consider the homological effect of modifying a $w^1_p$ to  $w^2_p$.  When $p=i\le k$, the latter is represented by $w^1_i \cup [S_j]$.  Similar statements hold for $ i+1, j, j+1$.  Next consider the various types of intersections of $\mF$ with $\sigma_j$ as in Figure \ref{fig:solidfinger2}.  Ones of type a), b) or c) each give 0 intersections with $S_j$ mod-2, while ones of type d) each give one  intersection.  Now $w_j$ is part of an immersed cycle of the form $w_{q_1}\cap G, f_ {q_1}\cap G, \cdots, w_{q_s}\cap G, f_{q_s}\cap G$, where  $w_{q_1}:=w_j$ and $s\ge 1$. Further, both of the type d) intersections in $\sigma_j$ come from  $f_p, f_p'$ in this cycle, where $p=p'$ if $s=1$.  In  the IA ordering either $f_p, f_{p'}<w^2_i$ or $w^2_i<f_p, f_{p'}$ and hence the new type d) intersections with $w^2_i$ cancel.  Similarly, the new type d) intersections with $w^2_{i+1}, w^2_j, w^2_{j+1}$ cancel.  \end{proof}

We now define the fifth  restandardization map.

\begin{definition}  \label{so3 twist} An \emph{SO(3)-twist} $\psi$ is the time one map of an ambient isotopy supported in a small neighborhood of a $D^2$ section of a solid finger $\sigma_i$ that fixes $\sigma_i$ setwise and twists it by $4\pi$.  See Figure \ref{fig:so31} which shows the ambient isotopy in a very small neighborhood of the $\sigma_i$ where the fourth dimension is used to enable the crossing.   \end{definition}

 \begin{figure}[!htbp]
    \centering
    \includegraphics[width=0.65\linewidth]{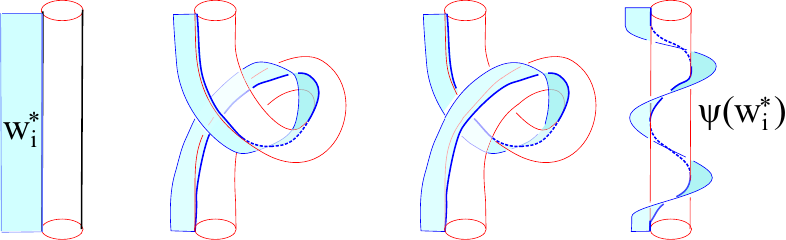}
    \caption{Very Local View of the SO(3) Twist}
    \label{fig:so31}
\end{figure}

\begin{lemma} \label{psid} If $\psi$ is the result of a SO(3)-twist about $\sigma_i$ and $\mW_1$ is the result of a k-switch to  $\mW$ and $\mW_2=\psi(\mW_1)$, then $w_j^2=\psi(w^1_j)$ unless $w^1_j$ is a switch disc and $j=i$ or $j=i+1$.  Here we assume $\mW_1, \mW_2$ are finger germed.  If $w^2_j\neq w^1_j,$ then up to $H_2$-equivalence, it is constructed as follows.   Starting with $w^1_j$ as in Figure \ref{fig:so32} a) first remove a rectangle as in Figure \ref{fig:so32} b).  Second, modify as in Figure \ref{fig:so32} c).  Third, add a product of the resulting boundary component into $y\times [0, \epsilon]$, fourth cap off with a disc in $y\times \epsilon$ and fifth do cut and paste with an appropriately oriented linking 2-sphere to $\sigma_j$.  If $j<k$, then $w^1_{j+1}$ is defined similarly.\end{lemma}

 \begin{figure}[!htbp]
    \centering
    \includegraphics[width=0.65\linewidth]{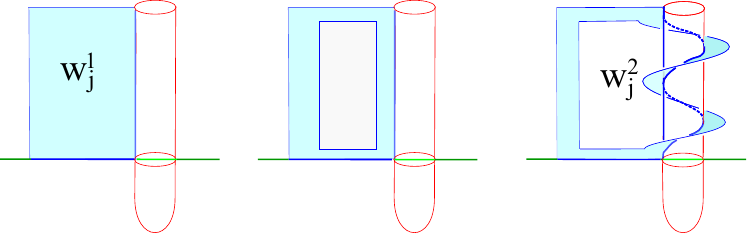}
    \caption{Partial Construction of $w_j^2$}
    \label{fig:so32}
\end{figure}

\begin{proof} The 2-sphere is needed so that the resulting disc is Whitney germed.  If $v^2_j  $ is constructed as above, then  $v^2_j$ and $w^2_j$ have the same boundary germ and hence their difference is homologous to a surface which lives in a small neighborhood of a finger.  Remembering that $\inte(B^4)\setminus \alpha=S^2\times \inte (D^2)$, where $\alpha$ is a properly embedded arc, this surface is homologous to a multiple of the linking sphere $S_j$ to $\sigma_j$. A non zero multiple would imply that one of $w^2_j, v^2_j$ is misframed.\end{proof}

\begin{lemma}\label{so3 twist invariance}  Given $(\mF, \mW)$ with minimal switching $\mW_1$ of $\mW$ and $\mW_2:=\psi(\mW_1)$, where $\psi$ is a SO(3)-twist, then $\I(\mF, \mW_1)=\I(\mF, \mW_2)$.\end{lemma}  

\begin{proof}  Assume that the twist is about $\sigma_i$.  Since either $\partial w^2_i=\partial w^1_i$ (resp. $\partial w^2_{i+1}=\partial w^1_{i+1}$) or they differ by a double Dehn twist, it follows that hypothesis iii) of Lemma \ref{parity} holds.

 We now show that hypothesis ii) also holds.  Suppose $w^2_i\neq w^1_i$.  By Lemma \ref{psid} we can assume that  the difference consists of a 2-disc $F$ and a 2-sphere $S$.  By considering the various components of a $f_p\cap \sigma_j$ we see that each one of types a)-d) contribute an even number of intersections with  $F$.  However, $S$ has a single intersection with each neighborhood of a type d) component.  There are now two cases.  
 
 If some $f\in \mF$ matches $w_i$, then each type d) component belongs to $f$ and hence both intersections contribute 0 mod-2 to $\hat \I(\mF, \mW_2)$.  If no $f\in \mF$ matches $w_i$, then $w_i$ is part of an immersed cycle of the form $w_{q_1}\cap G, f_ {q_1}\cap G, \cdots, w_{q_s}\cap G, f_{q_s}\cap G$, where  $w_{q_1}:=w_i$. Further, both of the type d) intersections come from distinct $f_p, f_p'$ in this cycle and in the IA order either $f_p, f_p'<w^2_i$ or  $f_p, f_p'>w^2_i$. In both cases these contributions to $\hat \I(\mF, \mW_2)$ cancel.  A similar argument proves the case $w^2_{i+1}\neq w^1_{i+1}.$\end{proof}

 \begin{remark} \label{half disc} (Half disc)  The space of half discs which coincide near the round boundary of a fixed embedding is contractible, a fact dating back to Cerf in the 1960's.  This motivates our sixth and final restandardization map. \end{remark}
 
   \begin{figure}[!htbp]
    \centering
    \includegraphics[width=.7\linewidth]{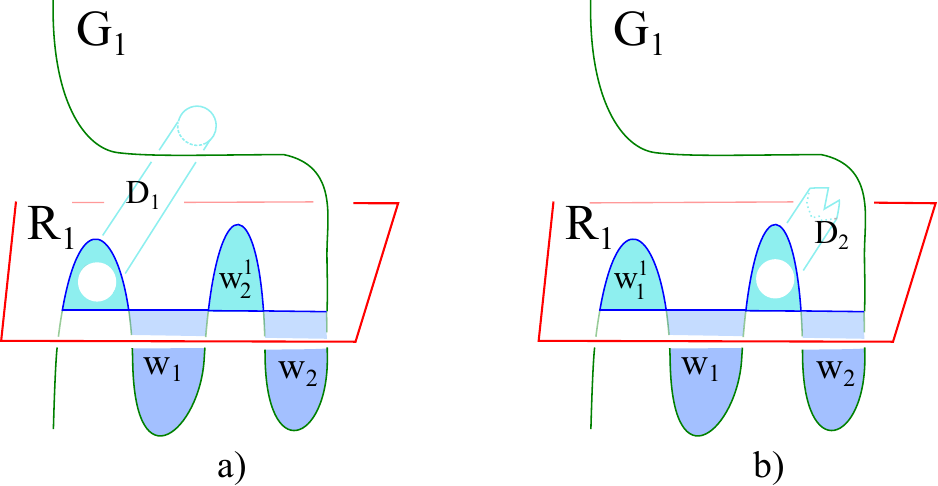}
    \caption{The half disc map}
    \label{fig:half disc}
\end{figure}

\begin{definition} \label{half disc map} (Half disc map)  Let $R\cup G$ be in framed finger form, $D_i$ a disc strongly $H_2$-equivalent to $w^1_i$, the $i$'th standard switch disc and $D_i\cap \mW\subset R\cap G$, where $\mW$ is the standard Whitney system.  Let $\sigma$ be the solid finger containing $w_i$, the $i$'th standard Whitney disc.  A \emph{half disc map} $\psi$ is the time one map of the ambient isotopy that fixes $N(\mG\cup \mW)$ pointwise, is supported near $\sigma\cup D_i\cup w^1_i$ and takes $D_i$ to $w^1_i$  such that $\psi|N(R\cup \sigma)=\id$.  See Figure \ref{fig:half disc} which suggests how the half disc map standardizes $w^1_i$ at the cost of making the standard $w^1_{i+1}$ non standard, where $i=1$.\end{definition}

\begin{lemma}  \label{half disc 1} \emph{(Half disc map invariance)}  If $\mW$ is the standard Whitney system  for $R\cup G$ in framed finger form and $\psi$ is the half disc map taking $D_i$ to $w^1_i$, then for any $\mF$,  $\I(\mF,\mW)=\I(\psi(\mF),\mW)$.\end{lemma}

\begin{proof}  Since $\psi$ fixes $N(R\cup G\cup \mW)$ pointwise, it suffices to show that $\hat I(\psi(\mF),\mW)=\hat I (\mF,\mW)$.  The only points of $\mF$ moved by the ambient isotopy are those that intersect a section of the solid finger and those that intersect $D_i\cup w^1_i$.  Since the ambient isotopy taking $D_i$ to $w^1_i$ essentially corresponds to spinning the $i$'th solid finger $\sigma$ around the immersed sphere $D_i\cup w^1_i$ and $[D_i\cup w^1_i]=0\in H_2(E)$ it follows that the discs of $\mF$ are replaced by strongly $H_2$-equivalent ones.\end{proof}

\begin{lemma} \label{half disc 2} Let $R\cup G$ be in framed finger form with $\mW$ the standard Whitney system and $\mW^1$ the standard $\kappa$-switch.  If $\mW^2$ is a $\kappa$-switch  such that the switch discs $\mW^2_S$ are strongly $H_2$-equivalent to the switch discs $\mW^1_S$, then a composition of half disc maps takes $\mW^2$ to $\mW^1$.\end{lemma}

\begin{proof} Suppose for $i=1,2$, $\mW^i_S=\{w_1^i, \cdots, w_\kappa^i\}$, where $w_i^2=w_i^1$ for $i<j$ for some $j\in \BN$.  Let $\psi$ be the half disc map taking $w_j^2$ to $w_j^1$.  Since $\psi|w_i^2=w_i^1$ for $i\le j$, $\psi(w_i^2)$ is strongly $H_2$-equivalent to $w_i^2$ for $i>j$ and $\psi$ fixes the non switch discs pointwise, the result follows by induction.\end{proof}

\noindent\emph{Proof of Lemma \ref{standardizing}}:  It suffices to consider the case that $R$ and $G$ are in framed finger form and $\mW$ is the standard ordered set of Whitney discs.  Let $\mW_0$ denote the standard $\mW$-framed $k$-switch and $\mW_1$ any $\mW$-framed $k$-switch.  Recall  this means that the switch discs of  $\mW_0$ and $\mW_1$ as finger germed and the non switch discs as Whitney germed.   We will find a sequence $\mW_1, \mW_2, \cdots, \mW_p$ such that $\mW_{i+1}$ is obtained from $\mW_i$ by a basic restandardization map and $\mW_p$ is strongly $H_2$-equivalent to $\mW_0$.  The result then follows by Lemmas \ref{half disc 1} and \ref{half disc 2}.  

To start with, by a sequence of $G$-braidings and twists we can transform $\mW_1$ to $\mW_2$ where $\mW_2\cap G=\mW_0\cap G$.  After a sequence of $R$-braidings applied to $\mW_2$ we  obtain $\mW_3$ such that $\mW_3\cap G=\mW_0\cap G$ and up to Dehn twisting along 2-disc sections of solid fingers, $\mW_3\cap R=\mW_0\cap R$. Using SO(3)-twisting we can assume that if there is Dehn twisting about a given solid finger, then it is a single right hand twist.  There can be no single Dehn twist along $\sigma_1$, else $w^3_1$ would contradict Lemma \ref{even winding} when comparing it with the standard finger disc $f_1$.  By induction, if there is no twisting along $\sigma_1,\cdots, \sigma_p$, then Lemma \ref{even winding} implies that there is no twisting along $\sigma_{p+1}$.  Since $\mW_0, \mW_3$ are untwisted and $\mW$ framed we conclude that  the boundary germs of $\mW_3$ are equal those of $\mW_0$.  

To complete the proof we modify $\mW_3$ by a sequence of spinnings to obtain $\mW_4, \cdots,  \mW_p$ where the latter  is strongly $H_2$-equivalent to $\mW_0$.  We can assume that $k>0$ otherwise the proof is trivial. Recall that   $E'=S^2\times S^2\setminus \inte(N(G\cup R\cup\mW))$.  Since by Lemma \ref{e' homology} $H_2(E')$ is freely generated by the linking spheres $S_1, \cdots S_n$ it follows that $[w^3_1]=[w^0_1]+ \sum_{j=1}^n m_{j}[S_j]\in H^2(E';\partial w^0_1)$.  By applying a suitable sequence of 1,j spinings where $j\ge 2$, we obtain $\mW_4$ which has the same boundary germs as $\mW_0$ and $[w^4_1] =[w^0_1]+m_1[S_1]\in H_2(E;\partial w^0_1)$.  Since $w^0_1$ and $w^4_1$ are  Whitney germed it follows that $m_1=0$.  If $k=1$, the proof is complete.  

Otherwise, as before  $[w^4_2]=[w^0_2]+\sum_{j=1}^n m_j[S_j]$.  By doing $2,j$ spinnings with $j>3$ we obtain $\mW_5$ with the same boundary germs as $\mW_0$ such that $[w^5_1]=[w^0_1]$ and $[w^5_2]=[w^0_2]+m_1[S_1]+m_{2}[S_2]$ and hence $m_1=\langle w^0_1, w^0_2+m_{1}[S_1]+m_{2}[S_2]\rangle=\langle w^5_1, w^5_2\rangle =0$.  Since $w^5_2$ is Whitney germed it now follows that $m_{2}=0$.  This proves the lemma for $k=2$.  The proof now follows by induction on $k$.\qed
 \vskip 10 pt
 
 \noindent\emph{Proof of Lemma \ref{ff independence}}:  Given $(\mF, \mW)$ with $\mW$ ordered and and order induced switchings $\mW_1, \mW_2$ of $\mW$, Lemma \ref{standardizing} shows that $\mW_2$ is obtained from $\mW_1$ by a sequence of basic restandardization maps.  Since $\I(\mF, \mW_1)$ is invariant under these maps the result follows.\qed
 \vskip 8pt

\section{Loops of multi-spheres}

\subsection{Introduction}  The content of \S 2 - \S 3 is that given  a finger/Whitney system $\mfw$ with  $\mW$ ordered, arising an $\alpha \in \Omega(\Emb (S^2,\stwostwo), \Rs)$  we can transform it to $(\mF, \mW')\in \EA$ and then define $\I\mfw=\I(\mF, \mW')$ whose value is independent of all choices made.    In this section, we consider finger/Whitney systems  $\mfw$ with $\mW$ ordered, arising from an $\alpha\in \Omega(\Emb(\sqcup_k S^2, \#_k (S^2 \times S^2)), \mRs)$.  Extending the methods of \S2-\S3 to this setting, we define $\bfI\mfw=(I_1\mfw, \cdots, I_k\mfw)$ and $\I\mfw=\sum_{i=1}^k I_i\mfw$, independent of choices made, by transforming it to an $(\mF, \mW') \in \EA$, where Definition \ref{invariant} applies.  Because $\mfw$ corresponds to a Cerf graphic with $k\ge 1$-eyes, we call this the \emph{multi-eye} case.    

\begin{definition} Given the representative $\alpha:\sqcup_{i=1}^kS^2\times I\to \#_kS^2\times S^2$ with the finger moves occuring before t=1/2 and the Whitney moves after t=1/2, let   $\mR=\alpha(1/2)(S^2_1, \cdots S^2_k)$ with $R_i$ denoting $\alpha(1/2)(S^2_i)$.  Unless said otherwise $\mG$ will denote $\Gs$ with the various components being $G_1, \cdots, G_k$.   View $\#_k \stwostwo$ as $k\ \stwostwo$'s summed to $S^4$ with $(\stwostwo)_i=\stwostwo\setminus \inte(B^4)$ as the $i$'th summand which contains $R_i\cup G_i$.

Finger (resp. Whitney) discs will be denoted $\mF=\{f_1,\cdots, f_n\}$ (resp. $\mW=\{w_1, \cdots, w_n\}$) and we will usually  supress the superscripts used in \S0.  We say that a finger or Whitney disc is a \emph{cross disc} if it cancels points of $R_i\cap G_j$ where $i\neq j$, otherwise call it a \emph{uncross disc}. Let $\mF_c, \mW_c$ (resp. $\mF_u, \mW_u$) denote the cross (resp. uncross) discs of $\mF, \mW$ and  let $\mF_{ij}=\{f\in \mF|\partial f\subset R_i\cap G_j\}$ and $\mW_{ij}=\{w\in \mW|\partial w\subset R_i\cap G_j\}$.  Elements of $\mF_{ij}$ (resp. $\mW_{ij})$ will often be denoted by $f_{ij}$ (resp. $w_{ij}$).  We will abuse notation by sometimes letting $\mF_{ij}, \mW_{ij}, \mF$, or $\mW$ sometimes denote the union of its elements, the distinction being clear by context.   Elements of $\mF_{ii}$ or $\mW_{ii}$ are often called \emph{i-discs}.  In a similar manner we have the notions of \emph{cross} and {uncross solid fingers}.

 By an \emph{ordering} on $\mW$ we mean a partial order that  totally orders each $\mW_{ii}$ and if  $w, w'\in \mW$ are related, then both are $i$-discs for some $i$.  In particular if $(\mF, \mW) \in IA$, then there is an induced IA-ordering on both $\mW$ and  $\mfii\cup\mwii$ for all $i$.

\end{definition}

\subsection{Immersed arc position}  In this subsection we indicate how to extend the material of \S \ref{immersed arc} to deal with multiple eyes.  As a warmup we have:  

\begin{lemma} \label{boundary clearing}  If $(\mF, \mW)\in \EA$, then by appropriately disc sliding cross discs over uncross discs we obtain $(\mF', \mW')\in \EA$ such that 
$(\partial \mF'_c\cup\partial \mW'_c) \cap (\partial \mF'_u\cup\partial\mW'_u)=\emptyset$ and $\bfI(\mF,\mW)=\bfI(\mF', \mW')$.\end{lemma}

\begin{proof}  Suppose $f_1, w_1,  f_2, w_2, \cdots,  f_p, w_p$  comprise the  1-discs such that their union intersects $G_1$ in an embedded arc with the indicated ordering.  Then any $f\in\mF_c$ with $f\cap w_1\cap G_1\neq\emptyset$ can be removed by sliding over $f_1$.  Then all $f\in \mF_c$ such that $f\cap w_2\cap G_1\neq\emptyset$ can be removed by sliding over both $f_1$ and $f_2$. Here we  abuse notation by continuing to call $\mF$ the  image of $\mF$ after sliding.  Inductively all the finger cross discs are disc slid off of $w_1\cap G_1,  \cdots, w_p\cap G_1$.  In a similar manner the Whitney cross discs are disc slid off of $f_p\cap G_1, f_{p-1}\cap G_1, \cdots, f_1\cap G_1$.    Inductively, for all $p$, slide the cross discs off of $(\mF_{pp}\cup\mW_{pp})\cap R_p$ and $(\mF_{pp}\cup\mW_{pp})\cap G_p$.  Since $\mF_u=\mF'_u$ and $\mW_u=\mW'_u$ we have   $\bfI(\mF,\mW)=\bfI(\mF', \mW')$.   \end{proof}

\noindent\emph{Comments on Lemma \ref{ia to ea}}:  The proof is essentially the same in the multi-eye case.  For example, consider   the immersed arc in $G_1$ coming from the 1-discs $f_1, w_1, \cdots, f_p, w_p$.  If $\inte(f_1\cap G_1)$ intersects $\mW$, then remove the first point of intersection as before.  I.e. if this point lies in $w$, then suitably slide $w$ over \emph{all} the other $w'\in \mW$ which intersect $G_1$ and then do an $R$-twist to eliminate boundary bigons.  Note that the  final $(\mF',\mW')\in \EA $ and $(\partial\mF'_c\cup\partial\mW'_c)\cap (\mF'_u\cup\mW'_u)=\emptyset$.  

\vskip 8pt

We restate Lemma \ref{no new ints} in the multi-eye form.  \vskip 8pt
\noindent\textbf{Multi-Eye Lemma \ref{no new ints}.}  \emph{Let $\mF,\mW$ be untwisted sets  of discs  such that for all i, $(\mfii\cup\mwii)\cap R_i$ is the embedded union of circles and possibly an arc.  Let $(\mF',\mW')$ be obtained from $(\mF,\mW)$ by either an untwisted  G-disc slide or  an R-twist supported near $\mR$, then for all $i$, $p,q$ if $f_p$ and $w_q$ are $i$-discs then,\ $|\inte(f_p)\cap\inte(w_q)|=|\inte(f_p')\cap\inte(w_q')|$ mod-2.  A similar statement holds with $R$ and $G$ reversed.}  
\vskip 8pt

\begin{proof} If the G-disc slide  is of an uncross disc $D$ over an uncross disc $E$, then the argument is as before.  If $D$ is an $i$-disc but $E$ is a cross disc that intersects $G_i$ and $R_j$, then after the slide $\inte(D)$ will pick up new intersections from these boundary intersections, but none of those are with $i$-discs, for suppose $F\cap\partial (E\cap R_j)\neq\emptyset$, then $F$ is either a cross disc or a $j$-disc, where $j\neq i$. The proof is immediate for cross discs sliding over cross discs and for R-twisting.  \end{proof}

\noindent\emph{Comments on Proposition \ref{hsf independence}}:  Lemmas  \ref{clifford homology}, \ref{clifford persists}, \ref{twist hsf}, \ref{clifford commutes} and \ref{clifford commutes two} are proved as before.  With the multi-eye Lemma \ref{no new ints} in place, the proof is as before.
\vskip 8pt
We restate Definition \ref{hat i} for multi-eyes.  Multi-eye versions of Lemmas \ref{reverse} and \ref{hati half ea} hold and are essentially proved as before.
\vskip 8pt
\noindent\textbf{Multi-Eye Definition \ref{hat i}}:  We say that the untwisted $(\mF,\mW)\in$ \emph{R-EA} (resp. $\in$ \emph{G-EA}) if for all $i$, $(\mfii\cup\mwii)\cap R_i$ (resp.  $(\mfii\cup\mwii)\cap G_i))$ is an embedded arc.  If either hold, then we say that $(\mF,\mW)\in$ \emph{Half-EA}.  Let $(\mF,\mW)\in IA$ be untwisted but not necessarily normalized such that $\mF$ is finger germed  and $\mW$ is Whitney germed or vice versa, then with the IA ordering we define  $$\hat I_p(\mF,\mW)=\sum_{i\le j} |\inte(f_i)\cap \inte(w_j)|\ \textrm{mod}\ 2 \in \BZ_2,\ \textrm{where}\ f_i\in \mF_{pp}\ \textrm{and}\ w_j\in \mW_{pp}.$$

\subsection{Getting to immersed arc position}  The main result of \S3 is Proposition \ref{switch invariance} which we now extend to the multi-eye setting.  The crucial observation is that passing from one eye to many is a mild form of stabilization allowing for little extra restandardization flexibility.  An intuition is that the affect on the closed $\mR\cup\mG$ complement of adding a trivial eye, i.e. $R_{k+1}=\Rs_{k+1}$ is to remove an open 4-ball.  In contrast, changing the ambient manifold by summing with an $S^2\times S^2$ without having a new $R_i$ and $G_i$ allows for a fundamentally new restandardization move, which can change the invariant.  
\vskip 8pt

\noindent\textbf{Multi-eye Parity Lemma \ref{parity}:}  Let $(\mF_1, \mW_1)$, $(\mF_2, \mW_2)$ be $k$-eyed finger/Whitney systems.   For $r=1,2$ let $\mF_{r_{ii}}=\{f^i_{r,1}, \cdots, f^i_{r,{n_i}}\}, \mW_{r_{ii}}=\{w^i_{r,1}, \cdots, w^i_{r,{n_i}}\}$.
\vskip 2pt  Assume that 
\vskip 8pt
i) For $i=1, \cdots, k$, $\mF_{1_{ii}}, \mF_{2_{ii}}$ are similarly matched as are $\mW_{1_{ii}}, \mW_{2_{ii}}$ with $f^i_{1,j}$ (resp. $w^i_{1,j}$) matching $f^i_{2,j}$ (resp. $w^i_{2,j}$).  Furthermore, $\mF_1, \mF_2, \mW_1, \mW_2$ are untwisted, though not necessarily normalized, such that both $\mF_1, \mF_2$ are finger germed and $\mW_1, \mW_2$ are Whitney germed or vice versa.  
\vskip 8pt
ii) $(\mF_1, \mW_1), (\mF_2, \mW_2)\in IA$, 
\vskip 8pt
iii) For each $1\le i\le k$, $p, q,\  |f^i_{1,p}\cap w^i_{1,q}\cap R_i|-|f^i_{2,p}\cap w^i_{2,q}\cap R_i|=|f^i_{1,p}\cap w^i_{1,q}\cap G_i|-|f^i_{2,p}\cap w^i_{2,q}\cap G_i|=0 \ \textrm{mod}\ 2$, 
\vskip 8pt
then $I_i(\mF_1, \mW_1)=I_i(\mF_2, \mW_2)$ if and only if $\hat I_i(\mF_1, \mW_1)=\hat I_i(\mF_2, \mW_2).$ 

\begin{remark} The existence of cross discs in $(\mF_1, \mW_1)$ and or $(\mF_2,\mW_2)$ do not affect  the hypotheses or conclusion.\end{remark}

\begin{proof} As shown in the proof of the Multi-eye Lemma \ref{no new ints}, disc sliding an i-disc over a cross disc of $(\mF_r, \mW_r)$ does not change  $\hat \bfI(\mF_1, \mW_1)$, nor does sliding a cross disc over an $i$-disc or cross disc.  The proof essentially now follows as before.  \end{proof}

\begin{definition} \label{multi switching} (Multi-Eye Switching)  Let $\mW=\{w_1, \cdots, w_n\}$ be a complete untwisted system of Whitney discs with $\mW_{pp}=\{w^p_1,\cdots, w^p_{n_p}\}$.  We say that the complete system $\mW^*=\{w_1^*, \cdots, w_n^*\}$ with $\mW^{*}_{pp}=\{w^{p*}_1,\cdots, w^{p*}_{n_p}\}$ is a \emph{\textbf{k}-switching}, with $\textbf{k}=\{k_1, \cdots, k_k\}$ if

\vskip 8pt

a) $a^p_{2i-1}\in w^p_i\cap w_i^{p*}$, where $R_p\cap G_p=\{a^p_0, a^p_1, \cdots, a^p_{2m_p}\}$ such that $\langle R_p,G_p\rangle=+1$ at $a^p_0, a^p_2, \cdots, a^p_{2n}$ and $w^p_i\cap R_p\cap G_p=a^p_{2i-1}\cup a^p_{2i}$.
\vskip 8pt
b) Up to isotopy $w_i^{p*}=w^p_i$ except for exactly $k_p\ w^p_i$'s.  The $w_j^{q*}$'s with $w^q_j\neq w_j^{q*}$, $1\le q\le k$ are called the \emph{switch discs} and their union is denoted $\mW^*_S$.  The others, $\mW^*\setminus \mW^*_S$ are called the \emph{non switch discs}.
\vskip 8pt
c) $\mW^*_S$, $\mW$ can be respectively isotoped to be finger germed and Whitney germed such that $\mW^*_S\cap \mW\subset \mR\cap \mG$ and $(\mW^*_S\cup \mW)\cap \mG$ contains no cycles.
\vskip 8pt
d) The cross discs of $\mW^*$ are equal to those of $\mW$.
\vskip 8pt
Call $\textbf{k}$  the \emph{switch index}.  We say that  $\mW^*$ is \emph{$\mW$-framed} if $\mW$ is Whitney germed and $\mW^*_S$ satisfies c) without the need to isotope $\mW$.  Given $(\mF,\mW)$, then $\mW$ is a \emph{minimal switching} if $(\mF, \mW^*)\in$ IA and $k_i $ equals the number of immersed cycles in $(\mF_{ii}\cup\mW_{ii})\cap G_i$.  A switching $\mW^*$ is a \emph{full switching} if $\mW^*\cap\mW=\mW_c$ and a \emph{full set of switch discs} are the switch discs of a full switching.  Given $(\mF, \mW)$, by viewing $\mF$ as a complete set of ordered Whitney discs we define in the similar way a $\textbf{k}$-\emph{switching} $\mF^*$ of $\mF$.  \end{definition}

\vskip 8pt

\noindent\emph{Comments on Lemma \ref{ff to ia} and Definition \ref{order induced}}:  Given $\mfw$ with an ordering on $\mW$, each $\mW_{ii}$ can be reordered as in Lemma \ref{ff to ia} and a switching constructed as in its proof.  Such a switching is called an \emph{order induced switching}.
\vskip 8pt

\noindent\emph{Comments on Convention \ref{minimal switch}}: The analogous Convention holds in the multi-eye setting.  In particular, unless otherwise said, given $(\mF, \mW)$ with $\mW$ ordered all switchings will be order induced.

\begin{definition} \label{multi framed finger}  (Multi-eye framed finger form) Figure \ref{fig:multi framed} shows the multi-eye analogue of  Figure \ref{fig:framed finger} b), which we   call \emph{$\mG$-framed finger form}.  There is the corresponding analogue of  the Figure \ref{fig:framed finger} a) version.  When $\mW_{ji}\neq\emptyset$, $R_j$ sends a single finger into the $i$'th $\stwostwo$ factor and from this finger emanates $|\mW_{ji}|$ fingers poking into $G_i$.  Figure \ref{fig:multi framed} shows the \emph{standard finger $i$-discs} and \emph{standard Whitney $i$-discs} and the \emph{standard Whitney cross-discs} which intersect the $i$'th $\stwostwo$ summand.  Framed finger form induces the \emph{standard ordering} on the Whitney discs.  The analogous version where $\mR=\mR^{std}$  and the fingers start from $\mG^{std}$ is called \emph{$\mR$-framed finger form}.\end{definition}

\begin{remark}  The unframed version was introduced in \cite{Ga3} and called \emph{arm hand finger form}.  The \emph{arm} is the finger from say $R_j$ into the $i$'th $\stwostwo$ factor.  The hand is the end of that finger together with the fingers poking into $G_i$.\end{remark}  

We have the multi-eye analogue of Lemma \ref{trivial w} with essentially the same proof.

\begin{lemma} \label{multi trivial w} Let $\mW$ be a complete ordered system of Whitney discs, where $\mG=\mGs$.  Then there exists an ambient isotopy fixing $\mG$ setwise whose time 1 map $\kappa$ has the property that $\mG$ and $\kappa(\mR)$ are in $\mG$-framed finger form with $\kappa(\mW)$ the standard ordered Whitney discs.  Furthermore, the ordering on $\mW$ equals the induced ordering on  $\kappa(\mW)$.
 
 Similarly if $\mF$ is a complete ordered system, then there exists an ambient isotopy fixing $\mG$ setwise whose time 1 map $\kappa$ has the property that $\mG$ and $\kappa(\mR)$ are in $\mG$-framed finger form with $\kappa(\mF)$ the standard finger discs.  Here, the ordering on $\mF$ equals the induced ordering on $\kappa(\mF)$.  
 
 Corresponding results hold for $\mR$-framed finger form.\qed\end{lemma}

   \begin{figure}[!htbp]
    \centering
    \includegraphics[width=0.5\linewidth]{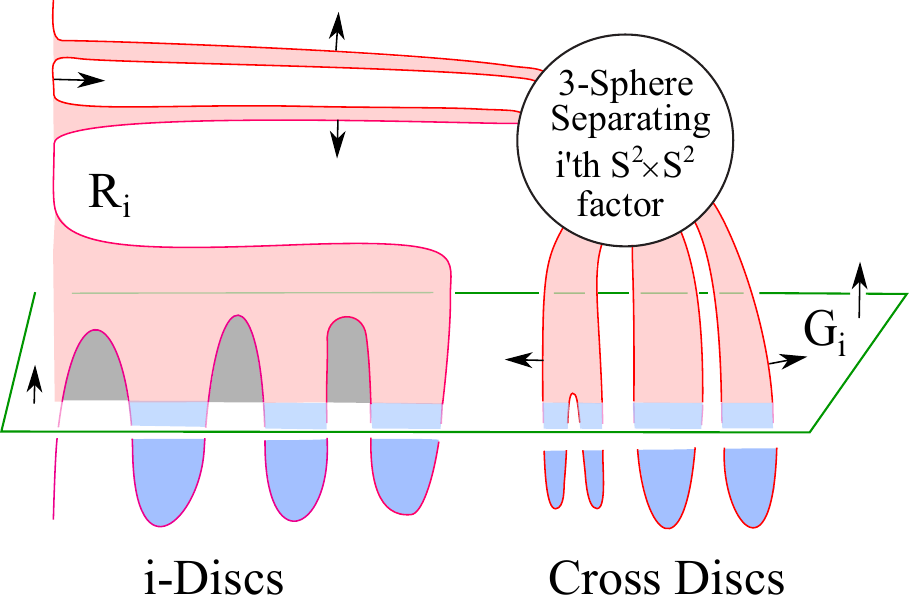}
    \caption{Multi-Eye Framed Finger Form}
    \label{fig:multi framed}
\end{figure}

Allowing for multi-components and multi-indices the material from Lemma \ref{trivial w} through Definition \ref{restandardization} extends to the multi-eye setting.
\vskip 8pt

\noindent\emph{Comments on Proposition \ref{switch invariance}}:  The structure of the proof is as before, as the six restandardization maps suitably extended suffice.  For G-finger twisting we have to account for twisting cross fingers, for G-braiding we have to account for braiding of and around cross fingers and the same for R-braiding.  We need only additionally deal with cross fingers for spinning.  The SO(3) map is as before.  Crucially, Definition \ref{solid finger} and Lemma \ref{f in solid} naturally extend to the multi-eye setting and since $w_2(\#_k\stwostwo)=0$, Lemma \ref{even winding} holds.

\begin{corollary}\label{i boundary twisting} (Multi-eye Corollary \ref{i-invariance boundary twisting}) If $\mW$ is ordered and $(\mF', \mW')$ is obtained from $\mfw$ by boundary twisting then $\I(\mF', \mW')=\I\mfw$.\qed\end{corollary}

  We now comment on the individual restandardization moves.

\vskip 8pt

\noindent\emph{Comments on Lemma \ref{twist invariance}}:  Let $\mW_1$ denote a \textbf{k}-switch of $\mW$.  If $w\in \mW_c$, then a $G$-finger twisting near $w$ acts trivially on $\mW_1$.  If $w\in \mW_{ii}$, then only the $i$-discs of $\mW_1$ are affected and the proof that $\bfI(\mF,\mW_1)=\bfI(\mF, \psi(\mW_1))$ is as before.  Note that Whitney cross discs do not get in the way in the proof of Case 2).  Indeed, since $G_i\setminus \mW_1$ and $R_i\setminus \mW_1$ are connected there exist paths enabling the disc slides of Case 2) and the choice of paths are immaterial for establishing iii) of Lemma \ref{parity}.  Further, using Lemma \ref{boundary compression}  the choices are immaterial for showing $\bfI(\mF, \mW_1)=\bfI(\mF, \mW_3)$.  

\vskip 8pt

\noindent\emph{Comments on Lemma \ref{gbraiding invariance}}:  Here a G-braiding can occur on any $G_i\in \mG$.  Let $w_p$, $w_q$, be two Whitney discs that intersect $G_i$ and $\psi$ denote the time one map sending the first around the second.  If both $w_p$ and $w_q$ are $i$-discs, then the argument follows as before.  If say $w_p$ is an $i$-disc and $w_q$ is a cross disc, or neither is an $i$-disc, then the argument of Case 3 applies. 

\vskip 8pt
\noindent\emph{Comments on Lemma \ref{rbraiding invariance}}:  R-braiding is supported in a small neighborhood of R and essentially isomorphic to G-braiding.   Therefore the proof follows as in Lemma \ref{gbraiding invariance}, modified according to the above comments.
\vskip 8pt

\noindent\textbf{Multi-Eye Lemma \ref{e' homology}}:    \emph{Let $\mG\cup \mR$ be in framed finger form where $k=|\mG|=|\mR|$.  Let $E'=\#_k S^2\times S^2\setminus \inte(N(\mG\cup \mR\cup\mW))$, where $\mW$ is the standard set of Whitney discs.  Let $S_{p,q}$ denote a 2-sphere which links the $q$'th solid p-finger and let $U_1, \cdots, U_m$ denote 2-spheres that link the solid cross fingers.  Then $H_2(E')$ is generated by the  $[S_{p,q}]$'s and $[U_i]$'s.  There is a subset $\mU$ of the $U_i$'s such that the $[S_{p,q}]$ and $\{[U]\in \mU\}$ freely generate}. 

\begin{proof}  This follows from the fact that $E'$ is diffeomorphic to the connected sum of $k$ 4-balls with neighborhoods of finitely many arcs removed and there is a natural 1-1 correspondence between the arcs and the fingers.\end{proof}

\noindent\emph{Comments on Definition \ref{ij spinning} and Lemma \ref{ij spinning homology}}:  Here any solid finger $f$ can spin about another solid finger $f'$.  If $w^1$ is a switch disc with boundary on $f$ and $w^2$ is the result of the spinning, then $w^2$ has the same boundary germ as $w^1$ and $[w^2]=[w^1] \pm [S]$ where $[S]$ is a 2-sphere linking the solid finger $f'$.

\vskip 8pt
\noindent\emph{Comments on Lemma \ref{spin invariance}}:  With notation as in the previous comment, if both $f$ and $f'$ are i-fingers, then the proof is  as before.  In all cases it suffices to see how adding $S$ to $w^1$ changes $\hat I$.  If $f$ is an $i$-solid finger and $f' $ is a $j$-solid finger, $i\neq j$, then both the type d) intersections arise from $j$-finger discs hence do not contribute to $\hat I$, since $w^2$ is an $i$-Whitney disc.  If $f'$ is a cross finger, then the associated type d) intersections come from cross finger discs and again do not contribute.  If $w^1$ has boundary on the j-finger $f'$, $f$ is an i-finger $i\neq j$ and $w^2$ is the result of the spinning, then $w^2$ has the same boundary germ as $w^1$ and $[w^2]=[w^1] \pm [S]$ where $[S]$ is a 2-sphere linking the solid finger $f$ and the argument follows as before.   
\vskip 8pt

\noindent\emph{Comments on Lemma \ref{so3 twist invariance}};  If the SO(3) twist is about an uncrossed finger, then the proof is as before.  If it is about a crossed finger with Whitney disc $w$, then $w$ is not a switch disc and hence $\psi$ fixes $\mW_1$ pointwise and hence $(\mF, \mW_1)=(\mF,\mW_2)$, so $\bfI(\mF, \mW_1)=\bfI(\mF,\mW_2)$.
\vskip 8pt 

\noindent\emph{Comments on the proof of Lemma \ref{standardizing}}:  With notation as in the first paragraph of that proof and following the same strategy, we obtain $\mW_3$ whose  boundary germs agree with the standard $\textbf{n}$-switch $\mW_0$, $I(\mF, \mW_3)=I(\mF, \mW_0)$ and  the unswitch discs are unchanged.  In particular, $\mW_{3c}=\mW_{0c}:=\mW_c$.  By doing spinnings, which neither changes I nor the unswitch discs,  we  obtain $\mW_4$ such that the elements of $\mW_{4_S}$ are strongly  $H_2$-equivalent to those of $\mW_{0_S}$.  The proof is then completed by the Multi-Eye versions of Lemmas \ref{half disc 1} and \ref{half disc 2}. 
\vskip 8pt

For later reference we record this as:

\begin{proposition}  \label{multi switch invariance} If $(\mF, \mW)$ is a finger/Whitney system and $\mW$ is ordered, then $I(\mF,\mW)$ is well defined.\qed\end{proposition}

\section{The Cross Term Invariant} \label{cross term section}

In this section we introduce the cross term invariant $C\mfw=c\mfw+CU\mfw$ where unlike the setting of the introduction, $\mfw$ may not be in full embedded arc position, hence the need for extra terms to account for the difference.  We also introduce the special sum square move, a specialization of Quinn's sum square move \cite{Qu} and show invariance of $D\mfw=\I\mfw+C\mfw$ under cross special sum square moves.

\begin{lemma}\label{dual sphere}\emph{(Dual sphere)}  Given the $k$-eyed finger/Whitney system $(\mF, \mW)$ and full switches $\mF^*$ of $\mF$ and $\mW^*$ of $\mW$ there exist pairwise disjoint oriented embedded spheres $R_1^*, \cdots, R_k^*$, $G_1^*, \cdots, G_k^*$ with  $|R_i^*\cap R_j|=\delta_{ij}$, $|G_i^*\cap G_j|=\delta_{ij}$ and for all $i,j,\ R_i^*\cap G_j=\emptyset, G_i^*\cap R_j=\emptyset$   such that $\cup_{p=1}^k R_p^*\cap (\mF\cup\mF^*_S)=\emptyset$ and $\cup_{p=1}^k G_p^*\cap (\mW\cup\mW^*_S)=\emptyset$.  Here $\mW^*_S$ (resp. $\mF^*_S$) are the finger germed (resp. Whitney germed) switch discs. \end{lemma}

\begin{proof}  The construction of the $G_p^*$'s is immediate if $\mR\cup\mG$ is in $\mG$-framed finger form with $\mW\cup\mW^*_S$ standard.  Since an ambient isotopy transforms $\mW \cup \mW^*_S$ to such a position the result follows.  A similar argument shows the existence of the $R_p^*$'s.  \end{proof}

\begin{definition}\label{dual system}  Call $\mR^*:=\{R_i^*\}$ (resp. $\mG^*:=\{G_j^*\})$ a $\mR_{\mF}$-dual (resp. $\mG_{\mW}$-dual) system if it satisfies the conclusions of Lemma \ref{dual sphere} for some full switching of $\mF$ (resp. $\mW$).  If $\mW$ is ordered and we require that the switch system $\mW^*$ be order induced, then we call this an \emph{order induced  $\mG_{\mW}$-dual system}.  If $\mW^*$ is specified, then we call this an \emph{ $\mG_{\mW, \mW^*}$-dual system}.   \end{definition}

\begin{remark} \label{dual remark} The lemma was stated in a form to minimize future notation.  The same argument shows that shows that there are systems $\mG_{\mW}, \mR_{\mW}$, where $\mG_{\mW}$ is as above and $\mR_{\mW}:=\cup R_{i,\mW}$ posseses similar properties except that it is a dual system for $\mR$ and  $\mR_{\mW}\cap (\mW\cup\mW^*_S)=\emptyset$.  Similarly, there is a dual system $\mG_{\mR}$ for $\mR$.  \end{remark}

\begin{remarks}  \label{homology dual spheres} i) These $R_p^*$'s  depend on how $\mF$  was put into framed finger form and the representative dual spheres for a given framed finger form.  However, since  $(R_1, G_1), \cdots, (R_k, G_k)$ are algebraically dual spheres that generate $H_2(\#_k\stwostwo)$ it follows that for each $p$, $[R_p^*]=[G_p]\in H_2(\#_k\stwostwo)$, $[G_p^*]=[R_p]\in H_2(\#_k\stwostwo)$ and hence $\langle R_i^*, G^*_j\rangle=\delta_{ij}$.

ii) The lemma does not assume an ordering on $\mW$ or $\mF$, however that is immaterial by the following.  \end{remarks}

\begin{lemma} \label{common switch dual} i)  Let $\mW$ be a complete set of Whitney discs with orderings $\mW_1$ and $\mW_2$.  Then there exists a common dual system $\mG^*$ for $\mW_1$ and $\mW_2$ with respect to some full switchings $\mW_1', \mW_2'$ of $\mW_1$ and $\mW_2$, i.e. $\mG^*\cap (\mW_1\cup\mW_1')=\mG^*\cap (\mW_2\cup \mW_2')=\emptyset$.

ii) Let $\mW$ be a complete set of ordered Whitney discs.  Let $\mW' $ be a switching of $\mW$, then there exists a common dual system $\mG^*$ for both $\mW$ and $\mW'$. In particular, this is true if $\mW'$ is a full switching.  \end{lemma}

\begin{proof}  i) Put $\mR\cup \mG$ in framed finger form respecting the ordering of the uncross discs of  $\mW_1$  and let $\mW_1'$ the standard full switch with $\mG^*$ the standard dual spheres.  There is an ambient isotopy taking $\mW_2$ to $\mW_1$ supported away from $\mG^*$ and the cross discs.  If $\tau_1$ is the time one map, then $\mW_2':=\tau_1(\mW_1')$ is a full switch of $\mW_2$ and $\mG^*\cap (\mW_2\cup \mW_2')=\emptyset$.

ii) If $\mW'$ is a full switch of $\mW$ and $\mG^*$ is the corresponding $\mG_{\mW}$-dual system, then since $\mW$ is a full switch of $\mW'$, $\mG^*$ is also a $\mG_{\mW'}$-dual system.  In general, if $\mW'_S$ are the switch discs of $\mW'$, then there is a full switch system $\mW''$ for $\mW$ such that $\mW'_S\subset \mW''_S$.  \end{proof}

\begin{definition} Given $(\mF, \mW)$ let  $\mF$ be finger germed,  $\mW$ Whitney germed and $\partial\mF$ being transverse to $\partial \mW$. 
Let $\infty_j\in G_j\setminus (\mF\cup\mW)$ (resp. $\infty_i\in R_i\setminus (\mF\cup\mW)$) denote $G^*_j\cap G_j$ (resp. $R^*_i\cap R_i$).  For $i\neq j$ 2-color $G_j$ with respect to the immersed curves $(\mF_{ij}\cup\mW_{ij})\cap G_j$ letting $B^G_{ij}\cap G_j$ denote the compact region whose mod-2 winding number about $\infty_j$ equals 0.   Let $W_{ij}^G$ denote the closed complementary region.  Here $B^G_{ij}$ and $ W^G_{ij}$ will be respectively called \emph{black} and \emph{white} regions.  
  In a similar manner 2-color $R_i$ with respect to $(\mF_{ij}\cup\mW_{ij})\cap R_i$ by $B_{ij}^R, W_{ij}^R$.
  
\vskip 8pt
Define $$\hat \mF_{ij}=\mF_{ij}\cup \bigcup_{p\neq i} |B_{ij}^G\cap R_p| R_p^* \quad \textrm{and}\quad \hat \mW_{ij}=\mW_{ij}\cup \bigcup_{p\neq j} |B_{ij}^R\cap G_p| G_p^*,$$

$$A_{ij}=|B^G_{ij}\cap B^R_{ji}|,\quad c_{ij}(\mF, \mW)=(|\hat\mF_{ij}\cap \hat\mW_{ji}|+|A_{ij}|)\mod 2$$

$$\ \textrm{and}\quad
c(\mF, \mW)=\sum_{j<i} c_{ij}(\mF, \mW)\ \textrm{mod}\ 2.$$

Call $\hat \mF_{ij}\setminus \mF_{ij}$ (resp. $\hat\mW_{ji}\setminus \mW_{ji}$)  \emph{extra spheres}.\end{definition}

\begin{remark} $A_{ij}$ is the number of points of $R_j\cap G_j$ common to both $B^G_{ij}$ and $B^R_{ji}$.  We may use notations such as $B_{ij}^G\mfw$ or $ |B_{ij}^G\cap R_p|\mfw$ when it needs to be clear what finger/Whitney system is being used.   \end{remark}

\begin{lemma} \label{c boundary twisting}  $c\mfw$ is invariant under boundary twisting.\end{lemma}

\begin{proof}  Boundary twisting does not effect the dual sphere, the values $|B_{ij}^G\cap R_p|$, $|B_{ij}^R\cap G_p|$ and hence $|A_{ij}|$.  If $f\in \mF_{ij}$ and $w\in \mW_{ji}$, then $\partial f\cap\partial w=\emptyset$ and hence boundary twisting does not effect $|\hat\mF_{ij}\cap \hat\mW_{ji}|$ and hence any $c_{ij}\mfw$.\end{proof}

\begin{lemma} (\emph{Independence of dual spheres and $\infty$'s}) \label{dual independence}  For all $i\neq j$, $c_{ij}(\mF,\mW)$ is independent of the choice of the dual spheres.\end{lemma}

\begin{remark} The proof will have two parts.  We first show invariance of dual spheres with common $\infty$'s, then show invariance with different $\infty$'s.  Note that there is ambiguity for the choice of $\infty\in G_i$ since $G_i\setminus(\mF\cup\mW)$ is never connected.\end{remark}

\begin{proof}  First choose $\infty$'s in each $R_p$ and $G_q$.  The $B_{ij}$'s hence $A_{ij}$'s are independent of the dual spheres, thus we ned to show that for $i\neq j,\ |\hat\mF_{ij}\cap\hat \mW_{ji}|$ is well defined mod 2.  We consider the representative case $i=2$ and $j=1$.  The intersections among the extra spheres are determined homologically.   Thus we need to show for $i\neq 2$, $|R_i^*\cap \mW_{12}|$ and for $j\neq 1$, $|G_j^*\cap\mF_{21}|$ are well defined.      Let $V=\mF_{12}\cup B^R_{12}\cup B^G_{12}\cup\mW_{12}$ and $R_i'$ another dual sphere to $R_i$.  Then $|\mW_{12}\cap R_i^*|=|V\cap R_i^*|=|V\cap R_i'|=|\mW_{12}\cap R_i'|$ mod 2 since $V$ is a closed surface with $(R_i^*\cup R_i')\cap V\subset \mW_{12}$ and $[R_i^*]=[R_i']$ by Remark \ref{homology dual spheres}.  A similar argument shows $|G_j^*\cap\mF_{21}|$ is well defined.

To show independence of $\infty$'s, it suffices to consider the case where a single $\infty$ is changed and is moved from some component of $G_j\setminus (\bigcup_{i\neq j} \mF_{ij}\cup \bigcup_{i\neq j} \mW_{ij})$ or $R_j\setminus (\bigcup_{i\neq j}\mF_{ji}\cup \bigcup_{i\neq j}\mW_{ji})$ to another whose closure shares a common edge $e$.  Without loss we can assume that $j=1$ and $e\subset\mF_{21}\cap G_1$.  Further, by the first assertion, we can assume that this is realized by an isotopy of $G_1^*$ to $G_1'$ by an isotopy supported near $e$.  Since this change only modifies $\hat\mF_{21}, \hat \mW_{12}, B_{21}^G$ and $W_{21}^G$ it follows that only $c_{21}$ can possibly change.  

We now show $c_{21}$ is unchanged.  In what follows, integral values are mod 2.  First we enumerate differences in its defining data.
\vskip 6pt

a) $B_{21}^G$ becomes $W_{21}^G$ and vice versa.  Denote the new ones by $B_{21}'$ and $W_{21}'$. 
\vskip 6pt
b) Since $G_1\cap R_1=1$,  $|B_{21}^G\cap R_1|=|B_{21}'\cap R_1|+1 $,  and hence the number of $R_1^*$ components added to $\mF_{21}$ changes by  $1$.  
\vskip 6pt
c) $|G_1'\cap \mF_{21}|=|G_1^*\cap \mF_{21}|+1$.
\vskip 6pt
d) If $a=|B^R_{12}\cap G_1|, b=|B^R_{12}\cap B_{21}^G|$ and $c=|B^R_{12}\cap B_{21}'|$, then $a=b+c$.  Therefore, if $A'_{21}$ denotes the changed $A_{21}$, then $A'_{21}-A_{21}=b-c$. 
\vskip 6pt
To complete the proof we need to show that if $\mF'_{21}$ (resp. $\mW'_{12}$) denotes the changed $\hat\mF_{21}$ (resp. $\hat\mW_{12}$), then $|\hat\mF_{21}\cap \hat \mW_{12}|-|\mF'_{21}\cap \mW'_{12}|=b-c$.  To do this we will show that $|R_1^*\cap \mW'_{12}|=|R_1^*\cap \hat\mW_{12}|=0$  and hence $|\hat\mF_{21}\cap \hat \mW_{12}|-|\mF'_{21}\cap \mW'_{12}|=|\mF_{21}\cap \hat \mW_{12}|-|\mF_{21}\cap \mW'_{12}|=a=b+c=b-c$.

Finally, we show $|R_1^*\cap \hat\mW_{12}|=0$, the argument showing $|R_1^*\cap W'_{12}|=0$ being similar.  Let $V_{12}=\mF_{12}\cup\hat\mW_{12}\cup B^R_{12}\cup B^G_{12}$.  Then $V_{12}$ is a closed surface whose $[R_1]$ coefficient in $H_2(\#_k S^2\times S^2;\BZ_2)$ equals 0.  Indeed, this coefficient is the number of times $V_{12}$ intersects $G_1$, but that is $|B_{12}^R\cap G_1|$ plus the number of components of $G_1^*$ added to $\mW_{12}$.  By definition these are equal.  This implies that $|R_1^*\cap V_{12}|=0$.  Since $R_1^*$ is disjoint from $B_{12}^R\cup B_{12}^G\cup \mF$ it follows that $|R_1^*\cap \hat W_{12}|=0$.  \end{proof}

\begin{remark}\label{weak dual} This argument only used that the $\mF$ (resp. $\mW$) dual spheres are disjoint from $\mF$ (resp. $\mW$) rather than also disjoint from a full switch.\end{remark}

\begin{lemma}\emph{(Disc slide invariance)} \label{slide ct invariance} Each $c_{ij}$ is invariant under $\mF$ $G$-disc slides and $\mW$ $R$-disc slides.

Each $c_{ij}$ is anti-invariant under $\mF$ $R$-disc slides and $\mW$ $G$-disc slides.  This means that if $\mF'$ (resp. $\mW'$) is the result of a single $R$-disc (resp. $G$-disc) slide, then 
$$ c_{ij}(\mF', \mW)-c_{ij}(\mF, \mW)=c_{ji}(\mF', \mW)-c_{ji}(\mF, \mW)=|\mF_{ij}\cap \mW_{ji}|-|\mF'_{ij}\cap \mW_{ji}|\mod 2.$$ 
$$  (\textrm{resp.}\quad c_{ij}(\mF, \mW')-c_{ij}(\mF, \mW)=c_{ji}(\mF, \mW')-c_{ji}(\mF, \mW)=|\mF_{ij}\cap \mW_{ji}|-|\mF_{ij}\cap \mW'_{ji}|\mod 2.)$$\end{lemma}

\begin{proof}  A dual sphere $R_i^*$ may change under $\mF$ disc sliding, thus $c_{ij}(\mF,\mW)$ is calculated using the original dual spheres while $c_{ij}(\mF', \mW)$ is calculated using the changed dual spheres which we denote by $\{R_i'\}$.  We will show that for every $i$, both $R^*_i, R'_i $ can be chosen disjoint from both $\mF'$ and $\mF$ and they have a common $\infty$.  Lemma \ref{dual independence} then implies that each of $c_{ij}(\mF, \mW), c_{ij}(\mF', \mW)$ can be equally calculated using either the $R^*_i$'s or the $R'_i$'s.  The similar argument will show that with suitable $G^*_i$'s, $G'_i$'s we can equally calculate $c_{ij}(\mF, \mW)$ and $c_{ij}(\mF, \mW')$. 

First consider the case that $\mF'$ is obtained by $R_p$-disc sliding $f_{pq}$ over some $f_{ps}$.  Recall that $f_{pq}'$ is obtained by banding $f_{pq}$ to a Whitney cap of $f_{ps}$ where the band and the cap boundary $\subset R_p$.  Since we can assume that each $R_i^*$ is disjoint from both the band and the half 3-disc bounded by the cap it follows that for all $i$, $R_i^*\cap \mF'=\emptyset$.  Since $\mF$ is obtained from $\mF'$ by $R_p$-disc sliding $f'_{pq}$ over $f_{ps} $ a similar argument shows that we can assume for all $i, \mF\cap R_i'=\emptyset$.  Since $(f_{pq}\cup f_{ps}\cup \textrm{band})\cap R_p$ and $(f_{pq}'\cup f_{ps}\cup\textrm{band})\cap R_p$ have isotopic compact neighborhoods $\subset R_p$,  each $i, R^*_i, R'_i$ can be chosen to also have the same $\infty$.  Here we are abusing notation by viewing the bands as solid rectangles in $R_p$.  The proof for $\mF'$ being obtained from $\mF$ by $G$-disc sliding is essentially the same.\qed

\vskip 8pt
\noindent\emph{Proof of Invariance.}  It suffices to consider the case that $f_{11}\in \mF_{11}$ or $f_{21}\in \mF_{21}$ $G$-disc slides over $f_{q1}\in \mF_{q1}$ as the other cases are similar.  In the former case, each $c_{ij}$ is unchanged since the defining data is unchanged.  We now consider the latter case. Let $f'_{21}, \mF'_{21}$ be the resulting disc and let $B'_{21}$ be the changed $B_{21}^G$.  Note that if $j\neq 2$, then $|B'_{21}\cap R_j|=|B^G_{21}\cap R_j| \mod 2$.  Indeed if $j\neq q$, then $|B'_{21}\cap R_j|$ is unchanged while if $j=q$, then $|B'_{21}\cap R_q|$ may increase by 2, decrease by 2 or stay the same.  The latter happens if and only if $|f_{q1}\cap G_1\cap \mW_{21}|=1 \mod 2$.  Thus $\hat\mF'_{21}$ is obtained from $\mF'_{21}$ by adding the same mod 2 number of $R^*_1$ components that are added to $\mF_{21}$ to obtain $\hat \mF_{21}$.  Since both $\mW$ and $\mF\setminus \mF_{21}$ are unchanged, it follows that only $c_{21}$ can possibly change.  
\vskip 6pt
We have $|\hat\mF_{21}\cap\hat \mW_{12}|-|\hat\mF'_{21}\cap\hat \mW_{12}|=|f_{21}\cap \hat \mW_{12}|-|f'_{21}\cap \hat\mW_{12}|=|f_{q1}\cap R_q\cap \mW_{12}|\mod 2$, which equals 0 unless $q=1$.  Now assume $q=1$ and let $e=f_{q1}\cap R_1$.  If $|e\cap \hat\mW_{12}|$ is even, then $|\hat \mF_{21}\cap \hat \mW_{12}|=|\mF'_{21}\cap \hat\mW_{12}| \mod 2$ and either $\partial e \subset B_{12}^R$ or $\partial e\subset W_{12}^R$.  Under the disc slide the color of each point of $\partial e\cap G_1$ is reversed and so $A_{21}$ and hence $c_{21}$ are unchanged.  If $|e\cap \mW_{12}|$ is odd, then$ |\hat \mF_{21}\cap \hat \mW_{12}|=|\mF'_{21}\cap \hat\mW_{12}|+1 \mod2$ and $\partial e=a\cup b$ where $a\in B_{12}^R$ and $b\in W_{12}^R$.  Since the color of $a\in G_1$ changes under the disc slide it follows that $A_{21}$ also changes by 1.\qed
\vskip 8pt

\noindent\emph{Proof of Anti-Invariance.}  It suffices to consider the case $f_{21}\in \mF_{21}$ $R$-disc slides over $f_{2q}\in \mF_{2q}$.  Let $f'_{21}, \hat\mF'_{21}, \hat \mW'_{21}$ be the resulting disc and sets and $B'_{21}$ the changed $B_{21}^R$.  As before if $j\neq 1$, then $|B'_{21}\cap G_j|=|B_{21}^R\cap G_j| \mod 2$.  Thus for all $j$, $\hat\mW'_{21}$ is obtained from $\mW_{21}$ by adding the same mod 2 number of $G_j^*$ components that are added to $\mW_{21}$ to obtain $\hat \mW_{21}$.  Note that $\mW$ itself is unchanged.  Also, except for $f_{21}$ being replaced by $f'_{21}, \hat\mF$ is unchanged.  It follows that only $c_{21}$ or $c_{12}$ can possibly change and change happens only if $q=2$.  Further $A_{21}$ is unchanged as is $|\hat\mF_{12}\cap \hat\mW_{21}| \mod 2$.  

\vskip 6pt

Now assume $q=2$ and let $e=f_{2q}\cap G_2$.  An argument similar to that in the invariance proof shows that if $|e\cap \mW_{12}|=0 \mod 2$, then $c_{21}$ and $c_{12}$ are unchanged and if $|e\cap \mW_{12}|=1 \mod2$, then both $c_{21}$ and $c_{12}$ change by $1$.  Indeed, $c_{21}$ changes because $|\hat\mF_{21}\cap \hat\mW_{12}|-|\hat\mF'_{21}\cap \hat\mW_{12}|=|e\cap \mW_{12}|=1\mod 2$ and both $B_{21}^G$ and $B_{12}^R$ are unchanged. Also $c_{12}$ changes because $A_{12}$ changes by $1$ as $\partial e\cap G_2=a\cup b$ with $a\in B^G_{12}$ and $b\in W_{12}^G$ and the color of $a\in R_2$ changes under the disc slide.  \end{proof}

\begin{remark}\label{disc slide remark} With notation as in the proof  $c_{ij}$ changes if and only if  $c_{ji}$ changes and both change if and only if the disc slide is anti-invariant with $|e\cap R_i\cap \mW_{ji}|=1 \mod 2$ (resp. $|e\cap G_j\cap \mF_{ji}|=1\mod 2$) for a $f_{ij}\in \mF_{ij}$ $\mR$-disc-slide (resp. $w_{ij}\in \mW_{ij}$ $\mG$-disc slide).\end{remark}

\begin{definition}  \label{cu} (IA cross term correction)  Suppose $(\mF, \mW)\in$ IA and has $k$-eyes.  For $1\le u\le k$ and $i\neq j$ define $CU_{ij}^u(\mF,\mW):=CU_{ji}^u(\mF, \mW)$ as follows.  Let $f_1, w_1, \cdots, f_n, w_n$ denote the IA order on the $\mF_{uu}\cup\mW_{uu}$ discs.  Let $a_r:=|(\mW_{ij}\cup\mW_{ji})\cap f_r\cap G_u|$ and $b_s=|(\mF_{ij}\cup\mF_{ji})\cap w_s\cap R_u|$. Then mod 2: 
$$CU_{ij}^u(\mF,\mW):=\sum_{r\le s} a_r b_s,\ \  CU_{ij}(\mF,\mW):=\sum_{u=1}^k CU_{ij}^u(\mF,\mW),\ \  \textrm{and}\ \  CU(\mF,\mW):=\sum_{i<j} CU_{ij}(\mF,\mW).$$ \end{definition}

\begin{remark} \label{cu remark}  If $u\notin \{i,j\}$, then $CU^u_{ij}=0$.\end{remark}

\begin{definition}\label{cross term final}  Define the \emph{cross term} invariants 
$$C_{ij}(\mF, \mW)=CU_{ij}(\mF,\mW)+c_{ij}(\mF,\mW)\quad and\quad C(\mF, \mW)=\sum_{i<j} C_{ij}(\mF,\mW) \mod 2$$
and the $\mD$-invariant $$\mathcal{D}(\mF,\mW)=I(\mF,\mW)+C(\mF,\mW)\mod 2$$
\end{definition}

\begin{lemma} \label{d boundary twisting} If $\mW$ is ordered, then $\mD\mfw$ is invariant under boundary twisting.\end{lemma}
\begin{proof}  Since $((\mF_u\cup\mW_u)\cap\mR\cap\mG))\cap(\mF_c\cup\mW_c)=\emptyset$ it follows that CU is invariant under boundary twisting and hence by Lemma \ref{c boundary twisting} so is C.  The result now follows from Lemma \ref{i boundary twisting}.  \end{proof}

The main result of this paper is that $\mD$ is a nontrivial  invariant of its relative homotopy class as a loop of embedded multi-spheres.

\begin{remark} \label{asymmetry} For the purposes of this remark let $\mD_{\mF\mG}$ denote the $\mD$ as defined above.  That notation recognizes the asymmetry in the definition of $\mD$.  Indeed, there is the corresponding definition $\mD_{\mF\mR}$ where the roles of $\mF$ and $\mW$, or equivalently $\mG$ and $\mR$ are reversed, i.e.  $\mD_{\mF\mR}\mfw=\mD_{\mF\mG}(\mW,\mF)$. E.g. we could have defined $\mR_{\mW}, \mG_{\mF}$ to be dual spheres respectively disjoint from $\mW\cup\mW^*_S$ and $\mF\cup\mF^*_S$ and accordingly modified the definitions of $c_{ij}(\mF,\mW)$ and $CU^u_{ij}(\mF,\mW)$.  For example, $\hat \mF_{ij}=\mF_{ij}\cup\bigcup_{p\neq j} |B_{ij}^R\cap G_p| G^*_p$, where $G^*_p\in\mG_{\mF}$.   \end{remark}

\begin{lemma}  \label{cu vs cij} \emph{(Cross term IA disc slide invariance)}  Let $(\mF', \mW')$ be obtained from $(\mF,\mW)\in$ IA by a disc slide sequence.  Then for all $i\neq j$ 
$$ CU_{ij}(\mF', \mW')+c_{ij}(\mF', \mW')=CU_{ij}(\mF,\mW)+c_{ij}(\mF,\mW).$$  The similar statement holds with $c_{ji}$ in place of $c_{ij}$.  Finally, $C\mfw=C(\mF', \mW')$.  \end{lemma}

\begin{proof} If an uncross disc slides over another disc, then neither $c_{ij}$ nor $CU_{ij}$ changes since the defining data is unchanged mod 2.   Next we  consider the representative case that $f_{21}\in \mF_{21}$ disc slides over $f_{pq}\in \mF_{pq}$.  By  Definition \ref{cu} and Remark \ref{cu remark}  the only terms of $CU(\mF, \mW)$ that can change are $CU^v_{21}(\mF,\mW), v=1,2$ and this can only happen when $(p,q)=(1,1)$ or $(2,2)$.  By Lemma \ref{slide ct invariance} only $c_{21}$ and $c_{12}$ can change.  Thus, it suffices to consider the effect on these terms of sliding $f_{21}$ over a $f_{11}$ or $f_{22}$ disc.  
\vskip 6pt
\noindent\emph{Subcase 1.} $f_{21}$ $G$-disc slides over $f_{11}$:  By Lemma \ref{slide ct invariance} the $c_{ij}$'s are unchanged.  The $CU^v_{21}$'s are also unchanged since their $a_r, b_s$ values are unchanged.  \qed
\vskip 6pt
\noindent\emph{Subcase 2.} $f_{21}$ $R$-disc slides over $f_{22}$:  Here $CU^1_{21}$ is unchanged since its $a_r, b_s$ values are unchanged.  Let $e=f_{22}\cap G_2$.  If $|e\cap \mW_{12}|=0 \mod 2$, then by Lemma \ref{slide ct invariance} $c_{21}$ and $c_{12}$ are unchanged and there are no intersections of $e$ with $\mW_{21}$.   If $f_{22} = f_r$ in the IA order, then with primes denoting the effect of the disc slide, $b'_r=b_r+1$ and $b'_{r-1}= b_{r-1}+1$, unless $r=1$ in which case $b_{r-1}$ does not exist.    Since $|e\cap \mW_{12}|=0 \mod$ 2 It follows that $CU^1_{21}$ is unchanged.  

If $|e\cap\mW_{12}|=1\mod 2$, then both $c_{21}$ and $c_{12}$ are changed by Lemma \ref{slide ct invariance}.  Since $a'_r=1\mod 2$ it follows that $CU^1_{21}$ is also changed.  The conclusion $C\mfw=C(\mF', \mW')$ now follows by Definition \ref{cross term final}.\end{proof}

Combining Lemmas \ref{cu vs cij} and \ref{d boundary twisting} and multi-eye Proposition \ref{hsf independence} we obtain:

\begin{corollary}  \label{d disc slide} \emph{($\mD$-invariance under  IA disc sliding)}  Let $(\mF', \mW')$ be obtained from $(\mF,\mW)\in$ IA by a disc slide sequence, then $\mD\mfw=\mD(\mF',\mW')$.  \qed\end{corollary}

\begin{definition}  Given $(\mF,\mW)$ with $\mW$ ordered, define $CU_{ij}^u(\mF, \mW)=CU_{ij}^u(\mF,\mW')$ where $(\mF, \mW')$ is obtained by switching $(\mF, \mW)$.  \end{definition}

\begin{lemma} \label{cu switching} $CU_{ij}^u(\mF,\mW)$ is well defined, i.e. independent of the choice of switching.\end{lemma}

\begin{proof}  Let $\mW'$ be a switching of $\mW$.   It suffices to show that if $\psi$ is a basic restandardization map and $ \mW''=\psi(\mW')$, then $CU_{ij}^u(\mF, \mW')=CU^u_{ij}(\mF, \mW'')$.  Here and in  what follows primes and double primes refer to data associated to  the switches.  Recall that $\psi$ fixes $\mR\cup\mG$ setwise and the non switch discs point wise. In particular, $\psi$ fixes point wise all the cross discs.  It suffices to consider the representative case $u=1$ and $ij=21$.  Let $\mF_{11}\cup \mW'_{11}=f_1,w_1', \cdots, f_n, w_n'$ in the IA order.  Since $\mF$ and $\mW_{21}$ are unchanged  we have for each $i, f_i\cap \mW_{21}'\cap G_1=f_i\cap\mW_{21}''\cap G_1$.  Therefore, to complete the proof it suffices to show that for each $i, |w_i''\cap \mF_{12}\cap R_1|=|w_i'\cap \mF_{12}\cap R_1|\mod 2$.  We now consider the basic restandardization maps.
\vskip 6pt
i) $\psi$ is a $G$-finger twist:  If the twist is about $w\in \mW$, then $\psi|R_1$ is a Dehn twist on $\partial N(w\cap R_1)$.  If  $w\in \mW_{11}$, then  the intersections of $ \mF_{12}$ with $\partial N(w\cap R_1)$ arise in pairs, one pair for each point of $\mF_{12}\cap w\cap R_1$, so $CU^1_{21}$ is unchanged.  If $ w\notin\mW_{11}$, then $\mW_{11}'\cup \mW_{11}''$ are disjoint from the support of $\psi$.\qed
\vskip 6pt
ii) $\psi $ is a G-braiding:  Here $\psi$ fixes $R_1$ pointwise.\qed
\vskip 6pt
iii) $\psi$ is an $R$-braiding:  Here $\psi$ is a composition of Dehn twists about the simple closed curves $\alpha_i, \alpha_j, \alpha_{ij}\subset R_1$.  See Figure 15 which shows the $G$-braiding analogue.  Here $\alpha_i=\partial N(w_i\cap R_1), \alpha_j=\partial N(w_j\cap R_1)$ and $\alpha_{ij}$ is the result of banding $\alpha_i$ to $\alpha_j$.  We are abusing notation since $w_i$ need not be in $\mW_{11}$ as in that figure.  Since each of $\alpha_i, \alpha_j, \alpha_{ij}$ intersect $\mF_{12}$ in an even number of points, the result follows.  Indeed, if $w_i\in \mW_{12}$, then $\alpha_i$ has one intersection with $\mF_{12}$ for each point of $w_i\cap R_1\cap G_2$ and two intersections for each point of $\mF_{12}\cap \inte(w_i\cap R_1)$.  If $w_i\notin \mW_{12}$ it only picks up pairs of points of the latter type.  Since we can assume that the core of the band is transverse to $\mF_{12}$ the result for $\alpha_{ij} $ also follows.\qed
\vskip 6pt
iv), v) $\psi$ is a spinning or a half-disc map:  Here $\psi$ fixes $R_1$ pointwise.\qed
\vskip 6pt
vi) $\psi$ is an SO(3) twist:  Here $\psi $ acts on $\mR$ by a square of a Dehn twist.  This completes the proof of the lemma.\end{proof}  

\begin{proposition}  \label{cross switching} \emph{(Cross term   invariance under switching)} Let $(\mF, \mW)$ have $\mW$ ordered.  If $(\mF, \mW')$ is obtained from $(\mF, \mW)$ by switching, then for all $i\neq j$, $c_{ij}(\mF, \mW')=c_{ij}(\mF, \mW)$.  Furthermore, $C_{ij}(\mF, \mW')=C_{ij}(\mF,\mW)$ and $C(\mF, \mW')=C(\mF, \mW)$.    \end{proposition}  

\begin{proof}  We will prove the first assertion.  That  plus Lemma \ref{cu switching} implies the second and the third then follows by definition.  If $j\neq i$, then $\mW_{ji}'=\mW_{ji}$ since switching fixes cross discs point wise.  Since $\mR\cap \mG$ is fixed under switching we have for $i\neq j$ that  $B_{ij}^G$ and $B_{ji}^R$ are unchanged and hence so is each $\hat\mF_{ij}$ and $A_{ij}$.  Since there also exists a common dual system $\mG^*$ to $\mW'$ and $\mW$ by Lemmas \ref{common switch dual} and \ref{dual independence}, it follows that $\hat\mW_{ji}'=\hat\mW_{ji}$ from which we conclude $c_{ij}(\mF, \mW')=c_{ij}(\mF, \mW)$.   \end{proof}

\begin{corollary} \label{d switching} \emph{($\mD$-invariance under switching)} If $\mW$ is ordered, then $\mD\mfw$ is well defined, i.e. it's value is independent of the choice of switching to $(\mF, \mW')\in$ IA.\end{corollary} 

\begin{proof}  The previous result implies that $C\mfw$ is independent of the choice of switching and Proposition \ref{multi switch invariance} implies that I$\mfw$ is independent of the choice of switching.\end{proof}  

\begin{lemma}  \label{d boundary twisting2} \emph{($\mD$ invariance under boundary twisting)}  Let $\mF'$ (resp. $\mW'$) be obtained from $\mF$ (resp. $\mW$) by an $\mR$ or $\mG$-twist, then \textbf{c}$(\mF', \mW)=\textbf{c}\mfw,  \bC(\mF', \mW)=\bC\mfw$ and hence $\mD(\mF', \mW)=\mD\mfw$ with similar results holding for $(\mF, \mW')$.\qed\end{lemma}

\begin{definition}  Two sets $W_1, W_2$ of Whitney discs are said to have \emph{common boundary} if $\partial W_1$ can be isotoped to $\partial W_2$ via an isotopy fixing $\mR\cap\mG$ point wise.\end{definition}

\begin{remarks} \label{common disjoint} i) After an isotopy of either $\mW_1$ or $\mW_2$ fixing $\mR\cap\mG$ point wise we can assume that whenever  $w\in W_1$ and $w'\in W_2$ have common boundary, then $\partial w\cap (\mW_2\setminus w')=\emptyset$ and $\partial w'\cap(\mW_1\setminus w)=\emptyset$.  We will always assume that these slightly stronger properties hold.

ii) This condition is the same as saying that $\partial W_1$ can be homotoped to $\partial W_2$ via an isotopy fixing $\mR\cap\mG$ point wise and the same as saying that $W_1$ can be isotoped so that the boundaries coincide, however the germs might not coincide.\end{remarks}

\begin{definition}  We say $\mfw$ is in \emph{full embedded arc position} if it is in embedded arc position and $\mF_c$, $\mW_c$ have common boundary.  \end{definition}

\begin{proposition}  \label{full symmetry} \emph{(Full embedded implies $\mD$ is symmetric)}  i) If $\mfw$ is in full embedded arc position, then for $i\neq j, c_{ij}(\mF, \mW)=|\mF_{ij}\cap \mW_{ji}|=|\mF_{ji}\cap \mW_{ij}|=c_{ji}$ mod 2. 

ii) For all $i\neq j, c_{ij}(\mF,\mW)=c_{ij}(\mW, \mF)=c_{ji}(\mF, \mW)=c_{ji}(\mW,\mF)$, $C(\mF, \mW)=C(\mW, \mF)$ and $\mD(\mF, \mW)=\mD(\mW, \mF)$.\end{proposition}

\begin{proof}  If conclusion i) holds, then it immediately implies the first set of equalities of ii).  The other two equalities follow after noting that $CU\mfw=CU(\mW,\mF)=0$ and $I\mfw=I(\mW,\mF)$.  

We now prove i).  Choose the $\infty$ points to be disjoint from the various disjoint discs that contain each $(\mW_{ij}\cup\mF_{ij})\cap \mG_j$ and the various disjoint discs that contain each $(\mW_{ij}\cup\mF_{ij})\cap R_i$.  Then for all $i\neq j$,  $A_{ij}=0$, $\hat\mF_{ij}=\mF_{ij}$, $\hat\mW_{ij}=\mW_{ij}$ and hence $c_{ij}\mfw=|\mF_{ij}\cap\mW_{ji}|$ mod 2.   For $i\neq j$, let $V_{ij}=\mF_{ij}\cup\mW_{ij}\cup B_{ij}^G\cup B_{ij}^R$.  Using Lemma \ref{even winding} we see $[V_{ij}]=a_{ij}[R_j]+b_{ij}[G_i]$ where $a_{ij}+b_{ij}=0$ mod 2.  Since $V_{ij}\cap V_{ji}\subset (\inte(\mW_{ij}\cup\mW_{ji}))\cap (\inte(\mF_{ij}\cup\mF_{ji}))$, $|\mF_{ij}\cap\mW_{ji}|+|\mF_{ji}\cap\mW_{ij}|=|V_{ij}\cap V_{ji}|=a_{ij}b_{ji}+a_{ji}b_{ij}=0$ mod 2 and hence $|\mF_{ij}\cap\mW_{ji}|=|\mF_{ji}\cap \mW_{ij}|$ mod 2.  \end{proof}

\begin{lemma} \label{common fc} Given $\mfw$ with $\mW$ ordered and $\mF_c, \mW_c$ having common boundary, then each $C^u_{ij}=0$ and for $i\neq j, c_{ij}\mfw=|\mF_{ij}\cap\mW_{ji}|=|\mF_{ji}\cap\mW_{ij}|=c_{ji}\mfw$.\end{lemma}

\begin{proof}  If $\mW'$ is a switching of $\mW$, then $\partial \mW'_u$ can be isotoped to be disjoint from $\mF_c\cup\mW'_c$ and hence each $C^u_{ij}=0$.  Since  $\mW'_c=\mW_c$ the proof follows as in that of  Lemma \ref{full symmetry}.\end{proof}

The sum square move introduced by Frank Quinn \cite{Qu} p. 355 replaces two Whitney discs by two others using a \emph{square} disc.  We now introduce a special version which replaces two Whitney discs by two others using a third Whitney disc.  The new Whitney discs comprise the third Whitney disc plus a disc which is approximately the union of the original three.  We then show $\mD$-invariance under the move when $w_1$ and $w_2$  are cross discs.

\begin{definition} \label{sum square} (Special sum square move)  Given $\mfw$ let $w_1, w_2\in \mW$ be Whitney germed and $D$ a finger germed Whitney disc such that $\mW\cap D\subset \mR\cap\mG$, $w_1\cap D\neq\emptyset$ and  $w_2\cap D\neq \emptyset$.  A  $(w_1, w_2, D)$ \emph{special sum square move} transforms $\mW$ to $(\mW \setminus \{w_1,w_2\}) \cup \{w,D\}$ where $w$ is constructed as follows.  Start with $w_1\cup w_2\cup D$ and resolve the $w_1\cap D$ and $w_2\cap D$ intersections as in Figure \ref{fig:sss5}.  First, view these intersections as crossings of arcs in 3-space and resolve them with half twisted bands as in Figures \ref{fig:sss5} b) or c) to obtain $w'$.  Then eliminate the arcs $\alpha,\beta\subset w'\cap G_j$ by pushing the interiors of both into either the future or the past to obtain $w$.  Any of these four possibilities for $w$ is called a $(w_1, w_2, D)$ \emph{special sum square move}.  Here $D$ is called the \emph{sum square disc}.  If $w_1, w_2$ are uncross (resp, cross) discs, then we call this an \emph{uncross (resp. cross) special sum square move}.

In a similar manner, we can define the special sum square move for a pair of finger germed Whitney discs  where $D$ is now Whitney germed.  Similarly we can define this move for finger or Whitney germed finger discs.\end{definition}

\begin{figure}[!htbp]
    \centering
    \includegraphics[width=.9\linewidth]{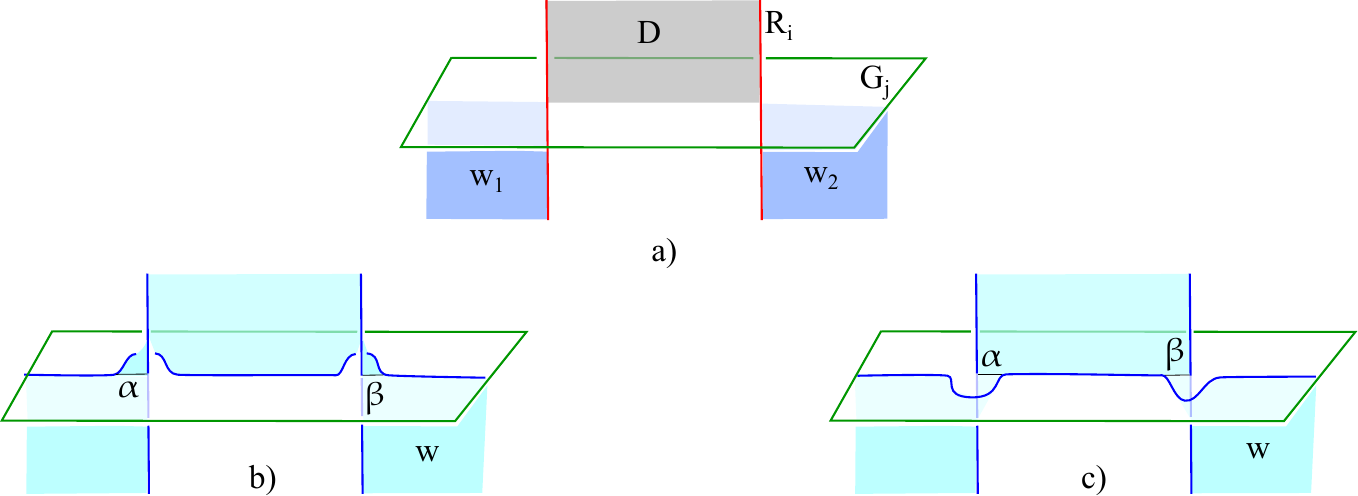}
    \caption{The special sum square move}
    \label{fig:sss5}
    \end{figure}

\begin{remarks}\label{quinn square}  i)  There are 4 ways to resolve each $w_i\cap D$ intersection.  First, we can use either a right or left half twist to \emph{resolve the crossing}.  Second the $\alpha$ or $\beta$ intersection can be eliminated by pushing its interior into either the past or future.  Thus there are a priori 16 $(w_1, w_2, D)$ special sum square moves.  However, only the four ways given in Definition \ref{sum square} will produce a w that is both correctly framed and can be isotoped to be disjoint from $D$.

ii)  If $\mR\cup\mG$ are in $\mG$-framed finger form with $\mW$ the standard Whitney discs with $w_1, w_2$ adjacent and $D$ the in between standard finger disc, then the special sum square move corresponds to a transformation shown in Figures \ref{fig:sss4} a) to b) or a) to c), exhibited as 1D representations.  The reason that figure shows only two transformations instead of four is that the 1D representation relies on a preliminary isotopy of $R_i$ to project to much of $G_j$ and there  are two ways to do this.  The next lemma asserts that every special sum square move can be conjugated by an ambient isotopy to be of this form.  

iii) The Whitney discs of  Figure \ref{fig:sss4} b) (resp. c)) are the result of doing the special sum square move using the sum square $D$ shown in Figure \ref{fig:sss4} e) (resp. d)).  Equivalently, they are obtained by using the Quinn sum square disc $S$ shown in those figures.  

The next two items show how the four ways are related and will not be used in the paper.

iv) Call the four possible resulting $w$'s as $w_a,w_b,w_c,w_d$. Then there is a $B^4$ containing $w_1,w_2,D,w_a,w_b,w_c,w_d$ and self-diffeomorphism of $B^4$, possibly non orientation preserving, taking $(w_1,w_2,D,w_a)$ to $(w_1,w_2,D,w_b)$ to $(w_1,w_2,D,w_c)$ to $(w_1,w_2,D,w_d)$. Hence, there is only one local model for the special sum square move.

v) All of the four $w$'s  are related by disc slides over $D$.  
\end{remarks}

\begin{lemma} \label{sum criterian} \emph{(The special sum square move is a Quinn sum square move)}  Given $\mfw$ and a sum square triple $(w_1, w_2, D)$ there exists an ambient isotopy conjugating the corresponding special sum square move to a sum square move depicted in Figure \ref{fig:sss4}.\end{lemma}

\begin{proof} First consider the case that $w$ and $w'$ are uncross discs, where it suffices to consider the case that $w_1\cup w_2 \in \mW_{11}$.  Choose an ordering on $\mW$ so that $w_1$ is a minimal element followed by $w_2$.  Next put $\mR$ into $\mG$-framed finger form with $\mW$ the standard Whitney discs.   After a sequence of the basic restandardization maps, $D$ becomes the standard $f_2$.  If $w_1$ and $w_2$ are cross discs, then put $\mR$ into $\mG$-framed finger form with $\mW$ standard such that $w_1$ and $w_2$ are adjacent Whitney discs.  Again a restandardization map takes $D$ to the standard finger disc between them.  Remarks \ref{quinn square} explains how a special sum square move on this standard triple corresponds to a sum square move of Figure \ref{fig:sss4}.  \end{proof}

 \begin{figure}[!htbp]
    \centering
    \includegraphics[width=.9\linewidth]{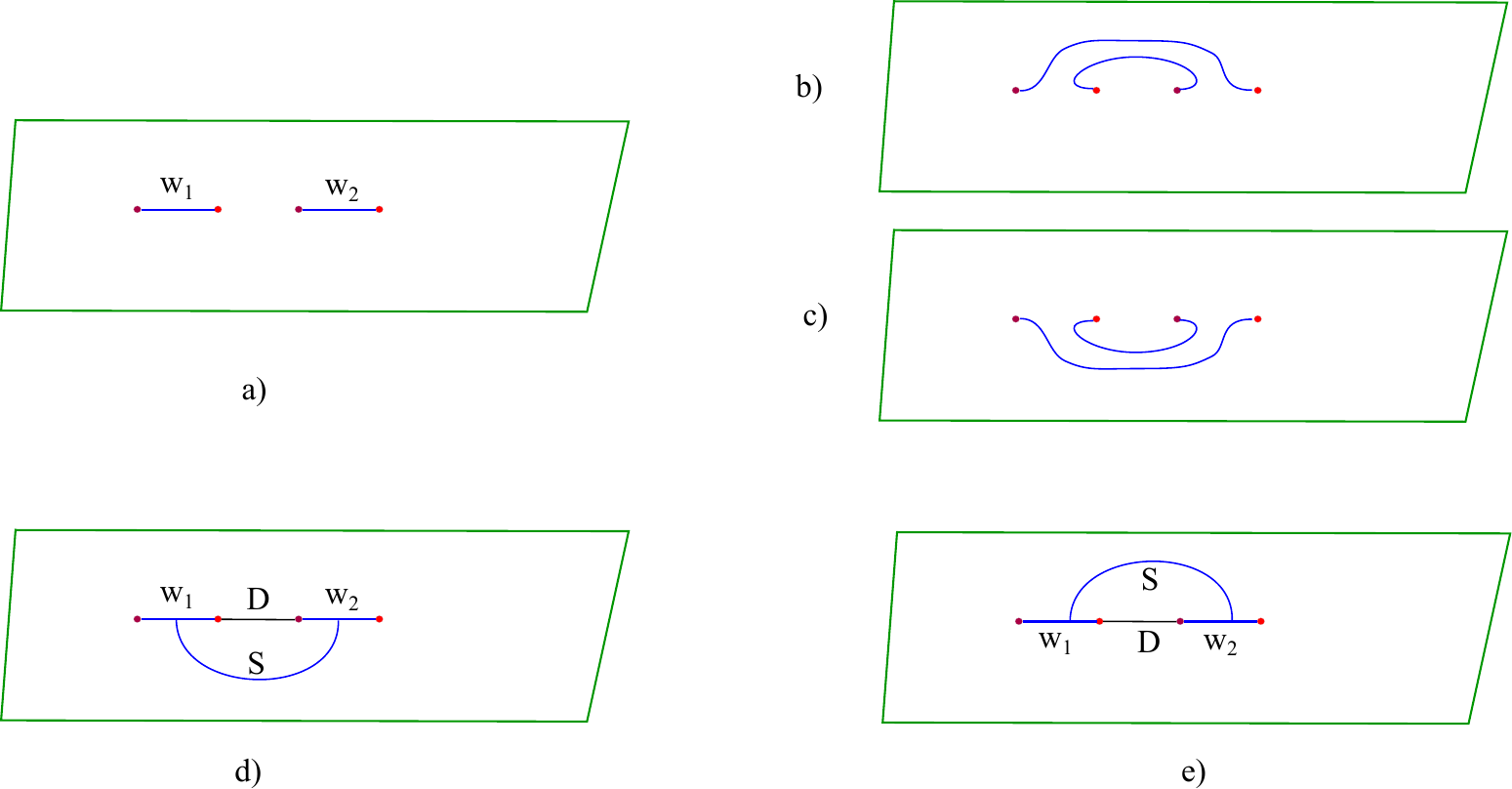}
    \caption{The special sum square move as a Quinn sum square move}
    \label{fig:sss4}
    \end{figure} 

\begin{remark} The use of the ordering is to inform us that there is an ambient isotopy taking $(w_1, w_2, D)$ to a standard position.\end{remark}

\begin{lemma} \label{common d} \emph{(Cross sum square moves preserve $\mD$)} If $(\mF, \mW')$ (resp. $(\mF', \mW)$) is obtained from $\mfw$ by a cross special sum square move on $\mW$ (resp. $\mF$), then with any ordering on $\mW$ we have $  \mD(\mF, \mW')=\mD\mfw$ (resp. $  \mD(\mF', \mW)=\mD\mfw$).\end{lemma}

\begin{proof} Suppose that $\mW'$ is obtained from $\mW$ by a  $(w_1, w_2, D)$ cross special sum square move.  As in the proof of Lemma \ref{sum criterian}, given any ordering on $\mW$,  there is a $\mG$-framed finger form in which $\mW$ is standard as are $w_1, w_2$ and $D$.  Since $w_1, w_2$ are cross discs and $D$ is standard we can use the same switch discs for switchings of $\mW$ and $\mW'$.  Since their uncross discs coincide,  it follows by the multi-eye parity lemma that $I(\mF, \mW')=I\mfw$.  We also have $C(\mF, \mW')=C\mfw$, since we have the same $|B_{ij}^G\cap R_p|, |B_{ij}^R\cap G_p|$ and $A_{ij}$ values and any new intersections arising in the computation of $CU$ come in pairs.  If $\mF'$ is obtained from $\mF$ by a special sum square move, then again we have by the parity lemma, $I(\mF', \mW)=I\mfw$ and the same cancellation argument gives $C(\mF', \mW)=C\mfw$.   \end{proof}

\begin{corollary} \label{cross matched} \emph{(Cross discs with common boundary)} Given $\mfw$ there exists $\mF'$ obtained from $\mF$ by cross special sum square moves such that $ \mF'_c$ and $ \mW_c$ are similarly matched.  Furthermore, for any such $\mF'$ and any ordering on $\mW$ we have $\mD(\mF', \mW)=\mD\mfw$.\end{corollary}

\begin{proof}  The proof of the first assertion is by upward induction on the number of immersed cycles of $(\partial \mF_c\cup\partial \mW_c)\cap \mG$.  Suppose $(\partial f\cup \partial w\cup \partial f')\cap \mG$ is part of an immersed cycle where $f\neq f'\in \mF_{ij}$ and $i\neq j$.  Put $\mR\cup\mG $ into $\mG$-framed finger form with $\mF$ standard and $f, f'$ adjacent.  Take $D$ as the standard Whitney germed Whitney disc between them and do a $(f, f', D)$-special sum square move to increase the number of immersed cycles by one.  The second assertion follows from Lemma \ref{common d}.\end{proof}

\section{More on Disc Sliding and Switching}

This section comprises technical lemmas needed for the next two sections, which prove $\mD\mfw$ is independent of the ordering on $\mW$.  Each lemma has two versions, one where $\mF_u, \mW_u$ are similarly matched and one where in addition $\mF_c, \mW_c$ have common boundary.  In the former case we obtain asymmetric conclusions, e.g. invariance of $\mD$ under only $\mW$ $\mR$-disc slides, in part reflecting the asymmetry in the definitions of c and CU.  For the latter we are freed of this asymmetry, e.g. can conclude invariance of $\mD$ under $\mW$ disc slides of both $\mR$ and $\mG$ type.  Intuitively, once $\mF_c, \mW_c$ have common boundary we can work as if there are no cross discs at all.

\begin{remark} To prove independence of ordering on $\mfw$ we will find a sequence $(\mF_1, \mW_1)=\mfw, (\mF_2, \mW_2), \cdots, (\mF_n, \mW_n)$ such that each $\mW_i$ has two orderings $\mW_i', \mW_i''$ such that if $\mD(\mF_1, \mW'_1)\neq \mD(\mF_1, \mW''_1)$, then for all $i$, $\mD(\mF_i, \mW'_i)\neq \mD(\mF_i, \mW''_i)$ and then obtain a contradiction by showing $\mD(\mF_n, \mW'_n)= \mD(\mF_n, \mW''_n)$. This motivates the following:\end{remark}

\begin{definition}Suppose that $\mW_1$ (resp.$ \mW_2$) is endowed with two  orderings $\mW_1', \mW_1''$ (resp. $\mW_2', \mW_2''$).  We say that $(\mF_1, \mW_1)$, $ (\mF_2, \mW_2)$ are \emph{$\mD$-equivalent} with respect to these orderings  if $\mD(\mF_1, \mW_1')=\mD(\mF_2, \mW_2')$ and  $\mD(\mF_1, \mW_1'')=\mD(\mF_2, \mW_2'')$. 

If $\mW_u, \mW'_u$  are  complete similarly matched sets of uncross Whitney discs, then an ordering on $\mW$ \emph{induces} one on $\mW'$.

 A \emph{transposition} on the ordered $\mW$ is a  reordering such that one differs from the other by changing the order of two  adjacent elements in a single $\mwii$. \end{definition}

\begin{lemma}  \label{multieye slide invariance} i) Let $\mF_u, \mW_u$ be similarly matched with $\mW$ ordered.  If $\mW'$ (resp. $\mF'$) is obtained from  $\mW$ (resp. $\mF$)  by an $\mR$ (resp. $\mG$) disc slide sequence, then with the induced ordering on $\mW'$, $\mD(\mF,\mW)=\mD(\mF,\mW')$ (resp. $\mD(\mF,\mW)=\mD(\mF',\mW)$). 

ii) If in addition $\mF_c, \mW_c$ have common boundary and $\mW'$ is obtained from $\mW$ by a $\mG$-disc slide sequence only involving uncross discs sliding over either uncross or cross discs, then with the induced ordering on $\mW'$, $\mD(\mF,\mW)=\mD(\mF,\mW')$. \end{lemma}

Unlike this result, Lemmas \ref{hsf independence} and \ref{cu vs cij} require  $(\mF, \mW)\in IA$.

\begin{proof} i) We first consider the case where $w_i\in \mW$, $R_p$-disc slides over $w_j$.  A disc slide of $w_i$ over $w_j$ is determined by an arc $\alpha$ from $w_i$ to  $w_j$, where $i\neq j$.  For now assume that $\alpha$ is efficient.  This means that if we are in  $\mR$-framed finger form with the $\mW$ discs standard, then $\inte(\alpha)\cap (\mW\cup \mW^1)=\emptyset$ where $\mW^1$ is the standard full switch. 

If $w_i$ is slid over $w_j$ and both are uncross discs, then it is evident that $\bC$ is unchanged. We will show that $\mW^1$ is $R_p$-disc slide equivalent to a $\mW^2$ which is an \textbf{n}-switch to $\mW'$. It then  follows that $I(\mF,\mW)=I(\mF, \mW^1)=I(\mF, \mW^2)=I(\mF,\mW')$, the second equality following from Lemma \ref{hsf independence}.  Let $\mW'_{pp}=\{w_1', \cdots, w_n'\}$ and $\mW^1_{pp}=\{w_1^1, \cdots, w_n^1\}$.  To construct $\mW^2$ need to remove the intersections $\mW'\cap \mW^1\setminus \mR\cap \mG$ and these occur between $\partial w_i'\cap R_p$ and $\partial w_j^1\cap R_p $ and $\partial w_{j+1}^1\cap R_p$ or just $w_j^1$, if $j=n$.  If $i<j$, then we obtain $\mW^2=\{w_1^2, \cdots, w_n^2\}$ by letting $w_q^2=w_q^1$ if $q\neq j$ or $j+1$ and respectively obtaining $w_{j+1}^2$  and $ w_j^2$ from $w_{j+1}^1$ and $w_j^1$ by $R_p$-disc slides over $w_i^1, w_{i-1}^1,  \cdots, w_1^1$.  In a similar manner if $j<i$, then we remove the non $\mR\cap\mG$  intersections of $\mW^1$ with $w'_i$ at the cost of creating  intersections with $w'_n$, but these are removed as in the proof of Lemma \ref{ia to ea}.  

If $w_i$ is an uncross disc and $w_j$ is a cross disc or both are cross discs, then $\mW^1$ serves as a full switch for both $\mW$ and $\mW'$ and hence $\mD(\mF, \mW)=\mD(\mF, \mW^1)=\mD(\mF, \mW')$ since these equalities hold at the $I$ and $C$ levels.   

It remains to consider the case of cross disc sliding over an uncross disc.   We can assume that $p=1$.  By Lemma \ref{slide ct invariance}, $\textbf{c}(\mF, \mW')=\textbf{c}\mfw$, hence we need  to show that $CU(\mF,\mW)$ changes if and only if $\I\mfw$ changes.     Now assume that the cross disc is $w_{12}\in \mW_{12}$ and the uncross disc is the $j$'th disc in the ordered $\mW_{11}=\{w_1, \cdots, w_n\}$.  Let $\mW^1=\{w_1^1, \cdots, w^1_n\}$ and let  $\mW^2$ denote the full-switch of $\mW'$ constructed as follows.  First, $w'_{12}$ intersects each of $w_j^1, w_{j+1}^1$ in single points and these lie in $R_1$, although only $w_j^1$ is intersected if $j=n$.  These points can be removed and $w_j^2, w_{j+1}^2$ can be constructed by respectively \emph{banding} them to \emph{caps} about a point of $w_{12}\cap G_2$ with the resulting intersections with $G_2$ \emph{tubed off} with copies of dual the sphere $G^*_2$.  Figure \ref{fig:cross uncross slide} shows an example where $j=2$.  Figure \ref{fig:cross uncross slide} a) shows $\mW_{11}, \mW_{11}^1$ and $w_{12}$.  Figure \ref{fig:cross uncross slide} b) shows the disc slid $w_{12}'$ obtained by disc siding $w_{12}$ over $w_2$ and Figure \ref{fig:cross uncross slide} c) shows the construction of $w^2_2, w^2_3$, though only parts of these discs are shown.  Note that $\mW_{11}$ is Whitney germed and $\mW^1_{11} $ is finger germed.

 \begin{figure}[!htbp]
    \centering
    \includegraphics[width=.6\linewidth]{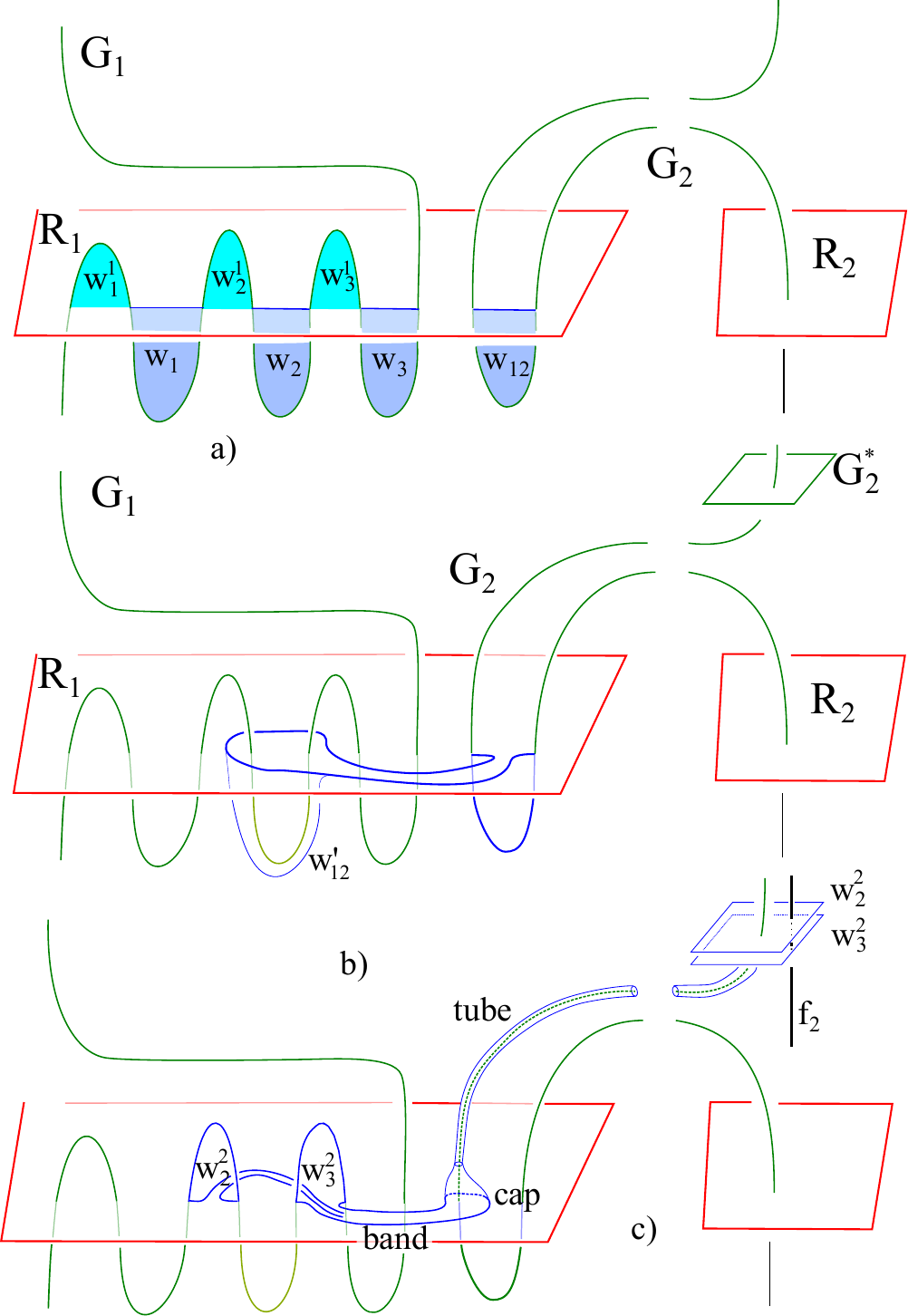}
    \caption{Constructing a switching for $\mW'$}
    \label{fig:cross uncross slide}
    \end{figure}

Now $(\mF, \mW^1), (\mF, \mW^2)\in$ IA with the IA order on $\mF_{11}\cup \mW_{11}^1$ being $f_n, w_n^1, f_{n-1}, \cdots, f_1, w_1^1$ where each $f_i$ matches $w_i$.  Similarly $f_n, w_n^2, f_{n-1}, \cdots, f_1, w_1^2$ gives the IA order on $\mF_{11}\cup\mW_{11}^2$.  To complete the proof it suffices to establish the following claims.  
\vskip 6pt
\noindent\emph{Claim 1}:  $\I\mfw=\I(\mF, \mW')$ if and only if $\hat \I\mfw=\hat \I(\mF,\mW')$ if and only if $|f_j\cap G_2^*|=0$ mod 2.
\vskip 6pt
\noindent\emph{Proof}.  The first if and only if follows from the multi-eye parity lemma.  Viewing $\mW^1,\mW^2$ as finger germed and $\mF$ as Whitney germed we see no intersections of $\mF$ with the bands and since the tubes follow arcs in $G_2$, they also have no intersection with $\mF_{11}$.  The two parallel $G_2^*$'s may intersect $\mF_{11}$, however except for intersections with $f_j$ all the others come in cancelling pairs for the calculation of $\hat I(\mF, \mW^2)$ compared with that of $\hat I(\mF, \mW^1)$.\qed

\vskip 6pt
\noindent\emph{Claim 2}:  $CU(\mF, \mW^1)=CU(\mF, \mW^2) $ if and only if $|f_j\cap G_2^*|=0 $ mod 2.
\vskip 6pt
\noindent\emph{Proof}.  It is immediate that $CU^u_{pq}(\mF, \mW^1)=CU^u_{pq}(\mF, \mW^2)$ unless $u=1$ and $p,q=2,1$.  We now calculate $CU^1_{21}(\mF, \mW^2)-CU^1_{21} (\mF, \mW^1)$ where numbers below are mod 2.  Ignoring cancelling intersections along the bands the points $\mW_{11}^2\cap \mF_{12}\cap R_1\setminus (\mW_{11}^1\cap\mF_{12}\cap R_1)$ come from the caps. Each of $w^2_j\cap R_1$ and $w^2_{j+1}$ picks up a single new intersection with $\mF_{12}$ and that's it.  From the definition of $C^1_{21}$ it follows that $C^1_{21}(\mF, \mW^2)-C^1_{21}(\mF, \mW^1)=|\mW_{21}\cap f_j\cap G_1|$.  

Let $V$ be the closed surface which is the union of $f_j, w_j$ and the checker board surfaces $G^j$ and $R^j$ which are respectively in $G_1$ and $R_1$ and are bounded by $(f_j\cup w_j)\cap G_1$ and $(f_j\cup w_j)\cap R_1$.      Then $[V]=\sum_{i=1}^k a_i[G_i] + b_i[R_i]\in H_2(\#_k\stwostwo,\BZ_2)$, where for $i\ge 2.\ a_i=|G^j\cap R_i|$ and $b_i=|R^j\cap G_i|$.  Since $(\inte(f_j)\cup\inte(w_j))\cap (\mR\cup\mG)=\emptyset, G_2^*\cap \mW=\emptyset$ and $G_2^*\cap (R_1\cup G_1)=\emptyset$ it follows that $|f_j\cap G_2^*|=|V\cap G_2^*|=a_2$.  Since $R_2\cap G_1=\partial \mW_{21}\cap G_1$ and the elements of $\mW$ are pairwise disjoint it follows from a $\BZ_2$-intersection argument that $|\mW_{21}\cap f_j\cap G_1|=a_2$. We conclude $|\mW_{21}\cap f_j\cap G_1|= |f_j\cap G_2^*|$. 
\vskip 6pt
It remains to consider the case that $\alpha=f(\alpha_1)$, where $\alpha_1$ is an efficient path from $w_i $ to $w_j$ and $f\in \Diff(R_1, N(\mG))$.   Let $\mW''$ denote $\mW$ modified by disc sliding $w_i$ over $w_j$ using the path $\alpha_1$.  There exists a diffeomorphism $\tau$ of $\#_k\stwostwo$ which fixes $\mW$ pointwise and $\mR\cup\mG$ set wise such that  $\tau|R_1=f$ and hence $\mD(\tau^{-1}(\mF),\mW)=\mD(\tau^{-1}(\mF), \tau^{-1}(\mW))=\mD(\mF,\mW)$.  Note that $\tau(\mW')=\mW''$ since  $\mW'$ and $\mW''$ are determined by the paths $\alpha$ and $\alpha_1$ and $\tau(\alpha_1)=\alpha$. It follows that $\mD(\mF,\mW'')=\mD(\mF, \tau(\mW'))=\mD(\tau^{-1}(\mF),\mW')=\mD(\tau^{-1}(\mF),\mW)=\mD(\mF, \mW)$, where the third equality follows from the efficient case.
\vskip 6pt

By Lemma \ref{d boundary twisting} $\mD$ is invariant under boundary twisting.
\vskip 6pt
Now suppose that $\mF'$ is obtained from $\mF$ by a $\mG$-disc slide sequence.  We have $\mD(\mF', \mW)=\mD(\mF', \mW^1)=\mD(\mF, \mW^1)=\mD\mfw$ where $\mW^1$ is a switching of $\mW$ where the middle equality follows from Lemma \ref{d disc slide} and the first and last equalities follow from Lemma \ref{d switching}.
\vskip 8pt
ii) Since $\mF_c,\mW_c$ have common boundary after isotopy the interior of $\alpha$ is disjoint from the cross discs, so $C\mfw$ is unchanged.  The proof that $\I\mfw$ is unchanged follows as in i).   \end{proof}

\begin{lemma}  \label{multieye multislide invariance} i) Let $\mF, \mW, \mW'$  be complete systems of  Whitney germed Whitney discs with $\mF_u, \mW_u, \mW'_u$ similarly matched and $\mW_c=\mW'_c$.   Suppose that the ordered $\mW$ has a full Whitney framed switching $\mW^*$ such that $\mW'\cap\mW^*\subset \mR$, then with the induced ordering on $\mW'$ we have $\mD(\mF,\mW)=\mD(\mF,\mW')$.
 
 ii)  In addition assume that  $\mF_c, \mW_c$ have common boundary.   If the ordered $\mW$ has a full Whitney framed    switching $\mW^*$ such that $\mW'\cap\mW^*\subset \mG$, then with the induced ordering on $\mW'$ we have $\mD(\mF,\mW)=\mD(\mF,\mW')$.\end{lemma}
 
 \begin{proof} i) We will show there exist $\mR$-disc slide sequences $\mW'=\mW'_1, \cdots, \mW_p'$ and $\mW^*=\mW^*_1, \cdots, \mW^*_q$ such that $\mW_q^*$ is a switching of $\mW_p'$.  It then follows that $\mD(\mF, \mW')=\mD(\mF, \mW_p')=\mD(\mF, \mW_q^*)=\mD(\mF, \mW^*)=\mD(\mF, \mW)$ where the first   equality follows from Lemma \ref{multieye slide invariance} i).  The third equality also follows from that lemma though proved earlier since $(\mF, \mW^*)\in$ IA.

The points of $\mW^*\cap \mW'\setminus \mR\cup\mG$ are among the uncross discs, since $\mW_c= \mW'_c$ and $\mW^*$ being a switching of $\mW$ has new uncross discs but the same cross discs.  	The proof of multi-eye Lemma \ref{ia to ea} shows that when $(\mW', \mW^*)\in IA$ and $\mW'\cap\mW^*\subset \mR$, a series of $\mR$-disc slides and $G$-twisting to both,  reduce them to $(\mW'_p, \mW^*_q)\in$ EA such that $\inte(\mW'_p)\cap\inte(\mW^*_q)=\emptyset$.  	Since the discs being slid are ones that have excess intersections, these discs are uncross discs, thus $\mW'_p$ and $\mW^*_q$ have the same cross discs.  It follows that $\mW'_p$ is a switching of $\mW^*_q$. 	
\vskip 8pt
ii) The proof follows as in i) except that the disc slides are $\mG$-disc slides and that all the disc slides from $\mW'$ to $\mW'_p$ and $\mW^*$ to $\mW^*_q$ can be taken to be of uncross discs over other discs, hence Lemma \ref{multieye slide invariance} ii) applies. \end{proof} 

\begin{section} {Clasping}\end{section}

We now define \emph{clasping} which replaces two discs of either $\mW$ or $\mF$ by another pair of discs.  Like disc sliding the resulting say $\mW'$ is similarly matched with $\mW$.  Between the two operations we have considerable freedom in modifying $\mW\cap\mR$ and $\mW\cap \mG$  as well as $\mF\cap\mR$ and $\mF\cap \mG$.

\begin{definition}  \label{clasp} Given   a complete ordered system $\mW$ of Whitney germed Whitney discs with $\mW_1$ a full \textbf{n}-switch, we define the notion of a \emph{$\mW\ \mG$-clasping}. A $\mW\ \mR$-clasping is defined similarly with the roles of $\mR$ and $\mG$  reversed.  In the analogous way we define clasping for finger germed Whitney discs.  Roughly, if $\mW'$ is obtained from $\mW$ by clasping $w_i$ and $w_j$, then $w_i'$ is obtained from $w_i$ by first attaching a band beginning at $\partial w_i\cap \mG$ and ending at a disc which links a point $v\in w_j\cap\mR\cap \mG$, then the resulting intersection with $\mR$ is tubed off with a dual sphere.  Further, $w_j'$ is obtained in a like manner so that $w_i'\cap w_j'=\emptyset$.  See Figure \ref{fig:gclasp}.

More precisely construct  $\mW'$ as follows.  To start with, define $w_p'=w_p$ unless $p\in i,j$.  Here $w_i, w_j$ intersect the same component $G$ of $\mG$ and respectively intersect components $R_1$ and $R_2$ of $\mR$ with possibly $R_1=R_2$.   The data for $w_i' $ is  a \emph{clasp arc}, a \emph{clasp cap}, a \emph{clasp tube} and a \emph{dual sphere}.  See Figure \ref{fig:gclasp}.  The clasp arc $\alpha \subset G$ is an embedded path from $w_i$ to $v\in w_j\cap\mR\cap \mG$, with $\inte(\alpha)\cap \mW=\emptyset$.  A clasp cap for $w'_j$  is a 2-disc $D_v$ obtained from a small disc $\subset G$ about $v$ whose interior is pushed slightly down, i.e. to the -side of $G$.  We then  band $w_i$ to $D_v$  by a band that follows $\alpha$ and then the segment $\subset w_j$ with endpoint $v$ to obtain a disc $w_i^a$.  We perturb $w_i^a$ slightly to obtain an embedded disc, which is almost a Whitney disc except that it intersects $R_2$ in a point in its interior.  We abuse notation by continuing to call this disc $w_i^a$.  A clasp cap for $w'_i$  is a 2-disc $D_u$ obtained from a small disc $\subset G$ about a point $u\in w_i\cap R_1 \cap G$ whose interior is pushed slightly down.  Our $w^a_j$ is obtained by banding $w_j$ to $D_u$ by band that starts at $w^a_i\cap w_j$, then follows the segment of $w'_i\setminus w_j$  with endpoint $u$.  Again, $w_j^a$ is perturbed slightly to be a Whitney disc except for its intersection with $R_1$.   Note that the local 3D projections of $w^a_i, w^a_j$  intersect in an embedded arc $\beta$ and their intersections in $G$ interleave.    By pushing part of $\inte(w_i^a)$ (resp. $\inte(w_j^a)$) slightly into the past (resp. future) near $\beta$  we obtain disjoint discs $w_i^b, w_j^b$ each of which intersects $\mR$ at one point in its interior.  

 We eliminate these intersection points and obtain $w_i'$  (resp. $w_j'$) by tubing it to a copy of an  $\mG_{\mW, \mW_1}$ dual sphere $R_2^*$ (resp. $R_1^*$) to  $R_2$ (resp. $R_1$) along a path $\gamma_i\subset R_2$, (resp. $\gamma_j\subset R_1$) disjoint from $\mW\cup\mW_1\cup (\mG\cap\mR)$. Also, $(R_1^*\cup R_2^*)\cap (\mW\cup\mW_1)=\emptyset$.      We call $\gamma_i,\gamma_j$ \emph{tube guide paths}.  We say that $\mW'$ is obtained by \emph{clasping $w_i$ to $w_j$}.  If  instead we pushed $w_i^a$ (resp. $w_j^a$) slightly into the future (resp. past), then we call this a \emph{- clasping}.  Note that this is a clasping of $w_j$ to $w_i$.  Without specifying the clasp tube and dual spheres this clasping is called an \emph{$(\alpha, u,v)$-clasping}.
 
If the clasp arc is between two cross discs (resp. cross and uncross discs, resp. two uncross discs), then we call this a \emph{cross/cross} (resp. \emph{cross/uncross}, resp. \emph{uncross/uncross}) clasping.
 \end{definition}

 \begin{figure}[!htbp]
    \centering
    \includegraphics[width=0.7\linewidth]{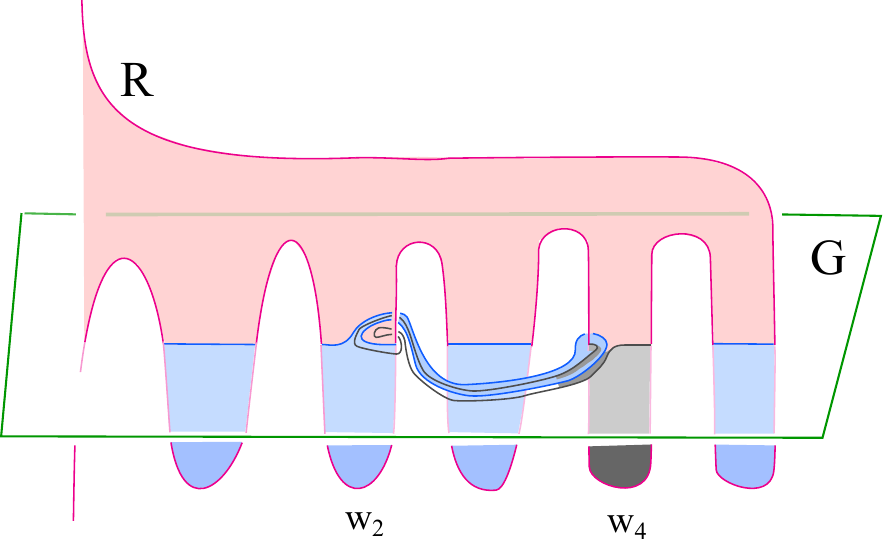}
    \caption{Partial construction of a $\mW$ $\mG$-clasping}
    \label{fig:gclasp}
\end{figure}

\begin{remarks} \label{clasp rea}  i) By construction $\mW$ has a full n-switch $\mW_1$ such that $(\mW', \mW_1)\in R$-EA and $\mW'\cap\mW_1\subset \mG$.   

ii) If $\mW$ is finger germed, then the construction is similar except that $w^a_i$, $w^a_j$ are pushed to the $+$side of $G$.  \end{remarks}

The following definition is given in the generality needed for this paper.

\begin{definition} \label{free point} (Cross free point clasping) Define $\mW\ R_i$-\emph{cross free point clasping} as follows where $\mW$ is Whitney germed.  Let $\alpha\subset R_i$ be an embedded path from $w\in \mW_c$ to $x\in R_i\cap G_i$ where $\inte(\alpha)\cap \mW=\emptyset$ and $x\notin\mW$.  We obtain $\mW':=(\mW\setminus w)\cup w'$ were $w'$ is built from $w$ and a \emph{clasp arc}, \emph{clasp cap}, \emph{clasp tube} and a dual sphere.  The clasp cap is a compact disc neighborhood $\subset G_i$ of $x$ pushed slightly down.  Construct $w^a$ by banding $w$ to the cap following $\alpha$.  The resulting $w^a\cap G_i$ intersection is then tubed off to a dual sphere $G_i^*$ via a tube disjoint from $\mW$.  This paper will only consider cross free point claspings for $\mW$  when $\mfw\in IA$.    \end{definition}

\begin{lemma} \label{cij clasp} \emph{($c_{ij}$ Clasp invariance)} Each $c_{ij}$ is invariant under $\mF$ $G$-clasping and $\mW$ $R$-clasping.\end{lemma}

\begin{proof}  It suffices to consider the representative case of $i,j=2,1$.  The only claspings that effect $\hat \mF_{21}, \hat \mW_{12}, B^G_{21}\cap (\mR\setminus R_2)$ or $B^R_{12}\cap (\mG\setminus G_1)$ are $G_1$-claspings of an $f_{21}$ with an $f_{p1}$ or $R_1$-claspings of a $w_{12}$ with a $w_{1q}$.  
\vskip 6pt
\noindent\emph{Case 1}.  An $f_{21}$ $G_1$-clasps an $f_{p1}$, possibly itself.
\vskip 6pt
\noindent\emph{Proof of Case 1}.  To start with $\hat W_{21}$ is unchanged as is $B^R_{12}$.  The only change to $\mF_{21}$ is on $f_{21}$ and $f_{p1}$, if $p=2$.  Recall that the change is as follows.  A small disc is removed from $f_{21}$ near $G_1$ and replaced by a feeler which intersects $R_p$ near $x\in R_p\cap G_1$.  The intersection with $R_p$ is then removed by tubing to a copy of $R^*_p$.  Since the feeler lies close to the $+$-side of $G_1$ it is disjoint from $\mW$.  The tube may intersect $\mW$, however it can only intersect $\mW_{12}$ if $p=1$.  In that case the number of intersections $= 1$ mod 2 if and only if $x\in B^R_{12}$.  This will be addressed in the next paragraph.  Finally, the added copy of $R_p^*$ may intersect $\hat\mW_{12}$.  If $p\neq 2$, then under the clasp $B_{21}^G$ either gains or loses $x\in R_p$, hence $\hat \mF_{21}\setminus \mF_{21}$ gains or loses a copy of $R^*_p$.  Hence mod 2, the intersections of $\hat\mW_{12}$ with the $R^*_p$ added to $f_{21}$ are counter balanced by those from the change in $\hat \mF_{21}\setminus \mF_{21}$.  If $p=2$ and distinct $f_{21}$'s clasp, then each picks up a copy of $R_2^*$ while if $f_{21}$ self-clasps, then it picks up two copies of $R_2^*$.
\vskip 6pt
Now suppose $p=1$ and $x\in B_{12}^R$.  If $B_{21}'$ (resp. $A_{21}'$) denotes the $B_{21}^G$ (resp. $A_{21}$) following the clasp, then either $x\in B_{21}^G$ and $x\notin B_{21}'$ or vice versa.  In either case $A'_{21}-A_{21}=1 \mod$ 2.  Thus, the intersections of the tube with $\hat W_{12}$ is balanced by the change in $A_{21}$.  If $x\notin B_{12}^R$, then $A_{21}$ is unchanged.\qed
\vskip 6pt
\noindent\emph{Case 2}.  A $w_{12}$ $R_1$-clasps a $w_{1q}$, possibly itself.
\vskip 6pt
\noindent\emph{Proof of Case 2}.  Here $\hat \mF_{21}$ and $B_{21}^G$ are unchanged.  As in Case 1, adding the feeler to $\mW_{12}$ does not create intersections with $\hat\mF_{21}$.  The tube will create intersections with $\hat\mF_{21}$ only if $q=1$.  These intersections will be exactly  balanced by changes in $A_{21}$.  The $G_q^*$ added to $w_{12}$ will be balanced by the $G_q^*$ added or deleted from $\hat \mW_{12}\setminus \mW_{12}$ if $q\neq 2$.  If $q=2$, then two $G_2^*$'s will be added to $\mW_{12}$.\end{proof}

\begin{definition}  (Clasp arc sliding) Let $(\mF', \mW)$ be obtained from $(\mF, \mW)$ by $\mF\ G_i$-clasping from $f_p$ to $f_q$ using clasp arc $\alpha$.  We say that $(\mF'', \mW)$ is obtained from $(\mF', \mW)$ by \emph{clasp arc sliding} if it arises by $G_i$-clasping $f_p$ to $f_q$ using $\alpha_1$ which is constructed by sliding $\alpha$ over $f\cap G_i $ where $f\in \mF\setminus f_p, f_q$ as in Figure \ref{fig:clasp arc}  Here the same cap, tube and dual sphere are used in both claspings. In a similar manner define obtaining $(\mF, \mW'')$ from $(\mF, \mW')$  by clasp arc sliding, where $(\mF, \mW')$ is obtained from $(\mF, \mW)$ by $\mW\ R_j$-clasping.  \end{definition}

  \begin{figure}[!htbp]
    \centering
    \includegraphics[width=.8\linewidth]{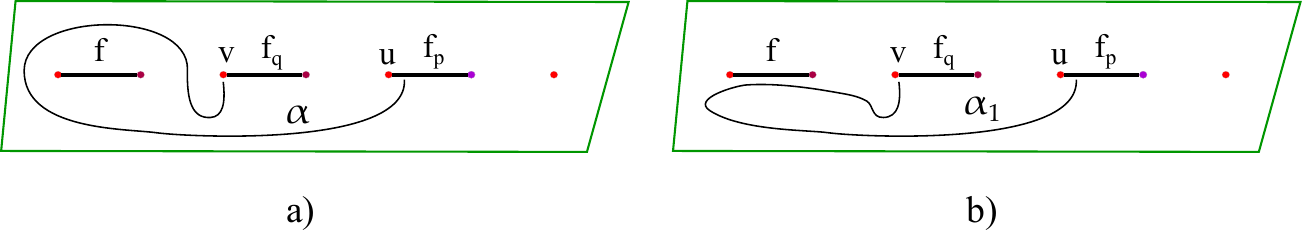}
    \caption{Clasp arc sliding}
    \label{fig:clasp arc}
\end{figure}
Since clasp arc slidings are the result of two disc slides we obtain:

\begin{lemma}  \label{clasp arc} If $(\mF'',\mW)$ is obtained from $(\mF', \mW)$ by $G$-clasp arc sliding, then $\bI(\mF'', \mW)=\bI(\mF', \mW)$ and $\bC(\mF'', \mW)=\bC(\mF', \mW)$.  The analogous statement holds when $(\mF,\mW'')$ is obtained from $(\mF, \mW')$ by $R$-clasp arc sliding. \qed\end{lemma}

\begin{proposition} \label {d invariance}  \emph{($\mD$ invariance under clasping)}   i)  Suppose $(\mF,\mW)$  has either $\mF, \mW$ similarly matched and $\mW$ ordered or $(\mF,\mW)\in$ IA.  If $(\mF', \mW')$ is obtained from $(\mF, \mW)$ by either an $\mF\ G$-clasp or a $\mW\ R$-clasp of uncross/cross or cross/cross type then, $\mD(\mF, \mW)=\mD(\mF', \mW')$.   The clasping may be a cross free point clasp.  When $\mF$ and $\mW$ are similarly matched, $\mW'$ has the ordering induced from $\mW$.

ii a) If $\mF_u$ and $\mW_u$ are similarly matched, $\mW$ ordered and $\mW'$  obtained from $\mW$ by an uncross/uncross $\mR$-clasping, then with the ordering induced from $\mW,\ \mD\mfw=\mD(\mF, \mW')$.

ii b)  If  $\mF_c$ and $\mW_c$ also have common boundary and   $\mW'$  is obtained from $\mW$ by an uncross/uncross $\mG$-clasping, then with the ordering induced from $\mW,\ \mD\mfw=\mD(\mF, \mW')$.\end{proposition}

\noindent\emph{Proof of ii)} a)  By definition  $\mW$ has a Whitney framed   switching $\mW^1$ such that $\mW'\cap \mW\subset \mR$.  The conclusion now follows from Lemma \ref{multieye multislide invariance} i).  For ii b) apply Lemma \ref{multieye multislide invariance} ii). \qed
\vskip 8pt

The following result reduces the rest of proof to the case of $(\mF, \mW)\in$ IA.

\begin{lemma} \label{clasp commutation 2}\emph{(Clasp - switching commutation)} Suppose that $\mF$ and $\mW$ are similarly matched, $\mW$ is ordered and $\mfw$ has $k$-eyes.  

a)  If  $(\mF', \mW)$ is obtained from $(\mF, \mW) $ by  a $\mF\ G$-cross/cross (resp. uncross/cross) clasping, then there is a  switching $\mW_1$ of $\mW$ such that $(\mF', \mW_1)$ is obtained from $(\mF, \mW_1)$ by an $\mF\ G$ cross/cross (resp. uncross/cross) clasping.

b) If $(\mF, \mW')$ is obtained from $(\mF, \mW) $ by  a $\mW\ R$-cross/cross (resp. uncross/cross) clasping, then there exists $(\mF, \mW'')$ obtained by clasp arc sliding and switchings $(\mF, \mW_1), (\mF, \mW''_1)$ respectively of $(\mF, \mW), (\mF, \mW'')$ such that $(\mF, \mW''_1 )$ is obtained from $(\mF, \mW_1)$ by a $\mW\ R$-cross/cross (resp. uncross/cross, resp. uncross/cross or cross free point) clasping.  \end{lemma}  

  \begin{figure}[!htbp]
    \centering
    \includegraphics[width=.9\linewidth]{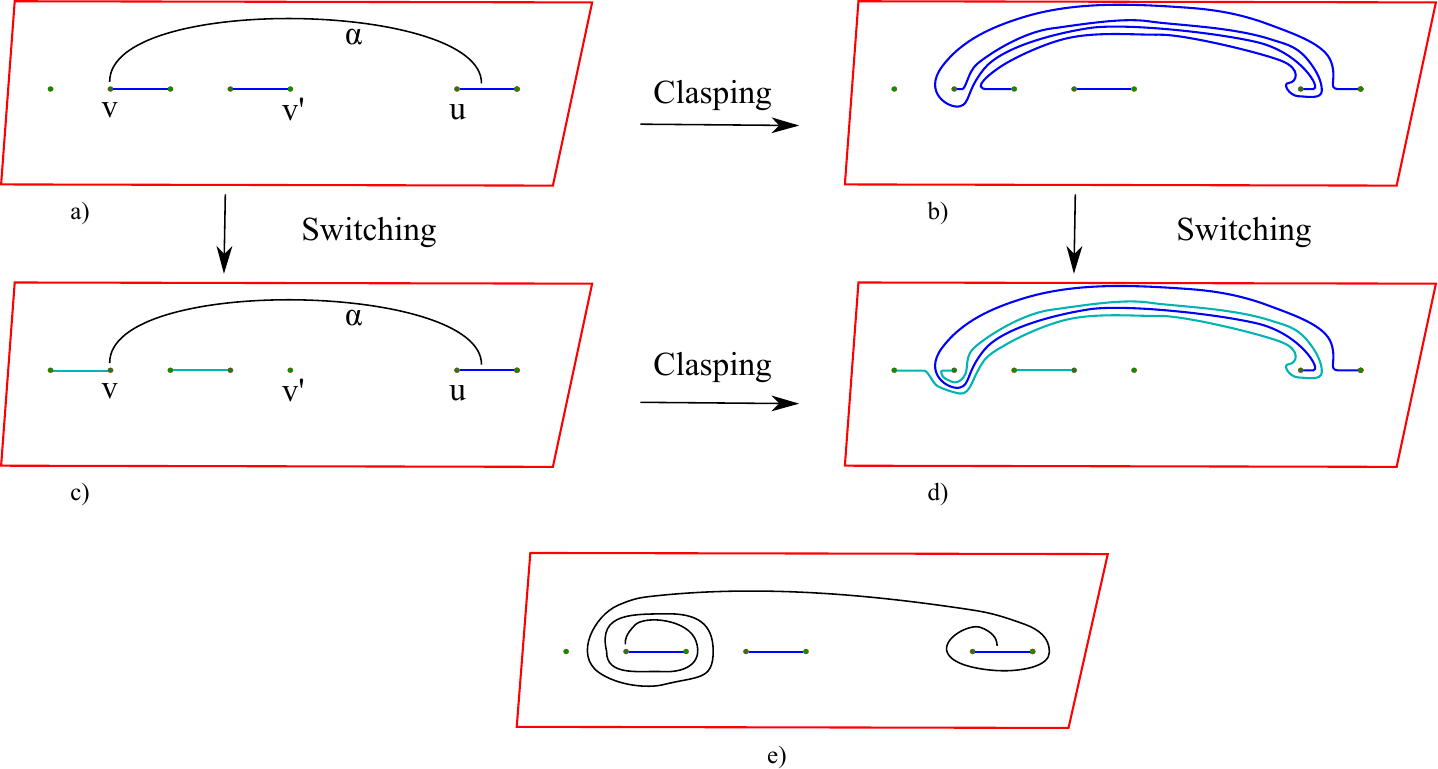}
    \caption{Clasp - switching commutation}
    \label{fig:clasp commutation}
\end{figure}

\begin{proof}  Part a) follows since switching commutes with $\mF$  $\mG$-clasping.  For b) we first consider the uncross/cross case.  Put $\mR\cup\mG$ in framed finger form with the $\mW$ discs in standard position.   Let $\mW^*$ denote the standard full switching with $\mW^*_S$ the switch discs.  We consider the representative case where $w_{11}\in \mW_{11}$, $w_{12}\in \mW_{12}$ and they $(\alpha, u,v)$ clasp with $v\in w_{11}\cap R_1\cap G_1$.  If $\inte(\alpha)\cap \mW^*_S=\emptyset$, then the conclusion holds without clasp arc sliding.  See Figures \ref{fig:clasp commutation} a) - d).  If $v\notin \mW^*_S$, e.g. $v=v'$ in Figure \ref{fig:clasp commutation} a),  then this is a cross free point clasping.  

If $\alpha\cap \mW^*_S\neq\emptyset$, then by a sequence of clasp arc slides we obtain a clasp arc $\alpha''$ with $\inte(\alpha'')\cap \mW_S^*=\emptyset$ up to winding about $w_{11}$.  See Figure \ref{fig:clasp commutation} e).  Let $\mW''$ denote the result of this $(\alpha'', u,v)$ clasping.  By suitably finger twisting $\mW^*$ we obtain the switching $\mW_1$ of $\mW$ such that there is a $\mW_1''$ which is both a clasping of $\mW_1$ and a switching of $\mW''$.  Again, this is a cross free point clasping if $v\notin \mW_{1_S}$.

When the clasping is of the cross/cross type, then after sliding the clasp arc off of $\mW^*_S$ the switching commutes with the clasping.\end{proof}

\begin{lemma} \emph{(IA Cross free point clasp $\mD$-invariance)} \label{free point invariance} Let $(\mF, \mW)\in$ IA with $k$ eyes and $(\mF, \mW')$ obtained by a  cross free point  $\mW\ R$-clasp.  Then for all $i, I_i(\mF, \mW)=I_i(\mF, \mW')$ and for all $i\neq j, c_{ij}(\mF, \mW)=c_{ij}(\mF, \mW')$ and $C_{ij}(\mF, \mW)=C_{ij}(\mF, \mW')$ and hence $\mD\mfw=\mD(\mF,\mW')$.  The similar conclusion holds for cross free point  $\mF$ $G$-clasping.  \end{lemma}

\begin{proof}  Let $x$ denote the free point.  It suffices to consider the representative case that $w_{12}\in \mW_{12}$ is clasping $x\in R_1\cap G_1$.  For all $i\neq j$ and $u$, each $CU^u_{ij}$ is unchanged since its defining data is unchanged.  Since only $w_{12}$ is changed, only  $c_{21}$ can possibly change.  Let $w_{12}'$ denote the clasped $w_{12}$.  It is obtained from $w_{12}$ by  attaching a band parallel to the clasp arc and then a cap about $x$ and then removing the new intersection with $G_1$  by tubing with a copy of $G_1^*$.  Since $B_{12}^R\cap G_1$ differs from $B_{12}'\cap G_1$ only at $x$ it follows that the copy of $G_1^*$ added to $w_{12}$ is canceled by a copy of $G_1^*$ added or deleted from $\hat W_{12}$.  Thus $|\hat\mF_{21}\cap \hat\mW_{12}'|-|\hat\mF_{21}\cap \hat\mW_{12}|$ is equal mod 2 to the number of times the tube crosses $\mF_{21}\cap G_1$.  That number equals 1 if and only if $x\in B_{21}^G$.  However, $x\in B_{21}^G$ if and only if $A_{21}'-A_{21}=1\mod 2$ and hence $c_{21}$ is unchanged.  Since only a cross disc is changed, the multi-eye parameter lemma implies that for each $i, I_i(\mF,\mW)=I_i(\mF, \mW')$.  \end{proof}

\noindent\emph{Proof of Proposition \ref{d invariance} continued.}  By Lemmas \ref{clasp commutation 2}, \ref{d switching} and \ref{clasp arc} we can assume that $(\mF, \mW)\in$ IA.  To complete the proof we will consider the case of an $\mF\ G$-clasp as the proof for the $\mW\ R$-clasp is similar.  We now show that the proposition holds for cross/cross claspings.   By the multi-eye parity lemma $I(\mF,\mW)$ is unchanged.   Also, each $CU^u_{ij}$ is unchanged since new intersections of $\partial \mF$ with $\partial \mW_u$ come in pairs from the band.  Finally by Lemma \ref{cij clasp}  each $c_{ij}$ is unchanged.  It remains to consider the uncross/cross clasping case.   By reindexing we can  assume that $f_{11}\in \mF_{11}$  and $f_{21}\in \mF_{21}$ are clasping.    Let $\mF_{11}\cup\mW_{11} = f_1, w_1, \cdots, f_n, w_n$ in the IA order and hence $f_{11}= f_p$ some $p$.  
\vskip 8pt
To complete the proof of the proposition it suffices to prove:
\vskip 6pt
a)  $C(\mF',\mW)=C(\mF, \mW)+\sum_{j=p}^n s_j\quad$ and 
\vskip 6pt
b) $I(\mF', \mW)=I(\mF, \mW)+\sum_{j=p}^n s_j$.
\vskip 6pt
\noindent\emph{Proof a)}.  By Lemma \ref{cij clasp} only CU can change and by our parametrization $CU_{ij}^u$ can only  change when $u=1$ and $i,j=2,1$.   Here $r_{p}'=r_p+1$ with the other $r_i$'s and $s_j$'s unchanged, the new $f'_p \cap w_{21}\cap G_1$ intersection coming from the boundary of the $f'_p$ clasp cap.  Therefore, $C'(\mF', \mW) -C(\mF, \mW)=\sum_{j=p}^n s_j$.\qed

\vskip 8pt
\noindent\emph{Proof of b)} This follows from the following three assertions, computed mod 2.
\vskip 6pt
i)  $\I(\mF',\mW)-\I\mfw=\hat \I(\mF', \mW)-\hat \I\mfw$,
\vskip 6pt
ii) $\hat I(\mF', \mW)=\hat I(\mF, \mW)+\sum_{j=p}^n |R_2^*\cap w_j|$  and 
\vskip 6pt
iii) for each $j\ge p, |R_2^*\cap w_j|=|w_j\cap\mF_{12}\cap R_1|=s_j$.
\vskip 6pt

Assertion i) follows from the parity lemma, since all new intersections of an $f\in \mF_u$ with a $w\in \mW_u$ come in pairs and arise from intersections with the band following the clasp arc.  

For ii)  note that  $f_p' $ is obtained from $f_p$ by attaching a band and clasp cap and then tubing off the resulting intersection with $R_2$ by with a copy of $R_2^*$.  The band and clasp cap have interiors disjoint from $\mW_{11}$ as does the small disc initially removed from $f_p$, thus the change in $\hat I$  comes from intersections of the $w_j$'s with $R_2^*$, but only those with $ j\ge p$ contribute to the count.

We now show iii).   Ambiently isotope $\mR$ to be in framed finger form such that $f_1, \cdots, f_n$ are the standard linearly ordered $\mF_{11}$ discs.  Let $w_j^{\textrm{std}}$ denote the standard Whitney disc for $\mR\cup\mG$ that matches $w_j$.  Now $(w_j^{\textrm{std}}\cup w_j)\cap G_1$ bounds a checker board surface $S_G$ and $(w_j^{\textrm{std}}\cup w_j)\cap R_1$ bounds a checkerboard surface $S_R$.  Define $V=w_j^{\textrm{std}}\cup w_j\cup S_R\cup S_G$.  After a small perturbation, $V$ is a closed surface such that $[V]=\sum_{i=1}^k a_i[R_i]+b_i[G_i]\in H_2(\#_k\stwostwo, \BZ_2)$ where $a_i=|V\cap G_i|$ and $b_i=|V\cap R_i|$.  
First, $|S_R\cap G_2|=|V\cap G_2|=|w_j\cap R_2^*|$.  The first equality follows since $G_2\cap(G_1\cup w_j\cup w_j^{\textrm{std}})=\emptyset$.  The second by homology and the third by construction, i.e. $R_2^*\cap (w_j^{\textrm{std}}\cup G_1\cup R_1)=\emptyset$.  Now $R_1\cap G_2$ is the boundary of the disjoint union of arcs $\mF_{12}\cap R_1$ and these arcs are disjoint from $w_j^{\textrm{std}}$.  A mod 2 intersection argument gives $|S_R\cap G_2|=|f_{12}\cap w_j\cap R_1|$.  This completes the proof of the proposition.  \qed
\vskip 8pt

\begin{lemma}  \label{cross clasping} \emph{(More cross cross clasping)} Let $\mfw$ be in embedded arc position with  $\partial \mF_c\cap \partial\mW_u=\partial \mW_c\cap\partial \mF_u=\emptyset$ and $\mF_c, \mW_c$ similarly matched. If $(\mF', \mW')$ is obtained from $\mfw$ by a cross cross clasping with clasp arc disjoint from $\partial \mF_u\cup\partial \mW_u$, then $\mD\mfw=\mD(\mF', \mW')$ and in particular $\I\mfw=\I(\mF',\mW'), c\mfw=c(\mF', \mW')$ and $CU\mfw=CU(\mF',\mW')=0$.\end{lemma}

\begin{remark}Unlike Proposition \ref{d invariance} this clasping can be either a $\mW$ $\mG$-clasping or an $\mF$ $\mR$-clasping.\end{remark}

\begin{proof} Since both $\mfw$ and $(\mF', \mW')$ are in embedded arc position, $CU\mfw=CU(\mF',\mW')=0$.  Since $\mF_u=\mF'_u$ and $\mW_u=\mW'_u$ it follows that $I(\mF,\mW)=I(\mF',\mW')$.  Thus it remains to prove $c(\mF,\mW)=c(\mF',\mW')$.  If $(\mF',\mW')$ is obtained by either an $\mW$ $\mR$-clasping or $\mF$ $\mG$-clasping, then the result follows from Proposition \ref{d invariance}.  As the proof for $\mF$ $\mR$-clasping is similar, we only give the proof for $\mW$ $\mG$-clasping. For concreteness assume that $w\in \mW_{pr}$ and $v\in \mW_{qr}$ clasp where $r=1, p=2$ and $q=2$ or $3$.

\begin{notation}  In what follows a ' will denote data associated to $\mW$, e.g. $w', v'$ denote the clasped $w,v$ and $B_{ij}^R$ (resp. $B_{ij}^{R'}$) denotes the $B_{ij}^R$ associated to $\mfw$ (resp. $(\mF', \mW')$). Also,  $G^*_{i, \mW}$ or $G^*_{i, \mW'}$ will denote the order induced dual spheres to $G_i$ with respect to $\mW$ or $\mW'$.  That notation  suppresses the full switchings which are in the background.  See Definition \ref{dual system} and Remark \ref{dual remark}.  \end{notation}

Construct $G^*_{i, \mW}$ and $R^*_{i,\mW}$ by giving $\mW$ the IA order, then  putting $\mW$ in $\mG$-framed finger form with $\mW$ the standard Whitney discs and $\mW^*_S$ the standard full switch discs and finally using the natural dual spheres.  Note that $R^*_{q,\mW}$ and $R^*_{p,\mW}$ are respectively  the dual spheres used to construct $w'$ and $v'$.  For all $i$ we can take $R^*_{i, \mW'}=R^*_{i,\mW}$ and for $i\neq 2,3$ we can take $G^*_{i, \mW'}=G^*_{i,\mW}$; however, for $i=2,3$ the dual spheres $G^*_{i, \mW}$ are problematic  because $v'\cap G^*_{2,\mW}\neq\emptyset$ and $G^*_{3,\mW}\cap w'\neq \emptyset$.

We now construct the  dual spheres $G^*_{p,\mW'}$ and $G^*_{q,\mW'}$.   Suppose first that  $p=2$ and $q=3$.    Construct $G^*_{2,\mW'}$ by starting with $G^*_{2,\mW}$ and tubing off $v'\cap G^*_{2,\mW}$ with a dual sphere $S_{v'}$ of $v'$.    In a similar manner construct $G_{3,\mW'}^*$ by starting with $G^*_{3,\mW}$ and tubing off the intersection with $w'$ with a dual sphere $S_{w'}$ of $w'$. As in \cite{Qu},  $S_{w'}$ is obtained by taking two parallel copies $\Rs_a, \Rs_b$ of $G^*_{1,\mW'}=G^*_{1,\mW}$ and tubing off their intersections with $G_1$ by a tube that follows an arc $\gamma_w\subset G_1$ such that $|\gamma_w\cap w'\cap G_1|=1$ and $\gamma_w\cap (\mW'\setminus w')=\emptyset$ with $\partial \gamma_w=(\Rs_a\cup \Rs_b)\cap G_1$.  The construction of $S_{v'}$ is similar and we can construct   $S_{w'}$ and $S_{v'}$ to be disjoint. If $p=q=2$, then set $G^*_{i,\mW'}=G^*_{i, \mW}$ for $i\neq 2$ and construct $G^*_{2,\mW'}$ by starting with $G^*_{2,\mW}$ and tubing off the two   intersections with $v'\cup w'$  with $S_{w'}$ and $S_{v'}$.  Note that since $\mfw$ is in embedded arc position with the cross discs disjoint from the uncross discs this can be done so $\gamma_w\cup\gamma_v$ are disjoint from the boundaries of the uncross discs.

\vskip 6pt
We now show $c\mfw=c(\mF',\mW')$.  Since $\mfw$ and $(\mF',\mW')$ are in embedded arc position, we have for $i\neq j, A_{ij}=A'_{ij}=0$.  In what follows all intersection counts are mod 2.

\vskip 8pt
\noindent\emph{Case 1}: $p=q=2$
\vskip 6pt
\noindent\emph{Proof of Case 1}.  First observe that $\mF=\mF' $ and for all $i\neq j$ and $p\neq i$, $|B_{ij}^G\cap R_p|=|B_{ij}^{G'}\cap R_p|$  and hence $\hat\mF_{ij}=\hat\mF'_{ij}$.  

We now show for all $r\neq s$ and $(r,s)\neq (1,2)$ that $c_{rs}\mfw=c_{rs}(\mF',\mW')$.  We do this by first showing for all $i$ and $r\neq s$,  $|\hat\mF'_{rs}\cap G^*_{i,\mW'}|=|\hat\mF'_{rs}\cap G^*_{i,\mW}|$. Assuming that,  since $|B_{ij}^R\cap G_p|=|B_{ij}^{R'}\cap G_p|$ and $\mW_{pq}=\mW'_{pq}$ for  $(p,q) \neq (2,1)$,  it then follows that $c_{rs}(\mF',\mW')=|\hat\mF'_{rs}\cap\hat\mW'_{sr}|=|\hat\mF'_{rs}\cap(\hat\mW'_{sr}\setminus\mW'_{sr})|+|
\hat \mF'_{rs}\cap \mW'_{sr}|=|\hat\mF_{rs}\cap(\hat\mW_{sr}\setminus\mW_{sr})|+|\hat\mF_{rs}\cap\mW_{sr}| =|\hat\mF_{rs}\cap\hat\mW_{sr} |=c_{rs}\mfw$.

We have $|\hat\mF'_{rs}\cap G^*_{i, \mW'}|=|\hat\mF'_{rs}\cap G^*_{i,\mW}|$ unless $i=2$.  Next $|\hat\mF'_{rs}\cap G^*_{2, \mW'}|=|\hat\mF'_{rs}\cap G^*_{2, \mW}|+|\mF'_{rs}\cap(S_{v'}\cup S_{w'})|+|(\hat\mF'_{rs}\setminus\mF'_{rs})\cap (S_{v'}\cup S_{w'})|$, since the tubes from $G^*_{2,\mW}$ to $S_{v'}$ and $S_{w'}$ can be chosen disjoint from $\hat\mF'_{rs}$.  Since $\mF_c$ and $\mW_c$ are similarly matched the $f\in \mF'_c$ which matches $w' $  intersects $S_{w'}$ once but all the other $f\in \mF'_c$   intersect $S_{w'}$ 0 times.  Similarly, $S_{v'}$ intersects the matching disc to $v'$ in $\mF'_c$ once  and the other discs in $\mF'_c$ 0 times, thus for all $r\neq w$, $|\mF'_{rs}\cap (S_{v'}\cup S_{w'})|=0$.  Finally $S_{v'}$ and $S_{w'}$ being homologically trivial implies $|\hat\mF'_{rs}\setminus \mF'_{rs})\cap(S_{v'}\cup S_{w'})|=0$.

Since $c=\sum_{j<i} c_{ij}$ we need to also consider the case $r>p$.  Given that for $r=1$ and $p=2$ we have already shown for $(i,j)\neq (1,2), c_{ij}\mfw=c_{ij}(\mF', \mW')$ establishing Case 1 is equivalent to additionally showing that $c_{12}\mfw=c_{12}(\mF', \mW')$.  To do this we need to also account for the change $w, v\in \mW_{21}$ to $w', v'$  when comparing $|\hat\mF_{12}\cap\hat\mW_{21}|$ with $|\hat\mF'_{12}\cap\hat\mW'_{21}|$.  First recall that $\hat\mF'_{12}=\hat\mF_{12}$.  Next, neither the clasp bands nor clasp caps pick up intersections with $\hat\mF_{12}$, since $\mW$ is Whitney germed and $\mF$ is finger germed.  Since $p=q=2$, the clasp tubes have no intersections with $\hat \mF_{12}$.  Since the dual spheres for the claspings are parallel, intersections with $\mF_{12}$ come in cancelling pairs.  It follows that $|w\cap\hat\mF_{12}|=|w'\cap\hat\mF_{12}|$ and $|v\cap\hat\mF_{12}|=|v'\cap\hat\mF_{12}|$.\qed
\vskip 8pt
\noindent\emph{Case 2}: $p=2$ and $q=3$
\vskip 6pt
\noindent\emph{Proof of Case 2}.  The result follows from the following four claims, the first of which lists some basic facts.  
\vskip 8pt
\noindent\emph{Claim 1}:  i) $|B_{21}^G\cap R_3|-|B_{21}^{G'}\cap R_3|=1$ and if $i\neq 3$, then $|B_{21}^G\cap R_i|=|B_{21}^{G'}\cap R_i|$
\vskip 6pt

ii) $|B_{31}^G\cap R_2|-|B_{31}^{G'}\cap R_2|=1$ and if $i\neq 2$, then $|B_{31}^G\cap R_i|=|B_{31}^{G'}\cap R_i|$
\vskip 6pt
iii) If $(i,j)\neq (2,1)$ or $(3,1)$, then $B_{ij}^G=B_{ij}^{G'}$
\vskip 6pt
iv) $B_{ij}^R=B_{ij}^{R'}$ all $i,j$
\vskip 6pt
v) a) $\hat\mF_{ij}'=\hat\mF_{ij}$, unless $(i,j)=(2,1)$ or $(3,1)$
\vskip 6pt
v) b) $\hat\mF_{21}'=\hat\mF_{21}\cup R^*_{3,\mF}$
\vskip 6pt
v) c) $\hat\mF'_{31}=\hat\mF_{31}\cup R^*_{2,\mF}$
\vskip 6pt
vi) a) $\mW'_{ij}=\mW_{ij}$, unless $(i,j) = (2,1)$ or $(3,1)$
\vskip 6pt
vi) b) $\mW'_{21}=w' \cup \mW_{21}\setminus w$
\vskip 6pt
vi) c) $\mW'_{31}=v' \cup \mW_{31}\setminus v$
\vskip 6pt
vii) $|\hat\mF'_{rs}\cap G^*_{i,\mW'}|=|\hat\mF_{rs}\cap G^*_{i, \mW}|$

\vskip 8pt
\noindent\emph{Proof of Claim 1}.  Items i)-iv) follow by definition and  $\BZ_2$-intersection calculations.  Item v) follows by i), ii) and definition.  Item vi) is by definition.  Item vii) is by v) and definition unless $i=2$ and $(r,s)=(3,1)$ or $i=3$ and $(r,s) =(2,1)$.  For the former we have $|\hat \mF_{31}'\cap G^*_{2,\mW'}|=|(\hat\mF_{31}\cup R^*_{2,\mF})\cap(G^*_{2,\mW}\cup S_{v'})|=|\hat\mF_{31}\cap G^*_{2,\mW}|+|R^*_{2,\mF}\cap G^*_{2,\mW}|+|\hat\mF_{31}\cap S_{v'}|+|R^*_{2,\mF}\cap S_{v'}|$.  The second term equals 1.  As explained in the proof of Case 1  the third term equals 1 and the fourth equals 0.\end{proof}

\noindent\emph{Claim 2}:  $c_{ij}(\mF,\mW)=c_{ij}(\mF',\mW')$ unless $(i,j) = (3,1)$, $(1,3)$, $(2,1)$ or $(1,2)$.

\vskip 8pt
\noindent\emph{Proof of Claim 2}.  If $(i,j)$ is not one of the four  values, then $c'_{ij}:=|\hat\mF_{ij}'\cap\hat\mW_{ji}'|=|\hat\mF_{ij}\cap\hat\mW'_{ji}|=|\hat\mF_{ij}\cap\mW'_{ji}|+\sum_{s\neq i} |B_{ji}^{R'}\cap G_s||\hat\mF_{ij}\cap G^*_{s,\mW'}|=|\hat\mF_{ij}\cap\mW_{ji}|+\sum_{s\neq i} |B_{ji}^R\cap G_s| |\hat\mF_{ij}\cap G^*_{s,\mW}|=c_{ij}$. The second equality follows by v) and  the fourth by iv),  vi) and vii).   \qed
\vskip 8pt

\noindent\emph{Claim 3}:  i) $c_{21}-c_{21}'=|G_3\cap B_{12}^R|$
\vskip 6pt
ii) $c_{31}-c_{31}'=|G_2\cap B_{13}^R|$
\vskip 6pt
iii) $c_{12}-c_{12}'=|G_3\cap B_{12}^R|$
\vskip 6pt
iv) $c_{13}-c_{13}'=|G_2\cap B_{13}^R|$
\vskip 8pt
\noindent\emph{Proof of Claim 3}.  i)  $c'_{21}=c_{21}(\mF',\mW')=|\hat\mF'_{21}\cap\hat\mW'_{12}|+|\hat\mF'_{21}\cap (\bigcup_{i\neq 2}|B_{12}^R\cap G_i|\ G^*_{i, \mW'})|=|\hat\mF_{21}\cap\mW_{12}|+|R^*_{3,\mF}\cap \mW_{12}|+|\hat\mF_{21}\cap(\bigcup_{i\neq 2}|B_{12}^R\cap G_i|\ G^*_{i, \mW})|=c_{21}+|R^*_{3,\mF}\cap \mW_{12}|=c_{21}+|G_3\cap B_{12}^R|$.  The second equality is by Claim 1 iv), third v) b), vi) a) and vii) and  the fourth is by definition.  
\vskip 6pt
We now prove the last equality.   Define $V=w_{12}\cup \mF_{12}\cup B_{12}^R\cup B_{12}^G$.  Then $[V]=a_1[R_1]+\cdots+a_k[R_k]+b_1[G_1]+\cdots+b_k[G_k]\in H_2(\#_k \stwostwo,\BZ_2)$, where for $i\neq2$, $a_i=|B_{12}^R\cap G_i|$ and for $j\neq 1, b_i=|B_{12}^G\cap R_i|$.   Since $[R_{3,\mF}^*]=[G_3]$, $|B^R_{12}\cap G_3|=\langle V, R^*_{3,\mF}\rangle=|R^*_{3,\mF}\cap \mW_{12}|+| R^*_{3,\mF}\cap (R_1\cup G_2\cup \mF_{12})|= |R^*_{3,\mF}\cap \mW_{12}|$.\qed
\vskip 6pt
ii) The proof is essentially that if i).\qed 
\vskip 6pt
iii) We have $c_{12}'=|\hat\mF_{12}'\cap \hat\mW'_{21}|=|\hat\mF_{12}\cap \hat\mW'_{21}|$.  Now $\hat\mW'_{21}=w'\cup(\mW_{21}\setminus w)\bigcup_{i\neq 1}|B^R_{21}\cap G_i|\ G^*_{i,\mW'}$ and $|\hat\mF_{12}\cap w'|=|\hat\mF_{12}\cap w|+|\hat\mF_{12}\cap R^*_{3,\mW}|=|\hat\mF_{12}\cap w|+|\mF_{12}\cap R^*_{3,\mW}|=|\hat\mF_{12}\cap w|+|G_3\cap B_{12}^R|$.  The second equality follows since homologically $[R^*_{3,\mW}]=[G_3]$ and $[\hat\mF_{12}\setminus \mF_{12}]=\sum a_i[G_i]$.  The third equality follows by considering $V$ as in the proof of i) and arguing as in that proof, except  here we have $R^*_{3, \mW}\cap \mW_{12}=\emptyset$.     Therefore, $c'_{12}=|\hat \mF_{12}\cap w|+|G_3\cap B_{12}^R|+|\hat\mF_{12}\cap (\mW_{21}\setminus w)|+ |\hat\mF_{12}\cap (\bigcup_{i\neq 1}|B^R_{21}\cap G_i|\ G^*_{i,\mW'})|=|\hat\mF_{12}\cap\mW_{21}+|G_3\cap B_{12}^R|+ |\hat\mF_{12}\cap (\bigcup_{i\neq 1}|B^R_{21}\cap G_i|\ G^*_{i,\mW})|=c_{12}+|G_3\cap B^R_{12}|$, where the third equality follows from vii) of Claim 1, after again noting $\hat\mF_{12}=\hat\mF_{12}'$.\qed

\vskip 6pt
iv)  The proof is essentially that of iii).\qed
\vskip 8pt
\noindent \emph{Claim 4}:  $|G_3\cap B_{12}^R|=|G_2\cap B_{13}^R|$.
\vskip 6pt
\noindent\emph{Proof of Claim 4}.  See Figure \ref{fig:g trade} which shows the intersection of $R_1$ with $G_2, G_3, w_{12}, w_{13}, f_{12}$ and $f_{13}$.\qed

  \begin{figure}[!htbp]
    \centering
    \includegraphics[width=.4\linewidth]{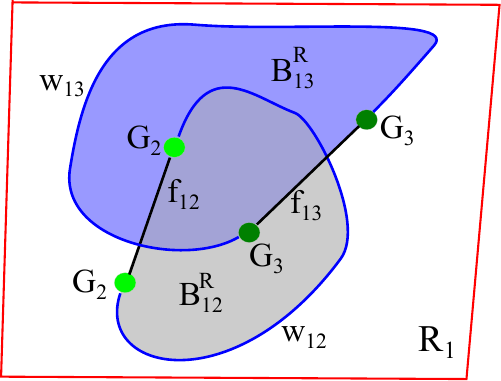}
    \caption{Equating certain $G_2$ and $G_3$ intersections with $R_1$}
    \label{fig:g trade}
\end{figure}
 
\vskip 8pt
\noindent\emph{Proof of Lemma \ref{cross clasping}}.  $c(\mF',\mW')-c\mfw=c_{21}-c_{21}'+c_{31}-c_{31}'=|G_3\cap B_{12}^R|+|G_2\cap B_{13}^R|=0$.  The first, second and third equalities follow respectively  from Claims 2, 3 and 4.

For concreteness we took $r=1$, $p=2$ and $q=3$, however there are six orderings of 3 distinct integers and since $c\mfw=\sum_{j<i}c_{ij}\mfw$, we need to consider all six.  Two correspond to equating i), ii) of Claim 3, two require equating iii) and iv), one needs equating ii) and iii) and the last  i) and iv).  \qed

\section {Independence of ordering} \label{ordering}

In this section we  show that for multi-eye finger/Whitney systems  $\mD$ is independent of the ordering on $\mW$.   

\begin{proposition}\label{order} \emph{(Independence of order)} If $(\mF,\mW)$ is a finger/Whitney system, then  $\mD(\mF, \mW)$ is independent of the ordering on $\mW$.\end{proposition}

\begin{lemma} It suffices to consider the case that $(\mF,\mW)$  has two orderings  $\mW_1$ and $\mW_2$ on $\mW$ which differ by the transposition on $\mW_{11}=\{w_1, \cdots, w_n\}$ that switches $w_r, w_{r+1}$ for some $r\le n-1$.  \qed\end{lemma}

\noindent\emph{Idea of Proof of Proposition \ref{order}}.   We first reduce to the case that $\mF$ and $\mW$ are similarly matched.  By a sequence of $\mW$, $\mR$-disc slides and $\mR$-uncross/cross claspings; $\mF$, $\mG$-disc slides and $\mG$-uncross/cross claspings; and cross/cross claspings we reduce to the case that $\mF_c$ and $\mW_c$ have common boundary thereby enabling use of the more powerful versions of the results of the previous section.  Then after $\mW$-disc slides and $\mW$-uncross/uncross claspings we reduce to the case that $\mF$ and $\mW$ have common boundary.   A linear algebra calculation completes the proof, first in the case that the boundary germs of $\mF_u$ equal those of $\mW_u$ and then when one is twisted relative to the other.

\begin{lemma} \label{multi full switch}  It suffices to consider the case that $\mF$ and $\mW$ are similarly matched. \end{lemma}

\begin{proof} Given $(\mF,\mW)$ let $\hat\mW$ be as in  Multi-Eye Lemma \ref{unknotted finger}, which has $\mF_u,\hat\mW_u$  similarly matched and for $i=1,2$  an ordering $ \hat\mW_i$ such that with that ordering $\mW_i$ and $\hat \mW_i$ have a common order induced switching with respect to $\mF$.  Since both $C\mfw$ and $\I\mfw$ are independent of their switchings it follows that $\mD(\mF,\mW_i)=\mD(\mF,\hat\mW_i)$.   I.e. with these orderings $\mfw$ and $(\mF, \hat\mW)$ are $\mD$-equivalent.  Finally, apply Corollary \ref{cross matched}  to find $\mF'$ that similarly matches $\hat\mW$  and conclude $(\mF, \hat\mW)$ and $(\mF', \hat\mW)$  and hence $(\mF, \mW)$ and $(\mF', \hat\mW)$ are $\mD$-equivalent.  Thus to show $\mD(\mF,\mW_1) = \mD(\mF,\mW_2)$ it suffices to show $\mD(\mF',\hat\mW_1) = \mD(\mF',\hat\mW_2)$.  \end{proof}

\begin{lemma} \label{common clasp} \emph{(Different orderings same clasping)}  Given $\mfw$ with $\mF_u, \mW_u$ similarly matched, two orderings $\mW_1, \mW_2$ on $\mW$ which differ by a transposition on $\mW_{11}$, $\alpha\subset \mR$ an embedded arc from $w\in \mW$ to $w'\in \mW, w\neq w'$, with $\inte(\alpha)\cap \mW=\emptyset, u\in w\cap\mR\cap \mG, v\in w'\cap\mR\cap\mG$, then with appropriate choice of sign $\sigma$ there exists identical $\sigma(\alpha,u,v)$ $\mW_1$ and $\mW_2$ $\mR$-claspings.  The corresponding result holds for $\mF$ $\mG$-claspings.  \end{lemma}

\begin{remark}  A clasping is a geometric operation which takes an unordered $\mW$ to a $\mW'$.  It is a clasping for an ordered $\mW$ if it satisfies extra conditions.  The lemma asserts that there exists common $\sigma(\alpha,u,v)$ claspings for both $\mW_1$ and $\mW_2$.  The issue here is that a clasping for $\mW_1$ (resp. $\mW_2$) requires clasp tubes disjoint from a switching $\mW_1^*$ (resp. $\mW_2^*$) of $\mW_1$ (resp. $\mW_2$). The proof involves suitably choosing $\mW_1^*, \mW_2^*$, the sign and the clasp tubes.\end{remark} 

\begin{proof} Put $\mfw$ into $\mR$-framed finger form with $\mW_1$ the standard Whitney discs.  Let $\mW^1$ denote the standard full switch for $\mW_1$.  Construct the switching $\mW^2$ of $\mW_2$ so that $\mW^2\setminus \mW^2_{11}=\mW^1\setminus \mW^1_{11}$.  Construct $\mW^2_{11}$ as in Figure \ref{fig:common clasp} which shows the representative case of a 2,3 transposition, i.e. Figure \ref{fig:common clasp} a) (resp. \ref{fig:common clasp} b)) shows the intersections of $\mW^2$ (resp. $\mW^1$) and $\mW_2$ ($\mW_1$) with $G_1$.  Here $\mW^1, \mW^2$ are finger germed.  

A $w_i$, $ w_j$ $\mR$-clasping with $w_i, w_j\in \mW_{11}$ will have clasp tubes emanating from opposite sides of $w_j\cap G_1$ and $w_i\cap G_1, i\neq j$, with the side being a function of the sign of the clasping.  With appropriate sign, clasp tubes can be constructed disjoint from $\mW^1$ and $\mW^2$.  The sign is chosen so the that if the clasping involves $w_2$ (resp. $w_3$) with the $\mW_1$ ordering, then the tube emanates from the $y>0$ (resp. $y<0$) side.  

When one or both of $w, w'\notin \mW_{11}$, then we readily construct clasp tubes for a $w,w'$ clasping since $\mW^2\setminus \mW^2_{11}=\mW^1\setminus \mW^1_{11}$.\end{proof}

 \begin{figure}[!htbp]
    \centering
    \includegraphics[width=1.0\linewidth]{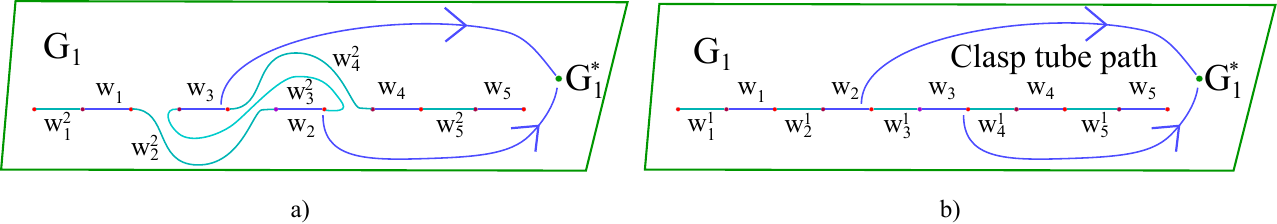}
    \caption{A common clasping for $\mW_2$ and $\mW_1$}
    \label{fig:common clasp}
    \end{figure}

\begin{lemma} \label{common boundary} \emph{(Common boundary)} Let $\mF$ and $\mW$ be both Whitney germed with $\mF$ and $\mW$ similarly matched.  Given two orderings of $\mW$ that differ by a $\mW_{11}$ transposition there exists an $(\mF', \mW')$ such that $\mF',\mW'$ have common boundary, $\mW$ and $\mW'$ are similarly matched  and with the induced orderings $(\mF', \mW')$ and $\mfw$ are $\mD$-equivalent.  \end{lemma}

\begin{proof} Assume that $\mfw$ has $k$-eyes.  

\vskip 6pt
\noindent\emph{Step 1}:  $\mfw$ is $\mD$-equivalent to $(\mF',\mW')$ such that $\mW'_c$ and $\mF'_c$ have common boundary.

\vskip 6pt
\noindent\emph{Proof of Step 1}.  We will find a sequence $\mW^1=\mW, \mW^2, \cdots, \mW^m=\mW'$ such that for each $1\le i\le m-1, \mW^{i+1}$ is obtained from $\mW^i$ by either a $\mW^i$ $\mR$-clasping of  uncross/cross or cross/cross type or a $\mW^i$ $\mR$-disc slide, such that $\mW'_c\cap \mR$ is isotopic to $\mF_c\cap \mR$ rel $\mR\cap \mG$.  Note that these operations do not change $\mW\cap \mG$ and by Lemma \ref{multieye slide invariance} and Proposition \ref{d invariance}, with respect to the orderings $\mW_1$ and $\mW_2$ induced on $\mW'$, $(\mF,\mW')$ is $\mD$-equivalent to $\mfw$.  The same argument gives a modification of $\mF$ to $\mF'$ using $\mF$ $\mG$-disc slides  and $\mF$ $\mG$-claspings of uncross/cross or cross/cross type to obtain the conclusion of the lemma.

Let $\mF_c=(f_1, \cdots, f_n)$ and $\mW_c=(w_1, \cdots, w_n)$ where $w_i$ matches $f_i$.  Suppose that for $j<i, w_j\cap \mR$ is isotopic to $f_j\cap \mR$ rel $\mR\cap \mG$.  We will show how to modify $\mW$ to $\mW^1$ by a sequence of $\mR$-claspings of uncross/cross or cross/cross type and $\mR$-disc slides such that each $w_j, j<i$ remains unchanged  and if $w_i^1\subset \mW^1$ is the disc matching $w_i$, then $w_i^1\cap \mR$ is isotopic to $f_i\cap \mR$ rel $\mR\cap\mG$.  The result then follows by induction on $i$.  Suppose $w_i\in \mW_{qp}$.

For this argument assume that $\mF$ is finger germed and $\mW$ is Whitney germed. 
Our proof is by downward induction on $(a,b)\in \BZ_{\ge 0}\oplus \BZ_{\ge 0} $ lexicographically ordered where $a=|(\mW\setminus w_i)\cap f_i\cap R_q|$ and $b=|w_i\cap f_i\cap R_q|$.  Let $\{v,v'\}=f_i\cap R_q\cap G_p$.  We describe two operations for modifying $\mW$ to $\mW^1$.  

If $\alpha \subset f_i$ is a segment with $v\in \alpha, \inte(\alpha)\cap \mW=\emptyset$ and $(\partial \alpha\setminus v)\in w\neq w_i$, then clasp $w$ to $w_i$ along $\alpha$ to remove this intersection, thereby reducing $a$ by 1, at the cost of possibly increasing $b$.  See Figures \ref{fig:common boundary} a) and b). This is either an uncross/cross or cross/cross clasping, the type depending on whether or not $w\in \mW_u$.  By Lemma \ref{common clasp} the same clasping can be used whether $\mW$ is ordered by $\mW_1$ or $ \mW_2$.   Now $w_i$ may wind about $v$ near a subarc $\beta\subset w_i$.  A variant of this operation is shown in Figures \ref{fig:common boundary} c) and d).  Here  $\alpha$ is as above except that $\inte(\alpha)\cap \mW \subset \beta$.  Again, $a$ is reduced by 1 and $b$ may increase.  

   \begin{figure}[!htbp]
    \centering
    \includegraphics[width=.8\linewidth]{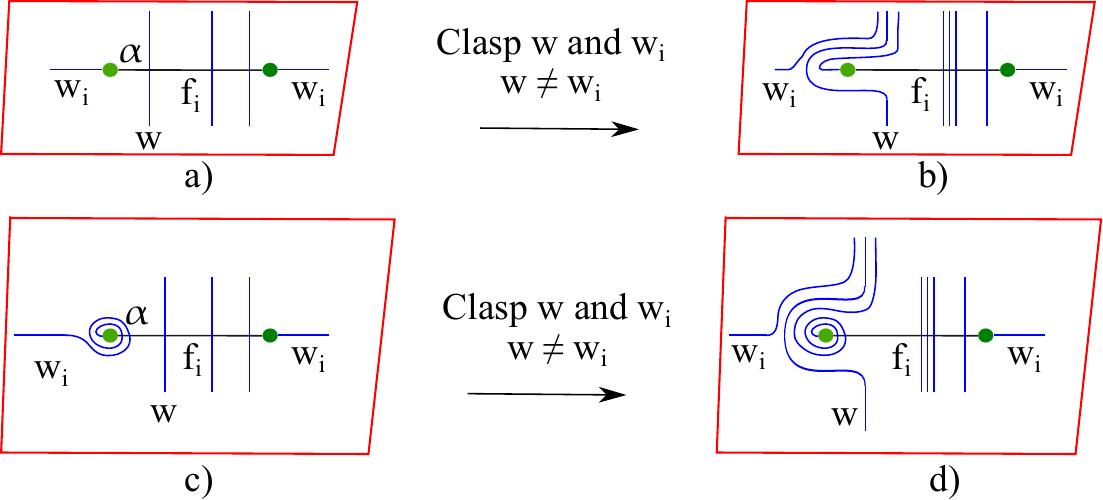}
    \caption{Reducing to a common boundary}
    \label{fig:common boundary}
\end{figure}

Our second operation occurs when there exists a subarc $\alpha\subset f_i$ with $v\in (\alpha\cap \mW)\subset \partial \alpha\subset w_i$.  Let $u=\partial \alpha \setminus v$ and $\gamma \subset w_i$ the segment with endpoints $v\cup u$.  Then $\alpha\cup\gamma$ is an embedded circle that bounds a disc $D \subset R_q$ disjoint from $v'$.  Since the non $\gamma$ components of $D\cap \mW$ are disjoint from $\partial D$, these arcs lie in $\mW\setminus w_i$.  After a sequence of $w_i$ $\mR$-disc slides, $\gamma$ can be isotoped to lie close to $v$.  This operation  neither changes $a$ nor increases $b$, though may decrease $b$ if $\inte(\gamma)\cap f_i\neq\emptyset$.  In any case the isotoped $\gamma$ gets incorporated into the winding of $w_i$ about $v$.    Again we require the variant of this operation when $w_i$ winds about $v$.   Here either $b$ is reduced at least by 2 or $\gamma$ is incorporated into the winding.  Eventually $\alpha=f_i$, in which case $w_i\cap R_q$ is isotopic to $f_i\cap R_q$ rel $R_q\cap\mG$.  \qed

\vskip 8pt
\noindent\emph{Step 2}:  If $\mF_c$ and $\mW_c$ have common boundary, then with the orderings induced by $\mW_1$, $\mW_2$ on $\mW'$, $\mfw$ is $\mD$-equivalent to $(\mF,\mW'')$ such that $\mF$ and $\mW''$ have common boundary.

\vskip 6pt
\noindent\emph{Proof of Step 2}.  Using the method of Step 1 we first find a sequence $\mW^1=\mW, \mW^2, \cdots, \mW^m=\mW'$ such that for each $1\le i\le m-1, \mW^{i+1}$ is obtained from $\mW^i$ by either a $\mW^i$ $\mR$-disc slide or a $\mW^i$ $\mR$-clasping, such that $\mW'\cap \mR$ is isotopic to $\mF\cap \mR$ rel $\mR\cap \mG$.  Since $\mF_c$ and $\mW_c$ have common boundary, we can take the claspings  to be of uncross/uncross type and the disc slides  to be of uncross discs over other discs.  By Lemma \ref{common clasp} we can assume that the claspings are the same for both $\mW_1$ and $\mW_2$.  Again these operations change neither $\mW\cap \mG$ nor $\mW_c$ and by Lemma \ref{multieye slide invariance} i) and Proposition  \ref{d invariance} ii a),  with respect to the orderings $\mW_1$ and $\mW_2$ induced on $\mW'$, $(\mF,\mW')$ is $\mD$-equivalent to $\mfw$.

The same process gives a modification of $\mW'$ to $\mW''$ where $\mW''\cap \mG$ is isotopic to $\mF\cap \mG$ rel $\mR\cap\mG$ and hence $\mW''$ and $\mF $ have common boundary.  Here we use $\mG$-disc slides of uncross discs over other discs and $\mG$-claspings of uncross/uncross type.    Therefore,  we can apply Lemma \ref{multieye slide invariance} ii) and Proposition \ref{d invariance} ii b) to conclude that with respect to the orderings $\mW_1$ and $\mW_2$ induced on $\mW'$, $(\mF,\mW'')$ is $\mD$-equivalent to $(\mF,\mW')$ and hence $\mfw$.\end{proof}

\begin{remark} \label{no free point}   Since $\mF$ and $\mW$ are similarly matched none of the claspings in the previous proof are free point claspings.  Free point clasping arise when we invoke Lemma \ref{clasp commutation 2} and these are of the $\mW$ $\mR$-cross free point type. \end{remark}

 The following result using the method of proof of the previous result will be used in the proof of Proposition \ref{symmetry1}

\begin{lemma}\label{ea common boundary}  Let $\mfw$ be in embedded arc position, $\partial \mF_c\cap \partial\mW_u=\partial \mW_c\cap\partial \mF_u=\emptyset$ and $\mF_c$ and $\mW_c$ be similarly matched, then by a sequence of cross/cross claspings and disc slides of cross/cross or cross/uncross type, where the claspings are only done when $\partial \mF_c\cap \partial\mW_u=\partial \mW_c\cap\partial \mF_u=\emptyset$,  we can obtain $(\mF', \mW')$  in full embedded arc position.  Furthermore, $\I\mfw=\I(\mF', \mW')$, $c\mfw=c(\mF', \mW')$ and $CU\mfw=CU(\mF', \mW')=0$.  Finally, all the disc slides and claspings can be chosen to involve only $\mF$ discs or only $\mW$ discs.\end{lemma}

\begin{proof}  The proof of Step 1 shows how to reduce $\mfw$ to $(\mF', \mW')$ in full embedded arc position.  Since $\mfw\in$ EA  all the claspings are of cross/cross type.  Disc slides are only done during the second operation, while bringing $D$ close to $v$, thus these are of cross/cross or cross/uncross type and all the claspings are done while the finger/Whitney system satisfies the initial hypothesis.  Since the uncross discs are unchanged in this process we have $\I\mfw=\I(\mF', \mW')$ and since we start and end in full embedded arc position we have $CU\mfw=CU(\mF', \mW')=0$.  The conclusion $c\mfw=c(\mF', \mW')$ follows from Lemma \ref{cross clasping}.  Since the process of transforming $\mfw$ to $(\mF', \mW')$ can be done using moves  involving only finger discs or only Whitney discs, the result follows.\end{proof}

\begin{remark} Lemma \ref{cross clasping} has hypotheses not satisfied by Lemma \ref{common boundary}.  Thus in the proof of that lemma we were constrained to using only $\mF$ $\mG$-claspings and $\mG$-disc slides and $\mW$ $\mR$-claspings and $\mR$-disc slides.\end{remark}

\begin{lemma} \label{homology1}   \emph{($H_2$-generators)} Let $\mR$ and $\mG$ be in $\mG$-framed finger form and $E:=\#_k\stwostwo\setminus\inte(N(\mR\cup \mG))$.  Then $H_2(E)$ is  generated by the spheres $R^i_1,\cdots, R^i_{n_i}, S^i_1,\cdots, S^i_{n_i}$ that $\delta_{ij}$ the standard finger  and Whitney $i$-discs, $1\le i\le k$, plus for each cross solid finger a linking sphere and a Clifford torus, where the Clifford torus is associated to $w\cap\mR\cap\mG$ where w is the Whitney cross disc $w$ within the solid finger.  Let $C_1,\cdots, C_m$ denote the Clifford tori and $U'_1, \cdots, U'_m$ denote the linking spheres to the solid cross fingers.    By restricting the $U'_i$'s to a subset denoted $U_1, \cdots, U_r$ we obtain a free generating set.  See Figure \ref{fig:homology1} and compare with the Multi-eye Lemma \ref{e' homology}.\end{lemma}

\begin{proof} Let $W=N(\mR\cup \mG\cup\mW\cup\mF_u)$ where $\mW$ are the standard Whitney discs and $\mF_u$ are the standard uncross finger discs.  Then $F:=\#_k(\stwostwo)\setminus \inte(W)$ is the result of starting with $S^4$ removing $k$ open 4-balls and then removing neighborhoods of properly embedded arcs, one for each cross $(i,j)$-finger, where this arc  goes from the $i$'th 4-ball to the $j$'th one.  Then $\mW\cup \mF_u$ serve as cohomology classes distinguishing the $R_p$'s, $S_q$'s and $C_r$'s.  On the other hand if $J$ is a 2-cycle in $E$, then by adding suitable copies of $R_p$'s, $S_q$'s and $C_r$'s it becomes homologous to a class in $H_2(F)$ which is freely generated by a subset of $U'_i$'s.\end{proof}

 \begin{figure}[!htbp]
    \centering
    \includegraphics[width=0.7\linewidth]{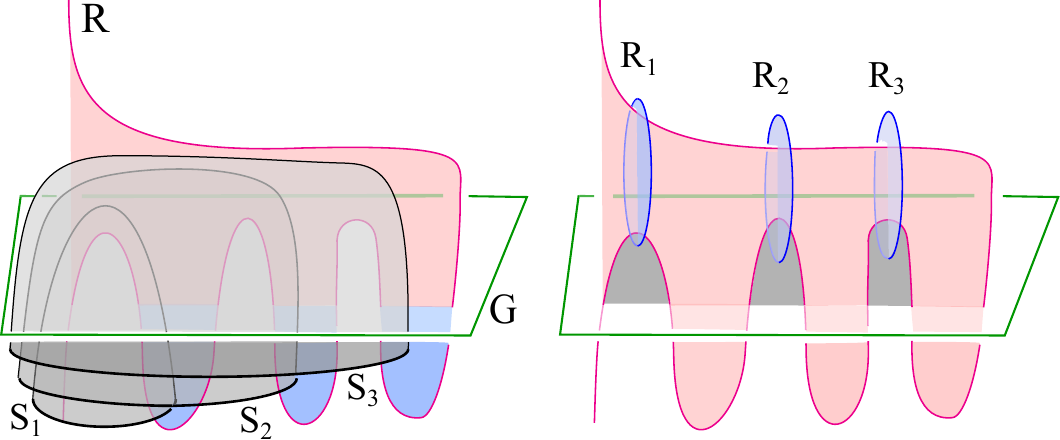}
    \caption{Free Generators for $H_2(\stwostwo\setminus(N(G\cup R)))$}
    \label{fig:homology1}
    \end{figure}

\begin{lemma}\label{homology2}  \emph{($H_2$-relative generators)} If $D$ is a properly embedded disc in $E$, then $H_2(E, \partial D)$ is freely generated by $D$, the $R_p$'s, $S_q$'s, $C_r$'s and $U_s$'s.\qed\end{lemma}

\begin{lemma}\label{homology3} \emph{(Whitney $H_2$-constraints)} i) If $D^i_u$ is a Whitney disc whose $\partial$-germ coincides with that of the standard $w^i_u$, then for appropriate values of $b^u_{j,q}$ we have $[D^i_u]=[w^i_u]+\sum_{j,q\neq i,u} b^u_{j,q} S^j_q\ +$ terms involving $R_p$'s, $C_r$'s, $U_s$'s.

ii) If $D^i_u$ and $D^i_v, u\neq v$ are disjoint Whitney discs whose $\partial$-germs coincide respectively with the standard Whitney discs $w^i_u$ and $w^i_v$, then $b^u_{i,v}=-b^v_{i,u}$.\end{lemma}

\begin{proof} Conclusion i) follows from Lemma \ref{homology2} after noting that a non zero $S^i_u$ factor would not give $D^i_u$ the Whitney framing.  If conclusion ii) failed, then $\langle D^i_u,D^i_v\rangle\neq 0$. \end{proof}

The following lemma will complete the proof of Proposition \ref{order}.  

\begin{lemma} \emph{(Linear Algebra)} \label{linear algebra}  If $\mF$ and $ \mW$ are complete sets of Whitney germed  Whitney discs such that   $\mF$ and $\mW$ have common boundary, then  $\mD(\mF, \mW)$ is independent of the ordering on $\mW$.
\end{lemma}

\noindent{Proof}: It suffices to consider the case that $\mW_1, \mW_2$ are orders on $\mW$ that differ by a transposition on $\mW_{1,1}=\{w_1, \cdots, w_n\}$ where $\mW_1$ has the standard ordering and $\mW_2$ transposes $w_r, w_{r+1}$.  Let $\{f_1,\cdots, f_n\}$ denote the similarly matched elements of $\mF_{11}$.  For concreteness, to minimize notation and since the proof of the general case is essentially the same we will only consider the case $n=5$ and $\mW_2$ has $w_2, w_3$ transposed, hence as ordered sets $\mW_{1_{1,1}}=\{w_1, w_2, w_3, w_4, w_5\}$ and $\mW_{2_{1,1}}=\{w_1, w_3, w_2, w_4, w_5\}$.  Put $\mR\cup \mG$ in $\mG$-framed finger form so that $\mW_1$ comprises the standard Whitney discs.   

We now describe switchings $\mW^*_1, \mW^*_2$ for $\mW_1, \mW_2$.  Let $\mW^*_1$ denote the standard switch system arising from the framed finger form with $\hat\mW^*_1 =\{w_1^1, \cdots, w_5^1\}\subset \mW^*_1$ denoting the standard 5-switch for $\mW_{1_{1,1}}$.  We now describe $\mW_2^*$.  Except for the switch discs associated to $\mW_{2_{1,1}}$, use the same discs as $\mW_1^*$, so the cross discs of $\mW_2^*$ are exactly those of $\mW_1^*$ and the switch discs associated to $\mW_{2_{ii}}$ are exactly those of $\mW_{1_{ii}}$ for $2\le i\le k$.   Let $\hat \mW_2^*=\{w_1^2, \cdots, w_5^2\}$ denote the switch discs associated to $\mW_{2_{1,1}}$ as shown in Figure \ref{fig:5switch}.  Figures \ref{fig:5switch} a), b) respectively show $(\mW_2\cup\hat\mW_2^*)\cap G_1$ and $(\mW_2\cup\hat\mW_2^*)\cap R_1$ and Figure \ref{fig:5switch} c) shows the projection of $w_1^2, w_2^2, w_3^2, w_5^2$ to a 3D slice  and Figure \ref{fig:5switch} d) shows the projection of $w_4^2$ to that  slice.  We use the notation that for each $i$, $w_i$ and $w_i^2$ share a common negative $R_1\cap G_1$ point and the labels in Figure \ref{fig:5switch} a), b) are near the common point.  The labels in Figures \ref{fig:5switch} c), d) are near the switch  discs themselves.  Again, we use the convention that if a point on $R_1$ has y-coordinate $< 0$ (resp. $> 0$), then it has $t$-coordinate $< 0$ (resp. $> 0$).  In Figure \ref{fig:5switch} c), d), the points of $\mW_2^*\cap R_1$ with $y$-coordinate $>0$ are shown dashed.  

Since $(\partial \mF_u\cap\partial \mW^*_c)\cup(\partial \mW^*_u\cap\partial \mF_c)=\emptyset$ it follows that $CU(\mF, \mW_i)=0$ for $i=1,2$.  Since each $c_{pq}(\mF,\mW)$ depends only on $\mF_c$ and $\mW_c$, it is independent of order and switching.  It follows that $C(\mF, \mW_1)=C(\mF, \mW_2)$.  Since $\mW^*_{1,ii}=\mW^*_{2,ii}$ for $2\le i\le k$ it follows that  $I_i(\mF, \mW_1)=I_i(\mF, \mW_2)$ for $2\le i\le k$.  \vskip 6pt

   \begin{figure}[!htbp]
    \centering
    \includegraphics[width=1.0\linewidth]{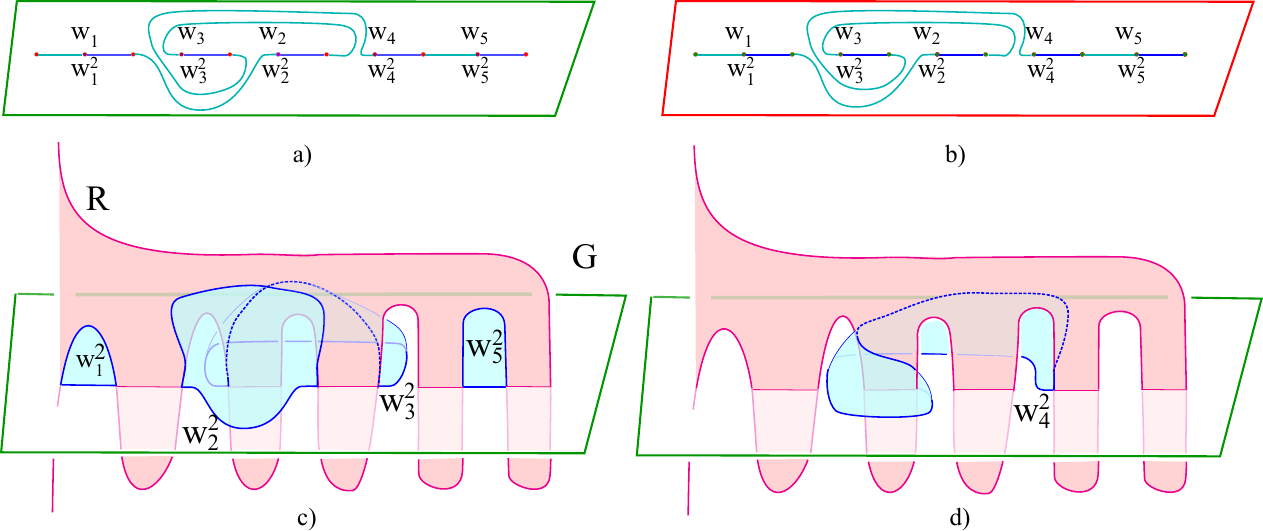}
    \caption{Constructing the 5-Switch $\mW_2^*$ for $\mW_2$}
    \label{fig:5switch}
\end{figure}

To complete the proof we show $I_1(\mF, \mW_1)=I_1(\mF, \mW_2)$. 

\vskip 8pt
\noindent \emph{Case 1}:  For each $i$, the boundary germ of $f_i$ coincides with that of $w_i$.  
\vskip 8pt

By Lemma \ref{homology3}, for $1\le i,j\le n$ and appropriate $a_{ij}, b_{ij} \in \BZ$, we have $[f_i]=[w_i]+\sum_{j=1}^n a_{ij} [R^1_j]+\sum_{j\neq i} b_{ij} [S^1_j]$ where $b_{ij}=-b_{ji}$ + terms involving $[R^j_p]$'s, $[S^j_q]$'s for $j\ge 2$ + terms involving $C_r$'s and $U_s$'s. Since for $i=1,2$ and $1\le j\le 5$, $w^i_j\cap (\cup_{j,p} R^j_p\cup_{j,p}S^j_p\cup_r C_r\cup_s U_s)=\emptyset$ and the calculation of $I_1$ is a homological intersection count involving discs in a 4-manifold with boundary, where the boundary of the discs are disjoint, these terms play no role in the calculation of either $I_1(\mF, \mW_1)$ or $I_1(\mF, \mW_2)$, thus we will assume that each $[f_i]$ has no such terms. 

We calculate the mod 2 intersection of each $w_i^j$ with each $R^1_p$ and $S^1_q$.  

\begin{lemma} \label{homology4}  1) $\langle w^1_i, R^1_j\rangle=\delta_{ij}$ and $\langle w^1_i, S^1_j\rangle=0$ all $j$.
\vskip 10pt
\noindent 2) If $i=1$ or $5$, then  $\langle w^2_i, R^1_j\rangle=\delta_{ij}$ and $\langle w^2_i, S^1_j\rangle=0$ all $ j$.
\vskip 8pt
\noindent $\langle w^2_2, R^1_j\rangle=1$ if $j=2,3$ and $0$ otherwise and $\langle w^2_2, S^1_j\rangle=1$ if $j=2$ and $0$ otherwise
\vskip 8pt
\noindent $\langle w^2_3, R^1_j\rangle=1$ if $j=3$ and $0$ otherwise and $\langle w^2_3, S^1_j\rangle=1$ if $j=2,3$ and $0$ otherwise
\vskip 8pt

\noindent $\langle w^2_4, R^1_j\rangle=1$ if $j=3,4$ and $0$ otherwise and $\langle w^2_4, S^1_j\rangle=1$ if $j=3$ and $0$ otherwise. \qed \end{lemma}

To compute $I_1(\mF, \mW_i^*)$ we need to determine their IA orderings on $\mF_{11}\cup\hat\mW^*_i$ for $i=1,2$.  

\begin{lemma}  \label{ia orderings} Figures \ref{fig:iaorder} a), b)  record the IA orderings on $\mF_{11}\cup\hat\mW^*_1$ and $\mF_{11}\cup\hat\mW^*_2$.\qed\end{lemma}

\begin{notation} Recall that $f_i$ denotes the finger disc that matches $w_i\in \mW_{1_{11}}$, i.e., the $i$'th standard Whitney disc.  Recall that $w_j^1$ (resp. $w_j^2$) denotes the $j$'th switch disc from the $\mW_1$ (resp. $\mW_2$) ordering.     We let $f_{i,p}$ (resp $w_{i,q}^*$) denote the $p$'th finger (resp. $q$'th Whitney) disc in the IA order on $\mF_{11}\cup\hat\mW^*_i$ for $i=1,2$.  The figures show $\mF_{11}\cup\hat\mW^*_i\cap G_1$ with both the old and new names of the discs.  \end{notation}

  \begin{figure}[!htbp]
    \centering
    \includegraphics[width=.8\linewidth]{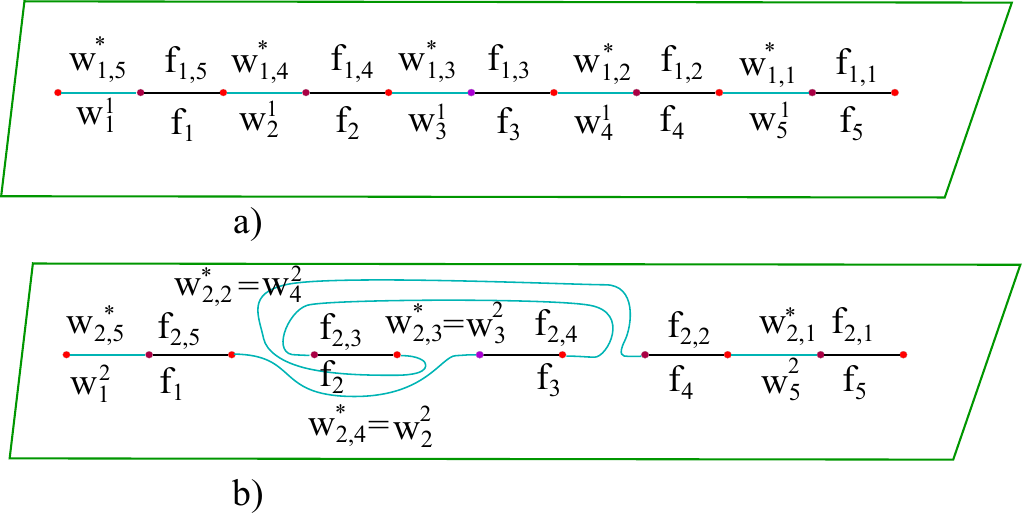}
    \caption{The IA Ordering on $\qquad$ a) $(\mF,\mW_1^*)\qquad$ and$\qquad$ b) $(\mF, \mW_2^*)$}
    \label{fig:iaorder}
\end{figure}

\begin{table}[h]
\centering
\begin{tabular}{ c|c|c|c|c|c| }
 
 $I(\mF,\mW_1^*)$ & $w^*_{1,1}=w^1_5$ & $w^*_{1,2}=w^1_4$ & $w^*_{1,3}=w^1_3$ & $w^*_{1,4}=w^1_2$ & $w^*_{1,5}=w^1_1$ \\
 \hline
  $ f_{1,1}=f_5$& $a_{5,5}$ & $a_{5,4}$ & $a_{5,3}$ & $a_{5,2}$ & $a_{5,1}$ \\

 \hline
 $f_{1,2}=f_4$ &  & $a_{4,4}$ & $a_{4,3}$ & $a_{4,2}$ & $a_{4,1}$ \\
 \hline
 $ f_{1,3}=f_3$ &  &  & $a_{3,3}$ & $a_{3,2}$ & $a_{3,1}$ \\
 \hline
 $ f_{1,4}=f_2$ &  &  &  & $a_{2,2}$ & $a_{2,1}$ \\
 \hline
 $f_{1,5}=f_1$ &  &  &  &  & $a_{1,1}$ \\
 \hline
 
\end{tabular}
\vspace{3mm}
\caption{Calculating $I(\mF, \mW_1)$}
\label{table1}
\end{table}

\begin{table}[h]
\centering
\begin{tabular}{ c|c|c|c|c|c| }
 
 $I(\mF,\mW_2^*)$ & $w^*_{2,1}=w^2_5$ & $w^*_{2,2}=w^2_4$ & $w^*_{2,3}=w^2_3$ & $w^*_{2,4}=w^2_2$ & $w^*_{2,5}=w^2_1$ \\
 \hline
  $f_{2,1}=f_5$& $a_{5,5}$ & $a_{5,4}+a_{5,3}+b_{5,3}$ & $a_{5,3}+b_{5,3}+b_{5,2}$ & $a_{5,3}+a_{5,2}+b_{5,2}$ & $a_{5,1}$ \\
 \hline
 $f_{2,2}=f_4$ &  & $a_{4,4}+a_{4,3}+b_{4,3}$ & $a_{4,3}+b_{4,3}+b_{4,2}$ & $a_{4,3}+a_{4,2}+b_{4,2}$ & $a_{4,1}$ \\
 \hline
 $f_{2,3}=f_2$ &  &  & $a_{2,3}+b_{2,3}+b_{2,2}$ & $a_{2,3}+a_{2,2}+b_{2,2}$ & $a_{2,1}$ \\
 \hline
 $f_{2,4}=f_3$ &  &  &  & $a_{3,3}+a_{3,2}+b_{3,2}$ & $a_{3,1}$ \\
 \hline
 $f_{2,5}=f_1$ &  &  &  &  & $a_{1,1}$ \\
 \hline
\end{tabular}
\vspace{3mm}
\caption{Calculating $I(\mF, \mW_2)$}
\label{table2}
\end{table}

\begin{lemma} \label{reorder calc}  For $i=1,2$ the matrices of Tables \ref{table1} and \ref{table2} respectively record the values of $\langle f_{i,p}, w_{i,q}^*\rangle$ for $p\le q$ viewed \emph{mod-2},  from which we conclude  $I_1(\mF, \mW_1)=I_1(\mF, \mW_2)$.\end{lemma}

\begin{proof}  By Definition \ref{invariant}, $I(\mF,\mW_1^*)$ (resp. $I(\mF,\mW_2^*)$) is the sum of the upper diagonal entries of Table \ref{table1} (resp. Table \ref{table2}), where these values are obtained by applying  Lemma \ref{homology4}.   Since $b_{ij}=b_{ji}$ mod-2 it is immediate that $I(\mF,\mW_1^*) = I(\mF, \mW_2^*)$.   Since $I_1(\mF, \mW_1)=I(\mF,\mW_1^*)$ and $I_1(\mF, \mW_2)=I(\mF,\mW_2^*)$ the result follows. \end{proof}

\noindent\emph{Case 2}:  The boundary germs do not all coincide.  
\vskip 8pt
\noindent\emph{Proof of Case 2}:  With respect to the conventions of Definition 5.1 \cite{Ga3}, we have  for each $i$, the boundary germ of $f_i$ twists $(p_i,q_i)\in \BZ\oplus\BZ$ relative to that of $w_i$.  By Lemma 5.3 \cite{Ga3} each $p_i+q_i$ is even.  
\vskip 8pt
\noindent\emph{Claim 1}:  Given $((p_1,q_1), \cdots, (p_n,q_n))$ such that for each $i,\ p_i+q_i=0$ mod 2,  there exists an $\mF_3=\{f^3_1, \cdots, f^3_n\}$ with this twisting relative to $\mW$ such that for each $i$, $\partial f^3_i=\partial f_i$.
\vskip 8pt
\noindent\emph{Proof of Claim 1}: The existence of such a $\mF_3$ follows as in the proof of Lemma 5.4 \cite{Ga3} and Figure 5 of that paper.  Figure \ref{fig:whitney1} shows the first step  to construct $\mF_3$ in our setting, rather it shows the construction of $f_2^3$, with the other $f_i^3$'s constructed similarly.  We start out with the standard Whitney discs and remove $|p_2|$ small open discs near $R$ and $
|q_2|$ small open discs near $G$.  Then for each of the $|p_2|
$ (resp. $|q_2|$) discs a small strip of the complement is spun once about $R$ (resp. $G$) the sign being the sign of $p_2$ or $q_2$.   A detail of this near $G$ is shown in the figure and slightly differently in Figure 5 \cite{Ga3}.  Thus, the boundary germ of the partially constructed $f_2^3$ spins $p_2$ times about $R$ and $q_2$ times about $G$.  In the figure $p_2=3$ and $|q_2|= 1$ the sign depending on the orientation of  $D$ in the detail.  Near $G$ we see $|q_2|$ tubes starting out near $f_2^3\cap G$.  The second step is connect these tubes to $|q_2|$ copies of $G'$ which is a dual sphere to $G$ to obtain a disc with $|p_2|$ open discs removed.  The tube paths are shown in Figure \ref{fig:whitney3} a).  To enable the construction of all the $f_i^3$'s, in the 3D slice, $R'$ rotates $p_i$ times about $R$ near $f_i^3\cap R$ and $p=-\sum_{i=1}^5 p_i$ times near $R$ at the bottom of the picture. 
   \begin{figure}[!htbp]
    \centering
    \includegraphics[width=.85\linewidth]{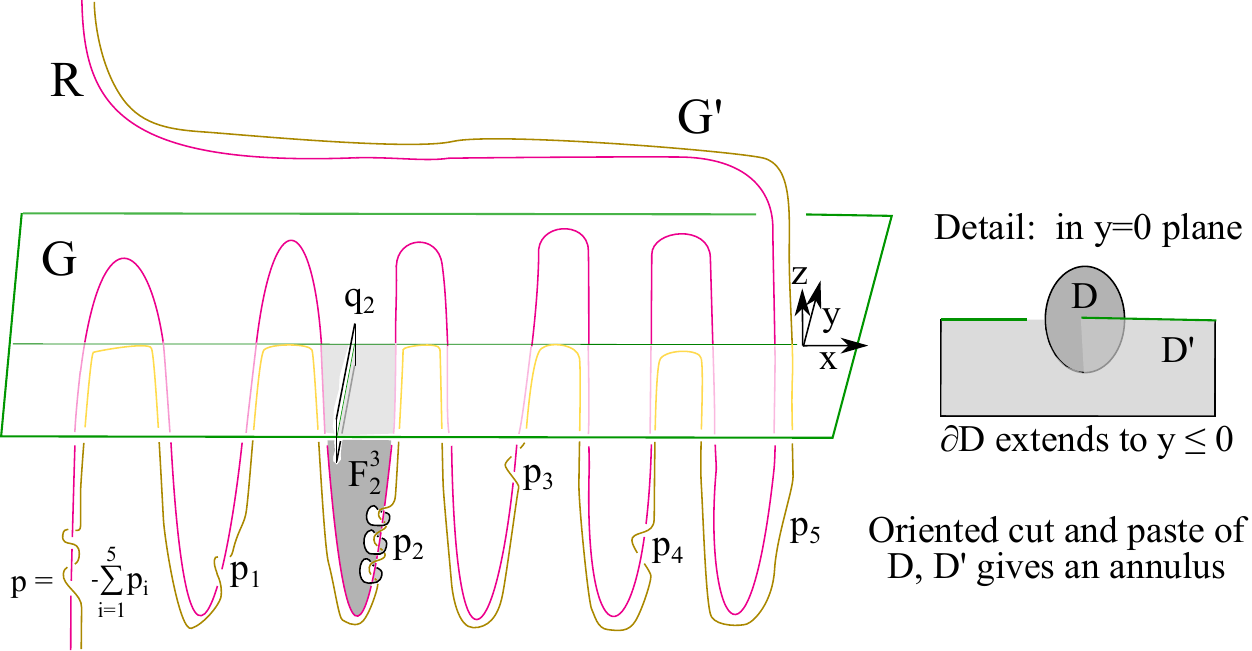}
    \caption{Constructing $f_2^3$ near $S^2\times S^1\times 0$}
    \label{fig:whitney1}
\end{figure}

Third extend these $|p_2|$ circles vertically in $S^2\times S^1\times t$ for $t\in [-\epsilon,0]$ where they are tubed to $|p_2|$ copies of the dual sphere $R' $ to $R$ as in Figure \ref{fig:whitney2}.  These tubes follow paths in $R$ as shown in Figure \ref{fig:whitney3} b).  
  \begin{figure}[!htbp]
    \centering
    \includegraphics[width=.85\linewidth]{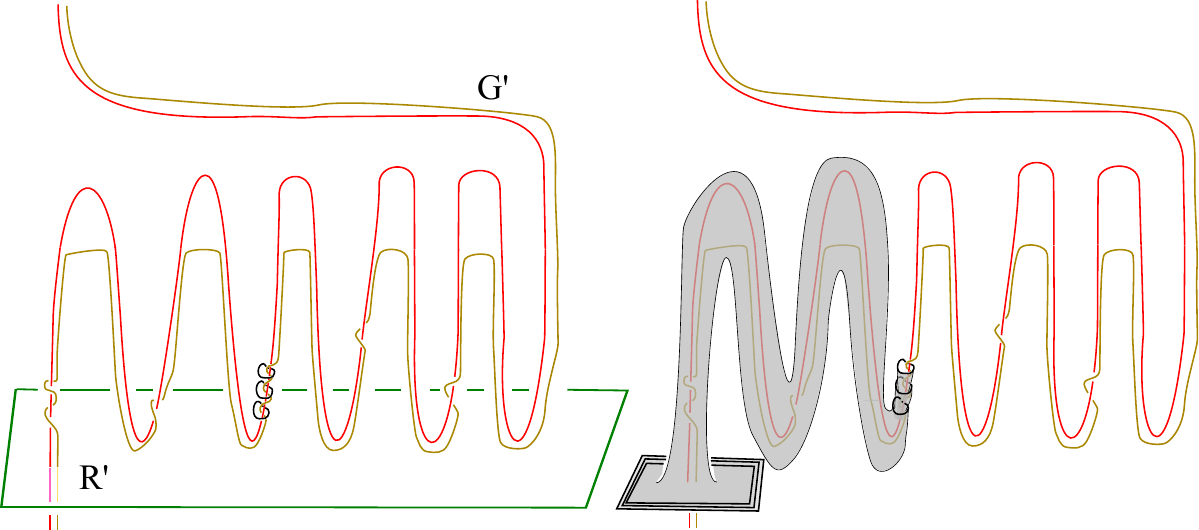}
    \caption{Constructing $f_2^3$ near $S^2\times S^1\times -\epsilon$}
    \label{fig:whitney2}
    \end{figure} 

Fourth, moving further into $S^2\times S^1\times t$, $t<-\epsilon$ we see that the parallel copies of $G'\cap S^2\times S^1\times t$ are parallel $R\cap S^2\times S^1\times t$ and just like $R$ they are capped off with discs within $S^2\times S^1\times [=\infty,t]$.  Similarly moving into $t>0$, these parallel copies are capped off with discs.  

The resulting disc may be misframed as a Whitney disc, so fifth and finally, do cut and paste with a suitable number of $S_2$'s to obtain the desired $f^3_2$, where $S_2$ is as in Figure \ref{fig:homology1}.   \qed
\vskip 8pt

\begin{figure}[!htbp]
    \centering
    \includegraphics[width=.85\linewidth]{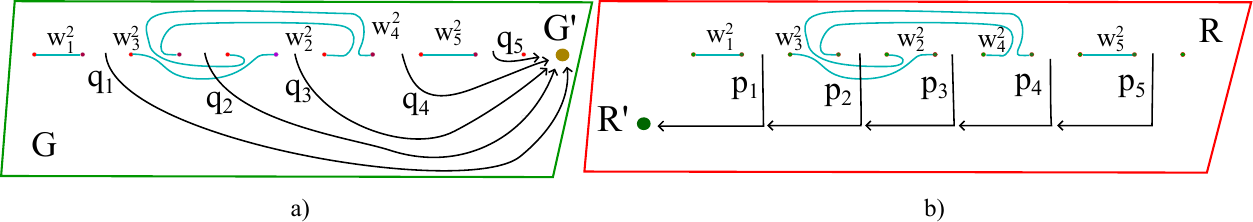}
    \caption{The tube paths, with multiplicity, in $R$ and $G$}
    \label{fig:whitney3}
    \end{figure}

\noindent\emph{Claim 2}:  The $\mF_3$ constructed in Claim 1 satisfies $I(\mF_3,\mW_1)=I(\mF_3, \mW_2)=0$.  
\vskip 8pt
\noindent\emph{Proof of Claim 2}:  If $\mW^*_1$ is the standard switching, then $(\mF_1, \mW^*_1)\in$ EA.  Since $\mF_3\cap\mW_1^*\subset G\cap R$ it follows that $I(\mF_3, \mW_1)=I(\mF_3, \mW_1^*)=0$.  

Similarly  $\mW^*_2=\{w^2_1, \cdots, w^2_n\} $ constructed above satisfies $(\mF_3, \mW_2^*)\in$ EA.  As before, Whitney and finger discs are denoted both with their old notations and their notations with the EA order.  In the EA order the finger discs are denoted $f_{3,1}, f_{3,2}, \cdots, f_{3,5}$.  Now each $f^3_i$ comes with its twisting $(p_i, q_i) $ and an $a_i$ where $a_i S_i$ corrects the framing.  There are three types of intersections that contribute to $I(\mF_3, \mW_2^*)$.  First there are the paths of tubes that intersect $\partial \mW_2^*$ as shown in  Figure \ref{fig:whitney3}.  With our upper diagonal count only  $w_{2,4}^*\cap f_{3,3}$ contributes to $I(\mF_3, \mW_2)$ and that  contribution is $p_2+q_2$.  Since $p_2+q_2=0$ mod-2  this contributes 0.  The second correspond to intersections of the correcting $a_i S_i$'s with $\mW_2^*$. By Definition \ref{invariant}, given $\mW_2^*$, only  $f_{3,3}$ contributes to the count and by Lemma \ref{homology4} the contributions cancel.  Third and finally, there are contributions from $\mW_2^*$ intersecting the dual sphere $G'$ which may cause intersections because the $f^3_i$'s may contain parallel copies.  For example $f_{3,3}$ has $q_2$ parallel copies of $G'$ and $|w_{2,3}^*\cap G'|=p_2+p_3$ mod-2, hence contributes $(p_2+p_3) q_2$ to the count.  Table \ref{table3} records the mod-2 count, denoted $I_{\textrm{third}}(\mF,\mW_2^*)$, of these intersections.  To conclude $I(\mF_3, \mW_2)=0$ it suffices to show that the sum of the upper diagonal entries   equals 0 mod-2.  Since this sum equals $p_2 q_3 + p_3 q_2$  and $p_2+q_2=p_3+q_3=0$ mod-2 the result follows.  \qed

\begin{table}[h]
\centering
\begin{tabular}{ c|c|c|c|c|c| }
 
 $I_{\textrm{third}}(\mF,\mW_2^*)$ & $w^*_{2,1}=w^2_5$ & $w^*_{2,2}=w^2_4$ & $w^*_{2,3}=w^2_3$ & $w^*_{2,4}=w^2_2$ & $w^*_{2,5}=w^2_1$ \\
 \hline
  $f_{3,1}=f^3_5$& $0$ & $(2p_2+p_3)q_5$ & $(p_2+p_3)q_5$ & $p_2 q_5$ & $0$ \\
 \hline
 $f_{3,2}=f^3_4$ &  & $(2p_2+p_3)q_4$ & $(p_2+p_3)q_4$ & $p_2q_4$ & $0$ \\
 \hline
 $f_{3,3}=f^3_2$ &  &  & $(p_2+p_3)q_2$ & $p_2q_2$ & $0$ \\
 \hline
 $f_{3,4}=f^3_3$ &  &  &  & $p_2q_3$ & $0$ \\
 \hline
 $f_{3,5}=f^3_1$ &  &  &  &  & $0$ \\
 \hline
\end{tabular}
\vspace{3mm}
\caption{Calculating   $I(\mF_3, \mW_2)$}
\label{table3}
\end{table}

\noindent\emph{Claim 3}: $I(\mF,\mW_1)=I(\mF, \mW_2)$, thereby completing the proof of Case 2.
\vskip 8pt
\noindent\emph{Proof of Claim 3}: Since the $R_j$'s are disjoint from the  sphere $R'$ and intersect $G'$ algebraically 0 times, it follows that for all $i,j,\ \langle f_i^3, R_j\rangle=0$.  By construction $\langle f_i^3, S_j\rangle=\delta_{ij}$.  Hence  the proof of  Lemma \ref{homology3} shows  that for each $i,\ [f_i]=[f^3_i]+\sum_{j=1}^n a_{ij} R_j+\sum_{j\neq i} b_{ij} S_j$ where for all $i,j,\ b_{ij}=-b_{ji}$.  The calculation of Lemma \ref{reorder calc} shows that $I(\mF,\mW_1)=I(\mF,\mW_2)$.  \qed

\section{$\mD$ is symmetric in $\mF$ and $\mW$}

Our $\mD\mfw$ is defined asymmetrically in two ways.  First $\I\mfw$ depends on an ordering of $\mW$, rather than one on $\mF$.  Second, $C\mfw$ is defined asymmetrically as detailed in Remark \ref{asymmetry}, however as noted there that asymmetry is equivalent to reversing the roles of $\mF$ and $\mW$.

\begin{proposition}\label{symmetry1} \emph{($\mD$ Symmetry)} If $(\mF,\mW)$ is a finger/Whitney system, then $\mD(\mF, \mW)=\mD(\mW, \mF)$. \end{proposition}

\begin{remark} The reader should note that the proof uses Lemma~\ref{L: inv under uncross bd} that asserts: If $(\mF_1, \mW_1)$ is obtained from $(\mF,\mW)$ by adding a new uncrossed finger with identical finger and Whitney discs, then $\mD(\mF, \mW)=\mD(\mF_1, \mW_1)$.  For organizational reasons its proof is given in Section \ref{homotopy invariance}. \end{remark}

\noindent\emph{Proof of Proposition \ref{symmetry1}}:    \emph{Step 1}:  Suppose that $\mfw$ has k eyes.  We can assume that for $1\le i\le k$,  each $\mF_{ii}\cup \mW_{ii}$ has an even number of immersed cycles.
\vskip8pt
\noindent\emph{Proof of Step 1}:  Put $\mfw$ in framed finger form.   Whenever $\mF_{ii}\cup\mW_{ii}$ has  an  odd number of cycles,   replace $R_i$ by $R'_i$  by introducing a new $i$-finger and replace $\mF_{ii}, \mW_{ii}$ with $\mF'_{ii}, \mW'_{ii}$ by introducing common canceling finger and Whitney discs.  If  $(\mF', \mW')$ denotes our new finger - Whitney system, then by Lemma \ref{L: inv under uncross bd} we have $I(\mF', \mW')=I(\mF, \mW)$ and $I(\mW', \mF')=I(\mW, \mF)$.\qed
\vskip 8pt
\noindent\emph{Step 2}:  $\mW$ has a switching $\mW^*$ and $\mF$ has a switching $\mF^*$ such that $(\mF, \mW^*), (\mF^*, \mW)$ and $(\mF^*, \mW^*)\in$ IA.  Furthermore, $\mW^*$ and $\mF^*$ can be chosen so that there exists $\mK$ which similarly matches $\mF^*$ so that with respective orderings $\mK_1, \mK_2$; $\mW$ is a switching of $\mK_1$ and $\mW^*$ is a switching of $\mK_2$.    Similarly, there exists $\mL$ which similarly matches $\mW^*$ with orderings $\mL_1, \mL_2$ which respectively are switchings of $\mF$ and $\mF^*$.  
\vskip 6pt
\noindent\emph{Proof of Step 2}:  Let $\gamma^i_1,\cdots, \gamma^i_{n_i}$ denote the immersed cycles of $(\mF_{ii}\cup\mW_{ii})\cap G_i$. Since each $\mF_{ii}, \mW_{ii}$ is handled similarly and the  modifications below can be done independently by working in terms of appropriate framed finger forms, we will suppress the $i$ superscripts and subscripts.   If the  $u\in \mW\cup\mF$ and $u\cap G\subset\gamma_j$, then we abuse notation by saying $u\in\gamma_j$.  Rename $\mF_u$ and define an order on it such that  $f_1<f_2<\cdots< f_n$ where $f_i\in \gamma_i$ and these are the first n elements in the ordering.  Let $\mF^*$ denote an order induced switching for $(\mW, \mF)$.  Next, rename and order $\mW_u$ so that the first $n$ elements are $w_n< w_{n-1} <\cdots<w_1$ where $w_j$ is the unique element $\in\gamma_j$ sharing with $f_j$  a +1 element of $R\cap G$.   Let $\mW^*$ be the order induced switching of $\mW$ with respect to $\mfw$, These switchings satisfy the first conclusion of Step 2.

For the second conclusion, see Figure \ref{fig:fw symmetry} which shows the basic construction.  Figure \ref{fig:fw symmetry} a) shows a 1D representation of $\mW$ and $\hat\mF$, where $\hat\mF$ is chosen so that its uncross discs similarly match those of $\mF$, and the switching $\hat\mF^*$ when finger germed is the standard set of finger discs.  Figure \ref{fig:fw symmetry} b) shows $\mW^*$ and $\hat\mF^*$.  By our rules for constructing order induced switching  $\hat\mF^*_u$ similarly matches $\mF^*_u$.  Define $\mK=\hat\mF^*$.  Since $\mK_u$ similarly matches $\mF^*_u$ any order induced switching of $(\mF^*, K)$ is a full switching.  By construction when $\mK$ is finger germed and $\mW^*$ and $\mW$ are Whitney germed, $\mK\cap\mW^*$ and $\mK\cap \mW\subset \mR\cap\mG$, thus $\mW$ and $\mW^*$ are switchings of $\mK$.  With suitable orderings $\mK_1, \mK_2$; $\mK$ full switches to $\mW$ and $\mW^*$.  \qed

\begin{figure}[!htbp]
    \centering
    \includegraphics[width=.85\linewidth]{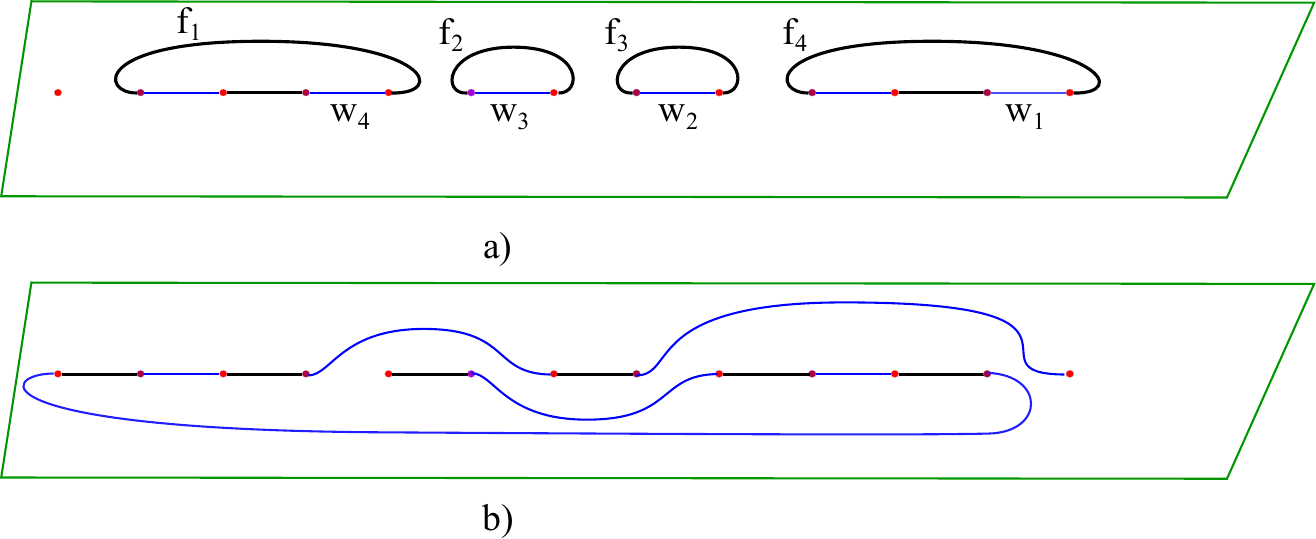}
    \caption{The construction of $\mW^*$ and $\mK$}
    \label{fig:fw symmetry}
    \end{figure}

\vskip 8pt
\noindent\emph{Step 3}:  It suffices to show $C\mfw=C(\mW,\mF)$ when $\mfw\in$ IA.

\vskip 6pt
\noindent\emph{Proof of Step 3}. 
We have $\mD(\mW, \mF)=\mD(\mW, \mF^*)=\mD(\mF^*, \mW)= \mD(\mF^*, \mK_1)=\mD(\mF^*, \mK_2) = \mD(\mF^*, \mW^*)$.  The first, third and fifth equalities follow from Proposition \ref{cross switching} and the fourth equality follows from Proposition \ref{order}.   To obtain the second equality observe that $I(\mW, \mF^*)=I(\mF^*, \mW)$ by multi-eye Lemma \ref{reverse}, since $(\mF, \mW^*)\in$ IA and by the hypothesis of Step 3 we also have $C(\mW, \mF^*)=C(\mF^*, \mW)$.   In a similar manner we have $\mD(\mF, \mW)=\mD(\mW^*, \mF^*)$.  Finally we have $\mD(\mW^*, \mF^*)=\mD(\mF^*, \mW^*)$ since $(\mF^*, \mW^*)\in$ IA.  \qed

\vskip 8pt
\noindent\emph{Step 4}:  It suffices to show $c\mfw=c(\mW,\mF)$ when $\mfw$ is in embedded arc position and $\partial \mF_c\cap \partial\mW_u=\partial \mW_c\cap\partial \mF_u=\emptyset$.

\vskip 6pt
\noindent\emph{Proof of Step 4}. By Step 3 we can assume $\mfw\in$ IA. Since $\mfw$ can be transformed to $(\mF', \mW')$ in embedded arc position by a disc slide sequence and by Lemma \ref{boundary clearing} we can further assume  $\partial \mF_c\cap \partial\mW_u=\partial \mW_c\cap\partial \mF_u=\emptyset$, it follows by Lemma \ref{cu vs cij} that it  suffices to prove $C(\mF',\mW')=C(\mW',\mF')$.  The proof follows after noting $CU(\mF', \mW')=CU(\mW', \mF')=0$.\qed

\vskip 8pt
\noindent\emph{Step 5}: Proof of Proposition \ref{symmetry1}.

\vskip 6pt
\noindent\emph{Proof of Step 5}.  By a sequence of cross special sum square moves maintaining  the embedded arc condition with $\partial \mF_c\cap \partial\mW_u=\partial \mW_c\cap\partial \mF_u=\emptyset$, we can assume that $\mF_c$ and $\mW_c$ are similarly matched.  By Lemma \ref{ea common boundary} we can reduce to the case that $(\mF,\mW)$ and $(\mW, \mF)$ are in full embedded arc position without changing either $\mD\mfw$ or $\mD(\mW, \mF)$.  Indeed, we can use the same set of claspings and discs slides whether we start with $(\mF,\mW)$ or $(\mW, \mF)$.   Disc slides preserve $\mD$ when in IA and since Lemma \ref{cross clasping} is agnostic to whether or not the clasping is of Whitney or finger discs, $\mD$ is also preserved.  
 
 To complete the proof apply Proposition \ref{full symmetry}. \qed

\begin{corollary} \label{disc slide invariance} If $(\mF', \mW')$ is obtained from $\mfw$ by disc sliding, then $\mD(\mF',\mW')=\mD\mfw$.\end{corollary}

\begin{proof}  Suppose that $(\mF,\mW')$ is obtained from $\mfw$ by a single $\mW$ disc slide.  If it is an $\mR$ disc slide, then apply Lemma \ref{multieye slide invariance}.  If it is a $\mG$-disc slide, then apply Lemma \ref{multieye slide invariance} using $\mD(\mW,\mF)$, where in that formulation the disc slide becomes a  $G$-disc slide of a finger disc.  Since $\mD\mfw=\mD(\mW,\mF)$ the result also holds for $\mF$-disc slides.\end{proof}

\begin{corollary} \label{full reduction} \emph{(Reduction to FEA)}  Any finger/Whitney system $(\mF,\mW)$ can be reduced to $(\mF',\mW')$ in  full embedded arc position by a sequence of untwisted isotopy, disc slides, boundary twisting, cross special sum square moves, order induced switching and clasping.  Each of those moves preserve $\mD$. \qed \end{corollary}

To summarize \S 1 - \S 9:  We built from scratch a calculus for computing $\mD$ which at a given stage produces a value independent of previously allowable choices.  By developing a sufficiently robust theory we showed that these operations preserve $\mD$ independent of the order performed.  Further, these six operations are sufficient to transform $(\mF,\mW)$ to $(\mF',\mW')\in$ FEA.  Thus, if  $(\mF'', \mW'')\in$ FEA is obtained by another such sequence of operations, then $\mD\mfw=\mD(\mF', \mW')=\mD(\mF'', \mW'')$.  

The next two sections will introduce new operations on finger/Whitney systems which are also $\mD$-invariant.  These will be needed to prove that $\mD\mfw$ is an invariant of the relative homotopy class of the loop of multi-spheres arising from $\mfw$.

\section {$\mD$ is invariant under general switching and the special sum square move} 

In this section we prove that $\mD\mfw$ is invariant under switching as stated in Definition \ref{switching} and more generally in  Definition \ref{multi switching}.  This contrasts with our until now requirement of minimal switching as stated in Convention \ref{minimal switch}.    

In \S 5 we proved that $\mD$ is invariant under cross special sum square moves.  Here we prove invariance of $\mD$ under uncross special sum square moves.

\begin{proposition} \label{sum invariance} i) Given $\mfw$, if $\mW'$ is obtained from $\mW$ by an \textbf{n}-switching, then $\mD(\mF,\mW)=\mD(\mF, \mW')$.

ii)  If $\mW'$ is obtained from $\mW$ by a special sum square move, then $\mD(\mF,\mW)=\mD(\mF, \mW')$.\end{proposition}

\begin{proof}  i)   We can assume that $\mR\cup\mG$ is in $\mG$-framed finger form with $\mW$ the standard Whitney discs.  Since any \textbf{n}-switching is a composition of switchings with a single switch disc, it suffices to consider the case that   $\mW'$ is obtained from $\mW$ by replacing $w_1\in \mW_{11}$ by $w'$ as in Definition \ref{switching}.  We will  find a switching $\mW^*$ of $\mW$ which is a minimal switching for both $\mfw$ and $(\mF, \mW')$ for suitable orderings on $\mW$ and $\mW'$ respectively.  Since by Proposition \ref{order}, we are free to choose the ordering and by Proposition \ref{multi switch invariance}, $\mD$ is invariant under minimal switching the result follows.  

View $w'$ for the moment as finger germed.  After an ambient isotopy, using a composition of the six restandardization maps, whose time one map fixes $\mW$ pointwise and $\mR\cup\mG$ setwise we can assume that $w'$ is the standard finger germed switch disc $\hat w_1$.  
\vskip 8pt
\noindent\emph{Case 1}:  $w_1$ lies in the immersed arc component of $\mW_{11}\cup\mF_{11}$.
\vskip 6pt
\noindent\emph{Proof of Case 1}:  A 1D-representation of $w_1'\cup \mW_{11}$ appears as in Figure \ref{fig:sss 1}  a), the black arcs indicating discs that similarly match of $\mF_{11}$ and $w_1'$ being finger germed.  In general $\mF_{11}\cup\mW_{11}$ may have multiple immersed cycles and  may intersect cross discs.  Define $\mW^*$ as follows.  First, $\mW^*_c=\mW_c=\mW'_c$.  For $i\ge 2$, set $\mW^*_{ii}$ to be the standard minimal switch system for $\mW_{ii}=\mW_{ii}'$.  Finally, Figure \ref{fig:sss 1} b) shows how to construct $\mW_{11}^*$.   Here $w_1'$ is drawn Whitney germed.   Note that it is a minimal switching of  $\mW$ with respect $(\mF, \mW)$ and a minimal switching of $\mW'$ with respect to $(\mF, \mW')$.  

 The passage from $\mW_{11}$ to $\mW_{11}^*$ has $n$ switch discs, where $n$ is the number of immersed cycles of $\mF_{11}\cup \mW_{11}$, while the passage from $\mW'_{11}$ to $\mW^*_{11}$ has $n+1$ switch discs, where $w_1$ is the extra switch disc.   Note that by Proposition \ref{order} we are free to choose the orderings which determine these switchings.  In Figure \ref{fig:sss 1} to go from $\mW$ to $\mW^*$, $w_2<w_3<w_1$ and $w_2, w_3$ are the switch out discs, while to go from $\mW'$ to $\mW^*$, $w_1'<w_2<w_3$ and these are the switch out discs.  \qed

\begin{figure}[!htbp]
    \centering
    \includegraphics[width=.85\linewidth]{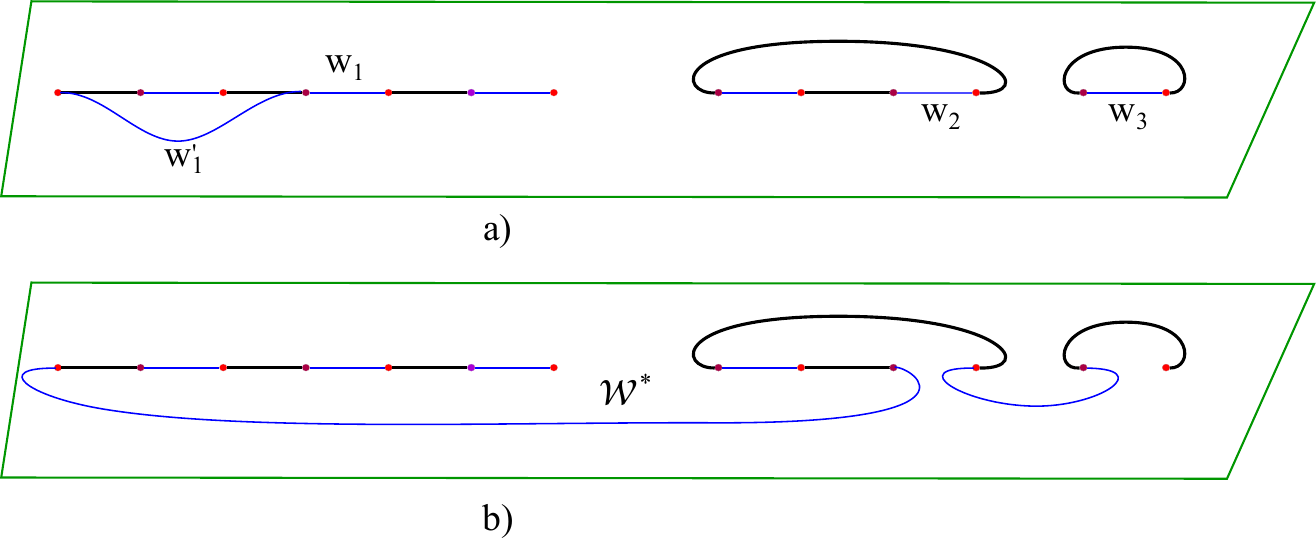}
    \caption{A 1-switch with switch out disc $w_1$ in the arc}
    \label{fig:sss 1}
    \end{figure} 
    
    \begin{figure}[!htbp]
    \centering
    \includegraphics[width=.85\linewidth]{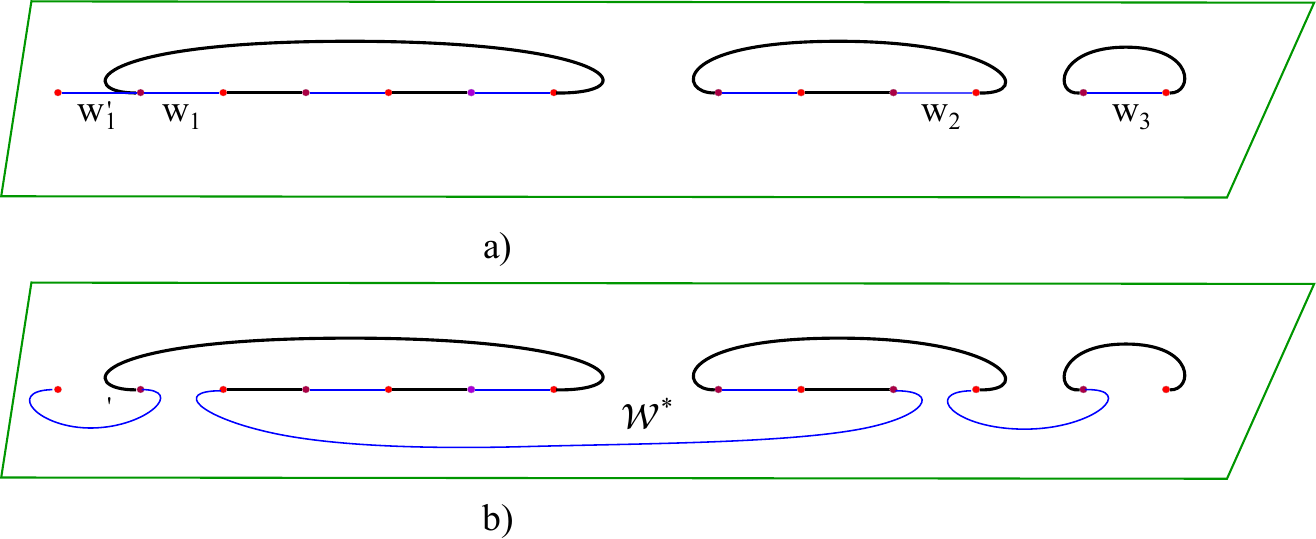}
    \caption{A 1-switch with switch out disc $w_1$ in a cycle}
    \label{fig:sss 2}
    \end{figure}

\noindent\emph{Case 2}:  $w_1$ lies in an immersed circle component of $\mF_{11}\cup\mW_{11}$.
\vskip 6pt
\noindent\emph{Proof}.  Figure \ref{fig:sss 2} a) shows a representative case with Figure \ref{fig:sss 2} b) showing a common minimal switching $\mW^*$ for both $\mW$ and $\mW'$ when restricted to $\mW^*_{11}$.  Here $\mW_{11}$ has $n$ switch out discs while $\mW_{11}'$ has n-1 switch out discs. \qed
\vskip 8pt
\noindent\emph{Proof of ii)}:  By Lemma \ref{common d} we can assume that $\mW'$ is obtained by an uncross special sum square move.   Figure \ref{fig:sss 3} then shows how to obtain $\mW'$ from $\mW$ by  three 1-switches, hence the result follows from i).\end{proof}

 \begin{figure}[!htbp]
    \centering
    \includegraphics[width=.85\linewidth]{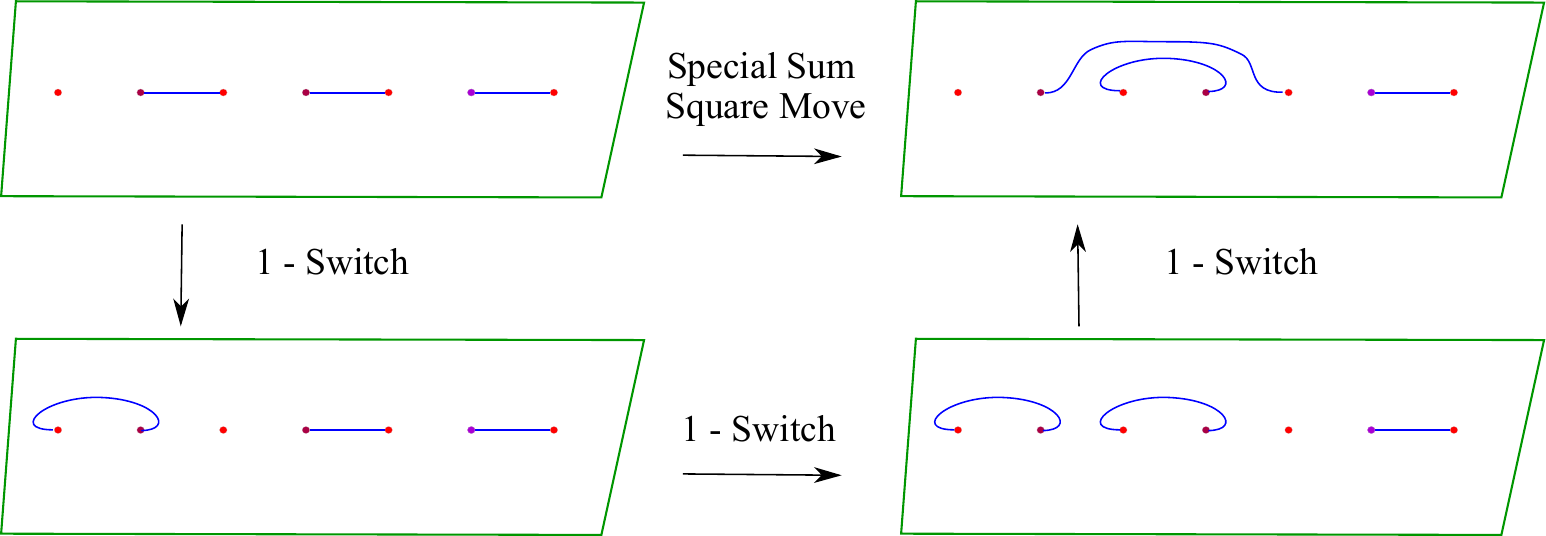}
    \caption{The special sum square move as three 1-switches}
    \label{fig:sss 3}
    \end{figure}

\section{2-parameter families of embeddings} \label{2-parameter families}

The previous sections showed how to produce an element of $\BZ_2$ given a representative $$\alpha: ([0,1],\{0,1\})\rightarrow (\Emb(\sqcup_k S^2, \#_k S^2\times S^2), \mR^{std})$$ in finger first position. In order to show that $\mD$ is well defined on homotopy classes and does not depend on the chosen representative, we will need to analyze generic homotopies between representatives. We take up this analysis in this section, studying homotopies between paths of embeddings in general.  The proof that $\mD$ depends only on the homotopy class $[\alpha]\in \pi_1(\Emb(\sqcup_{i=1} ^k S^2, \#_kS^2\times S^2), \mR^{std})$ as well as being a group homomorphism will be given in \S \ref{homotopy invariance}.

Throughout this section, $\mR$ will be a smooth closed orientable surface, $X^4$ a smooth closed connected orientable 4-manifold and $\mG$ a fixed smoothly embedded closed orientable surface with trivial normal bundle in $X$. We will let $\Emb(\mR, X)$ denote the space of smooth embeddings of $\mR$ in $X$. We will at times denote by $\alpha:[0,1]\rightarrow \Emb(\mR, X)$ a path of embeddings and other $H_s(t)$ where $H:[0,1]^2\rightarrow \Emb(\mR, X)$ is a 2-parameter family of embeddings and $H_s(t)$ is the restriction to the family at a fixed $s\in [0,1]$. 

There are two main theorems in this section, the 2-parameter ordering theorem, Theorem \ref{2-parameter families}, and the 2-parameter finger/Whitney system theorem, Theorem \ref{T: fw system moves}. Both are of independent interest; however, together they explain exactly how two finger/Whitney systems for the same relative homotopy class of embeddings are related. Up to isotopy, any two-finger/Whitney systems for the same homotopy class of embeddings differ by a finite set of five moves:
\begin{enumerate}
    \item disc slides,
    \item sphere slides (Definition \ref{D: sphereslide}),
    \item birth/death moves (Definition \ref{D: b/d move}),
    \item $x^3$-moves (Definition \ref{D: x^3}) and
    \item saddle moves (Definition \ref{D: saddle}).
\end{enumerate}

We will call these operations \emph{$\mF\mW$-moves}.

\subsection{Generic families of embeddings.}

In this part of \S6, we detail the interaction between a generic family of closed oriented smooth embeddings of the surface $\mR$ and the fixed smoothly embedded surface $\mG$ with trivial normal bundle in a closed smooth 4-manifold $X^4$. 

\begin{lemma}\label{L: Transverality Lemma} 
    Let $X^4$, $\mR$ and $\mG$ be smooth closed manifolds with $\mG\subset X^4$. A generic map $H: B^k\rightarrow \Emb(\mR, X^4)$ will be transverse to $\mG\subset X^4$. Moreover, if $\partial H$ is already transverse, then by a small deformation fixing $\partial H$, $H$ can be made transverse to $\mG$.
\end{lemma}
Note that $H$ being transverse to $\mG$ is the same as $\tr{H}$ being transverse to $B^k\times \mG$.
    \begin{proof}
        First, we consider the associated map $\tilde{H}:B^k\times \mR \rightarrow X^4$. As transversality is a generic open condition, if $\tilde{H}$ is not transverse, then it can be deformed to be transverse by an arbitrarily small deformation in $C^{\infty}(B^k\times \mR, X^4)$. Furthermore, since $H|_x: \mR \rightarrow X^4$ is an embedding for all $x\in B^k$ and the set of embeddings is an open subset of $C^{\infty}(\mR, X^4)$, we can choose our deformation of $\tilde{H}$ small enough so that $H|_x$ remains an embedding for all $x\in B^k$.
    \end{proof}

\subsection{Local analysis of paths of embeddings}
As a warmup, we will analyze the consequences of Lemma \ref{L: Transverality Lemma} for $k =1$. Let $\alpha:[0,1]\rightarrow \Emb(\mR, X^4)$ be a path of embeddings and denote by $\mR_t = \alpha(\mR, t)$ the image of $\alpha$ in $X$. Denote by $L = \alpha^{-1}(\mG)$ the preimage of $\mG$ in $[0,1]\times \mR$.

Dimension counting tells us that $L$ is an embedded 1-manifold with boundary. Next, after an additional perturbation, we may assume that the projection map $\pi_{[0,1]}:\mR \times [0,1]\rightarrow [0,1]$ restricts to a Morse function on $L$. Let $h:L\rightarrow [0,1]$ denote this function. Now, the regular values of $h$ correspond to times $t$ such that $\mR_t$ transversely intersects $\mG$, since at regular values, the vector $\partial_t$ decomposes as $v_L + v_\mR$ with $v_L\in T_pL$ and $v_{\mR}\in T_p\mR$. Since $\alpha_*(v_L)\in T_p \mG$ ($\alpha(L)\subset \mG$), $\alpha_*(\partial_t)\in T_p\mG + \alpha_*(T_p\mR_t)$. Therefore, the only way for $\alpha$ to be transverse to $\mG$ at time $t$ is if $\mR_t$ is transverse to $\mG$. The critical points of $h$ designate the moments when a finger move (Whitney move) happens. The index of the critical point determines the difference between the two moves; index 0 critical points correspond to finger moves, and index 1 critical points correspond to Whitney moves.

\subsubsection{Local finger-Whitney data}\label{S: 1par localFW data}

We now recall how transversality at the critical points gives rise to the local data for the finger/Whitney moves. 

\begin{definition}
Let $t_i$ be a critical value for $h$ and $p_i\in L\subset\mR_{t_i}$ a critical point. At a critical point, $T_{(p_i,t_i)} L\subset T_{(p_i,t_i)} \{t_i\}\times \mR $. Hence $\mR_{t_i}$ has a 1-dimensional tangency to $\mG$ at $\alpha_{t_i}(p_i)$. We call such a point $p_i$ a \emph{finger point (Whitney point)}, depending on the index.
\end{definition}
\begin{remark}
    We will often use both the terms critical point and finger/Whitney point to refer to the critical points of $h$ and their images in $X$.
\end{remark}

In order for $\alpha$ to be transverse  to $\mG$ at $(t_i, p_i)$, $\alpha_*(\partial_t)$ must have a non-zero component in the normal bundle of $\mG$. Let $v_{p_i}\in N_{p_i}\mG$ denote the non-zero component of $\alpha_*(\partial_t)$. Let $U_{p_i}\subset \mG$ be a neighborhood of $p_i$ that is disjoint from all other intersections of $\mR_{t_i}$ and $\mG$. As $\mG$ has a trivial normal bundle, we can fix a trivialization and fix an identification of the normal bundle with a tubular neighborhood. Using this identification of the normal bundle with a tubular neighborhood, we can identify a local 3-ball neighborhood of $\alpha(p_i)\subset X^4$ $U_{p_i}\times \{\lambda v_{p_i}\} $ for $-\delta \leq \lambda\leq \delta$ for some $\delta>0$. We will let $U = U_{p_i}$. As $\alpha$ is transverse to $\mG$, there exists an $\epsilon>0$ such that for $-\epsilon < \lambda < \epsilon$, $\mR_{t_i}$ is transverse to $U\times \{ \lambda v_{p_i} \}$. The stability of the transverse intersection then implies that there is a $\delta >0$ such that for $t_i - \delta< t< t_i+\delta$, $\mR_t$ is transverse to $U\times \{ \lambda v_{p_i} \}$ as well. 

We now describe how to obtain the local finger arc and finger disc when $p_i$ is an index 0 critical point of $L$.
\begin{construction}\label{Constrution of local fw data}
The local story for Whitney arc and disc for index 1 critical points is the same with the time parameter reversed. So we will assume that as $t$ moves from $t_i - \delta$ to $t_i + \delta$, the number of intersections between $\mR_t$ and $\mG$ increases by two. As $\mR_{t_i}$ is transverse to $U\times \{ \lambda v_{p_i} \}$, the intersection $A_{t_i} = \mR_{t_i}\cap \overline{U\times \{ \lambda v_{p_i} \}}$ will be a neatly embedded 1-mfld. By decreasing the range of $\lambda$ if necessary, we can assume that $A_{t_i}$ is a boundary parallel arc in $\overline{U\times \{ \lambda v_{p_i} \}}$ with both endpoints on $U\times -\epsilon$. Being boundary parallel implies that the complement of $A_{t_i}$ is diffeomorphic to $S^1\times B^2$. This allows us to choose an embedded disc $D_{t_i}$ in $\overline{U\times \{ \lambda v_{p_i} \}}$ disjoint from $U\times\{0\}$ with half its boundary on $A_{t_i}$ and the other half on $U\times \{-\epsilon\}$. Next, let $\gamma_{t_i}: B^1\rightarrow \overline{U\times \{ \lambda v_{p_i} \}}$ be a smooth unknotted embedding with the following boundary conditions: $\gamma_{t_i}(-1) = p_{i}\times \{0\}$, $\gamma(1) = p_i\times \{\epsilon\}$ and $(\gamma_{t_i})_*(\partial x) =v_{p_i}$ at $-1$ and $1$. Note that $\gamma_{t_i}(-1)\in \mR_{t_i}$. Now we extend $D_{t_i}$ to a family of embeddings $D_t$ for $t_i \leq t \leq t_i+\delta$ by isotopy extension of $\mR_t$ and similarly extend $\gamma_{t_i}$ to a family of embeddings $\gamma_t$ for $t_i-\delta \leq t\leq t_i$.

\end{construction}

\begin{definition}
For time $t_i-\delta \leq t < t_i$, the \emph{local finger arc} $f^{arc}_i$ for the embedding $\mR_t$ is the image of $\gamma_t$ restricted to $[-1,x(t)]$  where $x(t)$ is the smallest value of $x$ such that $\gamma_t(x(t))\in U\times \{0\}$. For time $t_i< t< t_i+\delta$, the \emph{local finger disc} $f_i$ for $\mR_t$ is the subset of $D_t$ restricted to $U\times \{ \lambda v_{p_i} \}$ with $0\leq \lambda \leq \epsilon$. 
\end{definition}

\begin{remark}\label{R: local FW data is contractible}
    We note that the 3-balls $\overline{U\times \{ \lambda v_{p_i} \}}$ depend smoothly on the vector $v_{p_i}$ in that when $v_{p_i}$ varies smoothly in $N_{p_i} \mG$, the local 3-ball also varies smoothly. The choice of disc $D_{t_i}$ and arc $\gamma_{t_i}$ also depends smoothly on the data. Moreover, the space of choices for $D_{t_i}$ and $\gamma$ are essentially the connected components of embedding spaces with fixed boundaries $\Emb(B^2, S^1\times B^2)$ and $\Emb(B^1, B^1\times B^2)$ containing the embeddings $pt\times B^2$ and $B^1\times pt $ respectively. By \cite{Ha1}, these components are contractible. Therefore, we can extend the local construction of the finger/Whitney data to any family of finger/Whitney moves.
\end{remark}

Let $C(M) = [0,1]\times M/(0\times M \sim *)$ denote the cone on the manifold $M$, and $*_M\in C(M)$ the cone point. Given a family of embeddings $H:B^k\rightarrow \Emb(\mR, X)$, let $\operatorname{tr}(H): B^k\times \mR\rightarrow B^k\times X$ denote the trace of the embedding. 

\begin{definition}\label{D: FW germs}
    Let $\mR$, $\mG$, and $X$ be as above and let $\alpha$ be a generic path of embeddings of $\mR$ into $X$. Let 
    $$C(\mF\mW) \colon = C(B^2)\sqcup C(B^1) / *_{B^2}\sim *_{B^1}.$$ 
    Working with the trace $\operatorname{tr}(\alpha)$ in $[0,1]\times X$,  a \emph{$\mF\mW$-germ} for a critical point $p_i$ with critical value $t_i$ is an embedding of $C(\mF\mW)$ into $[0,1]\times X$ satisfying the following (see Figure \ref{F: fw-germ} for an illustration):
    \begin{enumerate}
        \item $*_{B^2} = *_{B^1}$ maps to a finger/Whitney tangency between $\{t_i\}\times \mR_{t_i}$ and $\{t_i\}\times \mG$
        \item Each $B^2\times x(t) \subset C(B^2)$ smoothly embeds in $X\times \{t\}$ as a local finger disc for the finger/Whitney move at $t_i$
        \item Each $B^1\times x(t)$ smoothly embeds as the local finger arc for the finger/Whitney move at $t_i$.
    \end{enumerate}

\begin{figure}[!htbp]
    \begin{subfigure}{.45\textwidth}
    \centering
    \includegraphics[width=.9\linewidth]{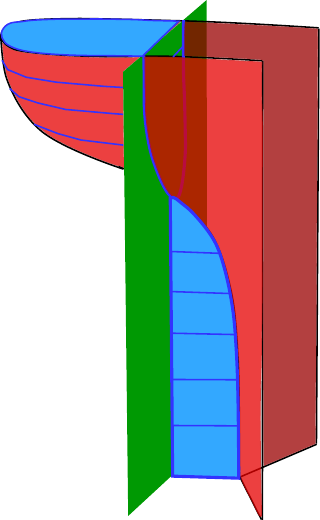}
    \end{subfigure}
    \begin{subfigure}{.45\textwidth}
        \includegraphics[width=.9\linewidth]{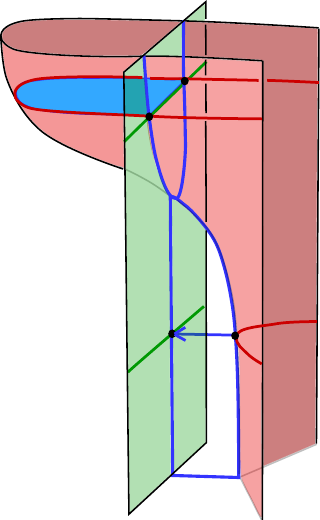}
    \end{subfigure}
    \caption{ The following image depicts the trace of the family $\mR_t$ near a finger/Whitney point $p$. On the left we have the image of the embedding $\mF\mW(p)$, where the $C(\mF\mW)$ is shown in blue. On the right is the same embedding, highlighting a time before the tangency with the local finger arc and a time after the tangency with the local finger disc. } \label{F: fw-germ}
\end{figure}

  Let $S_h$ denote the set of critical points of $h$. Then a $\mF\mW$ germ for $\alpha$ is an embedding
  \[
  \mF\mW(\alpha): S_h\times C(\mF\mW)\rightarrow [0,1]\times X
  \]  
  with $\mF\mW(\alpha)|_{p_i}$ is a $\mF\mW$ -germ for the critical point $p_i$.
  Technically this is a germ in the sense of being  an equivalence class of such embeddings with respect to restricting to arbitrarily small neighborhoods of the cone point, but we generally suppress this point.
\end{definition}

\begin{lemma}\label{L: 1-dim FW Germ}
    Given a 1-parameter family of embeddings $\alpha$, there exists a $\mF\mW$-germ for every finger/Whitney move.
\end{lemma}
\begin{proof}
    The lemma follows from Construction \ref{Constrution of local fw data} and Remark \ref{R: local FW data is contractible}.
\end{proof}

\begin{remark}
The $\mF\mW$-germs are heuristically the ``ascending/descending manifolds'' for the path of embeddings relative to $\mG$. We have defined them using cones; however, it can be shown that the cones are actually smoothly embedded $B^{m+1}$-half balls: half of the boundary on $[0,1]\times \mR\cup [0,1]\times \mG$ and the other half embedded in a level set $\{t\}\times X$ (see Figure \ref{F: fw-germ}). Note that other authors~\cite{perron1975pseudo, sharpe88, BPTlinkconcordance} have developed similar techniques to make a connection between finger and Whitney disks, more generally known as ``membranes'', and flow along gradient-like vector fields for certain relative Morse functions. The significantly new perspective in this work arises in the next section, in which we add an extra parameter.
\end{remark}

\subsection{Local analysis of 2-parameter families.}
Next, we turn our attention to the case where $k=2$ and perform the same analysis. We start with a generic relative homotopy $H: B^2 \rightarrow \Emb(\mR, X^4)$, which we rewrite as the map $H: B^2 \times \mR \to X$. We will let $(s,t)$ denote the coordinates of $B^2$, and $\mR_{s,t}$ the image of $H(\{(s,t)\} \times \mR)$ in $X$. Let us denote by $\Sigma$ the preimage $H^{-1}(\mG)$ in $B^2 \times \mR$ and by $h$ the restrictions of $\pi_{B^2}$ to $\Sigma$. From the transversality of $H$ and $\mG$ it follows that $\Sigma$ is a smooth properly embedded surface with boundary. 

\subsubsection{Folds and cusps.}

The assumption that $H$ is generic implies that the map $h:\Sigma \rightarrow B^2$ is also a generic map, and hence a stable map \cite{golubitsky}. It is a classical result that stable maps between surfaces will have singularities of ``fold'' or ``cusp'' type in the interior. Around a fold singularity, the function $h$ will locally behave like the map $(p_1,p_2)\mapsto (p_1, h_{p_1}(p_2))$ where $h_{p_1}$ is a 1-parameter family of Morse functions of constant index, while around a cusp singularity $h$ locally behaves like the map $(p_1,p_2)\mapsto (p_1, h_{p_1}(p_2))$ with $h_{p_1}$ a standard birth/death of a pair of canceling Morse critical points \cite{Whitneyclassification, golubitsky}. Along the boundary, the map will only have generic Morse-type singularities.

\begin{definition}
Let $H:B^2\rightarrow \Emb(\mR,X)$ be a generic smooth family of embeddings and $\mG$ a fixed embedding. We call $h:\Sigma \rightarrow B^2$ the singular map associated with $H$. The union of all the critical points of $h$ is called the \emph{singular set} and is denoted by $S_h\subset \Sigma$. This is a smooth 1-manifold with boundary, neatly embedded in $B^2\times \mR$. We call the image $h(S_h)$ of the singular set a \emph{Cerf graphic}. We will refer to both $p\in S_h$ and $(s(p),t(p)) \subset h(S_h)$ as a fold point when $h$ has a fold singularity at $p\in S_h$ and a cusp point if $h$ has a cusp singularity at $p$.
\end{definition}
\begin{remark}
A consequence of $h$ being a stable map is that $h(S_h)$ is an immersion with normal crossings away from the cusp points. 
\end{remark}
\begin{figure}
    \centering
    \includegraphics[width=0.8\linewidth]{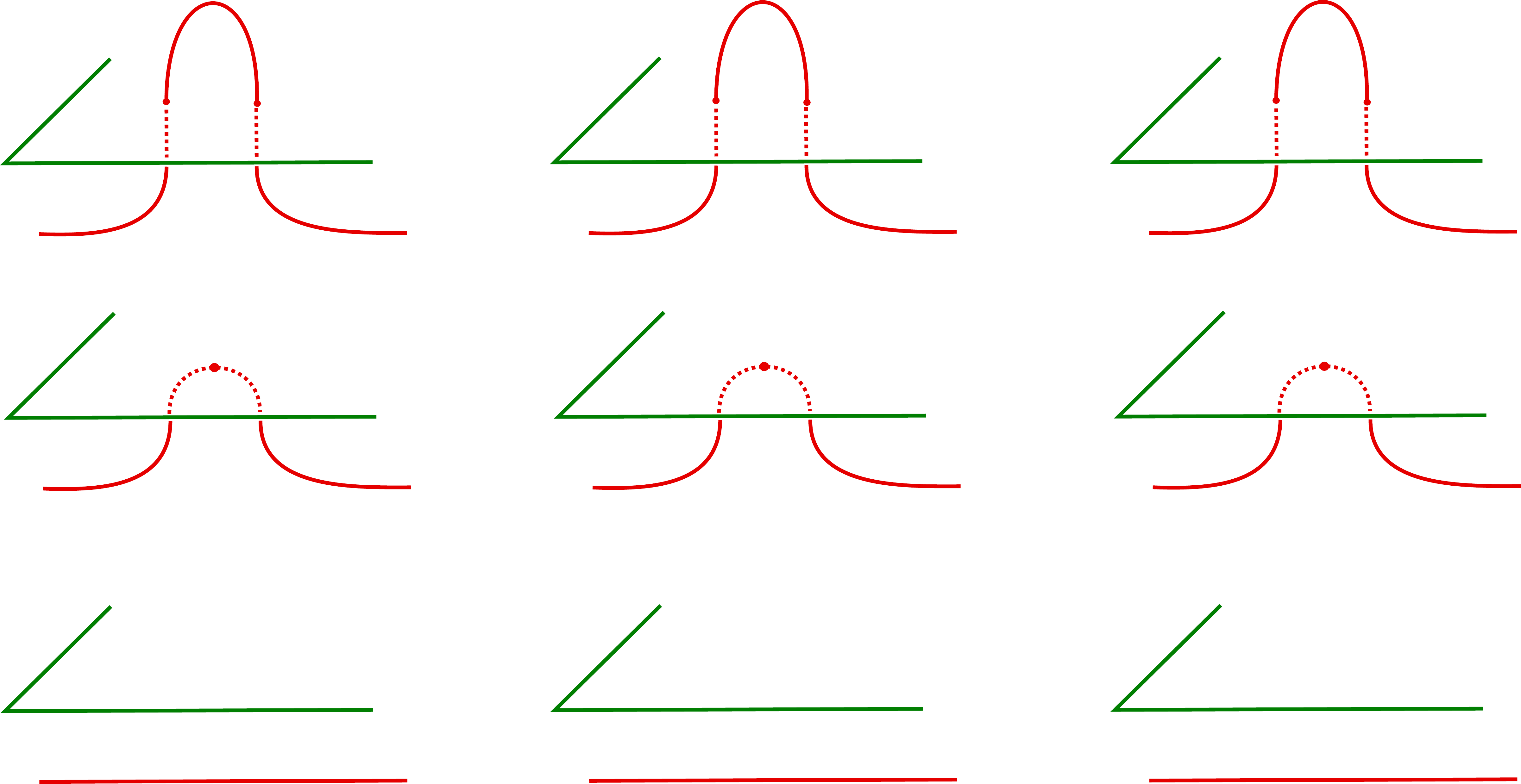}
    \caption{The local interaction between $\mR_{s,t}$ and $\mG$ near a fold singularity for $h:\Sigma\rightarrow B^2$. Each frame shows the configurations of $\mR_{s,t}$ and $\mG$ in a 3-ball slice of a 4-ball neighborhood around the point $\mR_{s_0,t_0}(p)$. Here, in terms of identifying the frames, the horizontal axis is $s$ and the vertical axis is $t$. The middle frame depicts the configuration corresponding to $\mR_{s_0,t_0}$ and $\mG$, and these configurations vary smoothly as one moves from frame to frame. Note that every frame along the middle row also corresponds to a fold point for $S_h$.}
    \label{F: Fold}
\end{figure}
\begin{figure}
    \centering
    \includegraphics[width=0.8\linewidth]{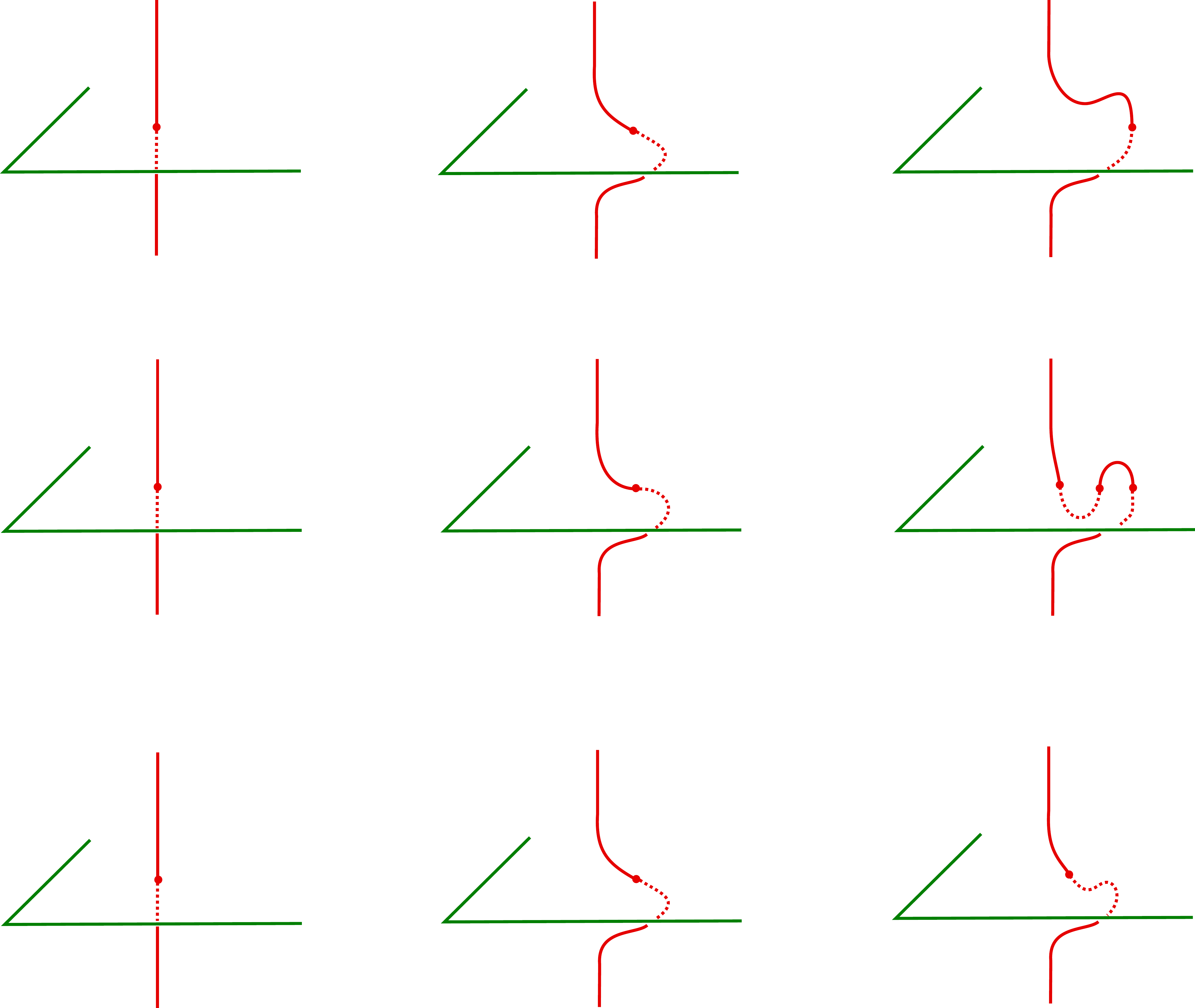}
    \caption{The local interaction between $\mR_{s,t}$ and $\mG$ near a cusp singularity for $h:\Sigma\rightarrow B^2$. Each frame shows the configurations of $\mR_{s,t}$ and $\mG$ in a 3-ball slice of a 4-ball neighborhood around the point $\mR_{s_0,t_0}(p)$. Again the horizontal axis is $s$ and the vertical axis is $t$. The middle frame depicts the configuration corresponding to $\mR_{s_0,t_0}$ and $\mG$, and these configurations vary smoothly as one moves from frame to frame.}
    \label{F: x cubed}
\end{figure}

The singular points of $\Sigma$ correspond to instances where $\mR_{s,t}$ has some increased tangency with $\mG$. At a fold or cusp point, $\mR_{s(p),t(p)}(p)$ will have a ``1-dimensional'' tangency with $\mG$; that is, $T_p\mR_{s(p),t(p)}\cap T_p \mG$ is 1-dimensional. The local models for the evolution of $\mR_{s,t}$ near a fold or a cusp are given in Figures \ref{F: Fold} and \ref{F: x cubed}. Each frame in the figures corresponds to a parameter value $(s,t)$ in a sufficiently small neighborhood of $(s(p),t(p))$. Shown is a $B^3$-slice of a 4-ball neighborhood of $H(p)$ that contains an open subset of $\mG$ and intersects $\mR_{s,t}$ in a 1-manifold. These neighborhoods result from the transversality of $H$ and are constructed in a similar manner as in the 1-dimensional case.

\begin{construction}\label{Construction 2-par local models}
The local models above can be derived from $h$ by a similar analysis as in Section \ref{S: 1par localFW data}. We will not repeat that, but remark on a few points. First, at a fold point $p\in S_h$ the role that $\partial_t$ plays in the 1-dimensional case is given by a unit vector $w(p)$ transverse to the component of $T_p S_h$ in $T_p B^2\oplus 0$. Then for some $\epsilon(p)>0$ we can restrict to the parameters $(s,t)$ of the form $(s(p),t(p))+ \delta w(p)$ for $-\epsilon(p) \leq \delta \leq \epsilon(p)$. Letting $\mR_{s,t}$ vary with $\delta$ along this line reduces to the 1-dimensional case. If we choose a smoothly varying transverse vector field $w(p)$ along the fold locus in $h(S_h)$, then $H_*(w(p))\in N_p(\mG)$ will vary smoothly. As such, the local neighborhood $U_p\subset \mG$ and the resulting 3-ball constructed using $H(w(p))$ will also vary smoothly. Furthermore, we can arrange the local 3-balls to be disjoint whenever there is a crossing by shrinking the open sets $U_p$ and $\epsilon(p)$ if necessary.

Let $p_0\in S_h$ be a cusp point and $(s_0,t_0)$ its image $B^2$ by $h$. Let $v_T(p)$ be a unit vector tangent to the fold locus $h(S_h)$ at $h(p)$ near $(s_0,t_0)$ and let $v_{p_0}$ be the limit of $v_T(p)$ as $p$ approaches $ p_0$. Note that since the graphic $h(S_h)$ near $(s_0,t_0)$ is a cusp, after translating $v_{p_0}$ to $(s(p),t(p))$, $\{v_T(p), v_{p_0}(p)\}$ forms a basis for $T_{s(p),t(p)}B^2$. Therefore the vector $v_{p_0}(p)$ must be transverse to the fold locus for $(s(p),t(p))$ near $(s_0,t_0)$. More importantly, we have that transversality of $H$ at $p_0$ is given by $$H_*(\langle v_{p_0}\rangle\oplus T_{p_0} (\{(s_0,t_0)\}\times \mR)) + T_{p_0} \mG = T_{p_0} X.$$
\emph{Therefore, $H_*(v_{p_0}(p))\in N_p(\mG)$ for all $p$ sufficiency close to and including $p_0$}. So for compatibility, we define $w(p)\in TB^2|_{S_h}$ as a smoothly varying vector field everywhere transverse to the fold locus that smoothly limits to $v_{p_0}(p)$ for $p$ sufficiently close to $p_0$. This will guarantee that the local 3-ball slices in $X$ used to describe the two different local models (Figure \ref{F: Fold} and Figure \ref{F: x cubed}) fit together compatibly.

\end{construction}

\begin{lemma}\label{L: 2-par fw germs}
    Let $\mR$, $\mG$, and $X$ be as above. Let $H:B^2\rightarrow \Emb(\mR, X)$ be a generic 2-parameter family and $h:\Sigma\rightarrow B^2$ the singular map associated with the family $H$. Let $\tilde{S_h}\subset B^2\times \mR$ denote the singular set of $h$ without the cusp points. We consider the trace $\operatorname{tr}(H)$ in $B^2\times X^4$. Then there exists an embedding 
    \[
    \widetilde{\mF\mW}: 
    \tilde{S_h} \times C(\mF\mW) \rightarrow B^2\times X \]
    such that
    \begin{enumerate}
        \item for $p\in \tilde{S}_h$, $\widetilde{\mF\mW}(p,*)$ maps to the finger/Whitney point $\mR(p)_{s(p),t(p)}$ and
        \item for a line segment $l_p$ through $(s(p),t(p))$ and transverse to $h_*(T_p S_h)$, $\widetilde{\mF\mW}$ embeds a $\mF\mW$-germ for the finger/Whitney point $H(p)$ for family $\mR_{l_p}$.  
    \end{enumerate}  
\end{lemma}

\begin{proof}
     First, we note that for any $p\in \tilde{S_h}$ the vector field $w(p)$ from Construction \ref{Construction 2-par local models} determines a 1-parameter family of embeddings $\mR_{l}$ where $l$ is a small line segment in $B^2$, the direction for which is determined by $w(p)$. Since $w(p)$ is transverse to the fold locus, the family $\mR_{l_{w(p)}}$ has a single finger/Whitney point. Construction \ref{Constrution of local fw data} then gives an embedding $\{p\}\times C(\mF\mW)$ in $l_{w(p)}\times X$. Now, since this vector field varies smoothly along $\tilde{S}_h$, the local data used to construct $\{p\}\times C(\mF\mW)$ will vary smoothly. Therefore, we get a family of finger/Whitney points $H(\tilde{S}_h)$ and smoothly varying local data. By Remark \ref{R: local FW data is contractible}, we know that the space of choices made to define $\{p\}\times C(\mF\mW)$ is contractible, allowing us to extend $\{p\}\times C(\mF\mW)$ locally along $\tilde{S}_h$. By choosing the length of $l_{w(p)}$ so that no two segments intersect in $B^2$ except near the normal crossings of $h(S_h)$, we can ensure the finger/Whitney cones are embedded. Near a crossing, the local neighborhoods containing the finger/Whitney points can be taken to be disjoint. This ensures that for every pair $p,p' \in \tilde{S}_h$, $\{p\}\times C(\mF\mW)$ and $\{p'\}\times C(\mF\mW)$ will remain disjoint. Together, we obtain the existence of an embedding $\widetilde{\mF\mW}$. 

     Property (1) is trivially satisfied since the construction above implies that $\widetilde{\mF\mW}$ maps every $p\in S_h$ to the corresponding finger/Whitney point $\{(s(p),t(p))\}\times \mR_{s(p),t(p)}(p)$. To verify property (2), we note that the line segment $l_p$ intersects the neighborhood $\cup_{p\in S_h}l_{w(p)}$. For $(s,t)$ along this intersection, $\tilde{\mF\mW}$ smoothly embeds the local discs and arcs in $\{(s,t)\}\times X$. If we restrict to the connected interval containing $(s(p),t(p)$, then moving along $l_p$ over these parameters, the embeddings will vary smoothly until one such embedding collapses to the finger/Whitney point $H(p)$. If it was a disc that collapsed, then passing through the finger/Whitney point expands to a locally arc and vise versa. Therefore, along the line $l_p$ $\widetilde{\mF\mW}$ does indeed embed a $\mF\mW$-germ for the family $\mR_{l_p}$ near the finger/Whitney point $H(p)$.
\end{proof}

\begin{construction}[Extension to cusps]\label{constr: FW germs near cusps}  Near a cusp, the vector field $w(p)$ is given by $v_0$, where $v_0$ is the limit of unit tangent vectors of the fold locus. The 1-dimensional $\mF\mW$-germs from Lemma \ref{L: 2-par fw germs} are embedded over small line segments determined by $v_0$. Let $l_0$ denote the local line segment through $(s_0,t_0)$ parallel to $v_0$. Now note that at a cusp, there are two folds coming together. If we look at $\{p\}\times C(\mF\mW)$ as $p$ approaches the cusp point $p_0$ along either fold, we see $C(B^1)$ shrinks smaller and smaller, collapsing to the cusp point $H(p_0)$  in the limit. For the local discs, the cones $C(B^2)$ do not collapse. Instead they both embed in $l_0\times X$ such that in the corresponding $\{(s,t)\}\times X$, the $B^2\times \{x_i\}$ slices are smoothly embedded in the local 3-ball determined by $H_*(v_0)$. If we denote by $f_{s,t}:= \widetilde{\mF\mW}(B^2\times \{x_1\})$ and $w_{s,t}:=\widetilde{\mF\mW}(B^2\times \{x_2\})$, then for each $(s,t)\in l_0$ where appropriate, $f_{s,t}\cap w_{s,t} \subset \mR\cap \mG$, $\partial f_{s,t}\cup \partial w_{s,t}\cap \mG$, $\partial f_{s,t}\cup \partial w_{s,t}\cap \mR$ are both embedded arcs. As $(s,t)$ approaches the cusp point $(s_0,t_0)$ in $l_0$, both $f_{s,t}$ and $w_{s,t}$ collapse to the cusp point $H(p_0)$. We call this configuration at a cusp a \emph{cusp extension}. 
\end{construction}

Putting Lemma \ref{L: 2-par fw germs} and Construction \ref{constr: FW germs near cusps} together, we have the following definition of the $\mF\mW$-germ for 2-parameter families. 

\begin{definition}
    Let $\mR$, $\mG$, $X$ be as above and let $H:B^2\rightarrow \Emb(\mR, X)$ be a generic embedding and $h:\Sigma\rightarrow B^2$ the corresponding stable map associated with the family. The $\mF\mW$-germ for the family is the embedding 
    \[\mF\mW(H):S_h \times C(\mF\mW)\rightarrow B^2\times X\]
    such that $\mF\mW(H)$ is an extension of an embedding $\widetilde{\mF\mW}$ to the cusp points given in Construction \ref{constr: FW germs near cusps}. 

\end{definition}

\subsection{Paths of paths.}
We now focus on the case where $H:[0,1]^2\rightarrow \Emb(\mR, X)$ is a generic relative homotopy between two paths $\alpha_0$ and $\alpha_1$, i.e., that $\alpha_i=H(i,t)$ and that $H(s,i) = \alpha_0(i)$ for all $s\in [0,1]$ and $i= 0,1$. Again, we denote the coordinates on $[0,1]^2$ by $(s,t)$ with $t$ the \emph{path parameter} and $s$ the \emph{homotopy} parameter. We will denote by $H_s$ the path of embeddings given by restricting $H$ to the line $\{s\}\times [0,1]$. Let $L_s$ denote $\Sigma\cap \{s\}\times [0,1]\times \mR$ and $h_s$ the function $h|_{L_s}$. Finally, let us denote by $L_{s,t}$ the intersection $\Sigma \cap \{(s,t)\}\times \mR$. We have that $\Sigma$ is a properly embedded surface with boundary, for which the boundary decomposes into four parts: $L_0$, $L_1$, $L_{0,0}\times [0,1]$, and $L_{0,1}\times [0,1]$. Finally, we find that since $H$ is a relative homotopy, the boundary of the singular set $\partial S_h$ is confined to the subset $L_0\cup L_1 \subset \partial \Sigma$. We will let $\mF\mW(H)$ be a fixed $\mF\mW$-germ for $H$.

 Now we examine how the intersection locus between $H_s(t)$ and $\mG$ changes during a generic homotopy. We will utilize the Cerf graphic for $h$ to organize and track the changes. By the genericity of $h$ we have that for all but finitely values of $s\in [0,1]$, the line $\{s\}\times [0,1]$ will cross all folds transversely. For such a fixed $s$, the path $\mR_{s,t}$ is a generic path where finitely many finger and Whitney moves occur between $\mR_{s,t}$ and $\mG$. We say \emph{a fold has index 0} if crossing a fold as $t$ increases corresponds to a finger move, and \emph{a fold has index 1} if crossing in the same manner corresponds to a Whitney move. As we vary $s$, two changes are possible. First, we see that the times $t$ when a given finger/Whitney move occurs can change. Second, the corresponding finger/Whitney point $p\in \mR_{s,t}\cap \mG$ can vary in both $\mR_{s,t}$ and $\mG$. However, these changes occur smoothly and, outside of our finite parameters $\{s_0,...,s_k\}$, do not change the number of finger/Whitney moves. The finite values $s_i$ correspond to parameters in which one of two possible phenomena appears in the graphic: Either the line $\{s_*\}\times [0,1]$ is tangent to a fold or the line contains a cusp point. These phenomena occur generically at different values of $s\in [0,1]$, so we can and will treat them individually.

 \subsubsection{Cusps.}\label{S: cusps}
The simpler of the two phenomena to describe at this point is when the line $\{s_*\}\times [0,1]$ contains a cusp singularity. The simplicity comes from the local model (Figure \ref{F: x cubed}) that completely describes the transition to passing through the singularity. In the graphic, the cusp will open either from left to right or from right to left. In the first case, the local model indicates that a pair of canceling finger/Whitney moves will be created, while the second case is just the opposite; a canceling pair of finger/Whitney moves are removed. The former case results in increasing the number of finger/Whitney moves by 2, and the latter decreases the number by two. Furthermore, the local finger/Whitney discs for the $\mF\mW$-germ has the following two properties: 
\begin{enumerate}
    \item[a.] The boundary arcs form a connected embedded arc on $\mR_{s,t}$ and $\mG$, and
    \item[b.] the interiors are disjoint.
\end{enumerate}

\begin{figure}
    \centering
    \labellist                             
            \hair 10pt
            \pinlabel $s$ at 300 -20
            \pinlabel $s$ at 1150 -20
            \pinlabel $t$ at -30 270
            \pinlabel $t$ at 800 270
        \endlabellist
    \includegraphics[width=0.7\linewidth]{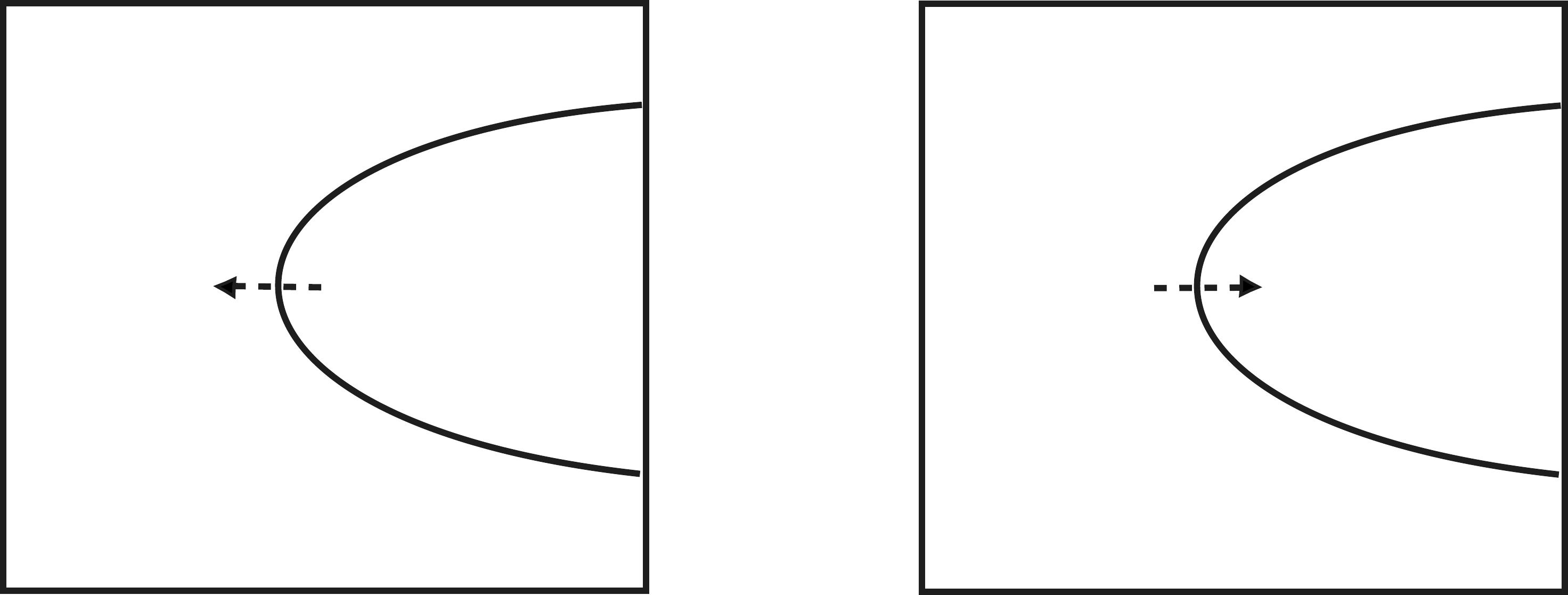}
    \caption{Local model for critical points of $\Sigma$. The arrows indicate the coorientation of the fold in the direction of decreasing intersection points between $\mR_{s,t}$ and $\mG$}
    \label{homotopy-Cerf-graphic}
\end{figure}

 \subsubsection{Birth, death, and saddles.}\label{S: b/d and saddle}
Next, we examine what happens when the line $\{s_*\}\times [0,1]$ is tangent to a fold. If we consider the function $s\circ h:\Sigma \rightarrow [0,1]$, then these vertical tangencies correspond to critical points $s\circ h$ (again, by the generic assumption on $H$, $s\circ h$ will have only nondegenerate critical points). In $[0,1]^2$, the fold tangency will appear locally as shown in Figure \ref{homotopy-Cerf-graphic} and we will restrict ourselves to the parameter values $(s,t)$ that are contained within the graphics of Figure \ref{homotopy-Cerf-graphic}. Let $(s_0,t_0)\in [0,1]^2$ denote the point where the line $\{s_0\}\times [0,1]$ is tangent to the fold. Then we have four cases to consider: 

\begin{enumerate}
    \item\label{Birth} Birth: The fold opens to the right and crossing the fold along the line $[0,1]\times \{t_0\}$ corresponds to a finger move of $\mR_{s,t}$ through $\mG$. 
    \item\label{Death} Death: The fold opens to the left and crossing the fold along the line $[0,1]\times \{t_0\}$ corresponds to a Whitney move of $\mR_{s,t}$ through $\mG$.
    \item\label{Saddle 1} Saddle: The fold opens to the right and crossing the fold along the line $[0,1]\times \{t_0\}$ corresponds to a Whitney move of $\mR_{s,t}$ through $\mG$.
    \item\label{Saddle 2} Saddle: The fold opens to the left and crossing the fold along the line $[0,1]\times \{t_0\}$ corresponds to a finger move of $\mR_{s,t}$ through $\mG$.
\end{enumerate}

 Case \ref{Birth} corresponds to the encounter of a local minimum of $\Sigma$ with respect to the function $s\circ h$. Let $\epsilon>0$ be small enough so that $s_0-\epsilon$ and $s_0+\epsilon$ are contained within the alotted parameters. For $s_0<s\leq s_0+\epsilon$, the line $\{s\}\times [0,1]$ will transversely cross the fold graphic, while for $s_0-\epsilon \leq s < s_0$, the line will be disjoint from the fold graphic. We know from the local model that crossing a fold corresponds to performing a finger/Whitney move. So we see that as $H_s(t)$ evolves from $H_{s_0-\epsilon}(t)$ to $H_{s_0+\epsilon}(t)$, we see that $\mR_{s,t}$ deforms from not intersecting $\mG$ locally to performing a finger move, then undoing the finger move. Hence, just after crossing $s_0$, the family $H_s(t)$ has a canceling pair finger/Whitney moves. Near $s_0$, the $\mF\mW$-germ gives an isotopy between the local finger disc and the local Whitney disc for $H_s$, but as $H_s$ progresses further, they may stop agreeing completely. However, the boundary arcs for the created finger and Whitney discs will always pair off the same two intersections between $\mR_{s,t}$ and $\mG$ and these arcs will be isotopic on the two spheres rel endpoints, and after arranging that the arcs to agree, the interiors of the two discs will be isotopic rel boundary in the complement of $\mR_{s,t}$ and $\mG$. Case \ref{Death} is the reverse of Case \ref{Birth} and corresponds to local maxima of $\Sigma$. The local finger/Whitney discs that exist for $s<s_0$ end up coinciding as $s$ approaches $s_0$, then cancel each other.

 The last two cases, Case \ref{Saddle 1} and Case \ref{Saddle 2}, occur when passing a saddle of $\Sigma$. This is the most interesting of the changes so far and can be thought of as a \emph{local factorization} (see Lemma 3.15 of \cite{Ga2}). We will now describe Case \ref{Saddle 1} with Case \ref{Saddle 2} being the reverse. Again, in our subset of $[0,1]^2$ we have $(s_0, t_0)$ the point in which $\{s_0\}\times [0,1]$ is tangent to the fold. In this case, crossing the fold left to right along the line $[0,1]\times \{t_0\}$ encodes a Whitney move that is performed between $\mR_{s,t_0}$ and $\mG$. Therefore, as $H_s$ deforms from $s-\epsilon$ to $s+\epsilon$, Whitney moves begin to form in the family $\mR_{s,t}$ for fixed values of $t$ near $t_0$. When $s$ equals $s_0$, the Whitney move occurs for the family $\mR_{s,t_0}$, and as $s$ passes $s_0$, the Whitney moves propagate along the fold, happening for the families of embeddings $\mR_{s, t_1}$ and $\mR_{s,t_2}$, where $t_1(s)<t_0<t_2(s)$. Focusing on the path $H_s$, for a fixed $s>s_0$ and increasing values of $t$ we find that $\mR_{s,t}$ first performs a Whitney move when $t = t_1$. As $t$ approaches $t_2$, $\mR_{s,t}$ performs a finger move as we now cross the fold \emph{in the opposite direction}. Thus, after crossing $s_0$, the path $H_s$ adds a Whitney move, followed by a finger move. Note that the local finger arcs given by the $\mF\mW$-germ will coincide for parameter values near $(s_0,t_0)$.

\subsubsection{Reordering}
We have discussed the local deformations that occur in the family $H_s$ as we encounter cusps and fold tangencies in the Cerf graphic. The last feature that appears in the Cerf graphic is normal crossings of the fold graphic. Crossings of the fold graphic occur when the order of a finger/Whitney move changes with another move. For example, a random path of embeddings $\alpha(t)$ may not be in finger-first position. However, it can be deformed to be in finger-first position and the 2-parameter family $\alpha_s(t)$ that realizes this deformation will have a Cerf graphic that contains only crossings of the fold graphic. At a crossing point $(s_0,t_0)$, there will be two finger/Whitney points between $\mR_{s_0,t_0}$ and $\mG$, and the local data given by $\mF\mW$ for these two points will be disjoint for all parameters near the crossing time.

\subsection{$\mF\mW$-vector fields} We want to now extend the local $\mF\mW$-germs to more globally defined data. To do this, we will construct a nonvanishing vector field on $[0,1]^2\times X\setminus \nu_\epsilon (S_h)$ contained in the subbundle $\langle \partial_t \rangle\oplus TX$ that projects to $\partial_t$, allowing us to translate the data along the $t$-direction. This will guarantee our extensions remain level-preserving. We also want to make sure that the boundary data on $[0,1]^2\times \mG$ and $\tr(H)$ remain on these submanifolds. Therefore, we require the vector field to be tangent along both submanifolds and agree on $\Sigma\setminus S_h$. The main technical point of this discussion is knowing that a vector field on $\Sigma\setminus S_h$ can be extended to a neighborhood $\nu_\Sigma : =\nu(\Sigma\setminus S_h)$ so that on both $\tr(H)\cap \nu_\Sigma$ and $[0,1]^2\times \mG\cap \nu_\Sigma$, it remains tangent to these submanifolds.

Throughout this section, we will denote $\Sigma_\epsilon:= \Sigma \setminus \nu_\epsilon(S_h)$ where $\nu_\epsilon(S_h)$ is an arbitrarily small open tubular neighborhood of $S_h$ in $\Sigma$. Furthermore, we assume that this neighborhood is the intersection of an open tubular neighborhood of $S_h$ in $[0,1]^2\times X$.

\begin{remark}\label{L: assumption on G}
    For the next lemma, we assume that $\mG \simeq \sqcup_i S^2$ for convenience.
\end{remark}

\begin{lemma}\label{L: sigma nbh model}
     There exists a parameterization $\phi: \Sigma_\epsilon\times B^2\times B^2\rightarrow \nu(\Sigma_\epsilon)$ of a tubular neighborhood of $\Sigma_\epsilon$ such that $\phi(\Sigma_\epsilon\times B^2\times 0)\subset [0,1]^2\times \mG$  and $\phi(\Sigma_\epsilon \times 0\times B^2)\subset \tr{H}$.
\end{lemma}
\begin{proof}
     Since $\Sigma_\epsilon \subset [0,1]^2\times \mG$ is transverse to $\{(s,t)\}\times \mG$ where they intersect, we can choose for each $p\in \Sigma_\epsilon$ a small open neighborhood $U(p)\subset \mG$. Since each component of $\Sigma_\epsilon$ is homotopy equivalent to either a point or a circle and by Remark \ref{L: assumption on G}, $\pi_1 \mG$ is simply connected, we can choose these open sets compatibly to yield an embedding $\widetilde{\phi}$ :$\Sigma_\epsilon \times \BR^2\rightarrow [0,1]^2\times \mG$. This gives a tubular neighborhood of $\Sigma_\epsilon$ in $[0,1]\times \mG$ where the normal fibers are tangent to $\{(s,t)\}\times \mG$. Now in $X$, $\mG$ has a trivial normal bundle. Thus, we can extend our embedding $\widetilde{\phi}$ to an embedding $\widetilde{\phi}:\Sigma_\epsilon \times \BR^2\times \BR^2\rightarrow [0,1]^2\times X$.

     Now, the points $p\in \Sigma_\epsilon$ are precisely the points such that $\mR_{s,t}$ is transverse to $\mG$ at $p$. Pulling $\tr(H)$ back by $\widetilde{\phi}$, we have a smoothly varying family of embeddings of $\BR^2$ intersecting $\BR^2\times 0$ transversely at $0$. Starting with $\pi: \Sigma_\epsilon \times B^2\times B^2\rightarrow B^2\times B^2$, we have a family of maps of $B^2\times B^2$ to itself. We linearize the family so that near $0\in \BR^2\times \BR^2$, the image of $\tr{H}$ looks like a family of linear subspaces. Afterwards, $\widetilde{\phi}$ gives a map $\Sigma_\epsilon \rightarrow G^{\pm}(2,4)$ where $G^\pm(2,4)$ is the space of oriented 2-planes in $\BR^4$. Again, each connected component of $\Sigma_\epsilon$ is homotopy equivalent to either a point or a circle. Since $\pi_1 G(2,4)=1$ for both components, we can isotope the family of planes to be constant over $\Sigma_\epsilon$. We choose the planes intersecting $0\times \BR^{\pm}$ depending on the intersection number of the component of $\Sigma_\epsilon$. Thus, we get a map $\psi:\Sigma_\epsilon \times B^2\times B^2\rightarrow \Sigma_\epsilon \times B^2\times B^2$ taking the image of $\tr{H}$ to $\Sigma_\epsilon \times 0\times B^2$. Therefore, $\widetilde{\phi}\circ\psi^{-1}$ gives the desired map.
\end{proof}

\begin{lemma}\label{L: v.f. extension lemma}
    Any section of $T \Sigma_\epsilon$ extends to $[0,1]^2\times X\setminus \nu_\epsilon(S_h)$ of $\Sigma_\epsilon$ preserving the tangency with both $\tr(H)$ and $[0,1]^2\times \mG$.
\end{lemma}
\begin{proof}
    Let $V: \Sigma_\epsilon\rightarrow T\Sigma_\epsilon$ be a section. We can trivially extend $V$ to a section $V':\Sigma_\epsilon\times \BR^2\times \BR^2\rightarrow T\Sigma_\epsilon \oplus T\BR^2\oplus T\BR^2$. Using $\phi:=\widetilde{\phi}\circ \psi^{-1}$ from the previous lemma, we can push forward $V'$ to the neighborhood $\nu(\Sigma_\epsilon)$. The tangency conditions are automatically fulfilled by noting that $\phi(\Sigma_\epsilon\times \BR^2\times 0)\subset [0,1]^2\times \mG$, $\phi(\Sigma_\epsilon\times 0\times \BR^2)\subset \tr(H)$. To obtain the extension to $[0,1]^2\times X \setminus \nu_\epsilon(S_h)$, we simply take a smooth bump function $\rho: \BR^2\times \BR^2\rightarrow \BR$ that is equal to $1$ on the unit 4-ball centered at $0$ and vanishes outside the 4-ball of radius 2 centered at $0$. Then $\phi_*(\rho V')$ can be extended by $0$ to the rest of $[0,1]^2\times X\setminus \nu_\epsilon(S_h)$. 
\end{proof}

Now we shall construct on $[0,1]\times X\setminus \nu_\epsilon(S_h)$ a nowhere vanishing vector field $V$ such that $(\pi_{[0,1]^2})_*(V) = \partial_t$. We do this by constructing $3$ independent vector fields on $[0,1]^2\times X \setminus \nu_\epsilon(S_h)$. Then, using appropriate smooth bump functions, piece them together.
\begin{construction}\label{construction of extension v.f.}
~

\begin{figure}
    \centering
    \labellist                             
            \hair 10pt
            \pinlabel \tiny{$supp(V_1)$} at 240 210
            \pinlabel {\color{orange}\tiny{$supp(V_2)$}} at 70 280
            \pinlabel {\color{blue}\tiny{$supp(V_3)$}} at 100 100
    \endlabellist
    
    \includegraphics[width=0.5\linewidth]{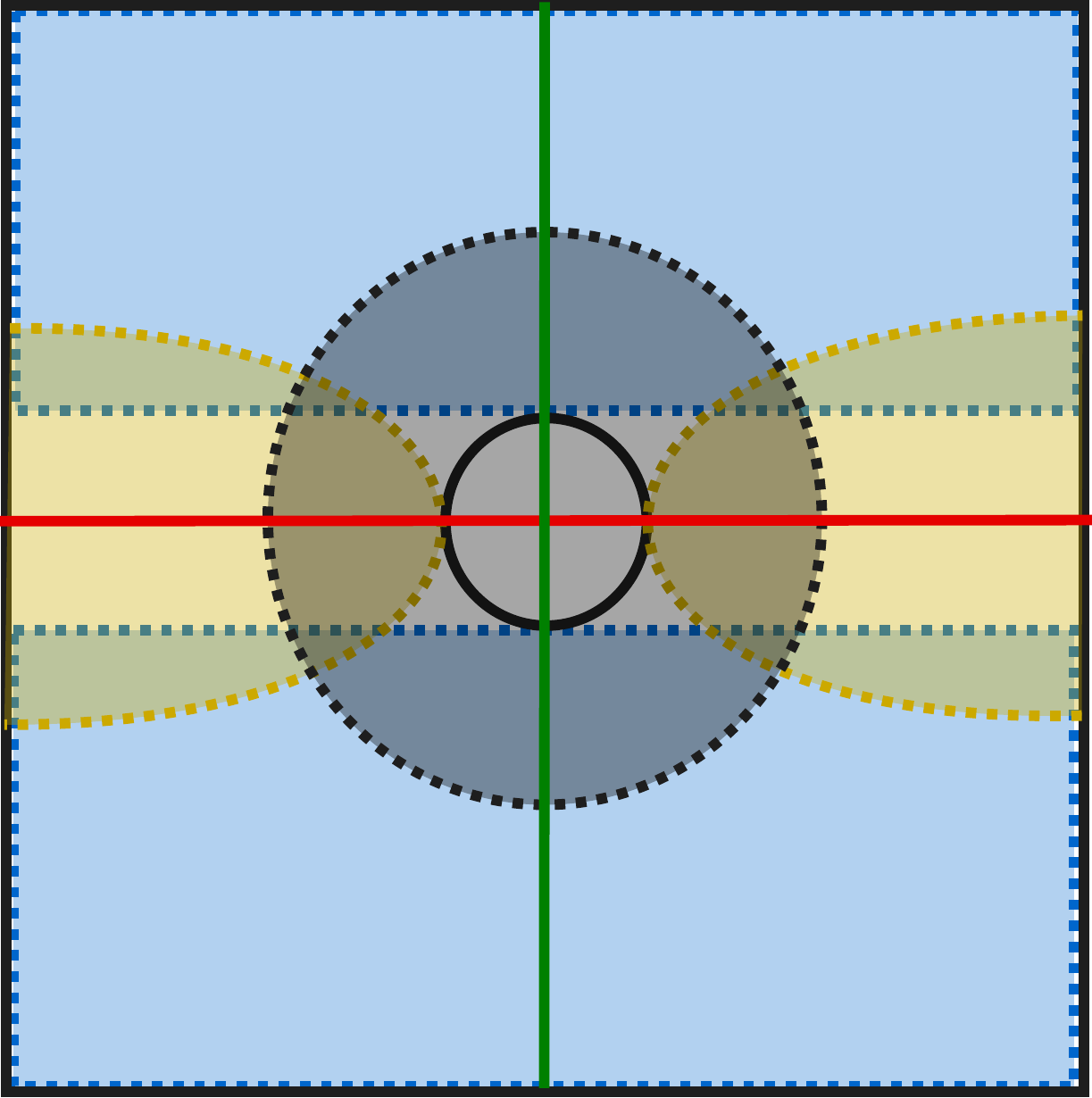}
    \caption{This figure illustrates the overlap of the supports for the vector fields constructed on $[0,1]^2\times X\setminus \nu(S_h)$ in Construction \ref{construction of extension v.f.} near the intersection of $\tr(H)$ (in red) and $[0,1]^2\times \mG$ (in green); in black we shown have the support of $V_1$, in orange the support of $V_2$ and in blue the support of $V_3$.}
    \label{F: v.f. neighborhoods}
\end{figure}
\begin{enumerate}
    \item On $\Sigma_\epsilon$, $h$ is by definition a submersion since the only place it fails has been removed. From this fact, we build a section $V_1$ of $T\Sigma_\epsilon$ such that $h_*(V_1)= \lambda(p) \partial_t$ where $\lambda(p)>0$ is strictly positive. This can be done locally in charts. Now Lemma \ref{L: v.f. extension lemma} allows us to extend this to a neighborhood in $[0,1]^2\times X$ and that when restricted to either $\tr(H)$ or $[0,1]\times \mG$, remains tangent to $\{s\}\times[0,1]\times \mG$ and $\tr(H)\cap \{s\}\times I\times X$ on the appropriate subsets where defined. This is because $h$ is the restriction of the projection. We denote the extended v.f. by $V_1$.

    \item It is a classical fact that any vector field on a submanifold (with or without boundary) can be extended to a neighborhood using the tubular neighborhood theorem so that the support is contained within the neighborhood. So we first consider the intersection of $\tr(H)\cap \phi$, where $\phi$ is the parameterization from Lemma \ref{L: sigma nbh model}. We choose a smooth bump function $\rho_2$ on $\tr (H)$ with the support of $\rho_2$ equal to $\tr(H)\setminus \phi(\Sigma \times 0\times B_1^2)$. We then let $V_2 = \rho_2 \tr(H)_* \partial_t$ be the smooth function on $\tr(H)$ that we wish to extend to a tubular neighborhood. We make sure the tubular neighborhood of $\tr(H)$ intersects $[0,1]^2\times \mG$ only in the image of $\phi$, and that $V_2$, when extended to the neighborhood is equal to $0$ on $\phi(\Sigma_\epsilon\times \BR^2\times B^2_1)$.

    \item For the third v.f., we start with the globally defined $\partial_t$ on $X$. We choose a smooth bump function $\rho_3$ on $[0,1]^2\times X$ with support equal to the complement of a closed neighborhood of $\tr$ contained within $\nu(\tr)$. $V_3 = \rho_3 \partial_t$.
\end{enumerate}

Figure \ref{F: v.f. neighborhoods} summarizes the overlap between the supports for $V_1$, $V_2$ and $V_3$. We can consider the sum $\widetilde{V} = V_1+V_2+V_3$ over $[0,1]\times X\setminus \nu_\epsilon (S_h)$. By construction, each vector field is a section of $\langle \partial_t \rangle \oplus TX|_{[0,1]^2\times X\setminus \nu(S_h)}$ and is nowhere vanishing. If we restrict to $\tr(H)$, we have that $\widetilde{V} = V_1+V_2$ and since both are tangent to $\tr(H)\cap \{s\}\times [0,1]\times X$ by construction, so is the sum. When we restrict to $[0,1]^2\times \mG$, we have that $\widetilde{V} = V_1+V_3$ and again this sum is tangent to $\{s\}\times [0,1]\times \mG$ since both are also. 

Finally, we note that by the construction, $\widetilde{V}$ takes the form $\lambda\partial_t + \xi$ where $\lambda>0$ is a smooth, strictly positive function. We use $\lambda$ to normalize $\widetilde{V}$ and define $V = \lambda^{-1} \widetilde{V}$.
\end{construction}

\begin{definition}
    We call the vector field $V$ resulting from Construction \ref{construction of extension v.f.} an $\mF\mW$ vector field if the domain of $V$ contains $\mF\mW(S_h\times \partial C(\mF\mW))$. 
\end{definition}

\begin{remark}
Construction \ref{construction of extension v.f.} gives a way to construct a nowhere vanishing vector field in the complement of the singular set $\nu_\epsilon(S_h)$ and this vector field will be used to flow the local $\mF\mW$-germ data to different $t$-levels. As such, we will not be able to flow all the data to every $t$-level. In particular, if we flow a finger disc into a $t$-level containing a Whitney disc and these two discs generically intersect, then the finger disc will not extend past the Whitney point for the corresponding Whitney disc. In what follows, we will implicitly use an $\mF\mW$-vector field to extend the local data without always saying so. Moreover, the existence of a $\mF\mW$-vector field for a single path of embeddings follows by taking a trivial product and then restricting the v.f. to the path. 
\end{remark}

\subsection{Deformation of families of embeddings}\label{S: deformations}
In this section, we describe the process of extending a $\mF\mW$-germ for 1- and 2-parameter families. Using the extensions, we can then deform a given family of embeddings to rearrange when the finger/Whitney move occurs. From this we recover Quinn's result that every path of embeddings has a finger first representative \cite{Qu}. The main theorem of this section is our 2-parameter ordering theorem.

\begin{definition}\label{D: ff ordering}
    We say that a path $H:[0,1]\rightarrow \Emb(\mR, X)$ has a \emph{finger-first ordering} if there exists a time $t_*\in [0,1]$ such that all finger moves occur in $[0,t_*)$ and all Whitney moves happen in $(t_*, 1]$. Note that this is subtly different than the previously defined ``finger-first position", in which all finger moves happen at the same time and all Whitney moves happen at the same time, with fingers before Whitneys.
\end{definition}

\begin{theorem}\label{T: 2-par ordering}
    Let $\mG$, $\mR$, $X$ be as above and let $H:[0,1]^2\rightarrow \Emb(\mR, X)$ be a generic relative homotopy between two paths $\alpha_0$ and $\alpha_1$ that have a finger-first ordering. Then $H$ is homotopic relative to $\partial [0,1]^2$ to a generic homotopy $\bar{H}$ that satisfies the following: \begin{enumerate}
        \item There exist finite disjoint sets $\{s_1,...,s_k\}$, $\{[a_1,b_1],...,[a_l,b_l]\}$ of the homotopy parameter such that in the complement of both sets, $\bar{H}_s$ has finger-first ordering.
        \item The parameters $\{s_1,...,s_k\}$ correspond to a birth/death or cusp point (see Section \ref{S: cusps} and Cases \ref{Birth} and \ref{Death} of Section \ref{S: b/d and saddle}).
        \item For $s\in (a_i,b_i)$, $H_s$ is in finger first position except for a single pair of finger/Whitney moves. For this pair, the Whitney move occurs before the finger move, and both moves happen after all after all other finger moves and before all other Whitney moves. When $s= a_i$ or $b_i$, $H_s$ contains a saddle point or a crossing (but not both). 
    \end{enumerate}
\end{theorem}

We will call a 2-parameter family satisfying 1-3 of Theorem \ref{T: 2-par ordering} an \emph{ordered homotopy}.

\subsubsection{Ordering of paths and finger-first position.}\label{S: paths ff position} We recount Quinn's proof that every path of embeddings can be deformed to be in finger first position using the language we have developed thus far.

Let $\alpha$ be a path of embeddings and let $\mF\mW(\alpha)$ be an $\mF\mW$-germ for $\alpha$ embedded in $[0,1]\times X$. Recall that the embedding $\mF\mW(\alpha)$ explicitly codifies the observation that just before a finger move, there is an embedded arc that locally describes the finger move, and just after the finger move, there are guiding discs that reverse it. Finally, we recall that each finger point corresponds to a index 0 critical point and a Whitney point to an index 1 critical point of the function $h:L\rightarrow [0,1]$ where $L=\alpha^{-1}(\mG)$ and $h$ is the restriction to $L$ of the projection onto the t-coordinate.

\begin{construction}[Extension of local arcs.]\label{Constr. arc-extensions}
 For a given index 0 critical point $p_i\in S_h$, let $t_i(p)$ be the critical value. Our goal is to extend the family of arcs $\mF\mW(p\times B^1\times (1/2,1]$ down so that $\mF\mW(p\times B^1\times \{1\})$ is embedded in the level set below $t= 1/4$. Let us denote the image of $\mF\mW(p\times B^1\times \{1\})$ by $f_i^{arc}$. Using a $\mF\mW$-vector field, we flow the arc downward along $-V$. The flow down can continue, as long as no flow line hits any other $\mF\mW$ data. Generically in a 4-manifold, an arc will always be disjoint from a disc or another arc, Thus there is no obstruction to flowing $f_i^{arc}$ down to below $t=1/4$. The same argument shows that we can flow all the Whitney arcs up past $t=3/4$ along the vector field $V$.

 Now we get an extended germ by taking $\mF\mW(p_i\times C(B^1))$ and adding the union of the flow lines for $-V$ (or $V$) starting in $f_i^{arc}$ ($w_j^{arc}$ for the Whitney arcs). Since $-V$ (or $V$) is tangent to both $\tr(H)$ and $[0,1]\times \mG$, the boundary points for $f_i^{arc}$ remain on $\{t\}\times \mR_t$ and $\{t\}\times\mG$. Since the flow is level-preserving, the image of $f_i^{arc}$ under the flow is still embedded in $t$-level sets. Hence, the extension is a level-preserving embedding of $C(B^1)$ into $[0,1]\times X$. Again, this holds for all other finger arcs and Whitney arcs as well.
\end{construction}

By extending the local finger arcs as above, we get an extension of the embeddings of $\{p_i\}\times C(B^1)$ for each finger move. Let $t_*\in [0,1]$ be the lowest critical value that contains an index 1 critical point for $h$. Then for each $p\in S_h$ with $t(p)\geq t_*$, we use the segment of $p\times C(B^1)$ contained between the level sets $\{t_*-\epsilon\}\times X$ and $\{t(p)\}\times *$ and deform the trace $\operatorname{tr}(\alpha)$ by performing a family of finger moves along the extended arcs that make up $C(B^1)$ in these levels. As each arc is embedded in a level set, this gives a level-preserving deformation of $\operatorname{tr}(\alpha)$, and hence creates a homotopy between paths of embeddings. After performing this deformation, all the index 0 critical points of the new $h$ will be below the index 1. Furthermore, since the arcs can be extended past the finger discs for other finger moves, we can arrange so that the finger moves all happen at the same time. Finally, the same argument shows that we could have raised the Whitney arcs up past any finger/Whitney move, and hence arrange that all the Whitney moves occur simultaneously. Hence, the new path is in finger-first position.

\begin{remark}\label{R: f order to f position}
    All that was needed to build the homotopy from $\alpha$ to a path of embeddings $\tilde{\alpha}$ that has finger first ordering was that the finger arcs could be extended past any Whitney arc. To go from a finger-first ordering to being in finger-first position, we need that the finger arcs could be extended past any finger disc that was below it.
\end{remark}

\subsubsection{Two parameter ordering.}
Our goal now is to perform the same analysis and build a deformation of a given homotopy $H:[0,1]^2 \rightarrow \Emb(\mR, X)$ between two paths $H(0,t)$ and $H(1,t)$ in the first position of the finger. Recall that a $\mF\mW$-germ for a 2-parameter family $H$ is an embedding $\mF\mW(H):S_h\times C(\mF\mW)\rightarrow [0,1]^2\times X$ such that for a fold point $p\in S_h\subset B^k\times \mR$, $\mF\mW(p\times C(\mF\mW)$ embeds into the subset $l_p\times X$ where $l_p = (s(p),t(p))+\lambda w_{s(p),t(p)}$, $-\epsilon(p)\leq \lambda \leq \epsilon(p)$ is the small line segment of parameters along the vector $w_{s(p),t(p)}$, and this data is compatibly extended near cusp points (see Construction \ref{Construction 2-par local models} and Construction \ref{constr: FW germs near cusps}). 

\begin{lemma}\label{L: GP for arcs}
    Let $\Sigma^2$ be an immersed surface in a 4-manifold $X^4$. Let $f,g:[0,1]\times B^1\rightarrow [0,1]\times X^4$ be two level-preserving embeddings such that $\partial f$ and $\partial g$ are contained in $\Sigma^2 \times [0,1]$ with interiors in the complement and boundaries disjoint from the double points. Then, generically, $f$ and $g$ are disjoint.
    \qed
\end{lemma}

The proof of the lemma is a standard general position argument. Using Lemma \ref{L: GP for arcs} repeatedly, we prove the following.

\begin{lemma}\label{L: local deformations}
    The deformations of the Cerf graphics shown in Figure~\ref{F: Deformation of graphics}) are performed by level-preserving deformations of the trace for 2-parameter families of embeddings with the given Cerf graphic.
    \begin{figure}
        \centering
        \labellist                             
            \hair 10pt
            \pinlabel \tiny$1a.$ at 825 1085
            \pinlabel \tiny$1b.$ at 2925 1085
            \pinlabel \tiny$2.$ at 825 630
            \pinlabel \tiny$3.$ at 2925 640
            \pinlabel \tiny$4.$ at 825 245
            \pinlabel \tiny$5.$ at 2925 210
        \endlabellist
        \includegraphics[width=0.8\linewidth]{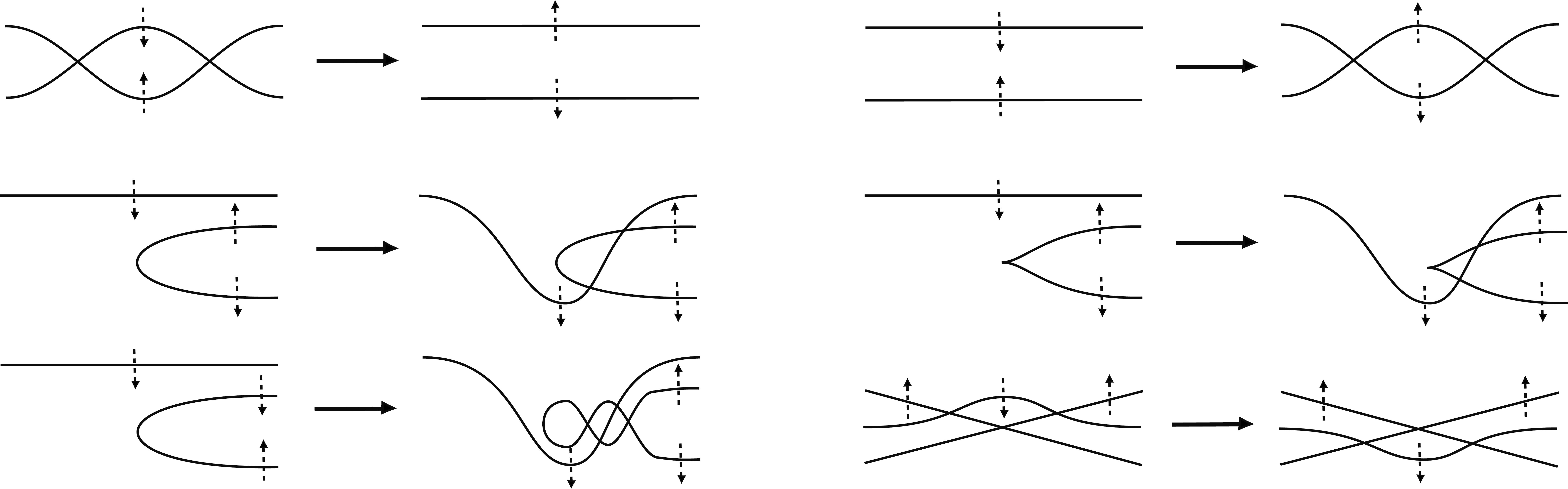}
        \caption{ The following shows the needed local deformations of the Cerf graphic for a 2-parameter family of embeddings. The dotted arrows indicate the coorientation of the folds that points in the direction of decreasing intersection points between $\mR$ and $\mG$.} 
        \label{F: Deformation of graphics}
    \end{figure}
\end{lemma}
\begin{proof}~
The following argument proves the lemma for deformations $1a.$, $1b.$ and $5$ Suppose we have a 2-parameter family $H$ with a Cerf graphic given by 1a. 1b. in Figure \ref{F: Deformation of graphics}. Let $S_0$ denote the index 0 fold and $S_1$ the index 1 fold. Then $\mF\mW(H)(S_0 \times C(\mF\mW))$ embeds a $C(B^1)$ in each $\{s\}\times [0,1]\times X$ with each $B^1\times \{x\}$ embedded in $\{(s,t)\}\times X$ and $t\leq t(p_0)$. Similarly $\mF\mW$ embeds a $C(B^1)$ in $\{s\} \times [0,1]\times X$ for the index 1 fold, and each $B^1\times \{x\}$ is embedded in $\{(s,t)\}\times X$ with $t> t(p_1)$. Let $f^{arc}_s$ be the image of $S_0\times (B^1\times \{1\})\subset S_0\times C(B^1)$ in $\{s\}\times [0,1]\times X$ and $w^{arc}_s$ the corresponding arc for $S_1\times B^1\times \{1\}$ in $\{s\}\times [0,1] \times X$.

For each $s \in [0,1]$, we drop $f^{arc}_s$ down just as we did in the 1-dimensional case to the level set $\{(s,t)\}\times X$ that contains $w_s^{arc}$. The $t$-parameter for this level set set will smoothly vary with $s$. We denote by $\tilde{f}_{s}^{arc}$ the image of $f_{s}^{arc}$ in $\{(s,t(s))\}\times X$. Working in the 5-manifold $\bigcup_{s\in [0,1]}\{(s,t(s))\times X$, we have the two level preserving embeddings $\cup_{s}\tilde{f}^{arc}_s$ and $\cup_s w^{arc}_s$. Lemma \ref{L: GP for arcs} implies that the family $\cup_{s}\tilde{f}^{arc}_s$ is disjoint from $\cup_s w^{arc}_s$. Therefore, the entire family can be lowered past the family of Whitney moves. Hence we can extending the embedding of $S_0\times C(B^1)$ past the embedding of $S_1\times C(\mF\mW)$ in $[0,1]^2\times X$. Finally, we can use the extended $S_0\times C(B^1)$ embedding to deform $\operatorname{tr}(H)$ by doing finger moves along the appropriate subset of finger arcs making up the extended family of arcs $\mF\mW(S_0\times C(B^1))$.

Deformation 2 again relies on Lemma \ref{L: GP for arcs} and noting that the family of finger arcs generically miss the Whitney arcs that lie above the vertical tangency at $s_0$. Furthermore, as the upper index 0 finger arcs can be extended below the index 1 fold, they will miss the corresponding local Whitney discs. As these discs agree with the Finger discs near the vertical tangency, the extended finger arcs can be lowered past the lower index 0 fold \emph{for $s$ values near $s_0$}.

A similar story holds for Deformation 3. In order to move the family of finger moves past the cusp, we need to make sure that when we extend the finger arcs below the index 1 fold, they miss the corresponding finger disc for the index. The local model (Figure \ref{F: x cubed}) for the cusp implies that this is the case for $(s,t)$ near the cusp point. 

Deformation 4 is achieved in two steps. First, we note that the finger arc contained in $\{s_0\}\times [0,1] \times X$ will generically be disjoint from the Finger disc for the index 0 fold and that this will be the case for some subset $[0, s_1]$ with $s_1>s_0$. Let $s' = \frac{s_1+s_0}{2}$. Now we can apply the Deformation 1a to the family over $[s',1]$. Over $[0,s']$, the top most finger arcs can be extended below the entire graphic, allowing one to deform the family until the top arc of finger moves below the saddle over $[\epsilon,s'+\epsilon]$. 
\end{proof}

\begin{remark}
    Note that turning each graphic upside down changes the $t$-direction. This swaps the indices, turning all index 0's into index 1 and vise versa. This does not change the extensions, only the direction of extending (up for down). Therefore, we find that the corresponding argument for raising a family of Whitney moves up as well.
\end{remark}

\begin{remark}
    One should note that a Reidemeister III type move for passing an index 0 fold below a crossing of index 1 folds is needed as well. The proof of this is the same general position argument as the Move 1a.
\end{remark}

\begin{proof}[Proof of Theorem \ref{T: 2-par ordering}.]
Let $H$ be a homotopy between two finger first paths of embeddings. Then using the local deformations in Lemma \ref{L: local deformations}, we can systematically lower all the index 0 folds below the index 1 folds, cusps, birth/death, and saddles and raise the index 1 folds above the index 0 folds, cusps, birth/death, and saddles. Once complete, the new $\tilde{H}$ will be have a finger first ordering except at the parameter values $s$ when $\{s\}\times [0,1]$ contains a birth/death, cusp, or small intervals containing a saddle point. Just after the birth/death and cusp points or the saddle intervals, $H_s$ will again have a finger-first ordering.
\end{proof}
 
\subsection{Moves between finger/Whitney systems.}
As a result of the preceding subsections, we can now show that paths of embeddings can be described by finger-Whitney systems, and we have shown explicitly where the finger and Whitney discs come from. In this final subsection, we state this precisely and then we give a complete set of moves for how finger/Whitney systems evolve under ordered homotopies. 

\begin{theorem}\label{T: fw system moves}
    Given two representatives $\alpha_0$ and $\alpha_1$ of $[\alpha]\in \pi_1(\Emb(\mR, X), \mR_0)$ in finger first position with $(\mF_i, \mW_i)$ finger/Whitney system for $\alpha_i$ Then $(\mF_0, \mW_0)$ differs from $(\mF_1, \mW_1)$ by isotopies of the data and a finite sequence of $\mR$- and $\mG$-disc slides, finger/Whitney sphere slides, birth/death moves, $x^3$-moves, and saddle moves, i.e., by a finite set of $\mF\mW$-moves.
\end{theorem}

To prove Theorem \ref{T: fw system moves}, we start with an ordered homotopy $H(s,t)$ between two paths $H_0(t)$ and $H_1(t)$ and try to put $H_s(t)$ into finger-first position for all $s$. This will not be possible for all $s\in [0,1]$ as $H_s(t)$ does not necessarily have a finger-first ordering for all $s$ (see (3) of Theorem \ref{T: 2-par ordering}). Instead, we will show that crossing the parameters $s'$ for which $H_{s'}(t)$ cannot be put into finger-first position changes the finger/Whitney system by one of the $\mF\mW$-moves listed in Theorem \ref{T: fw system moves}.  

\begin{construction}[Construction of a finger/Whitney system]\label{construction: fw systems from a path}
 Let $\alpha_t$ be a path of embeddings having a finger-first ordering with all finger moves happening below $t = 1/2$ and all Whitney moves above. Let $V$ be a $\mF\mW$-vector field for $\alpha_t$. We will assume that for every finger point $p_i\in S_h$, the flow lines starting on $f_{i} = \mF\mW(p_i\times B^2\times 1)$ extends past all other finger moves, that is, the flow lines are disjoint from the local finger arcs embedded in the level sets above. Similarly, we assume that for every Whitney point $p\in S_h$, the flow lines for $-V$ that begin on $w_{i} = \mF\mW(p_i\times B^2\times 1)$ are disjoint from the Whitney arcs that are embedded in level sets below the corresponding point. Now extend all the local finger discs up along $V$ to the level $\{1/2\}\times X$ and denote the image by $\mF = \{f_1,..., f_n\}$. Similarly, extend all Whitney discs down along $-V$ to $\{1/2\}\times X$ and denote their image by $\mW = \{w_1,...,w_n\}$. The tuple $(\{1/2\}\times \mR_{1/2}, \{1/2\}\times\mG,  \mF, \mW)$ is the finger/Whitney system determined by $\alpha$ and $V$.
\end{construction}

\begin{remark}\label{R: fw systems from a path}
    When a path $\alpha_t$ is in finger first position, the assumption that the flow lines extend past other finger/Whitney data is trivially satisfied. Moreover, the choice of a $\mF\mW$-vector field will change the extensions by an isotopy. Therefore, when a path is in finger first postion, the associated finger/Whitney system $(\mR_{1/2},\mG, \mF,\mW)$ is well defined up to isotopies of the disc systems $\mF$ and $\mG$.
\end{remark}

As discussed in Remark \ref{R: f order to f position}, to go from a finger-first ordering to finger-first position, we must be able to extend the local finger arcs that are above other finger moves, below the local finger discs and this is possible when the arcs have no flow lines to the local finger discs. For a path of embeddings, this amounts to knowing two things: that points and curves on a surface are generically disjoint, and in a 4-manifold, arcs and discs are generically disjoint. For the 2-parameter family, we need to be able to extend a path of arcs past a path of discs. The general position statement for the data is the following: 
\begin{lemma}\label{L: GP for arcs and disc}
    Let $\Sigma^2$ be an immersed surface in a 4-manifold $X^4$ and let $f:[0,1]\times B^1\rightarrow [0,1]\times X^4$ and $g:[0,1]\times B^2\rightarrow [0,1]\times X^4$ be level-preserving embeddings such that for both embeddings, $f\cap [0,1]\times \Sigma^2 = \partial f$ and $g\cap [0,1]\times \Sigma^2 = \partial g$, and $\partial f$ is disjoint from the double points of $\Sigma^2$. Then, generically, both $\partial f\cap \partial g$ and $f\setminus \partial f\cap g\setminus \partial g$ will consist of finitely many points.
    \qed
\end{lemma}

There are two types of generic intersections that can occur between finger arcs and finger discs when extending the $\mF\mW$-germs; boundary intersections and interior intersections. Moreover, these intersections will correspond to distinct changes in the finger/Whitney data. They are the analog of a handle slide that occurs in generic 1-parameter families of gradient like vector fields. To see this, we extend the local finger discs for the lower family up and the local finger arcs for the upper family down to a common set level $[0,1]\times \{t_*\}\times X$ between them. Here we apply Lemma \ref{L: GP for arcs and disc}. Then the family of finger arcs will intersect the family of finger discs at isolated points at isolated times. We have indicated this in the graphic with the purple arrow indicating a boundary intersection and a blue arrow an interior intersection.
\subsubsection{Boundary intersection.} 
\begin{figure}
    \centering
    \includegraphics[width=0.8\linewidth]{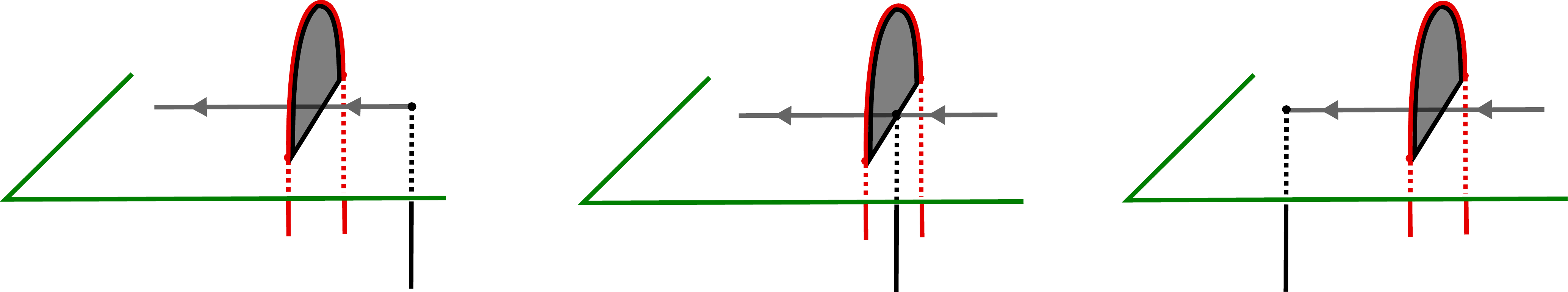}
    \caption{Shown is the neighborhood for the family of local finger discs and the family extended finger arc for a $\mG$ boundary intersection point happening at $s=s_0$.}
    \label{F: bdy_intersections}
\end{figure}
\begin{figure}
    \centering
    \includegraphics[width=0.7\linewidth]{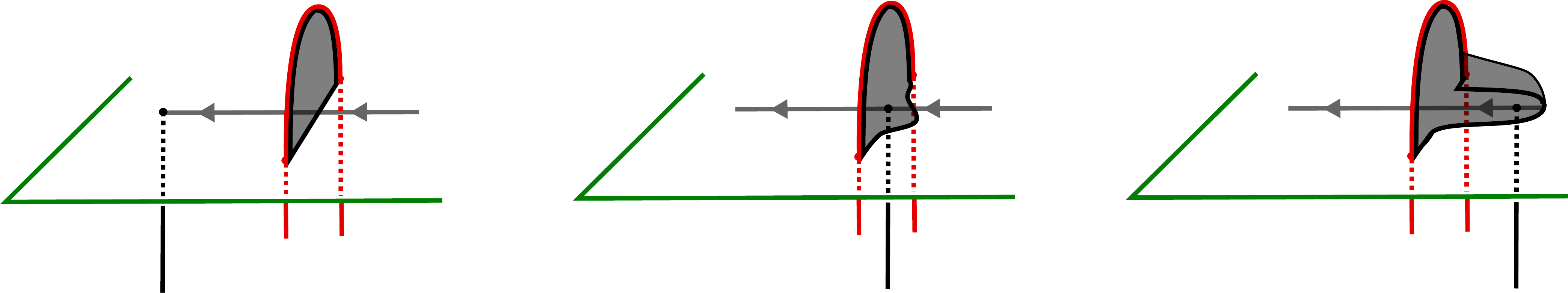}
    \caption{Shown here is the isotopy in the neighborhood of the local finger disc given by dragging the extended finger arc at $s>s_0$ back to $s<s_0$.}
    \label{F: Bdyint_arc_isotopy}
\end{figure}
\begin{figure}
    \centering
    \includegraphics[width=0.7\linewidth]{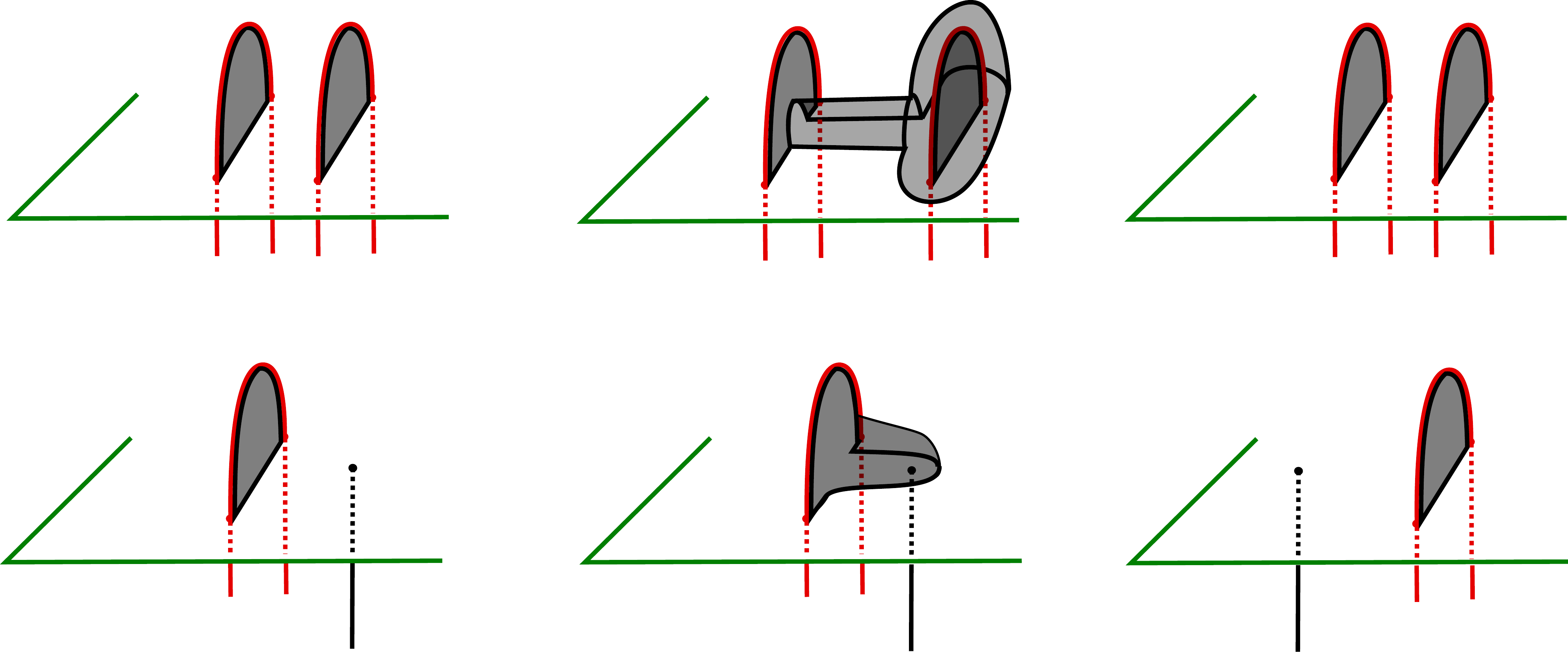}
    \caption{Shown here is the change in the configuration of the finger disc after passing a boundary intersection point. The top row of pictures shows the changes to the pair of discs $f_1$ and $f_2$ in the level set above both critical points while the lower row shows the changes in a level set in between the two critical points.}
    \label{F: Bdyint_implies_discslides}
\end{figure}
Let us analyze when a family of finger arcs intersects a family of finger discs along the boundary. Let $s_0$ denote the parameter when the intersection occurs. Note that the local finger discs have half of their boundary on $\mR$ and the other half on $\mG$. 
\begin{definition}
    If the boundary of the arc crosses the boundary of the disc on $\mG$, we call it a \emph{$\mG~ \partial$-crossing} and a \emph{$\mR ~\partial$-crossing} if the arc crosses the boundary of the disc on $\mR$.
\end{definition}  
Figure \ref{F: bdy_intersections} shows the local model for a $\mG~ \partial$ crossing. Without loss of generality, we can assume that the boundary of the arc moves right to left as $s$ increases and that we have arranged the tangent vector of the arc at the boundary to be constant and tangent to the disc at $s_0$. To compare the data before with the data after we take the finger arc for $s> s_0$ and isotope it back to $s'< s_0$, moving the data of the finger disc with the isotopy as shown in Figure \ref{F: Bdyint_arc_isotopy}. The effect on the finger discs is to perform a $\mG$-disc slide, as shown in Figure \ref{F: Bdyint_implies_discslides}. Thus, we have the following:

\begin{lemma}
    Let $H_s$ be a homotopy between paths of embeddings with all the Whitney moves occuring at the same time $t = 3/4$ for all $s$ and all the finger moves occuring at the same time $t = 1/4$ except for $s\in [s_0 -\epsilon, s_0 + \epsilon]$. Over this interval, there is a single finger move that occurs after all others and the extension of the finger arcs down has a single $\mR~\partial$-crossing ($\mG~\partial$-crossing) with a local finger disc $f$. Label the finger discs at $s_0-\epsilon$ $\{f_1,f_2,...,f_k\}$ such that $f_1$ is the finger disc for the upper finger move, and $f_2$ is the disc that $f_1$'s finger arc intersects at $s=s_0$. Then $(\mF, \mW)_1$ is obtained from $(\mF,\mW)_0$ by performing an $\mR$- ($\mG-$) disc slide of $f_2$ over $f_1$.\qed
\end{lemma}
\begin{remark}
    The same is true of the Whitney discs for the Whitney system. If when extending the local arcs up, they intersect the local Whitney discs on the boundary, then the data Whitney system on one side of the intersection differs from the data on the other side by an $\mR$- or $\mG$-disc slide.
\end{remark}
\subsubsection{Interior intersections.}
\begin{figure}
    \centering
    \includegraphics[width=0.8\linewidth]{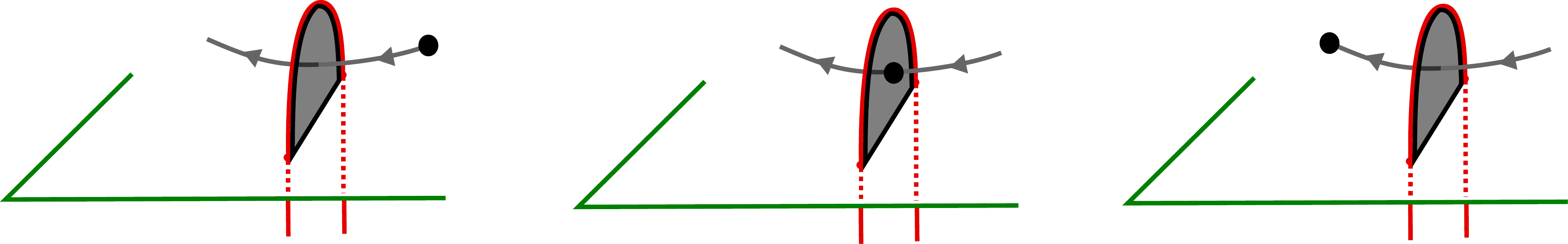}
    \caption{Shown is the neighborhood for the local finger disc (in blue) and the intersections of this neighborhood with the extended family of finger arcs (in black)}
    \label{F: interior intersections}
\end{figure}
\begin{figure}
    \centering
    \includegraphics[width=0.7\linewidth]{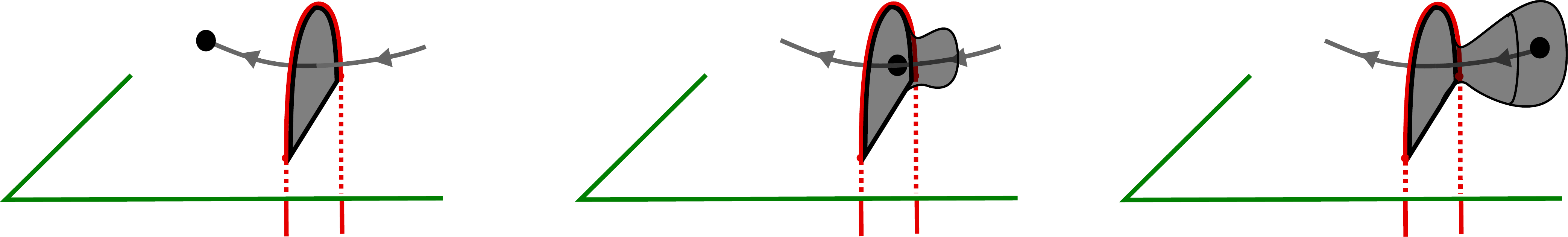}
    \caption{Here we depict the affect of isotopying the extended family of arcs at $s>s_0$ back to its position at $s<s_0$.}
    \label{F: II isotopy}
\end{figure}
\begin{figure}
    \centering
    \includegraphics[width=0.7\linewidth]{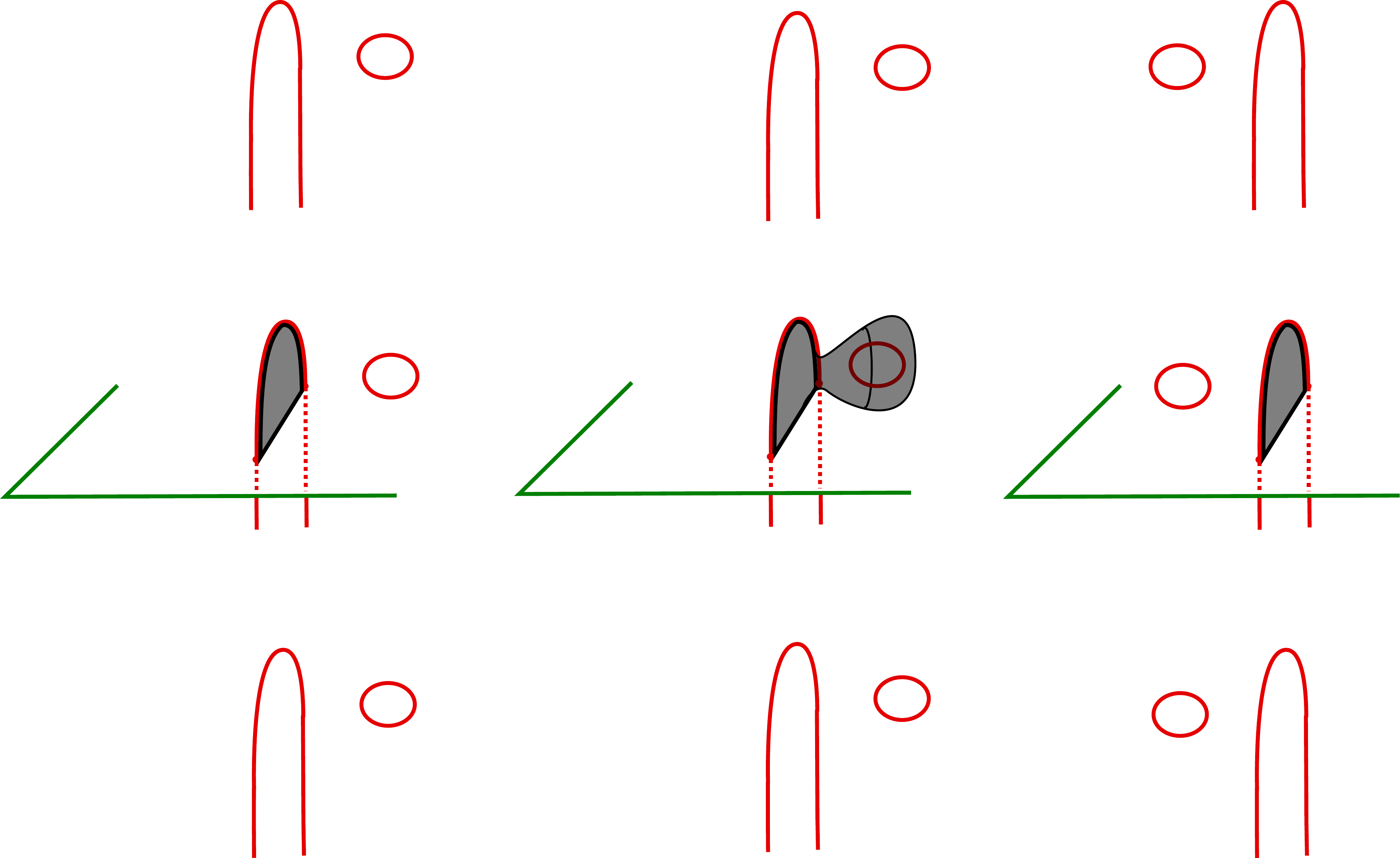}
    \caption{Each column of figures represents the 4-dimensional neighborhood of the lower finger disc in the level set above both critical points. The left column of figures is shows the local configuration of $\mR_{s,t_*}$ and the finger disc $f_2$ for $s<s_0$, while the right column shows the local configuration for $s>s_0$. The middle column and right column differ by a local isotopy. To get the middle column from the left column, one tubes $f_2$ to the 2-sphere separating the red circle in the middle figure of the middle column. This sphere is the finger sphere for $f_1$. }
    \label{F: II implies Sslide}
\end{figure}
Our goal now is to analyze the effect that interior intersections have on our finger/Whitney systems.
\begin{remark}
    Associated with a finger/Whitney move is a sphere called the Whitney sphere (see \cite{GGHKP}[Section 3.1]). To emphasis the difference between the data associated with a finger/Whitney move, we will call the Whitney sphere for a finger move a \emph{finger sphere} and reserving the name \emph{Whitney sphere} for the sphere corresponding to a Whitney move. 
\end{remark}
\begin{definition}\label{D: sphereslide}    
    Let $f_1$ and $f_2$ be two framed finger (Whitney) discs pairing off intersections between  $\mR$ and $\mG$. A \emph{sphere slide} of $f_1$ over $f_2$ is the operation of tubing $f_2$ to the finger sphere (Whitney sphere) of $f_1$. 
\end{definition}
\begin{remark}
    Tubing together surfaces in a 4-manifold is well defined up to the choice of tubing arc and homotopic arcs produce isotopic surfaces since homotopy implies isotopy for 1-manifolds in 4-manifolds. 
\end{remark}
Suppose that we have a family of finger arcs that has a single interior intersection with a family of finger disc at $s= s_0$. We will work in the local model for the lower finger disc as before. Generally, an arc will intersect a 3-manifold in finitely many points. Without loss of generality, we can ignore all other points of intersection of the finger arc with the local 3-ball containing the finger disc except for the point of intersection. For $s$ sufficiently close to $s_0$ we see in the local model as $s$ increases from $s<s_0$ to $s>s_0$, the family of finger arcs as a point that crashes through the disc as in Figure \ref{F: interior intersections}. To compare the data after with the data before, we isotope the arc at $s>s_0$ back, dragging the finger disc backwards as shown in Figure \ref{F: Bdyint_arc_isotopy}. The change in the finger disc is then to tube the original finger disc to the normal sphere for the arc as shown in Figure \ref{F: II implies Sslide}. Therefore, we have proved the following:

\begin{lemma}
    Let $H_s$ be a generic homotopy between paths of embeddings with all the Whitney moves occuring at the same time $t = 3/4$ for all $s$ and all the finger moves occuring at the same time $t = 1/4$ except for $s\in [s_0 -\epsilon, s_0 + \epsilon]$. Over this interval, there is a single finger move that occurs after all others and the extension of the finger arcs down has a single interior intersection with a local finger disc $f$.
    Label the finger discs at $s_0-\epsilon$ $\{f_1,f_2,...,f_k\}$ such that $f_1$ is the finger disc for the upper finger move, and $f_2$ is the disc that $f_1$'s finger arc intersects at $s=s_0$. Then $(\mF, \mW)_1$ is obtained from $(\mF,\mW)_0$ by performing a sphere slide of $f_1$ over $f_2$.\qed
\end{lemma}

\subsubsection{The birth/death move.}
As discussed previously, when the upper finger arcs (Whitney arcs) can be extended below (above) a lower (higher) finger move (Whitney move),
we are able to get a finger/Whitney system, and so long as there are no birth/death, cusps, or saddles pts, this is possible except at finitely many values of $s\in [0,1]$. Birth/death points also occur at finitely many values of $s$ and for generic extensions of the finger/Whitney arcs, the birth/death parameters will be distinct from the boundary and interior intersection parameters. Hence we can assume that just before a birth/death or cusp parameter and just after, we have finger/Whitney systems.

\begin{remark}\label{R: distinct cusp parameters.}
    The same general position argument above also shows that cusp points occur at distinct parameter values. Hence we can assume that just before and just after, we have finger/Whitney systems.
\end{remark}

We shall assume that the birth/death point happens when $s=s_0$ and when $t=t_0$. We consider the finger/Whitney system $(\mR_{t_0}, \mG, \mF,\mW)_s$ for $s<s_0$ and for $s>s_0$. Without loss of generality, we shall assume that $(s_0, t_0)$ is a birth point. Now if we fix $t=t_0$ and let $s$ vary from $s<s_0$ to $s>s_0$, we see a finger move occur between $\mR_{s,t_0}$ and $\mG$. In fact, $\mF\mW(H)$ embeds the local finger arc for this finger move in level sets $\{(s,t_0)\}\times X$. Generically this arc in $\{(s,t_0)\}\times X$ is completely disjoint from all the finger/Whitney discs for the finger/Whitney system $(\mR_{t_0}, \mG, \mF,\mW)_s$ $s<s_0$. Therefore, this data can be extended past the birth/death point compatibly. Now for a fixed $s<s_0$ and $t$ sufficiently close to $t_0$, the path $H_s(t)$ performs an isotopy that pushes $\mR_{s,t}$ along this arc part of the way until $t=t_0$, then retracts backwards. As $s$ increases, $\mR_{s,t}$ pushes further along the arc and then back until at $s=s_0$, $\mR_{s_0,t}$ pushes along until at $t=t_0$, $\mR_{s_0,t_0}$ touches $\mG$, then immediately retracts back. \emph{For any $s>s_0$, $\mR_{s,t}$ pushes along this arc, performs a finger move at $t<t_0$, then reverses and undoes the finger move at some $t>t_0$, creating a finger Whitney move pair where the local finger/Whitney discs coincide. As the previous data was disjoint from the local finger arc in the level set $\{(s,t_0)\}\times X$, the data extended past the birth/death point will be completely disjoint from the newly created finger disc and Whitney disc for $s>s_0$}. We record this as the following: 

\begin{lemma} \label{L: b/d move}
    Let $H_s$ be a homotopy between paths of embeddings that contains a single birth/death point at $(s_0,t_0)$. Furthermore, assume that the finger/Whitney moves other than the pair created at $s=s_0$ occur at time $t=1/4$ and $t=3/4$ respectively and that the local finger/Whitney arcs for the new pair have been extended down/up to $t=1/4$ and $3/4$ for all $s>s_0$ ($s<s_0$ for a death point). Let $\mF=\{f_1,...,f_k\}$ and $\mW=\{w_1,...,w_k\}$ denote the set of finger/Whitney discs for the finger/Whitney system at $s=0$. Then $(\mR,\mG,\mF, \mW)_0$ differ from $(\mR,\mG,\mF,\mW)_1$ by adding (removing) a pair of intersections between $\mR$ and $\mG$ and adding (removing) a pair of discs $f_{k+1}$ and $w_{k+1}$ that pair off the added (removed) intersections, are disjoint from all discs in $\mF$ and $\mW$ and agree as framed discs, modulo an isotopy of the data. 
    \qed
\end{lemma}

\begin{figure}
    \centering
    \includegraphics[width=0.8\linewidth]{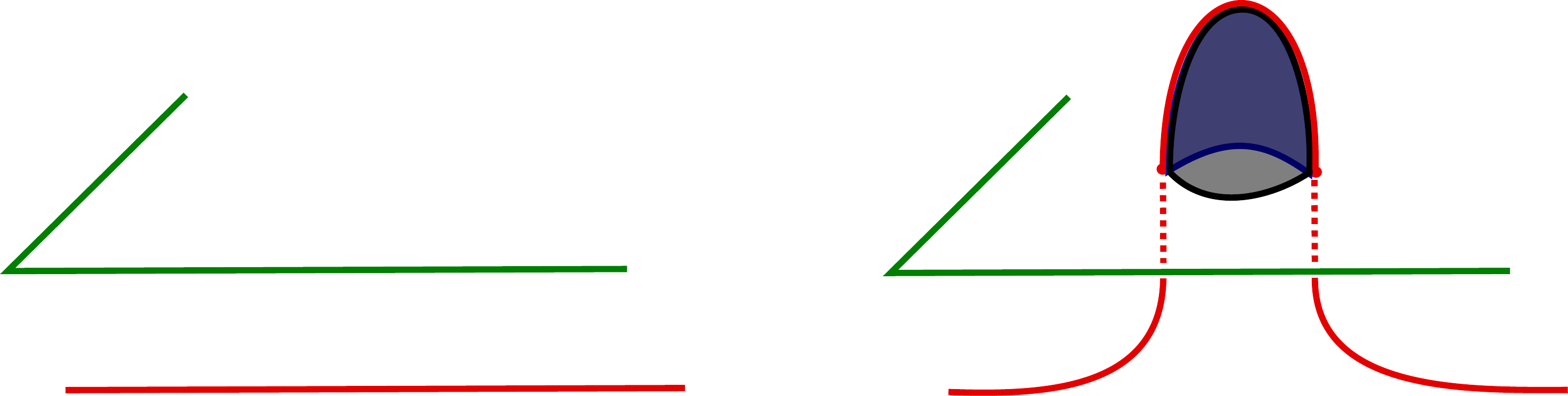}
    \caption{The shows the local modification of the Birth move (Death move). On the left we have no intersections between $\mR$ and $\mG$ in the local neighborhood. On the right, two extra intersections are created. To the $\mF\mW$ system, a finger disc and a Whitney disc are added that remove the two newly created intersections.}
    \label{F: bd move}
\end{figure}

\begin{figure}
    \centering
    \includegraphics[width=0.5\linewidth]{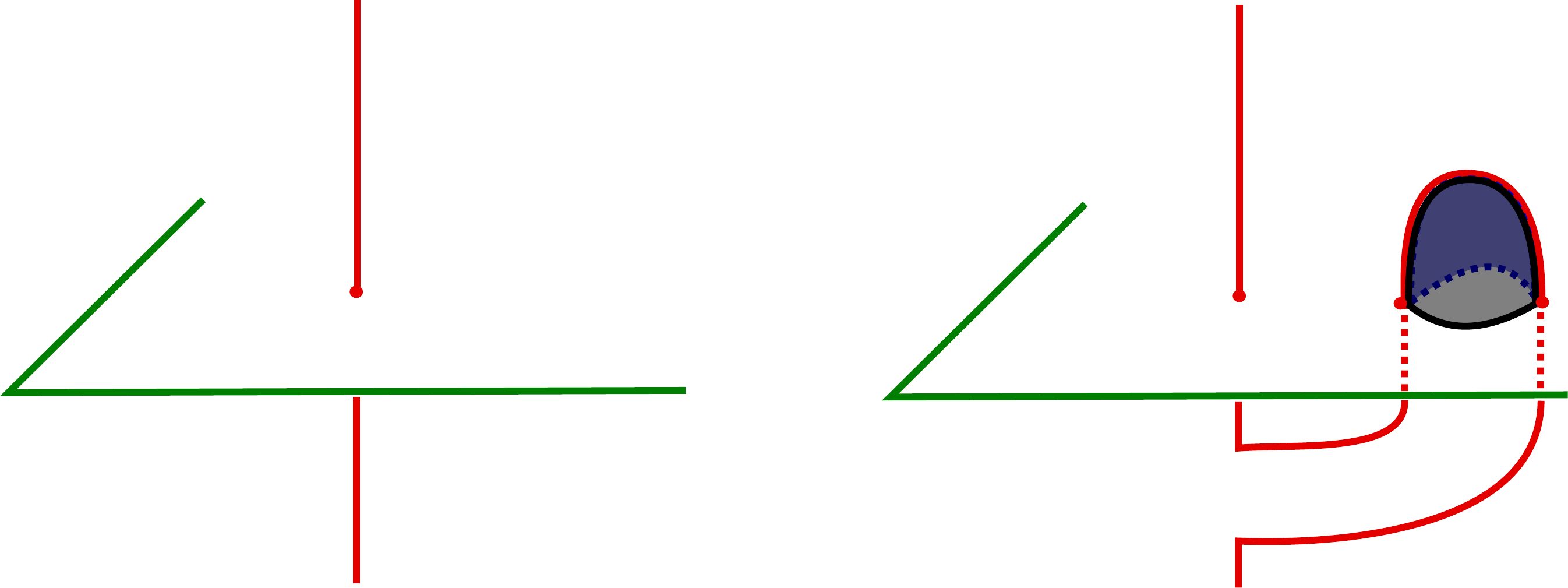}
    \caption{ Shown here is an example of applying the Birth move to the trivial $\mF\mW$ system. On the left we have $\mF = \varnothing = \mW$. After the move, we have $\mF' = \{f\}$ and $\mW' = \{w\}$.} 
    \label{F: bd_move_example}
\end{figure}

This motivates the following definition (see Figure \ref{F: bd move} and Figure \ref{F: bd_move_example} an example).
\begin{definition}\label{D: b/d move}
    A change in a finger/Whitney system as described in the previous lemma is called a \emph{birth move} or a \emph{death move} depending on whether a pair of intersections is created and discs added or if a pair is removed.
\end{definition}

\subsubsection{The $x^3$ move.}

As discussed in Remark \ref{R: distinct cusp parameters.}, cusp points generically occur at parameter values $s$ distinct from the phenomena discussed above. 
Let $(s_0,t_0)$ denote the parameters that contain the cusp point. Then for some small interval containing $s_0$, the only change in $H_s(t)$ that occurs is passing through a cusp. We will analyze the case where the cusp opens up to the right and work in a local 4-ball neighborhood $U_{p(s_0,t_0)}= B^3\times B^1$ of the cusp point $p(s_0,t_0)$ so that for the parameter $(s,t)$ sufficiently close to $(s_0,t_0)$, the frames given in Figure \ref{F: x cubed} describe the configuration of $\{(s,t)\}\times \mR_{s,t}$ and $\{(s,t)\}\times \mG$ in the 3-ball slice $\{(s,t)\}\times B^3\times \{0\}$ of $\{(s,t)\}\times U_{p(s_0,t_0)}$. We will implicitly take this perspective and not refer to the coordinates in $[0,1]^2\times X$ to simplify the notation.

\begin{figure}[!htbp]
\begin{subfigure}{.5\textwidth}
  \centering
  \includegraphics[width=.9\linewidth]{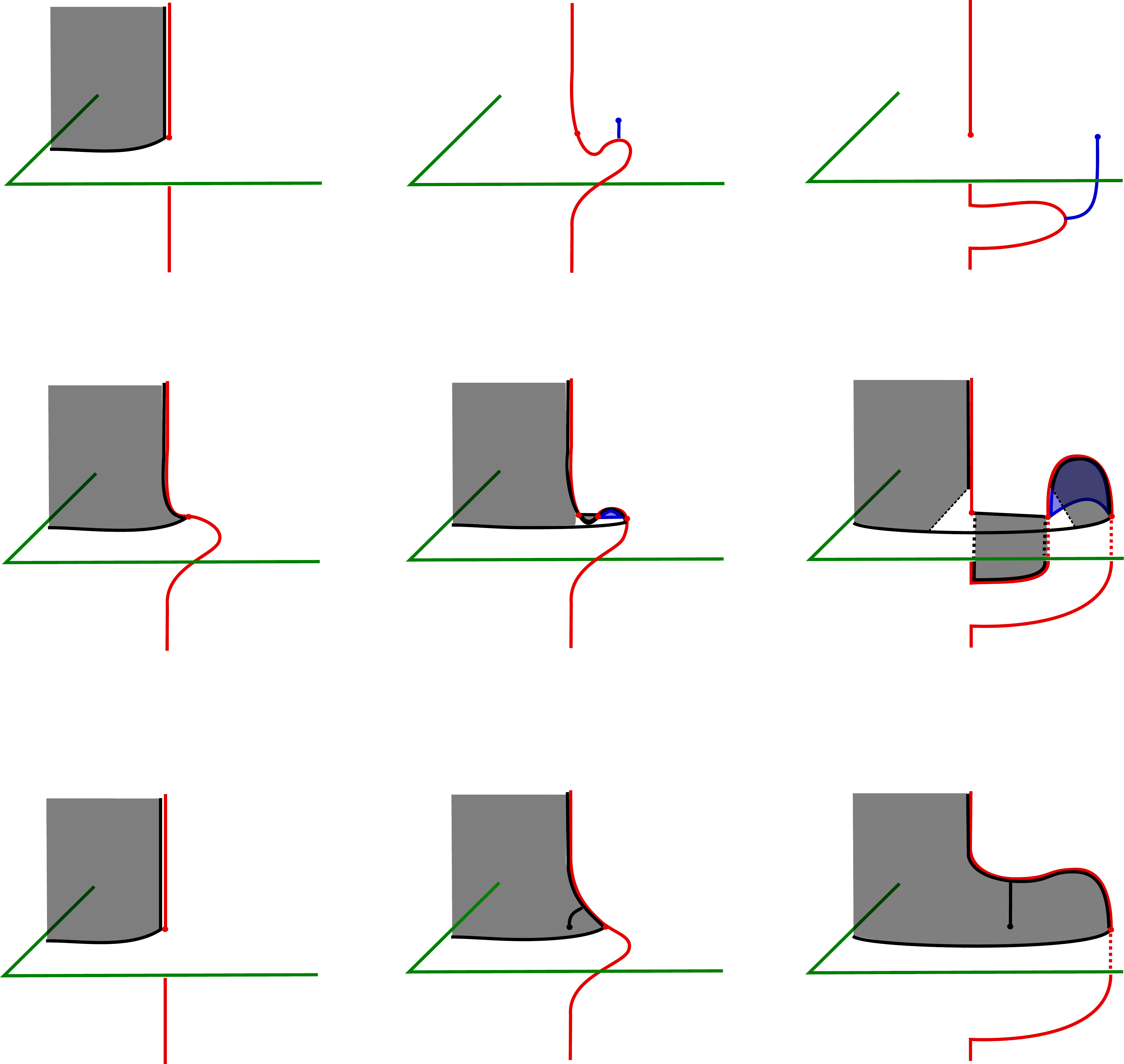}
  \caption{}
  \label{F: cuspextension finger corner whitney germed}
\end{subfigure}
\begin{subfigure}{.5\textwidth}
  \centering
  \includegraphics[width=.9\linewidth]{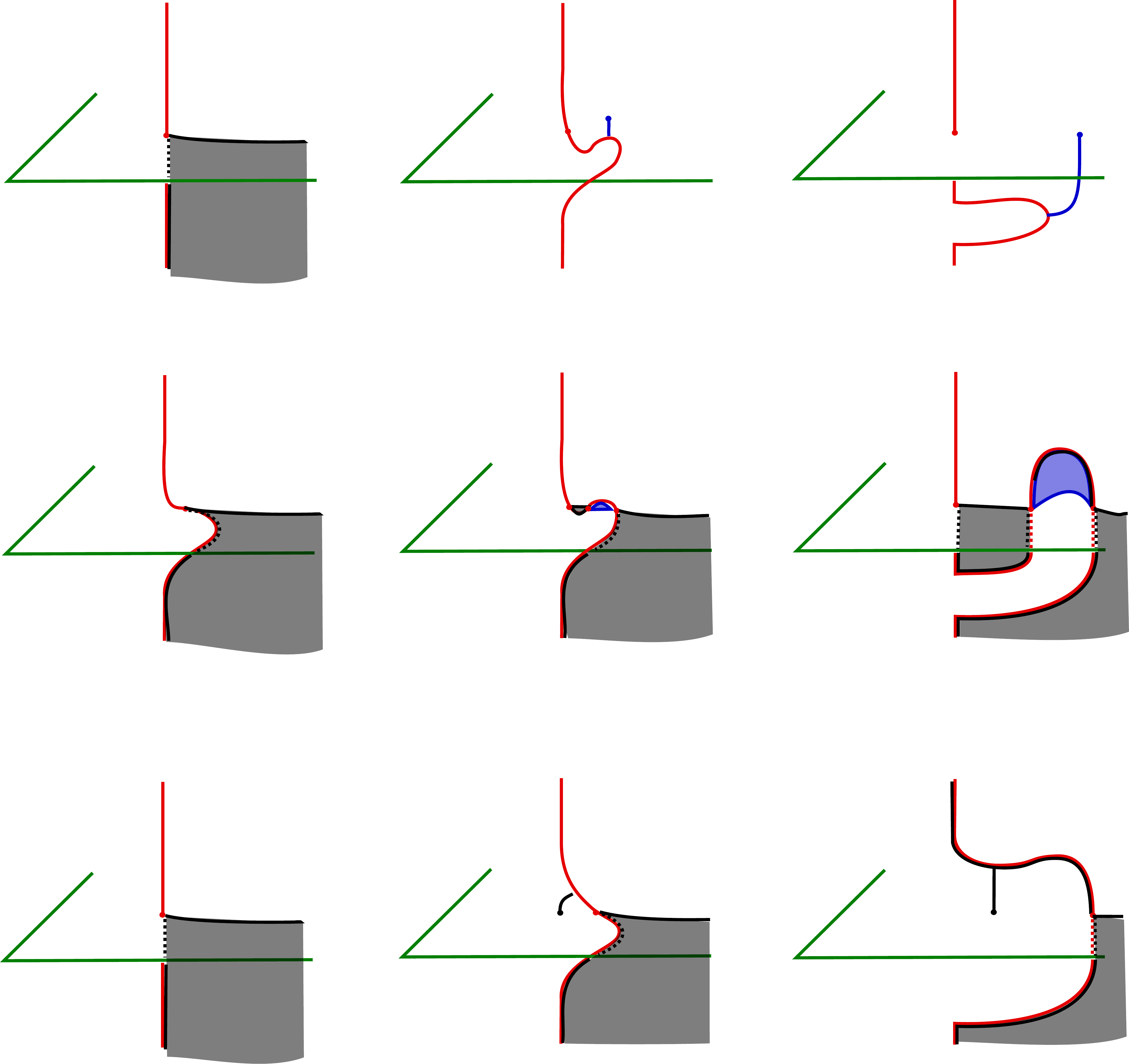}
  \caption{}
  \label{F: cuspextension finger corner finger germed}
\end{subfigure}
\caption{Depicted here is the local model for a cusp point (Figure \ref{F: x cubed}) in the presence of a finger disc. Shown in Figure \ref{F: cuspextension finger corner finger germed} is the corner for a finger disc that is finger germed while in Figure \ref{F: cuspextension finger corner whitney germed} is the corner that is Whitney germed. The two figures show how to extend the $\mF$ through the cusp singularity.}
\label{F: cuspextensionF}
\end{figure}

\begin{figure}[!htbp]
\begin{subfigure}{.5\textwidth}
  \centering
  \includegraphics[width=.9\linewidth]{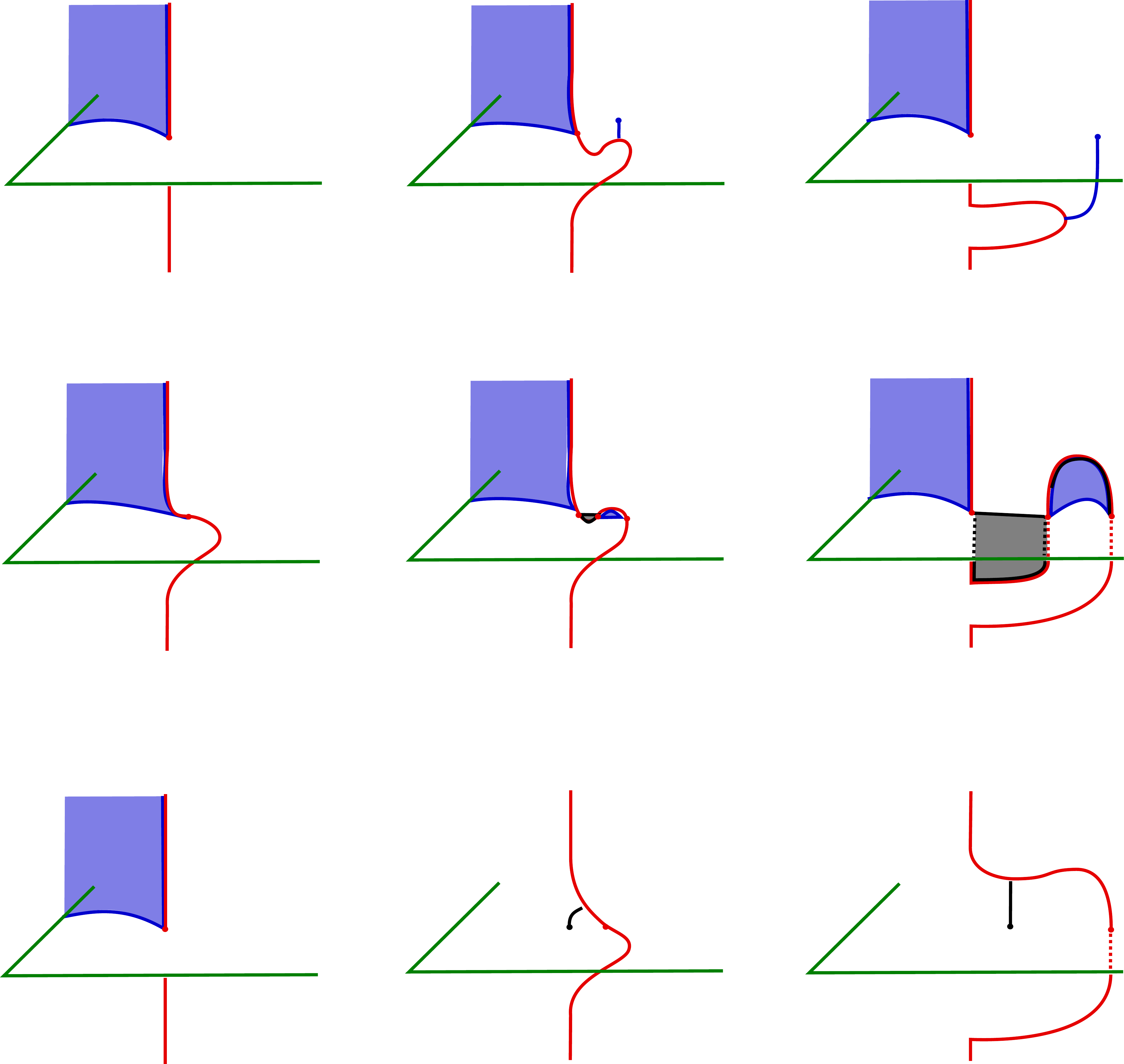}
  \caption{}
  \label{F: cuspextension whitney corner whitney germed}
\end{subfigure}
\begin{subfigure}{.5\textwidth}
  \centering
  \includegraphics[width=.9\linewidth]{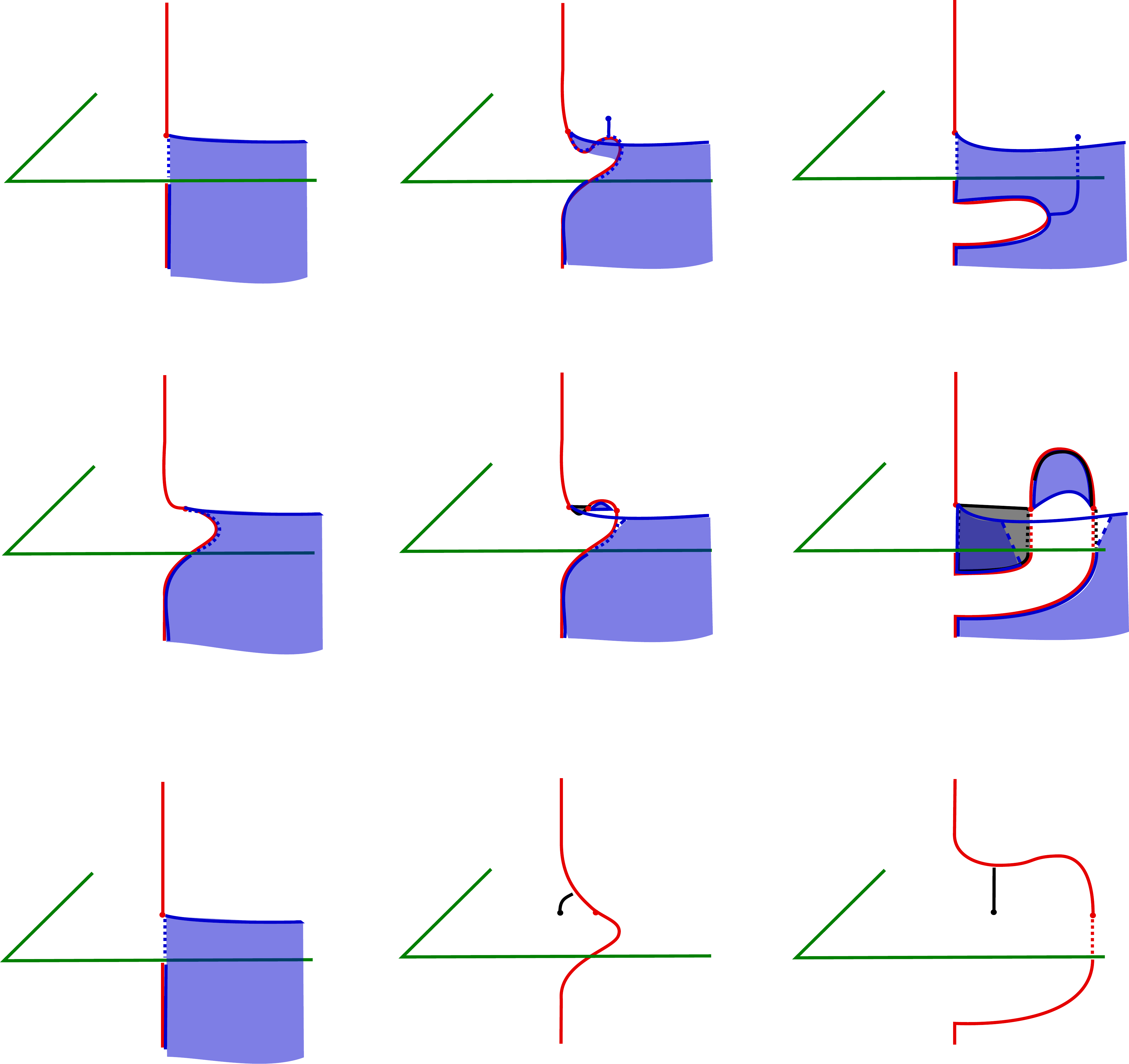}
  \caption{}
  \label{F: cuspextension whitney corner finger germed}
\end{subfigure}
\caption{Depicted here is the local model for a cusp point (Figure \ref{F: x cubed}) in the presence of a Whitney disc. Shown in Figure \ref{F: cuspextension whitney corner finger germed} is the corner for a Whitney disc that is finger germed while in Figure \ref{F: cuspextension whitney corner whitney germed} is the corner that is Whitney germed. The two figures show how to extend the $\mW$ through the cusp singularity.}
\label{F: cuspextensionW}
\end{figure}

\begin{figure}
    \centering
    \includegraphics[width=0.5\linewidth]{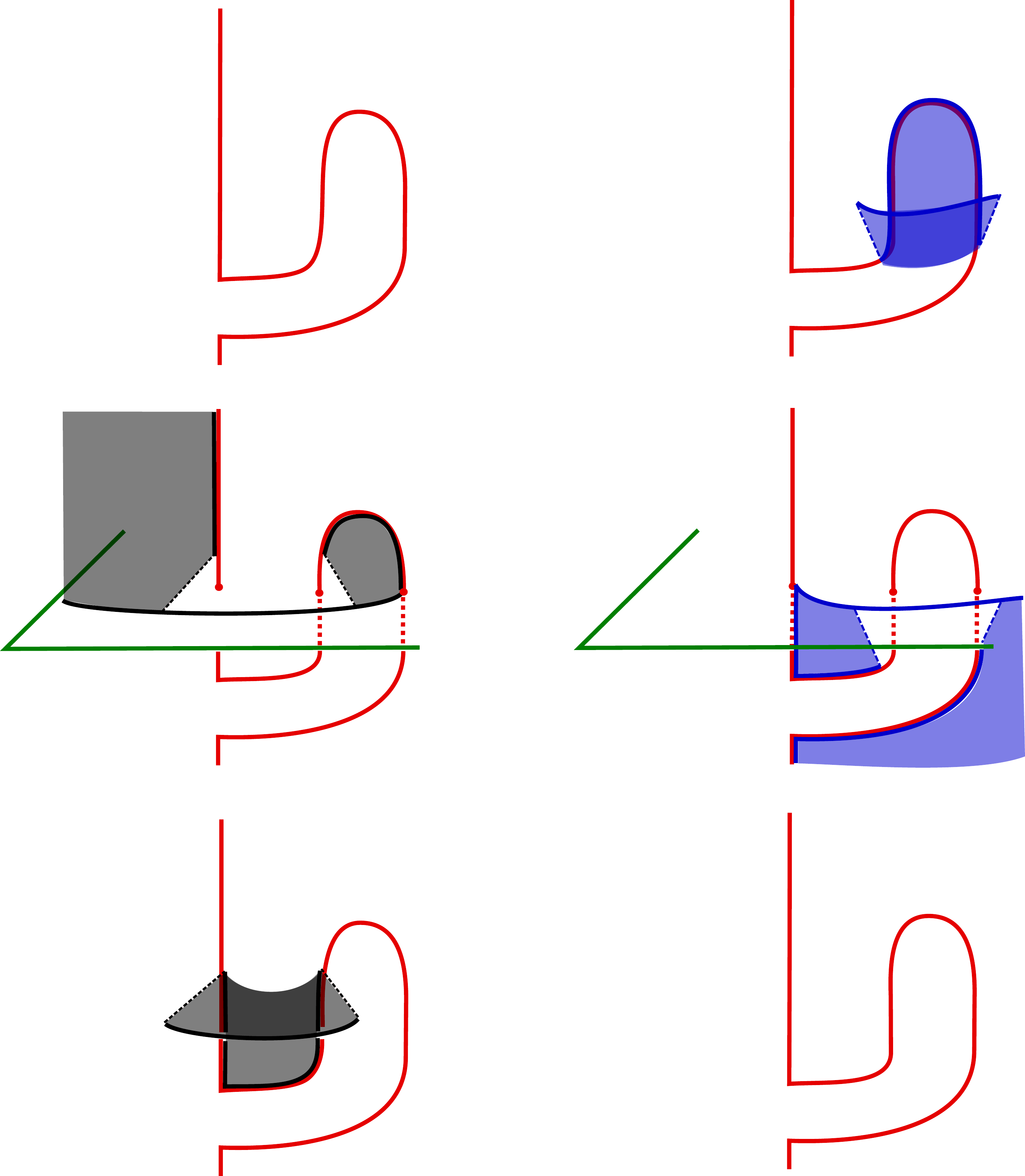}
    \caption{Shown here are the local neighborhoods of the finger/Whitney discs for the right-middle frame of Figures \ref{F: cuspextension finger corner finger germed} and \ref{F: cuspextension whitney corner whitney germed}. Here, each column of figures represents the 4-dimensional neighborhood.}
    \label{F: Dotted line notation}
\end{figure}

\begin{figure}
    \centering
    \includegraphics[width=\linewidth]{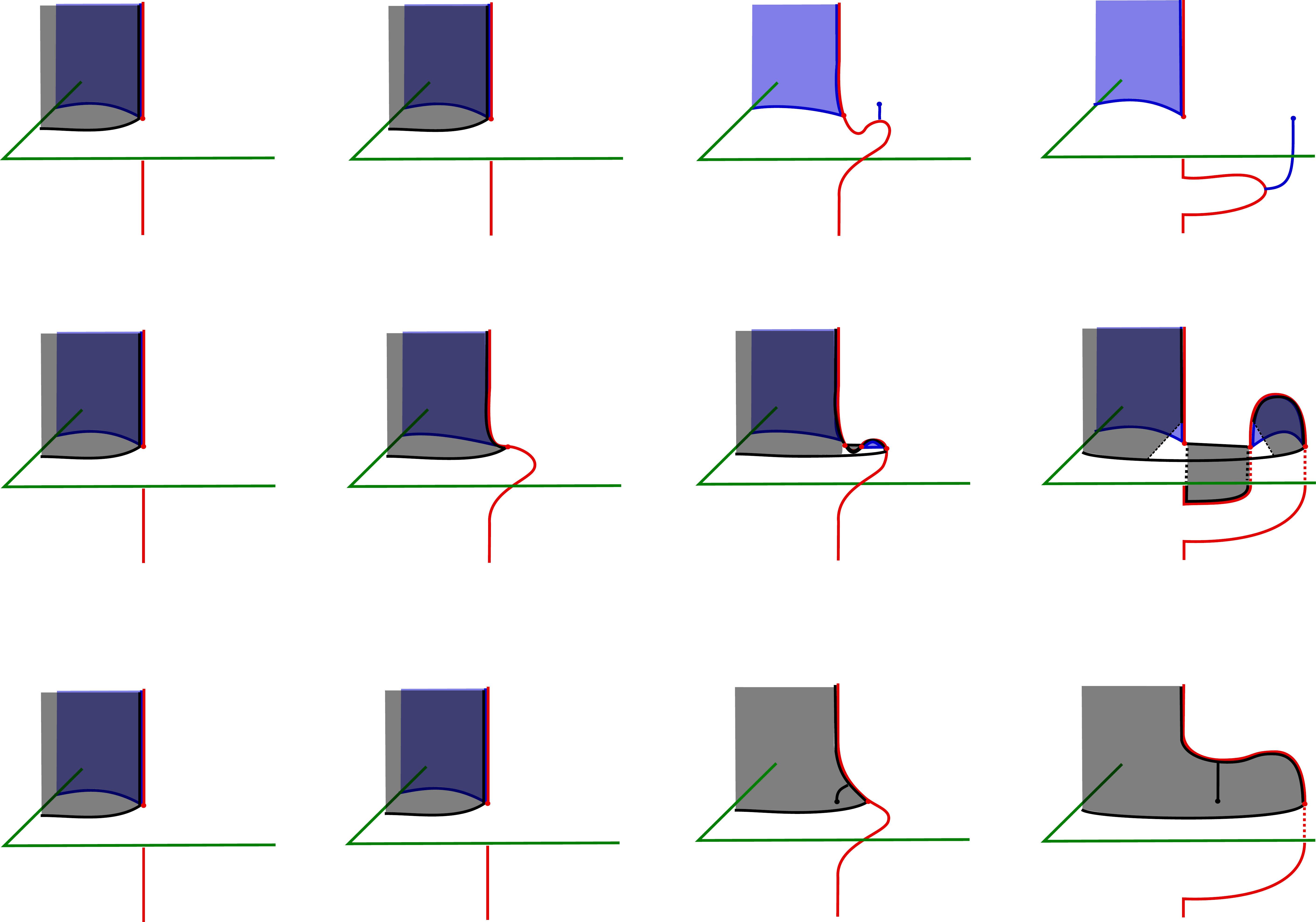}
    \caption{Shown here is an example of passing through an cubic singularity.}
    \label{F: x-cubedexample}
\end{figure}

Just as we did for the birth/death move, we want to extend the previous data across the cusp point, and then see how it interacts with the newly created data. In the local cusp neighborhood, we can assume that the only discs of $\mF$ or $\mW$ that intersect this neighborhood are those that pair off the cusp point. After a local isotopy of these discs, we may assume that they are locally germed.
\begin{definition}\label{D: Locally germed}
  Let $U\subset X$ be a small neighborhood of an intersection point $p\in \mR\cap \mG$ and let $d$ be a Whitney disc that pairs off $p$ with some other point $p'\in \mR\cap \mG$. We will call $U\cap d$ a corner of $d$. If $U = B^2\times B^1\times (-\epsilon, \epsilon)$ with $\mG\cap U = B^2 \times \{0\}\times \{0\}$. We say that the corner is locally Whitney germed if $U\cap d\subset B^2\times[0,1]\times \{0\}$ and is locally finger germed if $U\cap d\subset B^2\times[-1,0]\times \{0\}$. 
\end{definition}
\begin{remark}
    For a transverse intersection, the notion of a corner being locally finger or Whitney germed is not well defined, as we could always apply a local reflection to swap between. For a cusp, there is a well defined notion: The corner is finger germed (Whitney germed) if it is contained in the same side as the created finger disc (Whitney disc).
\end{remark}

\begin{remark}
    In Figures \ref{F: cuspextension finger corner whitney germed} and \ref{F: cuspextension whitney corner finger germed}, the right-middle frame depicts only part of the extended corner. The part of the corner not shown has been pushed into the 4th dimension. This is shown explicitly in Figure \ref{F: Dotted line notation}. 
\end{remark}
Shown in Figures \ref{F: cuspextensionF}, \ref{F: cuspextensionW} are the four cases of extending the corners of $\mF$ and $\mW$, up to the choice of parameterizing the local neighborhood $U$. For passing a cusp point, there are three cases to consider depending on the previous finger/Whitney data.

\begin{itemize}
\item[Case 1:] The first case to consider is when the transverse point $p$ is free of all previous finger/Whitney data. Then the previous data extends unaffected and will be disjoint from the created finger/Whitney. The new discs will have disjoint interiors and their boundaries will create an embedded arc on both $\mR_{s,t_0}$ and $\mG$ for $s>s_0$.

\item[Case 2:] Here we will assume that there is a single disk from $\mF$ or $\mW$ that pairs off the transverse point $p$. After performing a local isotopy, we may assume that this disc is either locally finger germed or locally Whitney germed. Figures \ref{F: cuspextensionF} and \ref{F: cuspextensionW} show how the corner is then extended. \emph{What determines the extension is whether the initial disc was a finger disc or a Whitney disc for the family $H_s(t)$ $s<s_0$ and how it is locally germed with respect to the cusp. If the disc was a Whitney germed finger disc, then the extension is obtained using a parallel of the new finger disc (see Figure \ref{F: Dotted line notation}); if the disc was a finger germed Whitney disc, then one uses a parallel of the Whitney disc.} In either case, the extended disc will be completely disjoint from both of the newly created finger/Whitney discs.

\item[Case 3:] The final case considers when the transverse point is paired by both a finger disc and Whitney disc. Just as in Case 2, the extension of the corner depends on the local germ with respect to the cusp and the extensions are given by Figures \ref{F: cuspextensionF} and \ref{F: cuspextensionW}. Figure \ref{F: x-cubedexample} shows an example When both the finger disc and Whitney disc are locally Whitney germed. \emph{In all cases, the newly created finger/Whitney pair is always disjoint from the extended data.}
\end{itemize}

\begin{definition}\label{D: x^3}
    Let $H_s(t)$ be an ordered homotopy between paths of embeddings that contains a single cusp that opens to the right. With the exception of the finger/Whitney pair created by the cusp, all other finger/Whitney moves happen at $t=1/4$ and $t=3/4$ respectively for all $s\in [0,1]$. We define the \emph{$x_{+}^3$ move} to be the modification of $(\mR,\mG, \mF,\mW)_0$ to $(\mR, \mG, \mF, \mW)_1$ by extending $\mF$ and $\mW$ appropriately (either Case 1, 2, or 3) and adding the local finger/Whitney discs created by the cusp. We define the \emph{$x_{-}^3$} to be the reverse.
\end{definition}

\begin{remark}
    Note that in all 3-cases, the extended data is disjoint from the finger/Whitney disc created by the cusp deformation. As such, we can follow the extended data back, letting $s$ vary from $s>s_0$ to $s<s_0$. The new finger/Whitney discs then cancel each other, and the extended discs return to their previous configuration. Moreover, we can compare the extended system at $s>s_0$ with a given system at $s>s_0$. Because all finger/Whitney moves other than the ones created by the cusps occur simultaneously by assumption and that all finger/Whitney systems are derived from a given germ $\mF\mW(H)$, the extended system and any other system for $\mF\mW(H)$ differ only by disc slides over the new discs and isotopy of the data. 
\end{remark}

We have thus proved the following lemma.
\begin{lemma}
    Let $H_s (t)$ be a homotopy between paths of embeddings that contains a single cusp point at $(s_0,t_0)$. Suppose that the finger/Whitney moves other than the pair created at $s=s_0$ occur at time $t=1/4$ and $t=3/4$. Suppose that for a fixed $s'>s_0$ ($s'<s_0$ if cusp opens to the left) the local finger/Whitney arcs for the new pair have been extended down/up to $t=1/4$ and $3/4$ for all $s>s'$ ($s<s'$). Then $(\mR,\mG,\mF, \mW)_0$ differ from $(\mR,\mG,\mF,\mW)_1$ by an $x_{+}^3$ move (or $x_{-}^3$), disc slides, and isotopy of the data. 
    \qed
\end{lemma}

\subsubsection{The saddle move.}

\begin{figure}[!htbp]
    \centering
    \labellist                             
            \hair 10pt
            \pinlabel $t_f$ at -20 80
            \pinlabel $t_0$ at -20 280
            \pinlabel $t_w$ at -20 480
            \pinlabel $s_0-\epsilon$ at 0 -20
            \pinlabel $s_0$ at 250 -20
            \pinlabel $s_1$ at 450 -20
            \pinlabel $s_0+\epsilon$ at 600 -20
        \endlabellist
    \includegraphics[width=0.5\linewidth]{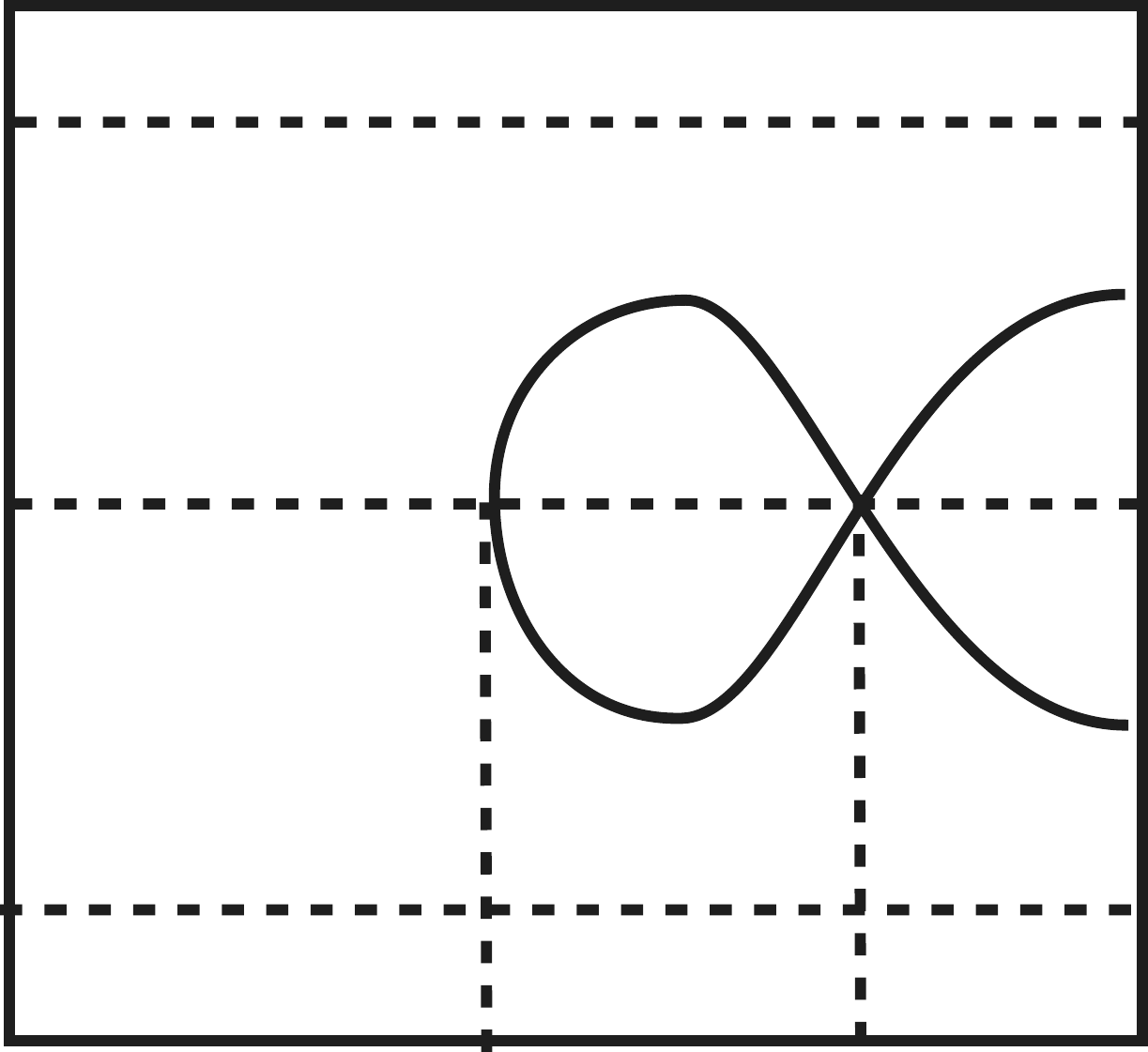}
    \vspace{.5cm}
    \caption{Shown is the local Cerf graphic for $H$ near a saddle point.}
    \label{F: saddlemove Cgraphic}
\end{figure}

The last case to consider are the saddle points. Let $(s_0,t_0)$ denote the parameters containing a saddle point. By Theorem \ref{T: 2-par ordering}, we have that over some interval $[s_0-\epsilon, s_0+\epsilon]$ the Cerf graphic will appear as in Figure \ref{F: saddlemove Cgraphic}. After some parameter value $s_1$ with $s_0< s_1<s_0+\epsilon$, the family $H_s(t)$ will again have a finger-first ordering. We may assume that over the interval $[s_0-\epsilon, s_0+\epsilon]$, none of the other previous phenomena occur between the other finger/Whitney data. We will work in a local model for the saddle point, assuming that as $s$ increases from $s<s_0$ to $s>s_1$, we increase the number of transverse intersection points between $\mR$ and $\mG$. If we track the embeddings $ \mR_{s, t_0}$ as $s$ increases past $s_0$, then we see a Whitney move performed across the local disc $D$ embedded by $\mF\mW(H)$. Just after $s_0$, we have the corresponding finger arc $\gamma_{D}$. As $s$ continues to increase, this arc duplicates and becomes two separate arcs: One for the newly created finger move, $\gamma^f_D$, the other for the Whitney move, $\gamma_D^w$. For the values of $s_0<s<s_1$, $H_s(t)$ does not have a finger-first ordering. As we continue, we see $\mR_{s,t_0}$ interchange when the finger move along $\gamma_D^f$ is performed, raising the Whitney move up and lowering the finger move until $s$ passes $s_1$. Then $H_s(t)$ has a finger-first ordering again. In the local model and for $s>s_1$, we see that $\mR_{s,t_0}$ now has two new local discs: one for the new finger move and one for the new Whitney move. We will denote these discs by $D_f$ and $D_w$ respectively.

\begin{figure}
    \centering
    \includegraphics[width=0.5\linewidth]{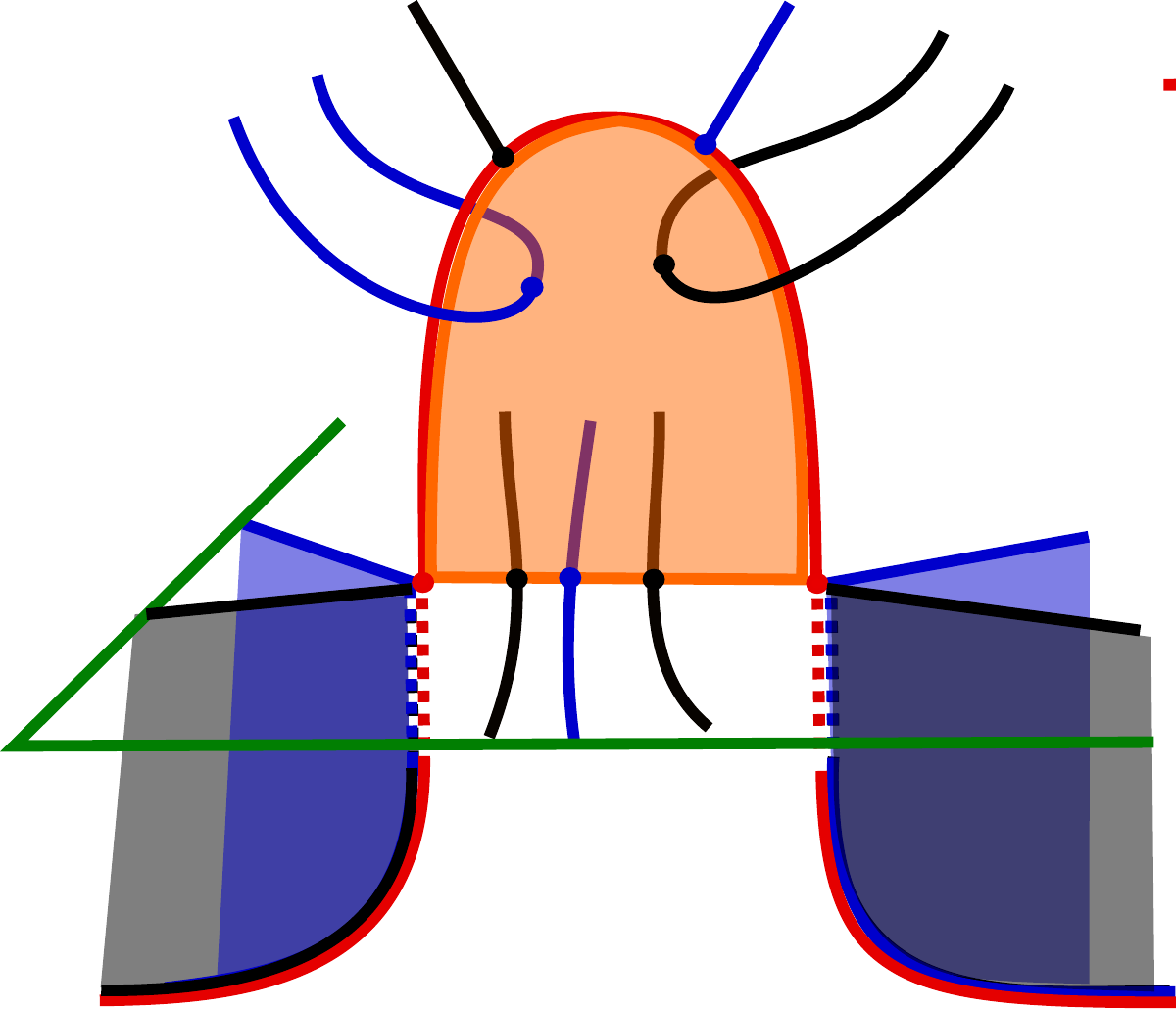}
    \caption{Here is an illustration of the possible configurations of the finger/Whitney data (in black and blue) in the local neighborhood of the saddle disc $D$(in orange).}
    \label{F: saddle disc possiblility}
\end{figure}
Let $t_f$ and $t_w$ be times such that over the whole interval $[s_0-\epsilon, s_0+\epsilon]$, all finger moves other than the one created by the saddle point happen before $t_f$ and all Whitney moves happen after $t_w$. In the Cerf graphic, the line segments $[s_0-\epsilon, s_0+\epsilon]\times \{t_f\}$ and $[s_0-\epsilon, s_0+\epsilon]\times \{t_w\}$ are disjoint from the fold graphic. As such, we can assume that all the lower local finger discs have been extended to the level set $[s_0-\epsilon, s_0+\epsilon]\times \{t_f\}\times X$ and all the upper Whitney discs extended into $[s_0-\epsilon, s_0+\epsilon]\times \{t_w\}\times X$. For $s< s_0$, we can extend the finger disc and Whitney discs into the common level set $\{(s,t_0)\}\times X$ and get a finger/Whitney system for $H_s(t)$. However, for $s\in (s_0,s_1)$, such an extension is not possible if the saddle disc $D$ intersects any other discs in $(\mF,\mW)_s$ $s<s_0$. Figure \ref{F: saddle disc possiblility} illustrates the general situation in the local model for the disc $D$.
\begin{figure}
    \centering
    \includegraphics[width=0.8\linewidth]{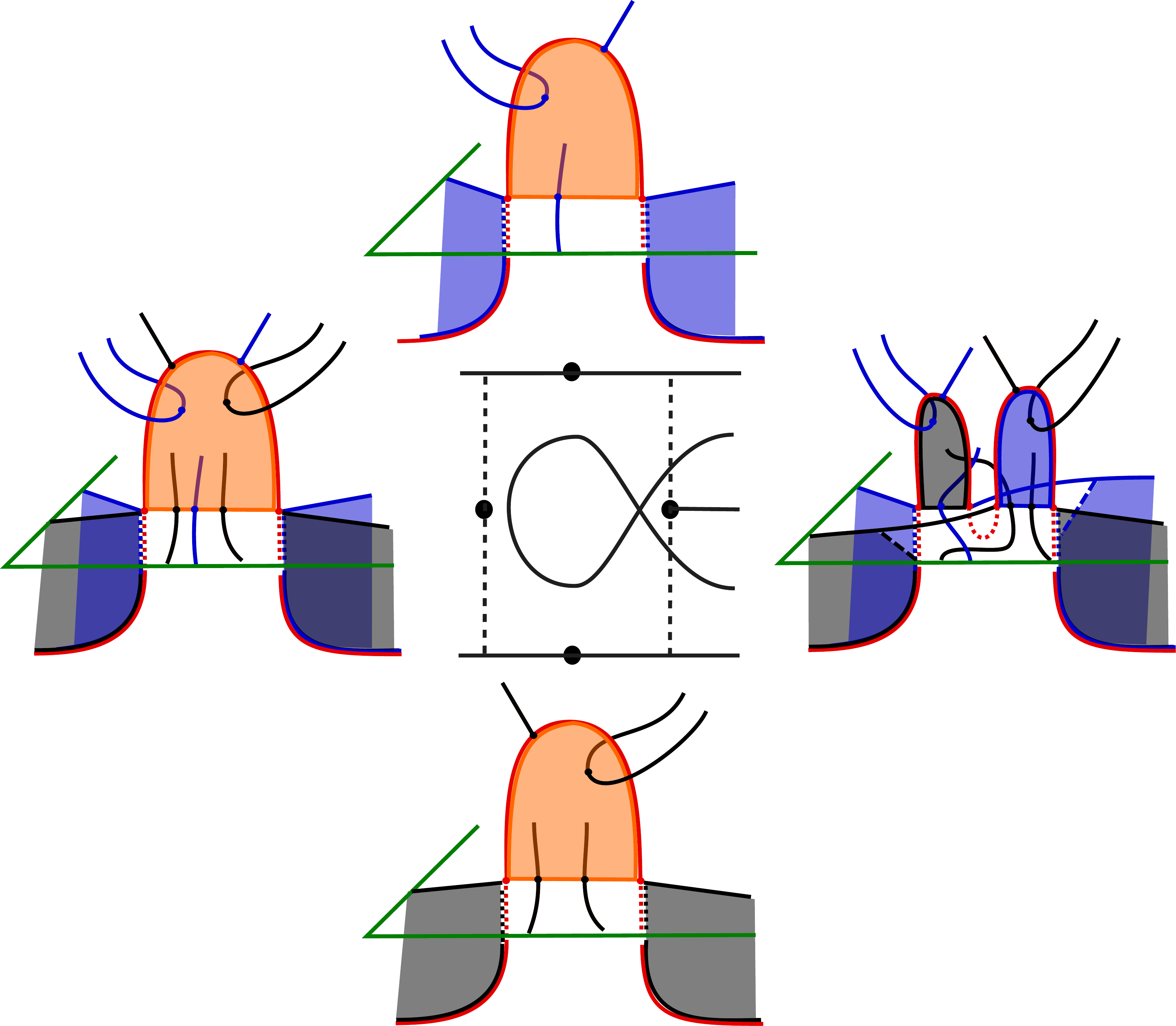}
    \caption{Shown here is an illustration of where the extended data lives. In particular, the Whitney discs for the moves happening at $t=3/4$ can be extended to the $t = t_w$ level sets, but not necessarly below for $s\in (s_0,s_1)$. A similar statement holds for the finger discs. For $s>s_1$, the data can be extended back into a common level and the configuration on the right depicts the extended data.}
    \label{F: saddle summary}
\end{figure}
We can however extend the saddle disc $D$ down to the level set $[s_0, s_0+\epsilon]\times \{t_f\}\times X$ and up to the set $[s_0, s_0+\epsilon]\times \{t_w\}\times X$ as the germ $\mF\mW(H)$ is defined all along the fold. The local picture in these sets differs only in that the set $[s_0-\epsilon, s_0+\epsilon]\times \{t_f\}\times X$ contains the finger discs for $\mF$ and the local disc $D$ while the set $[s_0-\epsilon, s_0+\epsilon]\times \{t_w\}\times X$ contains only the Whitney discs in $\mW$ and $D$. Note that the local disc $D$ becomes $D_w$ when $s>s_0$, $t=t_f$ and $D_f$ when $s>s_0$ and $t=t_w$. We've summarized this in Figure \ref{F: saddle summary}. For $s>s_1$, we can extend the data up from $t=t_f$ and down from $t=t_w$ to $t=t_0$ and get a new finger/Whitney system $(\mR, \mG,\mF, \mW)_s$. The data outside the local neighborhood does not change. Inside the neighborhood, we see two copies of the saddle disc, $D_f$ and $D_w$. Moreover the intersections (interior and boundary) with $D_f$ and the other Whitney discs $w_i\in \mW$ is preserved. In fact, interior and boundary intersection with $D$, while they can not be extended over the interval $[s_0,s_1]$, naturally extend into $\{(s,t_0)\}\times X$ for $s>s_1$ from the time slices $t = t_f$ and $t = t_w$. What will be important to understand is how the corners of the finger/Whitney discs change. \emph{This is because the corners determine how the new and old discs fit together around the intersection points between $\mR$ and $\mW$ once $H_s(t)$ has a finger-first ordering again}. By corners we mean 
open neighborhoods in the finger/Whitney discs of intersection points between $\mR$ and $\mG$. So in the following analysis, we will focus just on how the corners of the finger/Whitney data in the local model change passing across the saddle point and then back to finger-first ordering.

\begin{figure}
    \centering
    \labellist                             
        \hair 10pt
        \pinlabel $D$ at 420 350
    \endlabellist
    \includegraphics[width=0.5\linewidth]{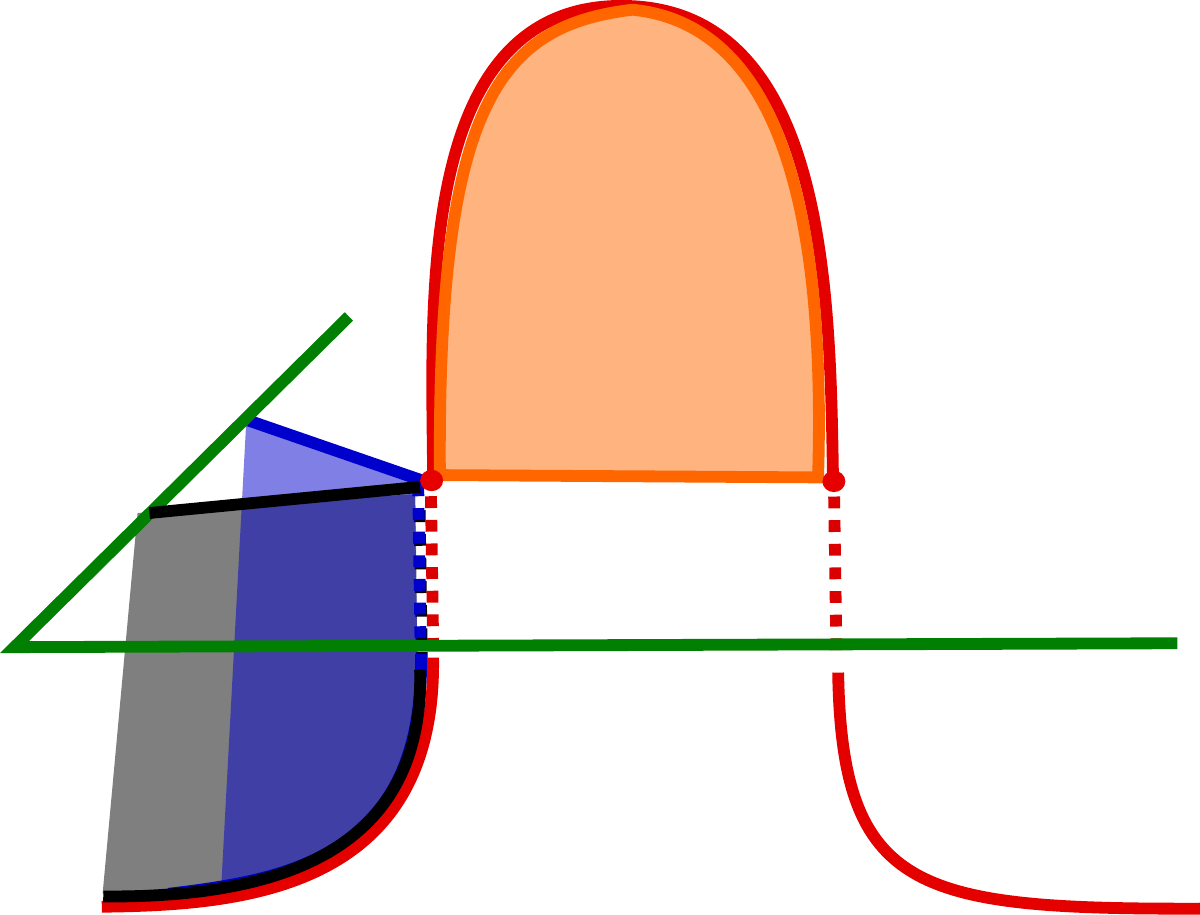}
    \caption{Shown is a local neighborhood for a saddle disc $D$ pairing of points in $\mR\cap \mG$ with one such point free of all other finger/Whitney data. In black and blue, we have corners for the finger disc and Whitney disc that pair off the other intersection point.}
    \label{F: saddlecase1corner}
\end{figure}
\begin{figure}
    \centering
    \labellist                             
        \hair 10pt
        \pinlabel $\gamma_D$ at 340 270
        \pinlabel $\gamma_{w_D}$ at 950 270
        \pinlabel $\gamma_{f_D}$ at 1100 270
        \pinlabel ${w_D}$ at 1650 310
        \pinlabel ${f_D}$ at 1800 310
    \endlabellist
    \includegraphics[width=.9\linewidth]{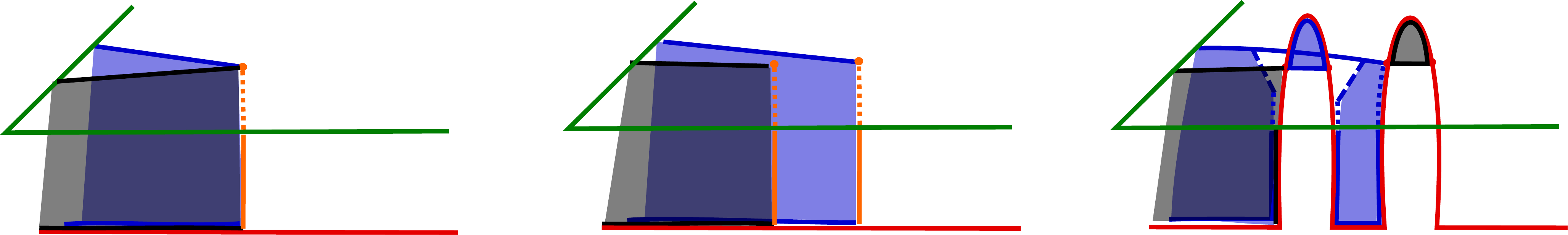}
    \caption{Depicted is the extension of the corners, and how the splitting of the finger arc $\gamma_D$ affects the corners in the local model. Shown on the right, part of the Whitney disc corner is not shown. Instead, the rest of the disc is outside of the 3-dimensional slice as depicted in Figure \ref{F: Dotted line notation}. }
    \label{F: saddlecase1fsplit}
\end{figure}

Just like the case of the $x^3$-move, we have a few cases to consider. We start with the local configurations for $s<s_0$ and describe how they extend to $s>s_1$. In fact, the corners can be extended through the local model in $\{(s,t_0)\}\times X$ for $s\in[s_0,s_1]$. We will make use of this.
\begin{itemize}
    \item[Case 1:] We will suppose that the disc $D$ pairs off a pair of transverse intersections with one of the intersection points free of all finger/Whitney discs. After a local isotopy, the finger/Whitney discs corners will appear as shown in Figure \ref{F: saddlecase1corner}. Let $f_c$ denote the corner for the finger disc, and $w_c$ the corner for the Whitney disc. Now just after performing the Whitney move across the disc $D$, part of both $\partial f_c$ and $\partial w_c$ will lie along the finger arc $\gamma$. When the arc $\gamma$ separates into $\gamma_f$ and $\gamma_w$, the part of $\partial f_c$ that was along $\gamma$ stays with $\gamma_w$ while the subset of $\partial w_c$ stays with $\gamma_f$. \emph{This is because two finger discs can not pair off the same intersection point.} We can arrange in the local model that when $\gamma$ separates, either $\gamma_w\subset w_c$ or $\gamma_f \subset f_c$ (see Figure \ref{F: saddlecase1fsplit}).  Then performing the finger moves along both $\gamma_f$ and $\gamma_w$, we get to the local model for $s>s_1$ with four possible configurations for the corners: If $\gamma_f\subset f_c$, then performing the finger move along $\gamma_f$ produces the new finger disc $D_f$ and the corner $f_c$ picks up a parallel copy $D_f$ as its boundary runs across to the correct intersection point created by doing the finger move along $\gamma_w$. Just like the $x^3$ move, there are two choices for the parallel. Hence with $\gamma_f \subset f_c$, there are two choices of extension. The same is true if $\gamma_w\subset w_c$ and together, these make the four possible extensions of the corner data. 

    \item[Case 2] Here we will assume that both intersection points between $\mR$ and $\mG$ have finger/Whitney disc data. Let $f_c^+$ and $f_c^-$ be the positive and negative corners for the finger disc respectively and $w_c^+$ and $w_c^-$ the corresponding Whitney data. We shall assume that, after a local isotopy, the corresponding positive and negative corner data agree, as shown in Figure \ref{F: saddlecase2corner}. We now extend the data past the saddle point, expanding the boundary of each corner to so that part of the boundary lies along the arc $\gamma_D$. Then when $\gamma_D$ separates into $\gamma_f$ and $\gamma_w$, the extended finger corners will lie along $\gamma_w$ with the Whitney corners lying along $\gamma_f$. We arrange locally so that $\gamma_f\subset f_c^+\cup f_c^-$ and $\gamma_w\subset w_c^+\cup w_c^-$. The union $f_c^+\cup f_c^-$ fit together to look like part of a single finger disc divided by the arc $\gamma_w$ (see Figure \ref{F: saddlecase2fsplit}). Furthermore, performing the Whitney move along $\gamma_w$ is akin to performing Quinn's ``splitting move'' \cite{Qu}. The same is true for the union of the Whitney corners $w_c^+\cup w_c^-$. Finally, since $\gamma_f\subset f_c^+\cup f_c^-$, after performing the finger move along $\gamma_f$, the corner containing $\gamma_f$ picks up a parallel of $D_f$. Similarly, doing the finger move along $\gamma_w$ will cause one of the Whitney corners to pick up a parallel copy of $D_w$.
\end{itemize}
\begin{figure}
    \centering
    \labellist                             
        \hair 10pt
        \pinlabel $D$ at 420 350
    \endlabellist
    \includegraphics[width=0.3\linewidth]{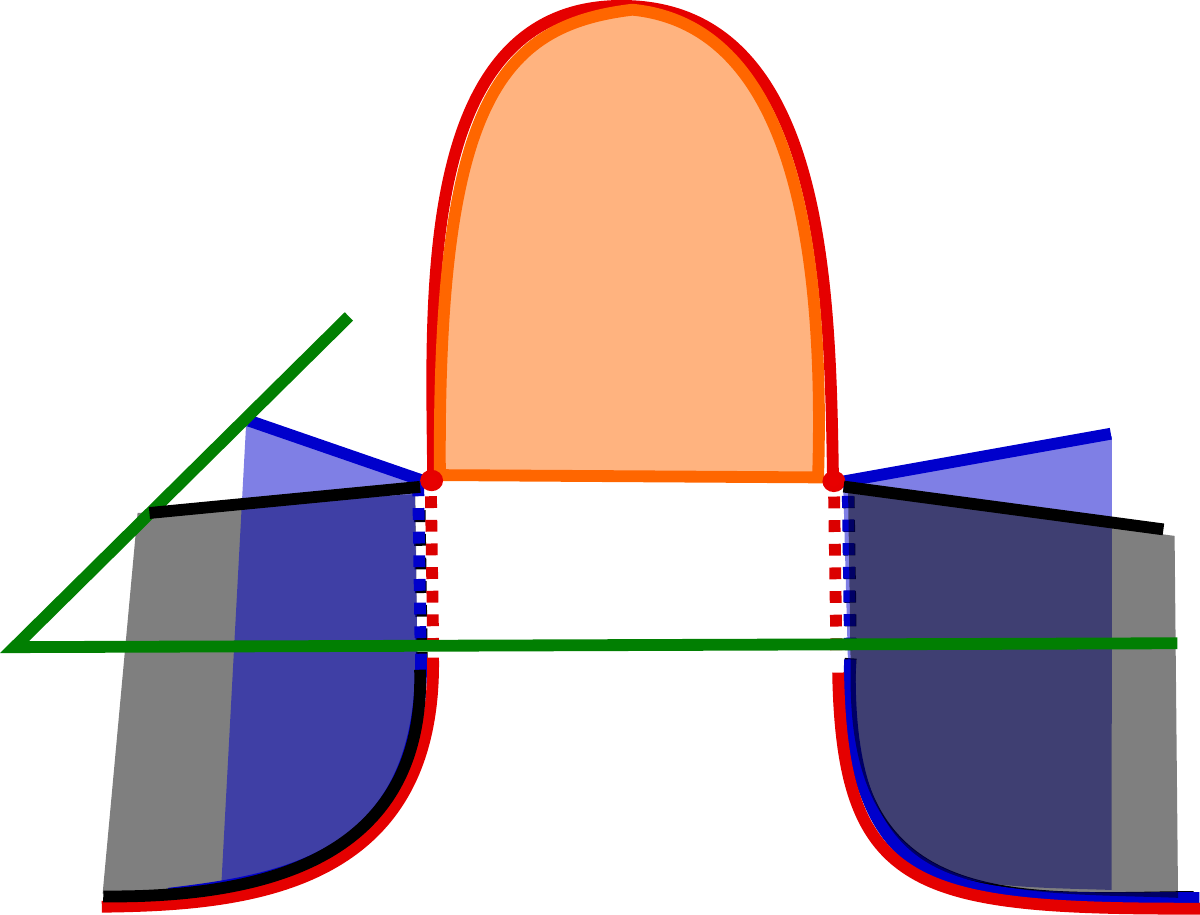}
    \caption{Shown is the configurations of the corners for the finger/Whitney data in the local neighborhood of a saddle disc in general.}
    \label{F: saddlecase2corner}
\end{figure}
\begin{figure}
    \centering
    \labellist                             
        \hair 10pt
        \pinlabel $\gamma_D$ at 340 270
        \pinlabel $\gamma_{w_D}$ at 950 270
        \pinlabel $\gamma_{f_D}$ at 1100 270
        \pinlabel ${w_D}$ at 1650 310
        \pinlabel ${f_D}$ at 1800 310
    \endlabellist
    \includegraphics[width=0.8\linewidth]{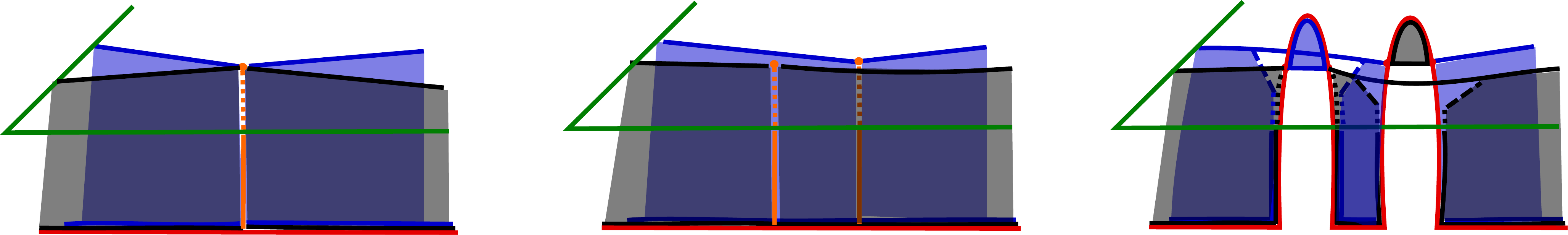}
    \caption{Depicted here is the sequence of extending the corners across the parameters $(s_0,s_1)$ and the effect the splitting of the finger arc $\gamma_D$ has on the data.}
    \label{F: saddlecase2fsplit}
\end{figure}
\begin{remark}[Key Observation]\label{R: saddle key observation}
    The single most important point of the analysis in both Case 1 and Case 2 is that \emph{the corners pick up parallel copies of either $D_f$ or $D_w$.} Therefore, when extend the $\mF$ information from up from $t=t_f$ and $\mW$-information down from $t_w$, any interior or boundary intersection data from a finger disc with $D_w$ or a Whitney disc with $D_f$ is then picked up by the parallel discs for the corresponding corner. The second critical observation is that in between the $D_f$ and $D_w$ in the local model exists a framed disc $D^*$ that is free of all $\mF$ and $\mW$ data that can be used as a switch disc for either $D_f$ or $D_w$.
\end{remark}

\begin{definition}\label{D: saddle}
    Let $H_s(t)$ be a homotopy between paths of embeddings that contains a single saddle point. We suppose that the Cerf graphic for $H_s$ appears as in Figure \ref{F: saddlemove Cgraphic}, so that $H_1(t)$ has one more additional finger/Whitney pair than $H_0(t)$ and that outside the local neighborhood containing the saddle point, $H_s(t)$ is constant for all $(s,t)$. We assume that the saddle point occurs at $(1/3, 1/2)$ and the finger/Whitney crossing at $(2/3, 1/2)$. Furthermore, with the exception of finger/Whitney pair created by the saddle, all other finger/Whitney moves happen at $t=1/4$ and $t=3/4$ respectively for all $s\in [0,1]$. We define the \emph{saddle move} to be the transformation of $(\mR, \mG, \mF, \mW)_{0}$ to $(\mR, \mG, \mF, \mW)_{1}$. 
\end{definition}

\begin{lemma}
    Let $H_s(t)$ be a homotopy between paths of embeddings that contains a single saddle point. We suppose that the Cerf graphic for $H_s$ appears as in Figure \ref{F: saddlemove Cgraphic}, so that $H_1(t)$ has one more additional finger/Whitney pair than $H_0(t)$. Furthermore, with the exception of finger/Whitney pair created by the saddle, all other finger/Whitney moves happen at $t=1/4$ and $t=3/4$ respectively for all $s\in [0,1]$. Then $(\mR, \mG, \mF, \mW)_0$ differs from $(\mR, \mG, \mF, \mW)_1$ by a saddle move, disc/sphere slides, and isotopes.
\end{lemma}

\begin{proof}
    The proof of this lemma relies on being able to reduce everything down to the local picture. Near the saddle point, we have the local picture and for $(s,t)$ sufficiently near the saddle point, $\mR_{s,t}$ can be fixed so that it is constant outside of the local model. We can within this local model apply Deformation 4 so that the Cerf graphic changes as shown in Figure \ref{F: saddlemove Cgraphic}. This uses only Lemma \ref{L: GP for arcs}. Afterwards, $H_s(t)$ will satisfy the definition of a saddle move. Therefore $(\mR, \mG, \mF, \mW)_{s_0-\epsilon}$ and $(\mR, \mG, \mF, \mW)_{s_0+\epsilon}$ will differ by a saddle move. Finally to go from $(\mR, \mG, \mF, \mW)_{0}$ to $(\mR, \mG, \mF, \mW)_{s_0-\epsilon}$ is just by an isotopy of the data (all finger/Whitney moves occur at $t= 1/4$ and $t=3/4$ and hence only differences in extending the local data up/down manifest). To go from $(\mR, \mG, \mF, \mW)_{s_0+\epsilon}$ to $(\mR, \mG, \mF, \mW)_{1}$, we needed to reorder after the local deformation. This could result in either disc or sphere slides. Hence $(\mR, \mG, \mF, \mW)_{s_0+\epsilon}$ differs from $(\mR, \mG, \mF, \mW)_{1}$ by disc/sphere slides and by isotopies of the data.
\end{proof}

\begin{proof}[Proof of Theorem \ref{T: fw system moves}]
    Starting with an ordered homotopy $H_s(t)$, we break the homotopy parameter $[0,1]$ into smaller intervals along parameter values $s_i\in [0,1]$ such that between $s_i$ and $s_{i+1}$, $H_s(t)$ has a finger-first ordering, or contains a single birth/death point, cusp point, or a saddle interval and no other phenomena. Each of the previous lemmas explains how the finger/Whitney system changes from $(\mR,\mG,\mF, \mW)_{s_i}$ to $(\mR,\mG,\mF, \mW)_{s_{i+1}}$ and these are exactly the moves in the theorem.
\end{proof}
\subsubsection{EA paths}
We finish this section by returning to the setting of $\mR = \sqcup S^2$, $\mG = \sqcup S^2\times \{r\}$ and $X = \#^k S^2\times S^2$ and give a proof that every relative homotopy class $[\alpha]\in (\Emb(\mR, X), \LB)$ has an EA representative.

\begin{proposition}\label{P: EA paths}
    Let $\mR = \sqcup_{i=1}^k S^2_i$ and $X = \#^k S^2\times S^2$, and $\mG = \sqcup_{i=1}^k(S^2\times \{r_0\})_i$.  Every class $[\alpha]\in \pi_1(\Emb(\mR, X), \LB)$ has a representative in EA.
\end{proposition}
In \S3 and \S2, we show that a given finger-Whitney system $(\mF, \mW)$ can be modified to be in EA using switching and disc slides. Here we show that these two operations are realizable by homotopies of the path $\alpha$ in $(\Emb,\LB)$.
 
\begin{lemma}\label{L: Factoring homotopy}
    Let $(\mF, \mW)$ be a finger-Whitney system and let $\alpha$ be the associated path of embeddings with endpoints in $\LB$. Let $\mW^*$ be a set of switch discs and $\alpha'$ the path of embeddings associated with $(\mF, \mW^*)$. Then $[\alpha] = [\alpha']$ in $\pi_1(\Emb(\mR, X), \LB)$.
\end{lemma}
\begin{proof}
    Here we use the process of factoring, an operation introduced by the first author in \cite{Ga3}, to split the path $\alpha$ into two paths. Since $\mW^*$ is a complete set of Whitney discs, we can factor the path $\alpha$ into the concatenation of 2 paths $\alpha'*\alpha''$ where $\alpha'$ has the system $(\mF, \mW^*)$ while $\alpha''$ comes from the system $(\mW^*, \mW)$. It remains to show that the system $(\mW^*, \mW)$ yields a path in $\LB$.

    By definition, the discs in $\mW^*$ agree with the discs in $\mW$ except for a subset of uncrossed discs. These uncrossed discs altogether form a collection that satisfies the Quinn's embedded arc condition with disjoint interiors. We may use the argument of Quinn \cite[Corollary to 4.5]{Qu} and perform the Whitney move along the discs in $\mW$ earlier in the time parameter. This deforms $\alpha''$ to a path with a finger-Whitney system in which the finger disks are equal to the Whitney discs and can be canceled in the same way. Thus, $\alpha''$ can be deformed to a path in $LB$. The original path $\alpha$ is then homotopic in $\Emb(\mR, X)$ to the path $\alpha'*\alpha''$ where $\alpha''$ is a path entirely in $\LB$.
\end{proof}
\begin{proof}[Proof of Proposition \ref{P: EA paths}]
Given a representative $\alpha'\in [\alpha]$ and associated finger-Whitney system $(\mF, \mW)$ we get a system in EA by first performing a switching, then a sequence of disc slides as per Lemma~\ref{ia to ea}. By Lemma \ref{L: Factoring homotopy}, the switching can be realized by a homotopy. Disc slides are an $\mF\mW$-move and as such can be realized by a homotopy by performing an isotopy of the disc that is sliding over in the level set below the disc that is being slid over (see Figure \ref{F: Bdyint_implies_discslides}).
\end{proof}

\section{Proof of homotopy invariance}\label{homotopy invariance}

The goal of this section is to prove Theorem \ref{T:mainthm}, which we restate here for the reader's convenience:

\noindent\textbf{Theorem \ref{T:mainthm}.}  \emph{For each $k$ there exists a surjective homomorphism 
    \[ \mD : \pi_1(\Emb(\sqcup^k S^2, \#^k (S^2 \times S^2)), \mRs) \to \BZ_2 \]
    with $\pi_1(\LB,\mRs)$ contained in the kernel of $\mD$. Furthermore, there exists a $\beta \in \pi_1(\Emb( S^2, S^2 \times S^2)$ such that for all $k\geq 1$, after taking the natural inclusion into $\pi_1(\Emb(\sqcup^{k} S^2, \#^k (S^2 \times S^2))$, $\mD(\beta) = 1$.} 
    
In building toward Theorem \ref{T:mainthm}, the previous sections defined $\mD$ (Definition \ref{cross term final}) on a given finger/Whitney system and showed that it remains invariant under certain transformations of the systems: Disk slides, switching, clasping, special sum square moves (SSS-moves), boundary twisting and untwisted isotopy.  These transformations were used to show that $\mD$ remained invariant under the choices made when defining $\mD$. Then in \S11, we proved quite generally that for an oriented surface $\mR$ in a 4-manifold $X$ and a simply connected (but possibly not path connected) orientable surface $\mG$, any two finger first representatives $\alpha_1, \alpha_2$ for a given homotopy class $[\alpha]\in \pi_1\Emb(\mR, X^4)$, the associated finger/Whitney systems for these loops transformed by $\mF\mW$-moves: Isotopies of the finger/Whitney systems, Disc slides, Sphere slides, the Birth/Death move, the $x^3$-ed move, and the Saddle move. Applying the results of \S11 to $\mR = \sqcup_{i=1}^k S^2_i$, $X= \#^kS^2\times S^2$, and $\mG = \sqcup_{i=1}^k(S^2\times \{r\})_i$, what is left to show to finish proving Theorem \ref{T:mainthm} is two--fold: that $\mD$ is well defined on a given homotopy class, and that it is a group homomorphism. We will show that $\mD$ is well defined on the homotopy class of the loop first, then we will show it defines a group homomorphism.

\begin{notation}
We will denote by $(\mR, \mG, \mF, \mW)_s$ the 1-parameter family of finger/Whitney systems $(\mR_s, \mG, \mF_s, \mW_s)$ where restricting to any of $\mR_s$, $\mF_s$, $\mW_s$ gives a smooth isotopy of the object. When $\mR$ does not change, we will suppress both $\mR$ and $\mG$ from the notaion and simply write $(\mF, \mW)_s$. The same will be done for all invariants: $\mD(\mR, \mG, \mF, \mW)$ for a given finger/Whitney system or $\mD(\mF, \mW)$ when $\mR$ and $\mG$ are understood from context.
\end{notation}

The first case to consider is when a finger/Whitney system $(\mR, \mG, \mF, \mW)_s$ changes by an isotopy.

\begin{definition}
    Given two embeddings $\mR$ and $\mG$ of complementary dimension. Let $$\mE_{\pitchfork} = \{\mR \in \Emb_{\mR^{std}}(\sqcup^k_{i=1} S^2, \#_k S^2\times S^2): ~\mR\pitchfork \mG \}$$ denote the subspace of $\Emb(\mR, X)$ that consists of all the embeddings $\mR$ which are transverse to $\mG$. 
\end{definition}

\begin{lemma}\label{L: invariance by isotopy.}
    Let $\mR_s$ be a path of embeddings in $\mE_{\pitchfork}$ for $0\leq s\leq 1$ and let $(\mR, \mG, \mF, \mW)_s$ be a 1-parameter family of finger/Whitney systems for $\mR_s$, with both $(\mR_0,\mG, \mF_0, \mW_0)$ and $(\mR_1,\mG, \mF_1, \mW_1)$ untwisted. Then $\mD(\mR_0,\mG, \mF_0, \mW_0) = \mD(\mR_1,\mG, \mF_1, \mW_1)$.
\end{lemma}

Since the notion of ``untwisted'' depends on framings of the normal bundles of $\mR_s$ and $\mG$, the implication here is that these framings vary continuously with $s$.

Note that we use the parameter $s$ for this path because we are thinking of the $1$--parameter family of finger-Whitney systems as coming from a $1$--parameter family $\mR_{s,t}$ of paths of red embeddings, with the paths of red embeddings parameterized by the parameter $t$. From this perspective, our $\mR_s$ is really $\mR_{s,1/2}$, but from a narrower perspective this lemma is just about finger-Whitney systems and not intrinsically about paths of red embeddings so we suppress the $t$.

\begin{proof}

We will relate $(\mR_0,\mG, \mF_0, \mW_0)$ to $(\mR_1,\mG, \mF_1, \mW_1)$ through a sequence of intermediary {\em untwisted} finger-Whitney systems connected by $1$--parameter families as follows:
\begin{enumerate}
    \item First we will go from $(\mR_0,\mG, \mF_0, \mW_0)$ to $(\mR_1,\mG, \mF'_0, \mW'_0)$ via a $1$--parameter family $\phi_s(\mR_0,\mG, \mF_0, \mW_0)$, where $\phi_s: \sqcup^k_{i=1} S^2 \to \sqcup^k_{i=1} S^2$ is an ambient isotopy with $\phi_0 = \id$ and $\phi_1(\mR_0) = \mR_1$, fixing $\mG$ setwise. Thus $\mF'_0 = \phi_1(\mF_0)$ and $\mW'_0 = \phi_1(\mW_0)$. Since all the data involved in computing $\mD$ is moving by an ambient isotopy it is obvious that $\mD(\mR_0,\mG, \mF_0, \mW_0) = \mD(\mR_1,\mG, \mF'_0, \mW'_0)$.
    \item Next we will go from $(\mR_1,\mG, \mF'_0, \mW'_0)$ to $(\mR_1,\mG, \mF'_1 = \mF''_0, \mW'_0)$ via a $1$--parameter family $(\mR_1,\mG, \mF'_s, \mW'_0)$, so that only the finger disks $\mF'_s$ change with $s$. Furthermore, this isotopy of finger disks is supported in a neighborhood of $\mR_1 \cup \mG$ and in this neighborhood is precisely a sequence of boundary twists, i.e. a {\em reuntwisting}. Thus, by Corollary~9.4, we know that $\mD(\mR_1,\mG, \mF'_0, \mW'_0) = \mD(\mR_1,\mG, \mF''_0, \mW'_0)$.
    \item Next we will go from $(\mR_1,\mG, \mF''_0, \mW'_0)$ to $(\mR_1,\mG, \mF''_1 = \mF_1, \mW'_0)$ via a $1$--parameter family $(\mR_1,\mG, \mF''_s, \mW'_0)$, so that again only the finger disks $\mF''_s$ change with $s$. This time, $\mF''_s$ remains finger germed for all $s$, so that this is an {\em untwisted isotopy}. Again, as stated in Corollary~9.4, untwisted isotopies do not change $\mD$ so $\mD(\mR_1,\mG, \mF''_0, \mW'_0) = \mD(\mR_1,\mG, \mF_1, \mW'_0)$.
    \item Finally we repeat the last two steps with Whitney disks instead of finger disks: We go from $(\mR_1,\mG, \mF_1, \mW'_0)$ to $(\mR_1,\mG, \mF_1, \mW'_1=\mW''_0)$ via a family $(\mR_1,\mG, \mF_1, \mW'_s)$ via a Whitney disk reuntwisting, and then we go from $(\mR_1,\mG, \mF_1, \mW''_0)$ to $(\mR_1,\mG, \mF_1, \mW''_1 = \mW_1)$ via $(\mR_1,\mG, \mF_1, \mW''_s)$ via an untwisted isotopy of the Whitney disks, and again neither of these changes $\mD$.
\end{enumerate}
Once  we explain how to construct these intermediate $1$--parameter families we will thus have succeeded in showing that $\mD(\mR_0,\mG, \mF_0, \mW_0) = \mD(\mR_1,\mG, \mF_1, \mW_1)$.

To prepare for the first step, we need to produce the ambient isotopy $\phi_s$ such that $\phi_s(\mR_0) = \mR_s$ and $\phi_s(\mG) = \mG$. Standard isotopy extension gives us the first condition but not necessarily the second, but since the proof of the isotopy extension theorem involves flowing along a vector field, we just need a vector field on $[0,1] \times \sqcup^k_{i=1} S^2$ which tracks the motion of $\mR_s$ and stays tangent to $[0,1] \times \mG$. This is essentially the same as an $\mF\mW$--vector field, and can easily be produced using the same ideas as in the proof of existence of $\mF\mW$--vector fields.

Equipped with the ambient isotopy $\phi_s$, we now have the $1$--parameter family $\phi_s(\mR_0,\mG, \mF_0, \mW_0)$ and simply declare $\mF'_0$ be to $\phi_1(\mF_0)$ and $\mW'_0$ to be $\phi_1(\mW_0)$. This completes the first step.

To prepare for the remaining steps, note that the concatenation of the reverse of the above family, i.e. $\phi_{1-s}(\mR_0,\mG, \mF_0, \mW_0)$, with the original $1$--parameter family $(\mR_s,\mG, \mF_s, \mW_s)$ is a $1$--parameter family going from $(\mR_1,\mG, \mF'_0, \mW'_0)$ to $(\mR_1,\mG, \mF_1, \mW_1)$ in which the red embedding starts at $\mR_1$, goes to $\mR_0$ via $\phi_{1-s}(\mR_0)$ and then returns to $\mR_1$ via the original family $\mR_s$. But $\phi_{1-s}(\mR_0) = \mR_{1-s}$ so in fact the path that the red embedding takes from $\mR_1$ back to $\mR_1$ is homotopic to the constant path at $\mR_1$. Thus, by a suitable application of parameterized isotopy extension (again using flow along an $\mF\mW$--vector field to keep $\mG$ fixed setwise) we can homotope our concatenated $1$--parameter family to obtain a $1$--parameter family $(\mR_1,\mG,\tilde{\mF}_s,\tilde{\mW}_s)$ going from $(\mR_1,\mG, \mF'_0, \mW'_0)$ to $(\mR_1,\mG, \mF_1, \mW_1)$ without moving $\mR_1$ or $\mG$. Furthermore, since the finger and Whitney disk isotopies are completely independent, we can easily arrange that all of the finger disk isotopies happen first and then all of the Whitney disk isotopies happen. Thus to justify steps (2) and (3), we suppose we have a $1$--parameter family $(\mR_1,\mG,\tilde{\mF}_s,\mW'_0)$ where $\tilde{\mF}_s$ is an isotopy of the finger disks going from $\tilde{\mF}_0 = \mF'_0$ to $\tilde{\mF}_1 = \mF_1$. We want to homotope this family to one in which first we see boundary twists and then we see an untwisted isotopy. Once we have explained how to do this, the exact same argument applies to the Whitney disks to justify step (4).

To simplify notation now, we ignore the Whitney disks and drop unnecessary subscripts and decorations and consider the following situation: $\mR$ and $\mG$ are fixed and we have a $1$--parameter family $\mF_s$ of finger disks which are untwisted when $s=0$ and $s=1$. We can clearly deal with each finger disk independently, so consider just one finger disk $f_s$, with boundary arcs $\partial_R f_s = f_s \cap \mR$ and $\partial_G f_s = f_s \cap \mG$. With respect to the fixed, given trivialization of the normal bundle of $\mR$, the germ of $f_s$ along $\partial_R f_s$ gives an $s$--parameterized family of maps $\gamma_s: [0,1] \to S^1$ with $\gamma_0$ and $\gamma_1$ being the constant map sending $[0,1]$ to the basepoint $0 \in S^1 = \mathbb{R}/2\pi \mathbb{Z}$. This $s$--parameterized family is thus homotopic to one for which $\gamma_s$ is constant for each $s$, i.e. $\gamma_s([0,1]) = \theta(s) \in \mathbb{R}/2\pi \mathbb{Z}$, and furthermore we can assume $\theta(s)$ is a linear function $\theta(s) = N_R 2\pi s$ where for some red reuntwisting number $N_R$. Similarly we can record a green reuntwisting number $N_G$. Note that if $N_R=0$ and $N_G=0$ then the $1$--parameter family $f_s$ can be isotoped to a family which remains untwisted for all $s$, using a parameterized isotopy extension argument similar to that in the preceding paragraph. Thus if we form the concatenation of three $1$--parameter families, the first doing $N_R$ standard red twists followed by $N_G$ standard green twists, the second reversing the first, and the third our given family $f_s$, we see that the concatenation of the second and the third has $0$ reuntwisting numbers for red and green and can be isotoped to an untwisted isotopy. In the end we get a new family $f_s$ which first does standard red followed by standard green twists and then does an untwisted isotopy. Doing this for all the finger disks completes the proof.

\end{proof}

\subsection{Finger-first representatives.}
Suppose that we are given an arbitrary representative $\alpha$ of a class $[\alpha]\in \pi_1(\Emb(\sqcup^k_{i=1} S^2, \#_k S^2\times S^2), \mR^{std})$. In Section \ref{S: paths ff position} we gave a recount of Quinn's procedure for producing a homotopy from $\alpha$ to a path $\alpha'$ that is in finger-first position relative to $\alpha_0$ and $\alpha_1$. Now suppose that we have an arbitrary relative homotopy $H:[0,1]^2\rightarrow \Emb(\sqcup^k_{i=1} S^2, \#_k S^2\times S^2)$ with $H_0(t)= \alpha(t)$ and $H_1(t) = \alpha''(t)$ such that $\alpha''(t)$ is in finger-first position as well. By Theorem \ref{T: 2-par ordering}, we may assume that $\widetilde{H}$ is an ordered homotopy since both $\alpha' = \widetilde{H}_0$ and $\alpha'' = \widetilde{H}_1$ are finger-first paths. We let $(\mR,\mG,\mF,\mW)_i$ denote the finger/Whitney system for $\widetilde{H}_i(t)$ for $i = 0,1$. To prove that $\mD$ is well defined independent of the finger-first representative chosen, it suffices to show that $\mD((\mR,\mG,\mF,\mW)_i) = \mD((\mR,\mG,\mF,\mW)_i)$. By Theorem \ref{T: fw system moves}, we know that $(\mR,\mG,\mF,\mW)_0$ differs from $(\mR,\mG,\mF,\mW)_1$ by a finite sequence $\mF\mW$-moves. In between the $\mF\mW$ moves are finger-Whitney systems $(\mR,\mG,\mF,\mW)_s$ such that $\mR_{s,1/2}$ varies in $\mE_{\pitchfork}$. It follows from Lemma \ref{L: invariance by isotopy.} that we only need to show that the $\mF\mW$-moves do not change $\mD$. By Corollary~\ref{disc slide invariance}, it remains to show $\mD$ is invariant under Sphere slides, Birth/Death, $x^3$, and saddle moves.
\subsubsection{Invariance under the Birth/Death move.}

Recall the birth move transforms a given system $(\mR, \mG, \mF, \mW)$ to a system $(\mR', \mG, \mF', \mW)$ as follows: First, $\mR'$ is obtained from $\mR$ by performing a finger move along a local arc $\gamma$, creating two additional intersection points between $\mR_i$ and $\mG_i$, and with a Whitney disc $d$ that undoes this finger move. The disc systems $\mF$ and $\mW$ differ from $\mF'$ and $\mW'$ respectively by adding the same disc $d$ to both; there is a copy of $d$ in $\mF'$ and a copy in $\mW'$ (see Lemma~\ref{L: b/d move}). Let us refer to the added discs as $f_d$ and $w_d$.

\begin{definition}
    Let $(\mR', \mG, \mF', \mW')$ differ from $(\mR, \mG, \mF, \mW)$ by a birth move.  Given a set of switch discs $\mW^*$ for $(\mR, \mG, \mF, \mW)$, we say $(\mR', \mG, \mF', \mW')$ has a compatible set of switch discs ${\mW'}^*$ if there exists a set of switch discs ${\mW'}^*$ for $(\mR', \mG, \mF', \mW')$ such that, as framed embedded Whitney discs, $\mW^* \subset {\mW'}^*$.
\end{definition}

\begin{definition}\label{D: cross_uncross birthmove}
    Let $\mR_t$ denote a path of embeddings. A cross birth move is a birth move between $R_i$ and $G_j$ for some $i \neq j$. An uncrossed birth move is one that involves $R_i$ and $G_i$ for some $1\leq i\leq k$. A cross (uncrossed) death move is a cross (uncrossed) birth move in reverse.
\end{definition}

\begin{lemma}\label{L: compatibleswitch}
    Let $(\mR, \mG, \mF, \mW)$ be a system with $\mW^*$ a set of switch discs for $(\mR, \mG, \mF, \mW)$. Suppose that $(\mR', \mG, \mF', \mW')$ is obtained from $(\mR, \mG, \mF, \mW)$ by a birth move. Then there is a compatible set of switch discs ${\mW'}^*$ for $(\mR', \mG, \mF', \mW')$
\end{lemma}
\begin{proof}
    First, let us assume that $\mF$ is finger germed and $\mW$ is Whitney germed. We break the proof up into two cases; cross and uncrossed birth moves. 

    Suppose that $(\mR', \mG, \mF', \mW')$ is obtained by a cross birth move. Let $\gamma$ denote the guiding arc for the birth move. Since $\gamma$ is an arc, we can assume that it will be disjoint from $\mF\cup \mW\cup \mW^*$ after performing a small isotopy of the discs if needed. In fact, every disc in $\mF\cup \mW\cup \mW^*$ can be assumed to be contained in the complement of a local neighborhood containing the birth move. Therefore we may perform the birth move, keeping all these other discs unchanged. The discs $\mW^*$ almost form a switch system for $\mW'$, except we need to add $w_d$ to the system since this is a crossed birth. Thus, ${\mW'}^* = \mW^* \cup \{w_d\}$ is a compatible switch system.

    We now suppose that $(\mR', \mG, \mF', \mW')$ is obtained by an uncrossed birth move. Suppose that $\mR'$ is obtained from $\mR$ by performing a finger move between $R_i$ and $G_i$ along the local curve $\gamma$. Since $\mW^*$ is a set of switch discs, we know that after replacing $\mW$ with $\mW^*$, the point $p_0\in R_i\cap G_i$ will only be paired off by $f_{ii,1}\in \mF$. That is $f_{ii,1}$ is the first finger disc in the induced IA ordering. We choose a small neighborhood $U$ of $p_0$ so that the corner of $f_{ii,1}$ is the only disc in $U$. In this neighborhood, we choose an arc $\gamma'$ so that performing a birth move along $\gamma'$ results in a new system $(\mR'', \mG, \mF'', \mW'')$ where the local change to $\mR$ and the finger and Whitney discs added to $\mF$ and $\mW$ is shown in Figure \ref{F: localstdbirth}. Let us denote these discs created by the birth move along $\gamma'$ by  the discs $f_{loc}$ and $w_{loc}$ . After performing the local birth along $\gamma'$ there is an obvious disc $w_{loc}^*$ (shown in Figure \ref{F: local birth}) that when added to $\mW^*$, gives a compatible set of switch discs for $(\mR'', \mG, \mF'', \mW'')$.

    The two arcs $\gamma'$ and $\gamma$ are isotopic by an isotopy $\gamma_z$, $z\in [0,1]$, keeping the end points on $R_i\cup G_i$ and interior in $X\setminus \mR\cup \mG$. The isotopy $\gamma_z$ then induces an isotopy $\mR_z$ between $\mR_0 = \mR'$ and $\mR_1 = \mR''$ by performing the finger move along $\gamma_z$ for every $z$, then isotoping the solid finger just as $\gamma_z$. Moreover, since $R_i\setminus \mW$ and $G_i\setminus \mW$ are both connected and $X\setminus (\mR\cup \mG\cup\mW)$ is simply connected, we can choose this isotopy so that it remains disjoint from $\mW$. Furthermore, as $\mW^*\cup \mW$ is a switch system, the discs in $\mW^*$ are either the same as those in $\mW$ or only intersect $\mW$ in $\mR'\cap\mG$. For the discs in $\mW^*$ that differ from the discs in $\mW$, they have dual spheres disjoint from the discs in $\mW$. We can use these spheres to modify the isotopy of $\gamma_z$ if necessary to make it disjoint from the switch system $\mW^*$. Therefore, we can assume that the isotopy remains completely disjoint from $\mW^*$ as well. The isotopy $\gamma_z$ might not remain disjoint from the discs in $\mF$. Consequentially, we can extend the discs $\mW$ and $\mW^*$ to a constant 1-parameter family of Whitney discs for $\mR_z$, but not $\mF$.
    
    To extend $\mF$, we must include disc slides and sphere slides. These occur at moments during the isotopy of $\gamma_z$ when the arc intersects one of the disc in $\mF$. Pulling the slides back to $\mR'$, we find that the discs in $\mF$ will move by a isotopy until a slide occurs.  The slide is always of the type of the disc $f\in \mF$ slide over $f_{d}$. After all the slides occur, we are left with a new configuration $\mF_{\text{after slides}}$ for $\mR'$. This system of discs can be extended to a 1-parameter family, thus giving a 1-parameter family of finger/Whitney systems.

    Using the constructed 1-parameter family of finger/Whitney systems, we get a switch system for $(\mR', \mG, \mF_{\text{after slides}}, \mW')$ by pulling back $\mW^* \cup w_{loc}^*$ along the 1-parameter family. Let $w_{d}^*$ the image of $w_{loc}^*$ after pulling back. To get the compatible switch system for $(\mR', \mG, \mF', \mW')$ we just reverse all the slides and isotopy of $\mF_{\text{after slides}}$. Since this does not change the discs $\mW'$ or $\mW^* \cup \{w_d^*\}$, we get a compatible switch system. 
\end{proof}

\begin{figure}
    \centering
    \labellist                             
            \hair 10pt
            \pinlabel $f_{ii,1}$ at 50 400
            \pinlabel $f_{ii,1}$ at 770 400
            \pinlabel {$f_{loc} \cup w_{loc}$} at 1210 370
            \pinlabel {$w^*_{loc}$} at 1080 200
            \pinlabel {$\gamma'$} at 460 100
        \endlabellist
    \includegraphics[width=\linewidth]{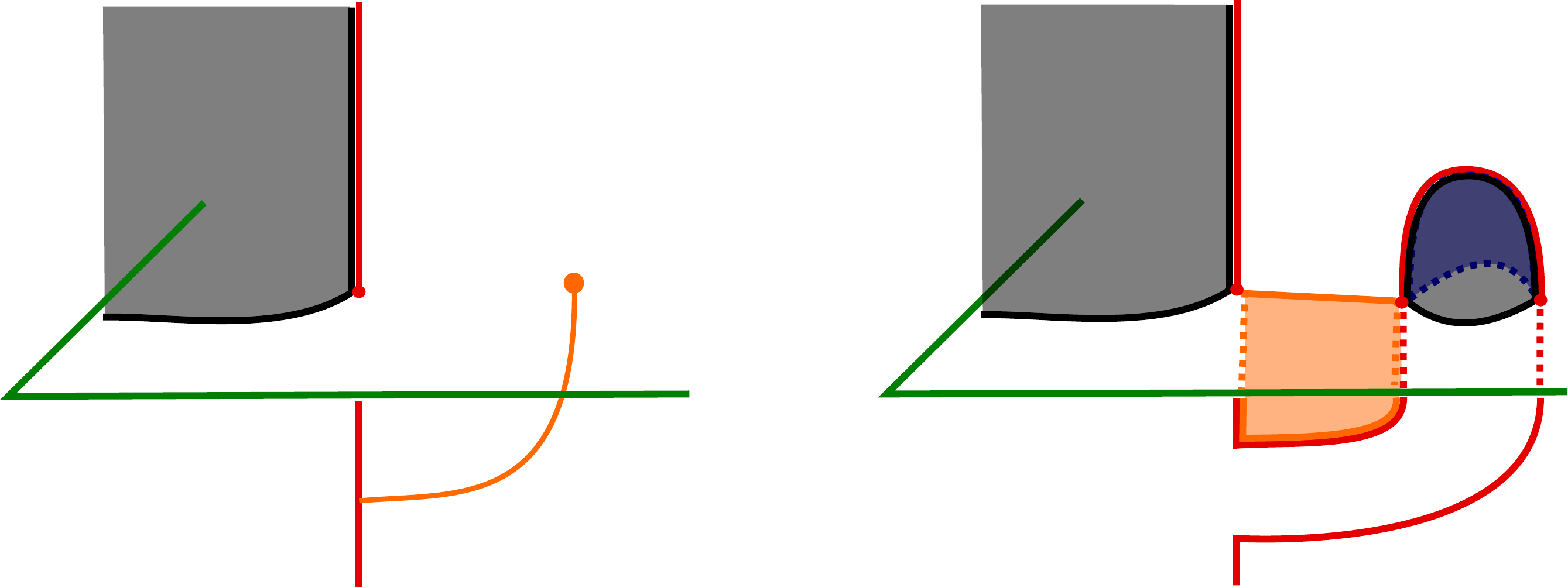}
    \caption{ Shown here is an uncrossed birth move localized to a corner of a Finger disc. On the left, we have the guiding arc $\gamma'$ which is used to perform a finger move of $R_i$ through $G_i$. On the left, we have the result, together with two copies of the the Whitney disc that reverses the finger move. These discs are labeled $f_{loc}$ and $w_{loc}$. Also shown on the right is a local ``switch disc'' $w^*_{loc}$.} \label{F: localstdbirth}
\end{figure}

\begin{remark}\label{R: localstdbirth}
    In constructing the compatible set of switch discs for the uncrossed birth, we used a local model (Figure \ref{F: localstdbirth}) for generating a compatible switch system $\mW^*\cup\{w_{loc}^*\}$, then using sphere and discs slides of the finger discs, together with an isotopy of the data to get the compatible system. Therefore, by reversing this process, we can assume that our birth occurs in the local model, at the expense of finger disc and sphere slides and an isotopy. 
\end{remark}

\begin{lemma}\label{L: inv under cross bd}
    Let $(\mR', \mG, \mF', \mW')$ be obtained from $(\mR, \mG, \mF, \mW)$ by a cross birth move. Then $\mD(\mR, \mG, \mF, \mW) = \mD(\mR', \mG, \mF', \mW')$
\end{lemma}
\begin{proof}
    First, we compute $c_{i,j}(\mF, \mW)$. In the definition of $c_{i,j}$ one must 2-color $\mR_i$ and $\mG_j$ depending on the boundary configurations of the cross discs. If the birth move creates $\mR_i$-$\mG_j$ intersections, then it may alter the two coloring. This can happen when the birth adds the two intersections in a white region. The birth move will add the cycles $(f_d\cup w_d)\cap G_j$ and $(f_d\cup w_d)\cap R_i$ to the region where the two points were added. If the region in $G_j$ or $R_i$ was a white region, then $B_{j,i}^\mG$ is modified to include the region bounded by the cycle $(f_d\cup w_d)\cap G_i$ and similarly $B_{i,j}^\mR$. If either region was a black region, then the black region does not change. This is because the cycles $(f_d\cup w_d)\cap G_j$ and $(f_d\cup w_d)\cap R_i$ always have a 0 winding number. In fact, the region bounded by both cycles never encloses $\mR\cap \mG$ intersections or contains the point at infinity because $f_d = w_d$. From this fact and the definition we deduce that the quantity $c_{i,j}$ remains unchanged by a birth move. 

    Next, we choose a compatible switch system $\mW^*$ for $(\mR, \mG, \mF, \mW)$ so that $(\mF, \mW^*)\in IA$. The previous argument then shows that for all $i \neq j$, $c_{i,j}$ is unchanged. Finally, since the discs $f_d$ and $w_d$ are completely disjoint from all other discs in $\mF'$ and $\mW'$, the quantities $CU_{i,j}$ remain unchanged with the addition of $f_d$ and $w_d$. Therefore, $\mC$ is invariant. Finally, we note that the $(i,i)$ discs for $(\mF, \mW)$ are exactly the same discs for $(\mF',\mW')$. It follows from the Multi eye parity lemma \ref{parity} that $\I$ remains invariant. This completes the proof of the claim.
\end{proof}

\begin{lemma}\label{L: inv under uncross bd}
    Let $(\mR, \mG, \mF, \mW)$ and $(\mR', \mG, \mF', \mW')$ be two systems that differ by an uncrossed birth move. Then $\mD(\mR, \mG, \mF, \mW) = \mD(\mR', \mG, \mF', \mW')$.
\end{lemma}
\begin{proof}
    Let $(\mR', \mG, \mF', \mW')$ be obtained from $(\mR, \mG, \mF, \mW)$ by an uncrossed birth. Suppose this birth happens between $R_u$ and $G_u$ for some $1\leq u\leq k$. Choose an ordering, a switch system $\mW^*$ and let ${\mW'}^*$ be the compatible switch system given by Lemma \ref{L: compatibleswitch}. By invariance under switching (Proposition~\ref{sum invariance}), we know that $\mD(\mF, \mW) = \mD(\mF, \mW^*)$ and $\mD(\mF', \mW') = \mD(\mF', {\mW'}^*)$. So it suffices to show $\mD(\mF, \mW^*) = \mD(\mF', {\mW'}^*)$.
    
    First, it follows from the definition that $c_{i,j}(\mF, \mW^*) = c_{i,j}(\mF',{\mW'}^*)$ for all $i,j$, since the birth added uncrossed discs and switching leaves the cross discs unchanged. Next, we note that by our choice of switch disc $w^*_{d}$, we have that $f_d = f_{uu,1}$ in the IA ordering for $\mF'_{uu}, {\mW'}^*_{uu})$. Since $f_{uu,1}$ is disjoint from all the discs in ${\mW'}^*$ we have that $a'_1= |(\mW'_{i,j}\cup \mW'_{j,i})\cap f_{uu,1}\cap G_u| = 0$. Furthermore, $\mF' = \mF$ and $\mW^* = {\mW'}^*$ for all discs other than $f_{uu,1}$ and $w_{uu,1}=w^*_{d}$. Therefore we have
    \begin{align*}
        CU^u_{i,j}(\mF', {\mW'}^*) &= \sum_{r\leq s}a'_r b'_s\\
        &=\sum_{r\leq s}a_r b_s\\
        &= CU^u_{i,j}(\mF, \mW^*).
    \end{align*}
    The above equality holds for all $u' \neq u$ trivially. Therefore $CU_{i,j}(\mF, \mW^*) = CU_{i,j}(\mF', {\mW'}^*)$ for all $i,j$. Thus $\mC(\mF, \mW^*) = \mC(\mF', {\mW'}^*)$.

    Now by Lemma~\ref{cu vs cij}, we have that $C$ does not change. Furthermore, as the finger sphere for $f_{uu,1}$ is null homotopic and no cross disc pairs off either intersection point paired by $f_{uu,1}$, $c_{i,j}$ is not changed by sphere slides over $f_{uu,1}$. Therefore we can perform all the finger disc slides and sphere slides used in the construction of ${\mW'}^*$ and the isotopy of discs to standardize the birth and this process does not change $C$. Thus we have $C(\mF, \mW^*) = C(\mF', {\mW'}^*) = C(\mF \cup \{f_{loc}\}, \mW^*\cup \{w^*_{loc}\})$.

    It is clear that $\hat{\I}(\mF', \mW^*) = \hat{\I}(\mF'\cup \{f_{loc}\}, \mW^*\cup \{w^*_{loc}\})$. Furthermore, this quantity is not changed by the isotopy. Since $f_{loc}$ has no boundary intersections (on either $\mR$ or $\mG$), disc slides over $f_{loc}$ does not change $\hat{\I}$. Finally, since $\hat{\I}$ only counts $|f_i\cap w_j|$ for $i\leq j$, $w^*_{loc} = w_{1}$ and $f_1 = f_{loc}$ in the IA ordering, the sphere slides over $f_{loc}$ do not change $\hat{\I}$ either. Therefore we have that $\hat{\I}(\mF', {\mW'}^*) = \hat{\I}(\mF\cup \{f_{loc}\}, \mW^* \cup \{w^*_{loc}\}) = \hat{\I}(\mF, \mW^*)$.

    Now we apply a sequence of disc slides as done in Lemma~\ref{ia to ea} to $(\mF, \mW^*)$ to get it into $EA$. We want to uses these same slides for $(\mF', {\mW'}^*)$, so to do this we first prepare the discs $(\mF', {\mW'}^*)$ by doing all the sphere slides first. By the Multi eye parity lemma \ref{parity}, $\I(\mF', {\mW'}^*) = \I(\mF'', {\mW'}^*)$ where $\mF''$ is the set of discs obtained by performing all the sphere slides first. So we can start with $(\mF'', {\mW'}^*)$ then apply the finger sphere slides which, after an  isotopy, gets us to $(\mF\cup \{f_{loc}\}, W^*\cup \{w^*_{loc}\}$. Since we can transport the disc slides along the isotopy, it follows that $\I(\mF'', {\mW'}^*)=\I(\mF\cup \{f_{loc}\}, W^*\cup \{w^*_{loc}\}$, so we will work instead with the latter set of discs. 
    
    Now we can apply the disc slide sequence to the discs in $(\mF, \mW^*)$ as Lemma~\ref{ia to ea}. Since $(\mF, \mW^*)\in \IA$ it has an IA ordering of the discs. Furthermore, $(\mF\cup \{f_{loc}\}, \mW^*\cup \{w^*_{loc}\}\in \IA$ as well with $f_{loc} = f_{uu,1}$ and $w_{loc} = w_{uu,1}$. We will use the indexing for $(\mF\cup \{f_{loc}\}, \mW^*\cup \{w^*_{loc}\}$ for $(\mF, \mW^*)$, i.e., instead of starting at 1, the discs will be index starting from 2 and ending at $n+1$ where $|\mF_{uu}| = n$.

    We recall the procedure outline in the proof of Lemma~\ref{ia to ea}. This process begins with a sequence of $\mW$ $\mG$-disc slides followed by swinging the boundary around $G_u$ remove the intersections with $f_{uu,2}$. The only way for the $\mW$ disc slides to create a boundary intersection with $f_{uu,1}$ is to slide over $w_{uu,1}$. But the above sequence has none of these slides. Thus, after these initial slides and boundary isotopy, we have sucessfully removed the $G_u \cap \mW_{uu}$ intersections with $f_{uu,2}$. Note that if any of these disc slides changed $\hat{\I}(\mF, \mW^*)$, they changed $\hat{\I}(\mF\cup f_{loc}, \mW^* \cup w_{loc})$ in exactly the same way. Therefore we have that after this first sequence of slides, $\hat{\I}(\mF, \mW^*)=\hat{\I}(\mF\cup f_{loc}, \mW^* \cup w_{loc})$.\footnote{To save on notation, we are not renaming the disc systems after each sequence of slides, even though the discs are different.}  
    
    Now for $(\mF, \mW^*)$, we have cleared $f_{uu,2}$ of all $\mW_{uu}$ intersections and for $(\mF\cup f_{loc}, \mW^* \cup w_{loc})$ we have $(f_{uu,1} \cup w_{uu,1}\cup f_{uu,2})\cap G_u$ forms an embedded arc. We now remove all the $\mF_{uu}$ intersections with $w_{uu,2}\cap G_u$ by sliding them all over $f_{uu,2}$. For  $(\mF, \mW^*)$, that is the last step but for $(\mF\cup f_{loc}, \mW^*\cup w_{loc})$, we must also slide them over $f_{uu,1}$. Since $f_{uu,1}\cap R_{u}$ is free of all boundary intersections with the discs in $\mW_{uu}$, these extra slides do not effect $\hat{I}$. Therefore the only slides that might change $\hat{\I}$ are the slides over $f_{uu,2}$. Therefore we have $\hat{\I}(\mF, \mW^*)=\hat{\I}(\mF\cup f_{loc}, \mW^* \cup w_{loc})$ after completing this set of slides. Continuing by induction, we get that after each step $\hat{\I}(\mF, \mW^*)=\hat{\I}(\mF\cup f_{loc}, \mW^* \cup w_{loc})$. It then follows that $$\I(\mF, \mW^*)=\I(\mF\cup f_{loc}, \mW^* \cup w_{loc}).$$

    We have thus shown that both $\mC$ and $\I$ remain unchanged. Therefore $\mD$ is unchanged by an uncrossed birth. 
\end{proof}

\begin{remark}
    Having proved Lemma~\ref{L: inv under uncross bd}, we have completed the proof of Proposition~\ref{symmetry1}. Therefore, we are free to use that $\mD(\mF,\mW) = \mD(\mW, \mF)$.
\end{remark}

\begin{lemma}
    Let $(\mF, \mW)$ and $(\mF, \mW')$ be two finger-Whitney systems that differ by a sphere slide. Then $\mD(\mF,\mW) = \mD(\mF', \mW')$
\end{lemma}
\begin{proof}
    We will assume that the two finger/Whitney systems differ by a finger sphere slide. The proof for Whitney discs is given by swapping the roles of fingers and Whitney's in the proof below. As we will switch the Whitneys to compute $\hat{I}$, this then relies on the fact that $\mD(\mF, \mW) = \mD(\mW, \mF)$. 

    First, recall some basic facts about finger spheres. First, a finger sphere $S_f$ for the finger move $f$ is embedded in the complement of $\mR\cup \mG$ and is null homologous in $\#^k S^2\times S^2$. Let $p_0, p_1\in \mR\cap \mG$ denote the intersection points that $f$ pairs off. If $p_0$ and $p_1$ are paired off by distinct Whitney discs $w_0$ and $w_1$, then $S_f \cdot w_i = 1 \mod 2$ for both $w_0$ and $w_1$ and for all other Whitney discs $w'$, $S_f \cdot w' = 0 \mod 2$. If $p_0$ and $p_1$ are paired off by a unique Whitney disc $w$, then $S_f\cdot w' =0\mod 2$ for all Whitney discs $w' \in \mW$.

    It follows from the basic facts of finger spheres and the definition of $c_{i,j}$ that for any $i\neq j$, $c_{i,j}(\mF, \mW) = c_{i,j}(\mF', \mW)$ when $\mF$ differs from $\mF'$ by a sphere slide. Since a sphere slide does not change the boundary configuration of any of the discs, $CU_{i,j}$ does not change under a sphere slide. Therefore $\mC(\mF, \mW) = \mC(\mF',\mW)$.

    It remains to show that $\I$ does not change under sphere slides. We choose an ordering on $\mW$ to switch. Note that switching changes the Whitney discs. Thus, if $\mW^*$ is a switch system for $\mW$ that takes $(\mF, \mW)$ to IA, then the same system takes $(\mF', \mW)$ to IA  as well. So we may assume that both systems are in IA. Now it follows that both systems now satisfy the hypothesis for the Multi-eye parity lemma \ref{parity}. Therefore $\I(\mF, \mW) = \I(\mF', \mW)$ if and only if $\hat{\I}(\mF, \mW)  = \hat{\I}(\mF', \mW)$. We now show that $\hat{\I}$ does not change under sphere slides.
    
    Now since both $(\mF, \mW)$ and $(\mF', \mW)$ are in IA, the finger discs in $\mF$ have an induced order. First, let us assume that $\tilde{f}$ is sliding over $f$. Let us suppose that  $f$ pairs off points in $\mR_i\cap \mG_j$ and $\tilde{f}$ pairs points in $\mR_k\cap \mG_l$, then $\hat{\I}(\mF, \mW) = \hat{\I}(\mF', \mW)$  if $(i,j) \neq (l,k)$ or if $i \neq j$ since, by definition, it only counts intersections between fingers and Whitneys that belong to the same immersed arc. Therefore, we may assume that $f$ and $\tilde{f}$ pair off intersections between $\mR_i$ and $\mG_i$. We will suppress the index of the spheres $\mR_i$ and $\mG_i$ and use subscripts to index the discs according to the IA condition. Let $f = f_i$ and $\tilde{f}=f_{j}$ for some $1\leq i,j \leq n$. Since $f_{j}$ is sliding over $f_i$, we have the discs in $\mF'$ are given by the formula
    \begin{align}
        &f'_k = f_k ~\text{ if }k\neq j\\
        &f'_j = f_j+S_{f_i}.
    \end{align}
    
    \begin{enumerate}
        \item[Case 1:] Let $j>i$. Then $\langle f'_j, w_l\rangle = \langle f_j, w_l\rangle + S_{f_i} \cdot w_l$. By the basic facts of Finger spheres we have $S_{f_i} \cdot w_l = 0$ for $l\geq j$. Therefore $\langle f'_j, w_l\rangle = \langle f_j, w_l\rangle \mod 2$ and $\hat{\I}$ remains unchanged.
        
        \item[Case 2:] Let $j< i$. Then $j\leq i-1$ and $i$. Therefore, $S_{f_i}\cdot w_{i-1} = 1 \mod 2$ and $S_{f_i}\cdot w_{i} = 1 \mod 2$. Consequentially we have $\langle f'_j, w_{i-1}\rangle = \langle f_j, w_{i-1}\rangle + 1\mod 2$ and $\langle f'_j, w_{i}\rangle = \langle f_j, w_{i}\rangle + 1\mod 2$. Hence $\hat{\I}((\mF, \mW))) =\hat{\I}((\mF', \mW))$.
    \end{enumerate}
    In all cases, we have shown that  $\hat{I}((\mF, \mW)) = \hat{I}((\mF',\mW))$. It then follows from the Multieye Parity Lemma \ref{parity} that $\I(\mF, \mW) = \I(\mF',\mW)$. We can now conclude that $\mD(\mF, \mW) = \mD(\mF', \mW)$.
\end{proof}

\begin{lemma}\label{L: x^3 equals bd plus SSS}
    If $(\mF', \mW')$ is obtained from $(\mF, \mW)$ by performing a $x^3_+$-move, then it can also be obtained from $(\mF, \mW)$ by performing a birth move followed by an SSS-move or a 1-switch.
\end{lemma}
\begin{proof}
    The proof is contained in Figures \ref{F: local_x3}, \ref{F: local birth}, and \ref{F: local_SSS}, with the finger discs are shown in black while the Whitney discs are shown in blue. Figure \ref{F: local_x3} shows the local change in $(\mF, \mW)$ after passing an $x^3_+$-move (Definition \ref{D: x^3}). Here, each column of figures depicts the 4-dimensional neighborhood where the transformation occurs, with each figure in a column being a 3-dimensional slice. In Figure \ref{F: local birth} we see the transformation of performing a local Birth move. Figure \ref{F: local_SSS} then shows how applying an SSS move to the finger disks after the birth move changes the $\mF\mW$-system locally. Notice that the left most column of Figure \ref{F: local birth} is in agreement with the left most column of Figure \ref{F: local_x3} and the right most column of Figure \ref{F: local_SSS} is in agreement with the right most column of Figure \ref{F: local_x3}. This shows that an $x^3_+$ move can be realized by a Birth move followed by an SSS move when the data of $(\mF, \mW)$ is in Case 3 of the $x^3$-move.

    We get the remaining cases for the $x^3$ move by removing the starting finger or Whitney data (or both) from Figures \ref{F: local_x3}, \ref{F: local birth}, and \ref{F: local_SSS}. Doing so, the orange disc in Figure \ref{F: local_SSS} can instead be used for a 1-switch instead of. 

\begin{figure}[!htbp]
    \centering
    \includegraphics[width=0.65\linewidth]{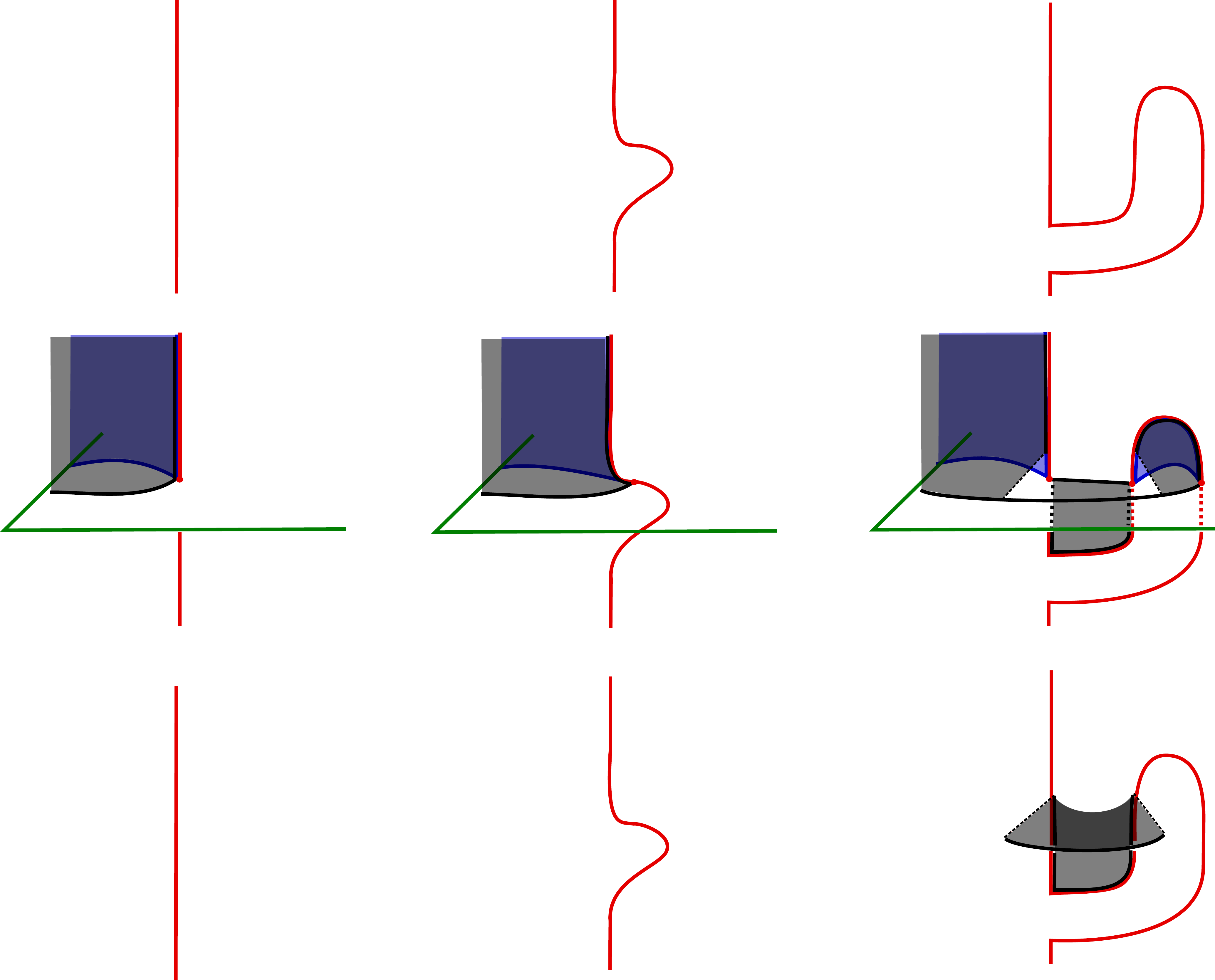}
    \caption{ Show here is the local neighborhood where the $x^3_+$ move takes place; each column of figures is the local 4-dimensional neighborhood.} \label{F: local_x3}
\end{figure}

    \begin{figure}[!htbp]
    \begin{subfigure}{.45\textwidth}
        \centering
        \includegraphics[width=.9\linewidth]{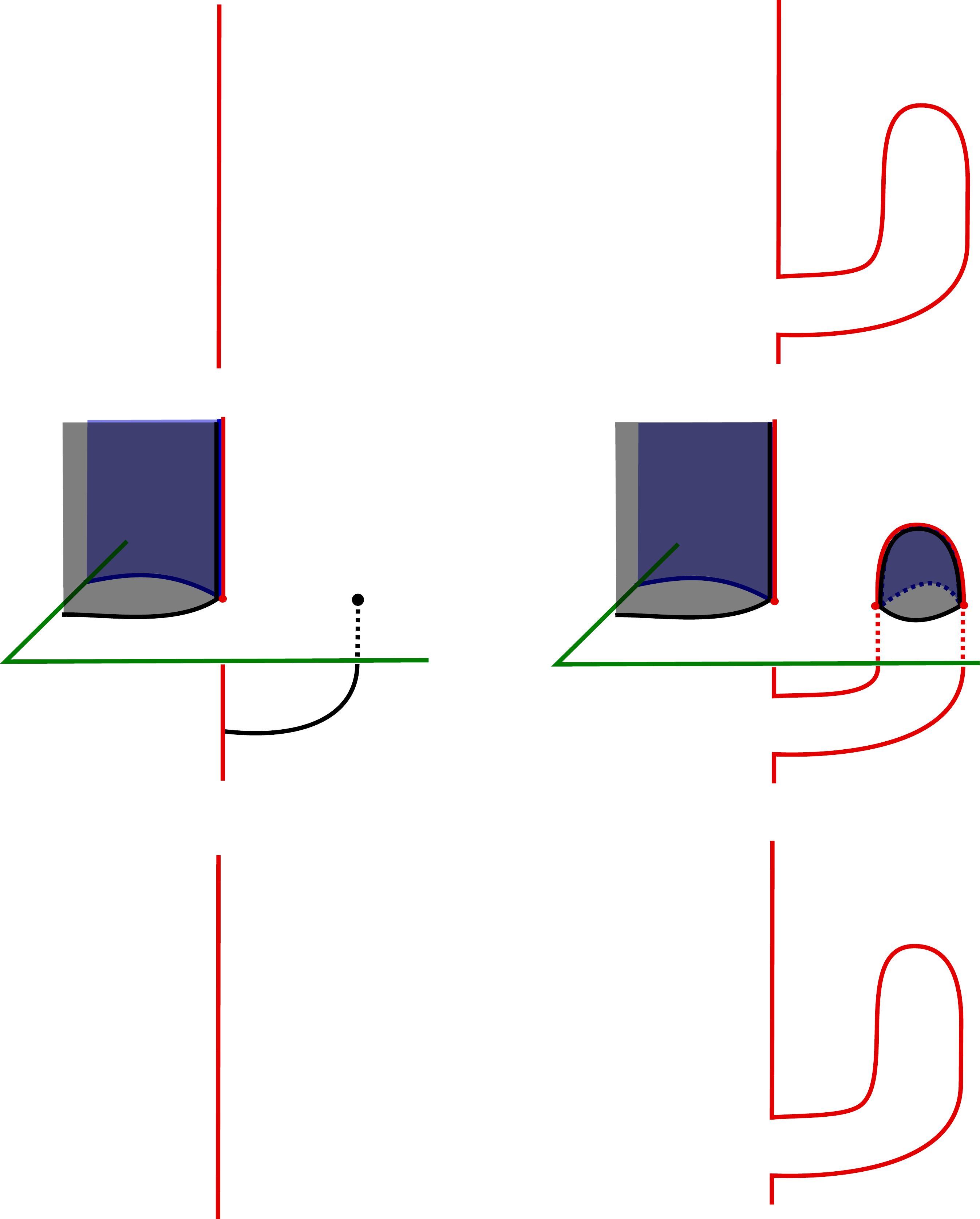}
        \caption{} \label{F: local birth}
    \end{subfigure}
    \begin{subfigure}{.45\textwidth}
        \centering
        \includegraphics[width=.9\linewidth]{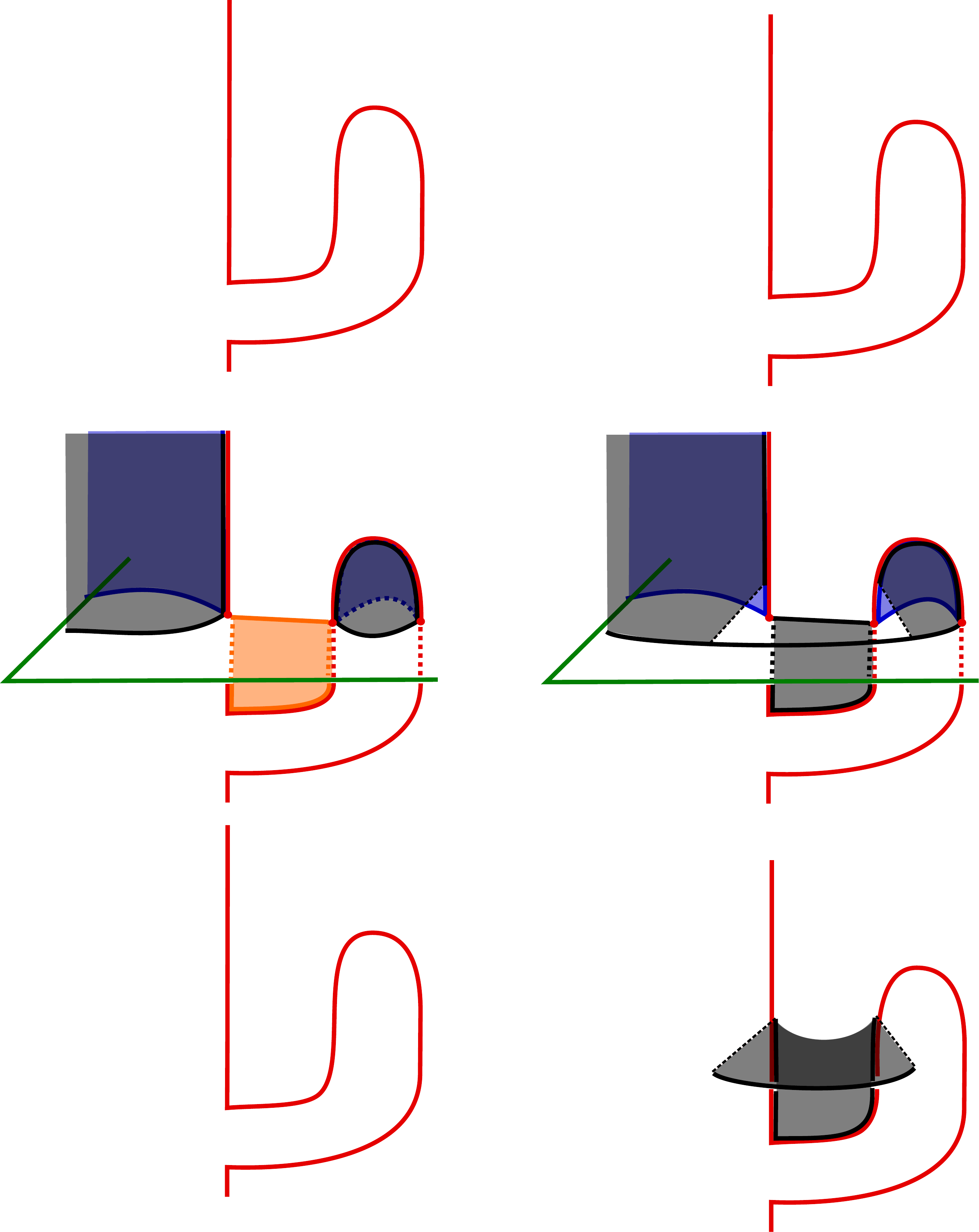}
        \caption{} \label{F: local_SSS}
    \end{subfigure}
    \caption{Shown is the sequence of a local birth move followed by a special sum square move. In both figures, each column of figures represents the local 4-dimensional neighborhood where the deformations occur. Figure (A) on the left shows the birth move creating a pair of finger with discs. In Figure (B) we see the affect of applying a SSS move to the finger discs using the Special Sum Square disc (Definition~\ref{sum square}).}
    \end{figure}
\end{proof}

\begin{lemma}
    If $(\mF', \mW')$ is obtained from $(\mF, \mW)$ by performing a $x^3_{+}$-move, then $\mD(\mF, \mW) = \mD(\mF', \mW')$
\end{lemma}
\begin{proof}
    By Proposition \ref{sum invariance} and Lemmas~\ref{L: inv under cross bd} and \ref{L: inv under uncross bd},  SSS, 1-switch and birth/death moves preserve $\mD$.
\end{proof}

\begin{lemma}\label{L: saddlemovesequence}
    If $(\mF', \mW')$ is obtained from $(\mF,\mW)$ by a saddle move, then it can be obtained from $(\mF,\mW)$ by first applying a local isotopy to the discs in $(\mF, \mW)$, then applying a Birth move, followed by two special sum square moves or a 1-switch and a SSS move or two  1-switches. In each case, one move is applied to the finger discs while the other is applied to the Whitney discs. 
\end{lemma}
\begin{proof}
    The proof of the first case is contained in Figures \ref{F: saddleinv1}, \ref{F: saddleinv2}, \ref{F: saddleinv3}. Every other case can be deduced from these figures by modifying the initial configurations of $(\mF, \mW)$ so that the saddle disc pairs off a free point or the end point of an arc.
    
     Figure \ref{F: saddleinv1} shows the local transformation of the saddle move (Definition \ref{D: saddle}) with the saddle disc shown in orange and the local configurations of the finger discs, shown in black, and the Whitney discs, shown in blue, before the saddle move (left column) and after the move (right column). Figures \ref{F: saddleinv2} and \ref{F: saddleinv3} will show a sequence of local transformations that start at $(\mF, \mW)$ (left column of Figure \ref{F: saddleinv2}) and end at $(\mF', \mW')$ (right column of Figure \ref{F: saddleinv3}). Starting with Figure \ref{F: saddleinv2} and moving from left to right, we see a local isotopy followed by a birth move. The right column of Figure \ref{F: saddleinv2} is the final result of these transformations and shows the first intermediate $\mF\mW$-system $(\mF_1,\mW_1)$. This system is then shown as the left most column in Figure \ref{F: saddleinv3}. Moving again from left to right, there is a sequence of two special sum square moves (Definition~\ref{sum square}). The two special sum square discs used are shown in orange and purple. The orange special sum square disc is used to transform the finger discs adjoint it, while the purple disc is used to transform the Whitney discs. The left column depicts the result of applying both moves, and clearly agrees with the left column of Figure \ref{F: saddleinv1}.
\end{proof}

\begin{figure}
    \centering
    \includegraphics[width=0.8\linewidth]{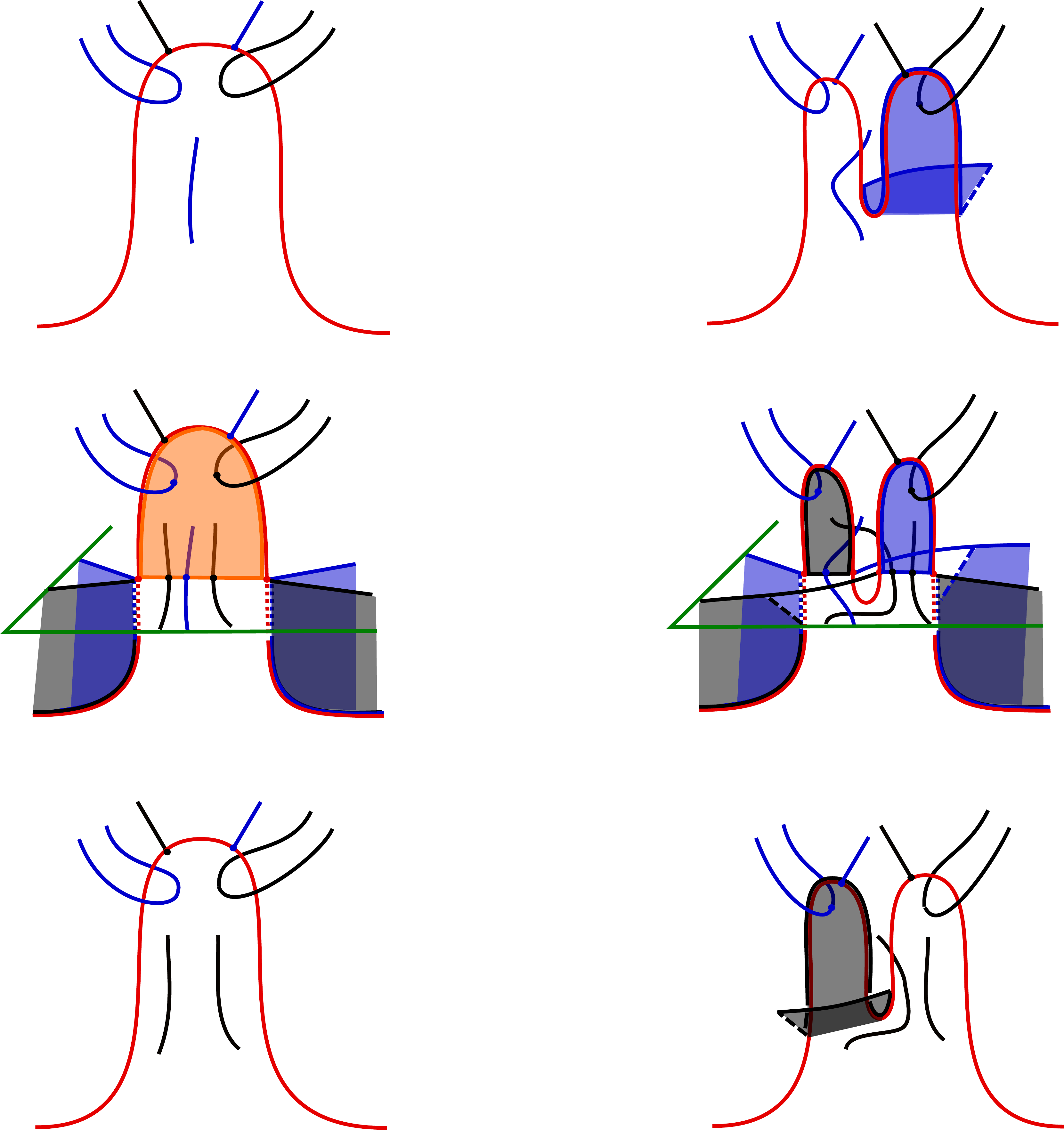}
    \caption{ Shown is the local change to $(\mF, \mW)$ under the saddle move. The left column is the local 4-dimensional neighborhood of the saddle disc $D$ (orange disc) before the saddle move occurs. In black are the local configurations of the discs in $\mF$ and in blue are the local configurations of the discs in $\mW$. The left column shows the new $\mF\mW$-system $(\mF', \mW')$ locally after passing the saddle. In both pictures, we have arranged so that the finger discs with boundaries only on $\mG$ are shown extending into the 4-dimension by moving downward in a column, while the Whitneys extend moving vertically upwards in a column.} \label{F: saddleinv1}
    
\end{figure}

\begin{figure}
    \centering
    
    \includegraphics[width=0.8\linewidth]{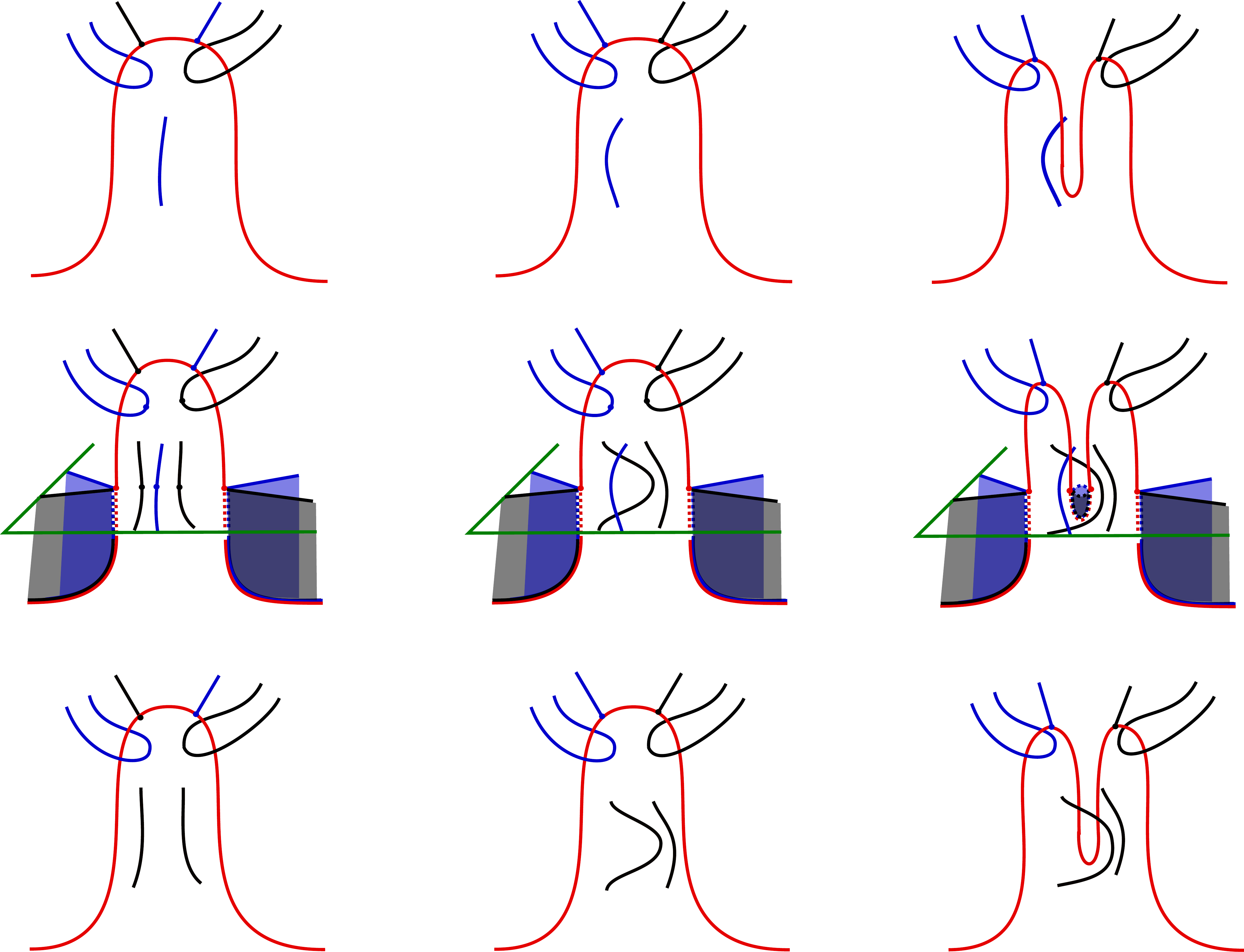}
    \caption{ Shown is a sequence of local changes to $(\mF, \mW)$ to an intermediate $\mF\mW$ system $(\mF_1,\mW_1)$. The system $(\mF_1, \mW_1)$ is obtained from $(\mF, \mW)$ by performing a local isotopy followed by a Birth move. Each column is the local 4-dimensional neighborhood of the saddle disc $D$ shown in Figure \ref{F: saddleinv1}. The left column shows the local configurations of the discs in $(\mF, \mW)$ with the finger discs shown in black and the Whitney discs in blue. The middle column shows the the result of applying the local isotopy to the discs of $(\mF,\mW)$ while the right most column shows $(\mF_1, \mW_1)$ that is obtained by applying a Birth move to the middle column configuration. In all pictures, we have arranged so that the finger discs with boundaries only on $\mG$ are shown extending into the 4-dimension by moving downward in a column, while the Whitneys extend moving vertically upwards in a column.} \label{F: saddleinv2}
    
\end{figure}

\begin{figure}
    \centering
    \includegraphics[width=0.8\linewidth]{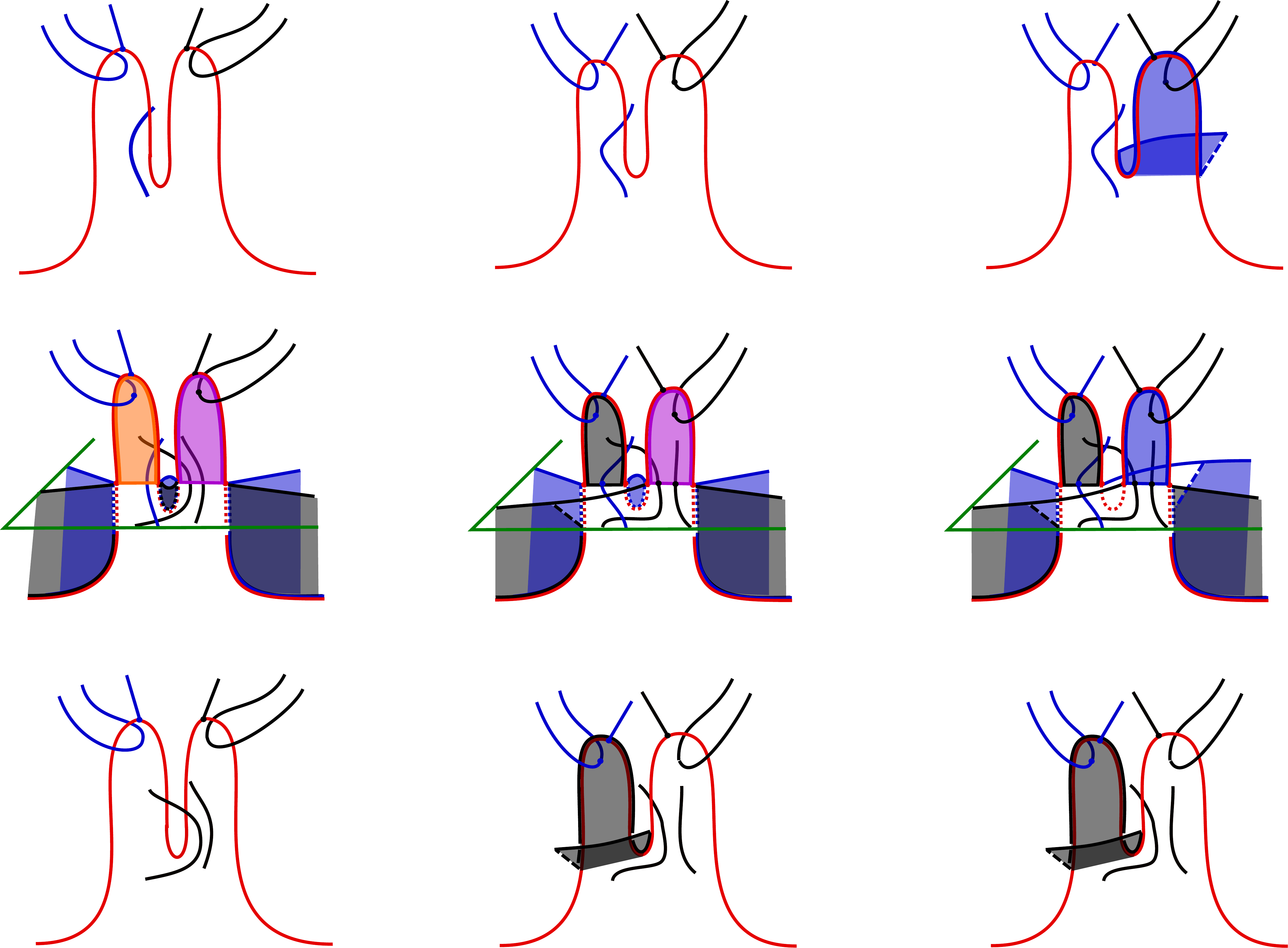}
    \caption{ Shown is a sequence of 2 special sum square moves applied to $(\mF_1, \mW_1)$. Each column is the local 4-dimensional neighborhood of the saddle disc $D$ shown in Figure \ref{F: saddleinv1}. The left column shows the local configuration of the discs in $(\mF_1,\mW_1)$ with two special sum square discs shown, one in orange and one in purple. The middle column is the local configuration of a $\mF\mW$ system $(\mF_2, \mW_2)$ obtained by applying the special sum square move to the two finger discs adjoining the orange SSS disc. The left most column is the system $(\mF', \mW')$ and is obtained from the middle by applying the special sum square move to the two Whitney discs adjoining the purple SSS disc.} \label{F: saddleinv3}
    
\end{figure}

\begin{lemma}\label{saddle invariance} 
    If $(\mF', \mW')$ is obtained from $(\mF, \mW)$ by a saddle move then $\mD(\mF, \mW)=\mD(\mF', \mW')$. 
\end{lemma}

\begin{proof} By Lemma \ref{L: saddlemovesequence}, $(\mF, \mW)$ is obtained from $(\mF,\mW)$ by a sequence of a Birth move and two SSS moves or a SSS and 1-switch or two 1-switches. Since by Proposition \ref{sum invariance}, and Lemmas~\ref{L: inv under cross bd} and~\ref{L: inv under uncross bd} these moves preserve $\mD$, the result follows.  
\end{proof}

\begin{remark} Note that the proof requires independence of order in that we had to choose an ordering on $\mW'$ to obtain $\mW_1$ and an ordering on $\mF_1$ to obtain $\mF_2$.  Also, it required symmetry of $\mD$, since we are doing operations to both the finger system and the Whitney system.  \end{remark}

\begin{proposition}\label{P: ff independence}
    $\mD$ is independent of the finger-first representative.
\end{proposition}
\begin{proof}
    Given two representatives in finger-first position, there exists an ordered homtopy $H_s(t)$ interpolating between the two (Theorem \ref{T: 2-par ordering}). Furthermore, we know that $\Delta_{0}$ differs from $\Delta_1$ by isotopies and $\mF\mW$-moves (Theorem \ref{T: fw system moves}). By the previous series of lemmas in this section, we now know that $\mD$ is preserved by isotopies and the $\mF\mW$-moves. Therefore, $\mD(\Delta_0)=\mD(\Delta_1)$.
\end{proof}
\begin{corollary}
    $\mD$ is well defined on homotopy classes $[\alpha_t]\in \pi_1(\Emb(\sqcup S^2, \#_k S^2\times S^2),\mR^{std})$. 
\end{corollary}
\begin{proof}
    Every class $[\alpha_t]$ has a finger-first representative (\cite[\S4]{Qu} or Construction \ref{Constr. arc-extensions}) and by Proposition \ref{P: ff independence} the choice of representative does not change $\mD$.
\end{proof}

\begin{figure}
    \centering
        \labellist                             
            \hair 10pt
            \pinlabel $R_i$ at 80 0
            \pinlabel $G_j$ at 0 280
            \pinlabel $f_{1,1}$ at 200 460
            \pinlabel $f_{1,2}$ at 350 460
            \pinlabel $f_{1,3}$ at 450 460
            \pinlabel $f_{2,1}$ at 650 460
            \pinlabel $f_{2,2}$ at 750 460
            \pinlabel $f_{n,1}$ at 970 460
            \pinlabel $f_{n,2}$ at 1070 460
            \pinlabel $f_{n,3}$ at 1170 460
            \pinlabel $f_{n,4}$ at 1270 460
            \pinlabel $\cdots$ at 870 100
        \endlabellist
    \includegraphics[width=.8\linewidth]{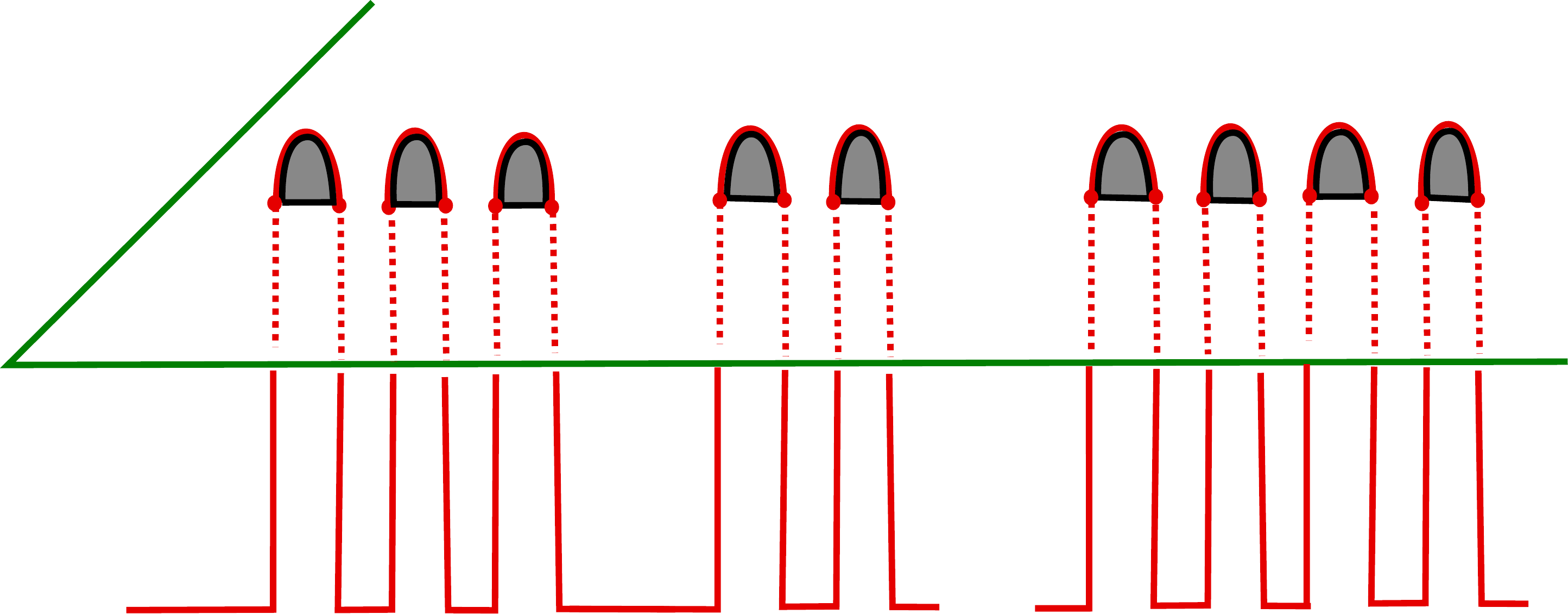}
    \caption{Shown is a local form for the $\mF_{ij}$ finger discs after performing a standardizing isotopy. We have shown the indexing convention where the disc $f_{l,k}$ is the $k$-th finger discs in the $l$-th circle.} \label{F: localstdcrossedfingers}
\end{figure}
\begin{figure}
    \centering
    \includegraphics[width=\linewidth]{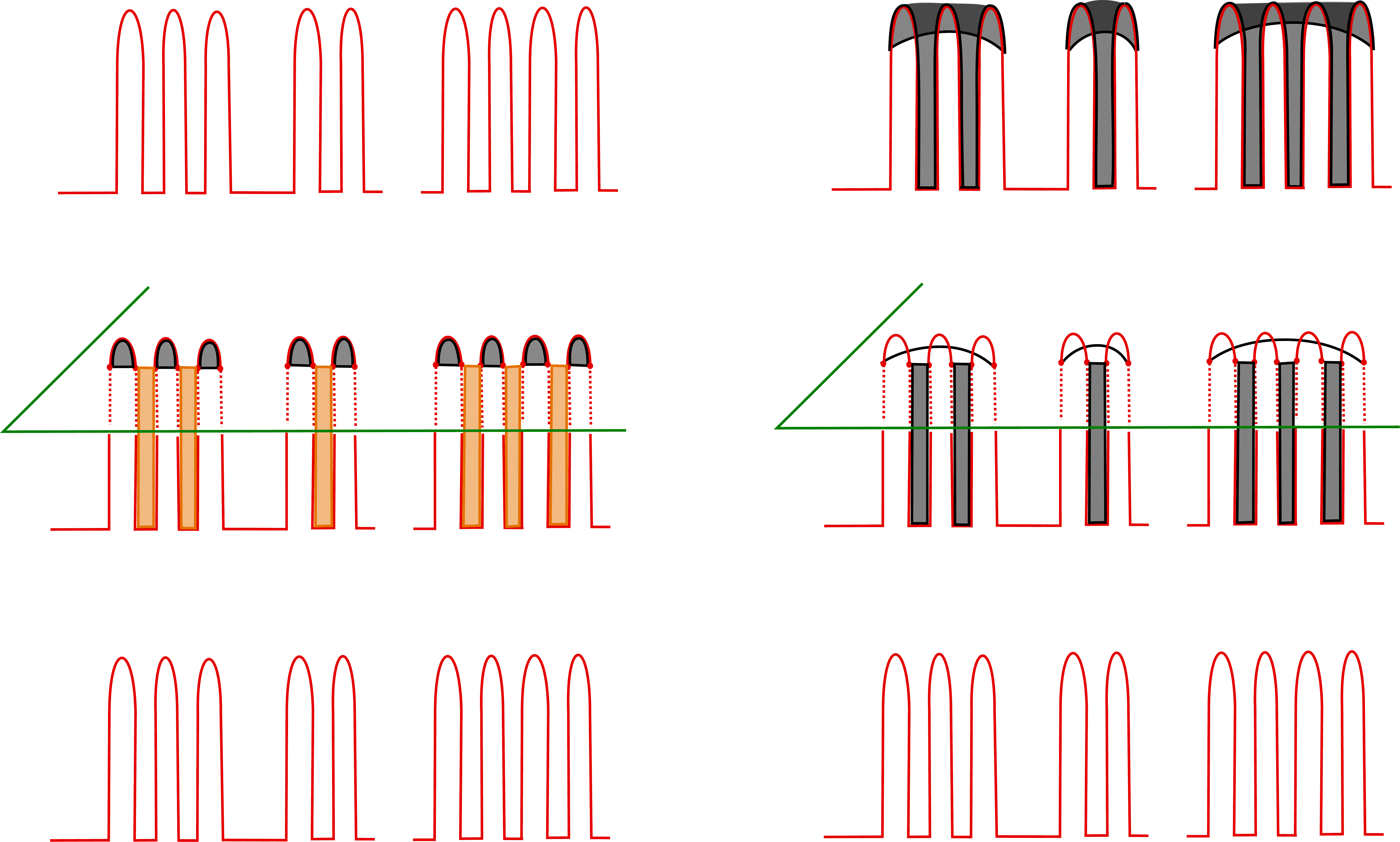}
    \caption{Shown in orange on the left are the local special sum square discs. On the right are the new finger discs obtained by performing the sum square move along the SSS discs.} \label{F: localstdcrossfingerSSSdiscs} 
\end{figure}

\begin{lemma}\label{L: Similarly matched}
    Given a finger-Whitney system $(\mF, \mW)\in EA$, it can be modified by crossed SSS moves to a new system $(\mF', \mW)\in EA$ such that the crossed discs are similarily matched and the uncrossed discs are unchanged.
\end{lemma}
\begin{proof}
    The first step in the proof is to organize the original finger discs appropriately. Then we will construct the SSS-discs locally so that when the sum squared move is applied, the resulting discs $\mF'$ will be similarly matched with the $\mW$.
    
    We will first index the $(i,j)$ crossed discs. Now note that $\mG_{j}\cap(\partial \mF_{i,j} \cup \partial \mW_{i,j})$ forms a collection of immersed circles. We choose an enumeration of the circles  and let $C_n$ denote the nth circle in the list. Next we index the discs so that $f_{n,l}$ denotes the $l$-th finger disc belonging to the $n$-th circle $C_n$. Let $p_{n,l}^+$ and $p_{n,l}^-$ denote the points in $R_i\cap G_j$ paired off by $f_{n,l}$. Then $w_{n,l}$ will denote the Whitney disc that pairs off $p_{n,l}^+$ and $p_{n,l+1}^-$. If $C_n$ has $m_n$ finger discs, then $w_{n,m_n}$ will pair off $p_{n,m_n}^+$ and $p_{n,1}^-$. 
    
    We now apply a standardizing isotopy to $\mR$ that arranges for each $i\neq j$, the crossed finger discs to appear in a local form as shown in Figure \ref{F: localstdcrossedfingers}. In this local form, the finger discs are group by circle, and arranged in increasing order from left to right. This is always possible since $\#^k S^2\times S^2 \setminus (\mR^{std}\cup \mG)$ is simply connected, so we are free to isotope the finger arcs into this local position. Then this isotopy of the finger arcs can be used to generate the standarizing isotopy.  
    
    Having arranged the finger discs as shown in Figure \ref{F: localstdcrossedfingers}, there are a set of SSS discs, shown in the left column of Figure \ref{F: localstdcrossfingerSSSdiscs}, that when applied to the crossed discs results in a new set of discs $\mF'$ that are similarly matched with the the Whitney crossed discs $\mW_c$. By construction, applying the SSS move to the cross discs does not change the uncrossed discs and thus preserves the EA ordering.
\end{proof}

\begin{remark}
    Note that the same argument can be applied to the Whitney discs as well, that is to say, we can apply the SSS move to a set of SSS discs for $\mW_c$ to change $(\mF, \mW)$ to $(\mF, \mW')$, leaving the uncrossed discs unchanged so that $\mF_c$ are similarly matched with $\mW'$.
\end{remark}

\begin{lemma}\label{L: FEA-position.}
    Let $(\mF, \mW)$ be an EA system such that the boundaries of the crossed disc are disjoint from the uncrossed discs and the crossed discs $\mF_c$ are similarly matched with $\mW_c$. Then there is a sequence of crossed clasping operations, disc slides and boundary twists applied to $\mF$ such that the new disc system $(\mF', \mW)\in FEA$ and preserves the uncrossed discs.
\end{lemma}

\begin{figure}
\begin{subfigure}{.5\textwidth}
    \centering
    \labellist
      \pinlabel {$\iddots$} at 420 130
      \pinlabel {$\iddots$} at 680 400
    \endlabellist
    \includegraphics[width=.8\linewidth]{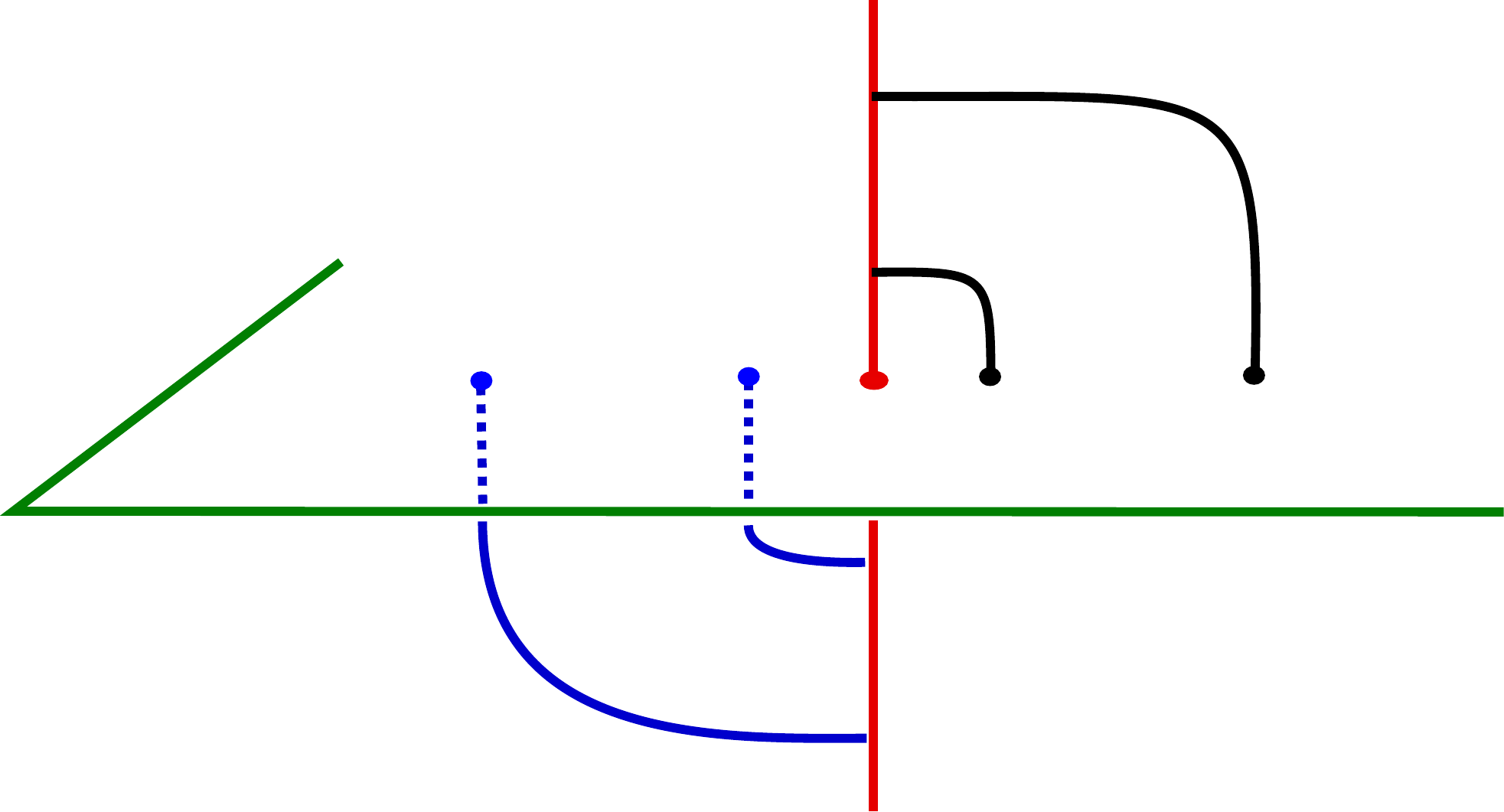}
    \caption{} \label{F: homfwarcconfig}
\end{subfigure}
\begin{subfigure}{.5\textwidth}
    \centering
    \labellist
      \pinlabel {$\cdots$} at 380 300
      \pinlabel {$\cdots$} at 700 220
    \endlabellist
    \includegraphics[width=.8\linewidth]{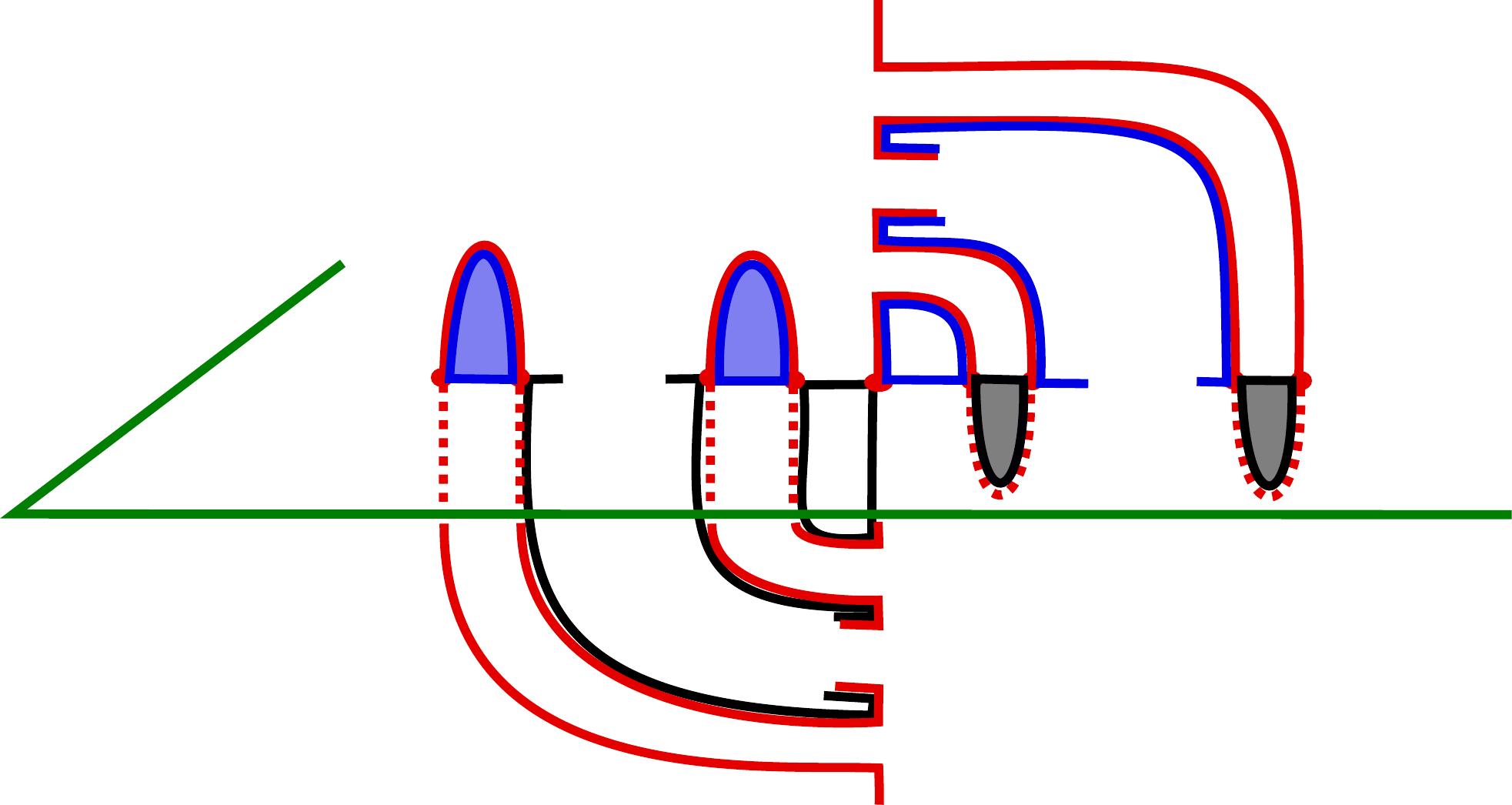}
    \caption{} \label{F: homfwdiscconfig}
\end{subfigure}
\caption{On the left shows the extended arc for the upper finger arcs (in black) and the lower Whitney arcs (in blue). On the right we see the finger/Whitney system after the deformation into finger-first postion.}
\end{figure}

\begin{proof}[Proof of Theorem \ref{T:mainthm}] We will work with the traces of our loops. First, by Proposition \ref{P: EA paths} and Lemma \ref{boundary clearing}, every homotopy class $[\alpha]\in  \pi_1\Emb(\sqcup_{i=1}^k S^2_i, \#^k S^2\times S^2),\mR^{std})$ has a representative in $EA$ where the boundaries of the cross discs are disjoint from the uncross discs (so $CU_{ij} =0$ for all $i,j$). For a given $\alpha$ with $(\mF, \mW) \in \EA$, we can organize the boundaries of the cross and uncross discs, both in the finger/Whitney system $(\mF, \mW)$ and after performing the Whitney moves along the discs in $\mF$ or in $\mW$. We will refer to this as working below the finger points or above the Whitney points, and will ignore the interiors of the discs for now. We note that on $R_{i}$, the complement of the arc of uncrossed discs is contractible. So we can perform an isotopy and concentrate all the cross term boundaries into a contractible region $U^\mR_{i}$. The same holds for $G_i$ as well as for any $i$. 

The same is true for the boundaries of the finger discs above the Whitney points or for the boundaries of the Whitney discs below the finger points. We shall explain this for the finger discs about the Whitney points with the same description holding for the situation below the finger points with the roles of fingers and Whitneys being interchanged. Above the Whitney points, i.e., after performing the Whitney moves along the discs in $\mW$, the boundaries of the uncrossed finger discs $\mF_{ii}$ create the edges of embedded linear graph, one on $R_i$ and the other on $G_i$, with the vertices consisting of the boundaries of the Whitney arcs and the single intersection $R_i\cap G_i$. The boundaries of the finger cross discs and Whitney cross arcs will create a complicated graph that is disjoint from the uncrossed paths and contained in $U^\mR_i$ ($U^\mG_i$ respectively).

Now since all the boundaries of the cross term discs are confined to $U_i^\mR \cup U^\mG_i$, we can perform the clasping operation and cross special sum square moves to the cross discs, keeping their boundaries inside $U^\mR_i\cup U^\mG_i$, and transform the given system to one in FEA. Note that the presence of clasping might change the homotopy class of $\alpha$. Never the less, we have by Corollary~\ref{full reduction} that $\mD(\alpha) = \mD(\alpha')$, where $\alpha'$ is the path in FEA obtained from $\alpha$. 
 
We now prove the homomorphism property. Take two EA paths $\alpha_1$ and $\alpha_2$ and consider the loop $\alpha_1*\alpha_2$ with finger-Whitney systems $(\mR_{1/4}, \mG, \mF_1, \mW_1)$ and $(\mR_{3/4}, \mG, \mF_2, \mW_2)$ respectively. These systems appear at height $t= 1/4$ and $t=3/4$ of the trace $\tr(\alpha_1*\alpha_2)$ respectively. At height $t=1/2$ we have $\mR^{std}$. As explained above, we group the boundaries for the finger cross disc for $\alpha_1$ into their contractible sets $U_{1,i}^\mR$ and $U_{1,i}^\mG$ above the Whitney points and group the boundaries for the Whitney cross discs into their contractible sets below the finger points for $\alpha_2$, extending these down to heights $t=1/2$. We also extend the Whitney arcs for $\alpha_1$ and finger arcs for $\alpha_2$ past $t=1/2$ respectively so that the following are true:
\begin{enumerate}
    \item For each $i$, the uncrossed finger arcs and Whitney disc boundaries for $\alpha_2$ and Whitney arcs and finger disc boundaries for $\alpha_1$ appear as in Figure \ref{F: homfwarcconfig} and
    \item $U^\mR_{i,1} \cap U^\mR_{i,2}=\varnothing$ and $U^\mG_{i,1} \cap U^\mG_{i,2}=\varnothing$. 
\end{enumerate}

Note that both (1) and (2) can be arranged by performing an isotopy to the various data in the level sets $(1/2 - \epsilon, 1/2]$ and $[1/2,1/2+\epsilon)$ for some small $\epsilon >0$. We now use the extended finger and Whitney arcs to move the finger moves for $\alpha_2$ down and the Whitney moves for $\alpha_1$ up in order to combine both systems. By our preparation, we have that for each $i$, the boundaries of the uncross discs will appear as in Figure \ref{F: homfwdiscconfig} while the cross disc systems boundaries will be contained in the disjoint sets $U_{i,l}^\mR$ ($U_{i,l}^\mG$ in $G_i$) for $l = 1,2$.

Now we can compute $\mD(\alpha_1*\alpha_2)$. To compute $\I (\alpha_1*\alpha_2)$, note that we have arranged for the system to be in $EA$ already. Moreover, the EA ordering of the discs is such that $(\mF_2,\mW_2)$ forms the first set of discs and $(\mF_1, \mW_1)$ forms the last set. Since $\mF_2$ was above $\mW_1$ originally, the interior of the disks in $\mF_2$ will be disjoint from the discs in $\mW_1$. There might be intersections between $\mF_1$ and $\mW_2$, but since the finger discs for $\mF_1$ will occur later in the EA ordering, those intersections do not contribute to $\I$. Therefore when computing $\I_i(\mF_1\cup \mF_2,\mW_1\cup \mW_2)$ for each $i,$ it follows from the definition that

\begin{equation}\label{Eq: I sum}
    \I_i(\mF_1\cup \mF_2,\mW_1\cup \mW_2) = \I_i(\mF_1, \mW_1)+\I_i(\mF_2, \mW_2).
\end{equation}
Since this holds for every $i$, It follows that $\I(\alpha_1*\alpha_2) = \I(\alpha_1)+\I(\alpha_2)$. 

For $C$, we again note that the Whitney discs for $\alpha_1$ will be disjoint from the finger discs for $\alpha_2$. Therefore the only additional intersections that may occur are between the fingers discs for $\alpha_1$ and the Whitney discs for $\alpha_2$. Moreover, these intersections will be all interior intersections.

In order to compute $\mC$ cleanly, we will convert the combined system $(\mF_1\cup \mF_2, \mW_1\cup \mW_2)$ to a system in FEA. As we argued above, each system of discs $(\mF_l, \mW_l)$, $l = 1,2$, can be converted to an FEA system using the clasping operation, disc slides, boundary twistsing, and cross special sum square moves with the defining information confined to the sets $U_{i,l}^\mR$ and $U_{i,l}^\mG$ before combining the systems. Since this can be done to either $\mF_l$ or $\mW_l$ (Lemma \ref{L: FEA-position.}, Lemma \ref{L: Similarly matched}) we will apply the transformations to $\mF_1$ and $\mW_2$, keeping $\mW_1$ and $\mF_2$ fixed. Performing the operations in this way preserves the way the two systems are combined since $\alpha_1([1/2, 1])$ and $\alpha_2([0,1/2])$ are not affected by the above operations.\footnote{Here we refer to the paths before concatenating.} We have then argued that transforming $(\mF_1, \mW_1)$ to $(\mF'_1, \mW_1)$ and $(\mF_2, \mW_2)$ to $(\mF_2, \mW'_2)$ is equivalent to transforming $(\mF_1\cup \mF_2,\mW_1\cup \mW_2)$ to $(\mF'_1\cup \mF_2,\mW_1\cup \mW'_2)$ by a sequence of clasping, cross-SSS moves, disc slides, and boundary twisting operations. Furthermore, by Corollary~\ref{full reduction}, we have that $\mD(\mF_1\cup \mF_2,\mW_1\cup \mW_2)=\mD(\mF'_1\cup \mF_2,\mW_1\cup \mW'_2)$. Since the transformations left the uncrossed discs unchanged, it follows that $\I(\mF_1\cup \mF_2,\mW_1\cup \mW_2)=\I(\mF'_1\cup \mF_2,\mW_1\cup \mW'_2)$. Thus it must be the case that $\mC(\mF_1\cup \mF_2,\mW_1\cup \mW_2)=\mC(\mF'_1\cup \mF_2,\mW_1\cup \mW'_2)$. Repeating this argument for each $\alpha_i$ before combining, it follows that $\I(\alpha_i) = \I(\alpha'_i)$ and $\mD(\alpha_i) = \mD(\alpha'_i)$ and therefore $\mC(\alpha_i) = \mC(\alpha'_i)$. 

We now show that $\mC(\mF'_1\cup \mF_2,\mW_1\cup \mW'_2) = \mC(\mF'_1,\mW_1) + \mC(\mF_2,\mW'_2)$. This will then complete the proof. First, we recall that $\mF_2$ and $\mW_1$ are completly disjoint. Furthermore, since $\partial (\mF'_1\cup \mF_2)_c = \partial (\mW_1\cup \mW'_2)_c$, each $f_{i,j}\cup w_{i,j}$ forms an immersed 2-sphere in the complement of $\mR \cup \mG$. Therefore $[f^l_{ij, k}\cup w^l_{ij,k}] = 0 \in H_2(\#^k S^2\times S^2)$ for $i\neq j$ and $l= 1,2$. Here the indices denote the $k$th disc in $\mF_l$ or $\mW_l$ between $R_i$ and $G_j$. Now, the only intersections which might change $\mC$ happen between $f_{ij,k}^1$ and $w_{ji, k}^2$ for $i\neq j$. Since $[f^l_{ij, k}\cup w^l_{ij,k}] = 0$ we must have the algebraic intersection $[f^1_{ij, k}\cup w^1_{ij,k}] \cdot [f^2_{ji, k'}\cup w^2_{ji,k'}] = 0$. And since $w^1_{ij,k}\cap f^2_{ji,k'}=\varnothing$, it must be the case that $|f^1_{ij,k}\cap w^2_{ji,k'}| = 0 \mod 2$. Since this holds always, it then follows that 
\[
\mC(\mF'_1\cup \mF_2,\mW_1\cup \mW'_2) = \mC(\mF'_1,\mW_1) + \mC(\mF_2,\mW'_2).
\]
\end{proof}

\printbibliography[heading=bibintoc]

\end{document}